\documentclass[10pt,leqno]{article}

\usepackage{lmodern}
\usepackage[french]{babel}
\usepackage{amsmath,amssymb,amscd}
\usepackage{a4wide}
\usepackage[utf8]{inputenc}

\usepackage[all]{xy} 
\usepackage{fancyhdr}
\usepackage[T1]{fontenc}
\usepackage{hyperref}
\usepackage{chngcntr}
\usepackage{comment}

\newenvironment{paragr}[1][]{\refstepcounter{subsubsection} \noindent \textbf{\thesubsubsection . \ #1}}{\medskip}

\newenvironment{theoreme}{ \medskip\refstepcounter{theo}  \noindent\textbf{Th\'eor\`eme \thetheo}. ---\em}{\em \medskip}
\newenvironment{proposition}{\medskip\refstepcounter{theo}   \noindent\textbf{Proposition \thetheo}. ---\em}{\em\medskip}
\newenvironment{corollaire}{\medskip\refstepcounter{theo}  \noindent\textbf{Corollaire \thetheo}. ---\em}{\em\medskip}

\newenvironment{lemme}{\medskip\refstepcounter{theo}   \noindent\textbf{Lemme \thetheo}. ---\em}{\em\medskip}

\newenvironment{preuve}[1][]{\noindent \textbf{Démonstration.} #1 --- }{\hfill
  \ensuremath{\square} \medskip}

\newenvironment{remarque}{\medskip\refstepcounter{theo}  \noindent\textbf{Remarque \thetheo}. ---}{\medskip}

\DeclareMathOperator{\rs}{rss}

\DeclareMathOperator{\vol}{vol}

\DeclareMathOperator{\Ad}{Ad}

\DeclareMathOperator{\End}{End}
\DeclareMathOperator{\ad}{ad}

\DeclareMathOperator{\supp}{supp}

\DeclareMathOperator{\Cent}{Cent}

\DeclareMathOperator{\Aut}{Aut}
\DeclareMathOperator{\Gal}{Gal}
\DeclareMathOperator{\Hom}{Hom}

\DeclareMathOperator{\Int}{Int}

\DeclareMathOperator{\Id}{Id}

\DeclareMathOperator{\Res}{Res}
\DeclareMathOperator{\cusp}{cusp}

\DeclareMathOperator{\Spec}{Spec}

\DeclareMathOperator{\trace}{trace}

\DeclareMathOperator{\Prd}{Prd}
\DeclareMathOperator{\Nrd}{Nrd}
\DeclareMathOperator{\Trd}{Trd}
\DeclareMathOperator{\unip}{unip}
\DeclareMathOperator{\nilp}{nilp}

\DeclareMathOperator{\diag}{diag}
\newcommand{\ZZ}{\mathbb{Z}}

\newcommand{\GmD}{\mathbb{G}_{m,D}}
\newcommand{\SG}{\mathfrak{S}}
\newcommand{\NN}{\mathbb{N}}
\newcommand{\RR}{\mathbb{R}}
\newcommand{\AAA}{\mathbb{A}}
\newcommand{\CC}{\mathbb{C}}

\newcommand{\QQ}{\mathbb{Q}}

\newcommand{\ga}{\gamma}

\newcommand{\oc}{\mathcal{O}}

\newcommand{\Sc}{\mathcal{S}}

\newcommand{\uc}{\mathcal{U}}
\newcommand{\vc}{\mathcal{V}}

\newcommand{\yc}{\mathcal{Y}}
\newcommand{\lc}{\mathcal{L}}

\newcommand{\fc}{\mathcal{F}}
\newcommand{\pc}{\mathcal{P}}

\newcommand{\nc}{\mathcal{N}}

\newcommand{\bc}{\mathcal{B}}

\newcommand{\PP}{\mathbf{P}}

\newcommand{\gl}{\mathfrak{gl}}
\newcommand{\ggo}{\mathfrak{g}}

\newcommand{\of}{\mathfrak{o}}

\newcommand{\mgo}{\mathfrak{m}}
\newcommand{\ngo}{\mathfrak{n}}
\newcommand{\ago}{\mathfrak{a}}
\newcommand{\cgo}{\mathfrak{c}}
\newcommand{\pgo}{\mathfrak{p}}

\newcommand{\sgo}{\mathfrak{s}}

\newcommand{\hgo}{\mathfrak{h}}

\newcommand{\vgo}{\mathfrak{v}}

\newcommand{\Xgo}{\mathfrak{X}}
\newcommand{\Fgo}{\mathfrak{F}}

\newcommand{\ov}{\overline}
\newcommand{\wt}{\widetilde}

\newcommand{\al}{\alpha}
\newcommand{\be}{\beta}

\newcommand{\om}{\omega}
\newcommand{\Om}{\Omega}

\newcommand{\La}{\Lambda}
\newcommand{\la}{\lambda}

\newcommand{\Ga}{\Gamma}
\newcommand{\back}{\backslash}

\newcommand{\Cc}{C_c^\infty}

\newcommand{\bg}{\langle}
\newcommand{\bd}{\rangle}

\newcommand{\eps}{\varepsilon}

\renewcommand{\leq}{\leqslant}
\renewcommand{\geq}{\geqslant}

\newcommand{\tlK}{\widetilde{K}}

\title{Une formule des traces pour les espaces symétriques.\\ Le cas de  Guo-Jacquet.}
\author{Pierre-Henri Chaudouard et Huajie Li}
\date{ }

\begin{document}
\counterwithin{equation}{subsubsection}

\maketitle

\begin{abstract}
Nous établissons, sur le modèle de la formule des traces d'Arthur, une formule des traces générale pour les espaces symétriques associés à la variété des involutions d'un module de type fini sur une algèbre à division centrale sur un corps de nombres $F$. Une telle formule devrait être utile pour étudier le spectre automorphe de ces espaces symétriques et les liens profonds entre périodes linéaires et valeurs spéciales de fonctions $L$ standard en leur centre de symétrie. De fait, notre formule donne une identité entre des distributions spectrales, qui généralisent les caractères relatifs construits sur les périodes linéaires, et des distributions géométriques, qui sont une extension des intégrales orbitales relatives. Nous montrons que les distributions spectrales sont, en un certain sens, asymptotiques à des intégrales  tronquées des composantes du noyau automorphe associées à une donnée cuspidale: ceci donne une prise sur ces distributions et a permis, dans un article annexe, d'exprimer certaines de ces distributions sous forme d'un caractère relatif pondéré.
Les distributions géométriques attachées à des données géométriques \og semi-simples régulières \fg{} s'expriment sous forme d'intégrales orbitales relatives pondérées. En général, pour une donnée géométrique non régulière, on introduit une procédure de descente au centralisateur ce qui permet d'exprimer toute distribution géométrique à la contribution nilpotente de formules des traces infinitésimales étudiées dans des articles antérieurs. 
\end{abstract}

{\selectlanguage{english}
\begin{abstract}
In the spirit of Arthur's trace formula, we establish a general trace formula for symmetric spaces associated with the variety of involutions of a finite  $D$-module where $D$ is a division algebra central over a number field $F$. Such a formula should be useful for studying the automorphic spectrum of these symmetric spaces and the deep links between linear periods and special values of standard $L$-functions at their center of symmetry. Indeed, our formula yields an identity between spectral distributions, which generalize relative characters built on linear periods, and geometric distributions, which are an extension of relative orbital integrals. We show that the spectral distributions are, in a certain sense, asymptotic to truncated integrals of the components of the automorphic kernel associated with a cuspidal datum: this provides a handle on these distributions and has allowed, in a companion paper, to express some of these distributions in the form of a weighted relative character.
The geometric distributions attached to "regular semi-simple" geometric data are expressed as  weighted relative orbital integrals. In general, for non-regular geometric data, we introduce a procedure of descent to the centralizer, which allows us to express any geometric distribution in terms of the nilpotent contribution of infinitesimal trace formulas studied in previous papers.
\end{abstract}
}

\tableofcontents

\section{Introduction}

\subsection{Espace symétrique et formule des traces relative}

\begin{paragr}\label{S-intro:D} Soit $F$ un corps de nombres et $\AAA$ l'anneau des adèles de $F$.   Soit $n\geq 1$ un entier et  $D$ une algèbre à division centrale sur $F$ d'indice $d$, c'est-à-dire $\dim_F(D)=d^2$.  Soit $G$ le groupe des automorphismes du $D$-module à droite $D^n$. C'est donc un  groupe algébrique affine, réductif, connexe et défini sur $F$. Soit $Z_G$ son centre. Soit
  \begin{align}\label{eq-intro:S}
    S=\{g \in G\mid g^2=1\}.
  \end{align}
  C'est une sous-variété fermée de $G$ qui est munie de l'action du groupe $G$ par conjugaison. La variété $S$ n'est ni connexe ni même, si $n>1$,  de dimension pure. Une  composante connexe de $S$ est indexée par un  couple d'entiers naturels $(p,q)$ tels que $p+q=n$ où $p$ et $q$ sont les dimensions des espaces propres des éléments de $S$ respectivement pour les valeurs propres $1$ et $-1$.
  Chaque composante connexe est  un espace homogène sous l'action de $G$ et possède un  point rationnel. Soit $\Xi$ un système de représentants dans $S(F)$ des classes de conjugaison de $G(F)$. L'application,  qui, à $\theta\in \Xi$, associe $S_\theta$ l'orbite de $\theta$ sous l'action de $G$, induit une bijection de $\Xi$ sur   l'ensemble (fini) des composantes connexes de $S$. On note $G^\theta$ le centralisateur de $\theta$ dans $G$.
\end{paragr}

\begin{paragr}     Soit $\Sc(S(\AAA))$, resp.  $\Sc(G(\AAA))$, l'espace des  fonctions de  Schwartz sur $S(\AAA)$, resp. $G(\AAA)$. Pour tout $\theta\in \Xi$,  on fixe $dh$ une mesure de Haar sur $G^\theta(\AAA)$. Pour tout $\Phi\in \Sc(S(\AAA)) $,  on fixe  une fonction $\Phi_\theta\in \Sc(G(\AAA))$ telle que
  \begin{align}\label{eq:intro-fxi}
    \forall g\in G(\AAA),    \   \Phi(g\theta g^{-1})= \int_{G^\theta(\AAA)} \Phi_\theta(gh)\,dh .
  \end{align}
\end{paragr}

\begin{paragr}[Périodes.] ---  Soit $(\pi,V_\pi)$ une représentation automorphe cuspidale  de $G(\AAA)$ de caractère central trivial sur $Z_G(\AAA)$.  Soit $\theta\in \Xi$. On définit la  $G^\theta$-période $\pc_{G^\theta}$  comme la forme linéaire sur $V_\pi$ donnée par 
 \begin{align}\label{eq-intro:H per}
    \phi\in V_\pi \mapsto \pc_{G^\theta}(\phi)=\int_{Z_G(\AAA)G^\theta(F)\back G^\theta(\AAA)  }  \phi(h)\, dh.
  \end{align}
  La fonction $\varphi$ étant cuspidale, l'intégrande est à décroissance rapide et  l'intégrale ci-dessus est bien définie.  Nous dirons que la représentation $\pi$ est $G^\theta$-distinguée si la forme linéaire $ \pc_{G^\theta}$ est non nulle sur $V_\pi$.   
\end{paragr}

\begin{paragr}   Pour tout $\Phi\in \Sc(S(\AAA)) $, on définit  une fonction $K_\Phi$ sur $G(F)Z_G(\AAA)\back G(\AAA)$ par
  \begin{align*}
    K_\Phi(g)&=\sum_{\theta \in S(F)} \Phi(g^{-1} \theta g)\\
    &=\sum_{\theta\in \Xi} \int_{G^\theta(F)\back G^\theta(\AAA)  } K_{\Phi_\theta}(g,h)\, dh
  \end{align*}
  où l'on introduit, pour toute fonction $f\in \Sc(G(\AAA))$ et  tous $x,y\in G(\AAA)$, le \og noyau automorphe \fg{}
  \begin{align}\label{eq-intro:noyau autom}
     K_{f}(x,y)=\sum_{\gamma\in G(F)} f(x^{-1}\ga y).
  \end{align}

  Le groupe $G(\AAA)$ agit par translation à droite sur l'espace engendré par les fonctions $K_\Phi$ lorsque $\Phi$ varie dans $\Sc(S(\AAA))$. Jacquet a soulevé le problème de décomposer spectralement  cet espace, cf.  \cite{Jac-Edin} par exemple.  Le  calcul suivant, qui vaut pour tout $\phi\in V_\pi$
  \begin{align*}
    \int_{Z_G (\AAA)G(F)\back G(\AAA)    } K_\Phi(g) \overline{\phi(g)}\, dg=\sum_{\theta\in \Xi}\int_{  G^\theta(\AAA)  \back G(\AAA)   } \Phi(g^{-1} \theta g) \overline{\int_{Z_G(\AAA)G^\theta(F)\back G^\theta(\AAA)  }  \phi(hg)\, dh} \, dg,
  \end{align*}
  montre que   la composante cuspidale de $K_\Phi$ est formée des représentations automorphes cuspidales    qui sont distinguées pour au moins un des groupes  $G^\theta$ avec  $\theta\in \Xi$.
\end{paragr}

\begin{paragr}[Formule des traces relatives.] --- La formule des traces relative devrait fournir un moyen commode d'exploiter l'information spectrale contenue dans les fonctions $K_\Phi$. Il y a plusieurs manières de présenter cette formule hypothétique. Si l'on suit  Sakellaridis dans  \cite[section 1.2.2]{Sak-Vietnam},  il s'agit d'étudier  le produit  hermitien
  \begin{align*}
 (\Phi,\Phi')\in \Sc(S(\AAA))^2 \mapsto  \bg K_\Phi, K_{\Phi'}  \bd= \int_{ G(F)Z_G(\AAA)\back G(\AAA)   } \overline{K_\Phi(g)} K_{\Phi'}(g)\, dg
  \end{align*}
qui,  en général,  ne converge pas. Néanmoins, en ignorant les problèmes de convergence et   en introduisant la fonction  $\Phi_\theta^*\in \Sc(G(\AAA))$ définie par $\Phi_\theta^*(g)=\overline{\Phi_\theta(g^{-1})}$, on peut écrire
  \begin{align*}
    \bg K_\Phi, K_{\Phi'}  \bd&= \sum_{\theta,\theta'\in \Xi}  \int_{Z_G (\AAA)G(F)\back G(\AAA)    }   \int_{ G^\theta(F) \back G^\theta(\AAA)   }  K_{\Phi_\theta^*}(h,g) \, dh   \int_{ G^{\theta'}(F)\back G^{\theta'}(\AAA)   } K_{\Phi_{\theta'}}(g, h ')\, dh' dg\\
    &=  \sum_{\theta,\theta'\in \Xi}   \int_{ Z_G (\AAA) G^\theta(F) \back G^\theta(\AAA)   }        \int_{ G^{\theta'}(F)\back G^{\theta'}(\AAA)   }     \int_{G(\AAA)    } \Phi_\theta^*(h^{-1}g) K_{\Phi_{\theta'}}(g, h ') \, dg dh'dh\\
   & =  \sum_{\theta,\theta'\in \Xi}   \int_{ Z_G (\AAA) G^\theta(F) \back G^\theta(\AAA)   }        \int_{ G^{\theta'}(F)\back G^{\theta'}(\AAA)   }   K_{\Phi_{\theta,\theta'}   } (h,h') \, dh'dh
  \end{align*}
  où $\Phi_{\theta,\theta'}$ désigne le produit de convolution $\Phi_\theta^**\Phi_{\theta'}$.

On est donc amené à étudier l'intégrale \emph{a priori} divergente
\begin{align}\label{eq-intro:integ noyau}
   \int_{ Z_G (\AAA) G^\theta(F) \back G^\theta(\AAA)   }        \int_{ G^{\theta'}(F)\back G^{\theta'}(\AAA)   }   K_{f  } (h,h') \, dh'dh
\end{align}
pour une fonction $f\in \Sc(G(\AAA))$ arbitraire, ce qui est plutôt le point de vue avancé par Jacquet dans \cite{Jac-Edin}. Le but de  cet article, est de construire une variante modifiée du noyau automorphe   $K_{f  } (h,h')$, qui dépend d'un paramètre auxiliaire,  pour laquelle l'intégrale ci-dessus converge absolument. Il est alors possible d'obtenir,  sur le modèle de la formule des traces d'Arthur  cf. \cite{Ar-cours}, une identité remarquable entre deux  sommes de distributions sur $\Sc(G(\AAA))$ de nature différente:   d'une part, des distributions  spectrales qui, en première approximation, sont des caractères \og relatifs \fg{} construits à l'aide des périodes $\pc_{G^\theta}$ et $\pc_{G^{\theta'}}$  et, d'autre part, des distributions géométriques, qui généralisent les intégrales orbitales \og relatives \fg{} associées aux doubles classes dans  $G^{\theta}(F)\back G(F)/ G^{\theta'}(F)$. Précisons toutefois que, dès qu'on prend $n>2$,  la formule que nous obtenons, si elle est invariante pour l'action à droite de $G^{\theta'}(\AAA)$ sur $\Sc(G(\AAA))$,  est non-invariante pour celle à gauche de  $G^{\theta}(\AAA)$. On peut aussi regarder notre formule comme une distribution sur $\Sc(S_{\theta'}(\AAA))$ : dans ce cas, la formule n'est pas invariante pour l'action par conjugaison de $G^\theta$ sur  $\Sc(S_{\theta'}(\AAA))$.  Autrement dit, la formule des traces pour les espaces symétriques (en tout cas ceux que nous considérons) semble bien plus proche de la formule des traces d'Arthur que les formules des traces pour certaines variétés qui apparaissent dans les conjectures de Gan-Gross-Prasad: rappelons que, dans son travail \cite{Z3}, Zydor obtient  des distributions invariantes. En pratique, les principaux termes géométriques et spectraux de notre formule sont respectivement des intégrales orbitales relatives pondérées et des caractères relatifs pondérés.
\end{paragr}

\begin{paragr} Avant de rentrer plus en détail dans les résultats, nous voudrions donner quelques motivations pour l'étude de ces formules des traces relatives. Revenons à la question fondamentale de la décomposition spectrale automorphe des espaces symétriques $G/G^\theta$, essentiellement celle de la décomposition spectrale de l'espace engendré par les fonctions $K_\Phi$.  Suite aux travaux, entre autres, de Nadler \cite{Nadler}, Sakerallidis \cite{Sak-spherical}, Sakellaridis-Venkatesh  \cite{SaVen} et Takeda \cite{takeda2023}, on s'attend à ce que le spectre automorphe de $G/G^\theta$  soit contrôlé par un certain groupe dual au travers duquel se factorisent les paramètres de Langlands  (bien sûr conjecturaux à l'heure actuelle) des représentations automorphes de $G(\AAA)$ qui apparaissent dans le spectre automorphe de $G/G^\theta$. Pour un peu plus de précisions, on renvoie le lecteur à  la discussion dans \cite[section 1.1]{caractererelatifpondere}.

  Supposons provisoirement  $D=F$,  de  sorte que $G$ est le groupe général linéaire $GL(n)$ sur $F$. Les conjectures ci-dessus devraient alors admettre  une formulation concrète. Illustrons cela  en écartant le cas trivial $n=1$ pour lequel l'espace symétrique est réduit à un point. Pour $n\geq 2$,  Friedberg-Jacquet dans   \cite{FriJa} observent  que  le spectre automorphe cuspidal de $G/G^\theta$ est vide sauf si $n$ est pair et $\theta$ est conjugué à la matrice
  \begin{align*}
  \theta_0=  \begin{pmatrix}
      I_{n/2} & 0 \\ 0 & -I_{n/2}
    \end{pmatrix}.
  \end{align*}
  De plus, dans cas, Friedberg-Jacquet  montrent qu'une représentation automorphe cuspidale $\pi$ de $G(\AAA)=GL(n,\AAA)$ apparaît dans le spectre de $G/G^{\theta_0}$ si et seulement si la représentation $\pi$ est de type symplectique et vérifie  $L(1/2,\pi)\not=0$ où $L(s,\pi)$ est la fonction $L$ standard. La condition d'être de type  symplectique se traduit par le fait que la fonction $L(s,\pi,\La^2)$ de carré extérieur a un pôle en $s=1$. Bien sûr, l'étude ne s'arrête pas au spectre cuspidal.  Considérons, par exemple, la situation étudiée dans  \cite{caractererelatifpondere}, pour laquelle  $n=2p+1$ avec $p\geq 1$ et
  \begin{align*}
  \theta=  \begin{pmatrix}
      I_{p+1} & 0 \\ 0 & -I_{p}
    \end{pmatrix}.                         
  \end{align*}
  Soit $\sigma$ la représentation de $G(\AAA)=GL(2p+1,\AAA)$ induite à partir de la représentation $1\boxtimes \pi $ du sous-groupe de Levi  $GL(1,\AAA)\times GL(2p,\AAA)$ où $1$ est la représentation triviale de $GL(1,\AAA)$ et $\pi$ est une représentation automorphe cuspidale de $GL(2p,\AAA)$. Alors, d'après \cite{caractererelatifpondere}, la représentation $\sigma$ appartient au spectre automorphe  de $G/G^\theta$ si et seulement si $\pi$ est de type symplectique.

  Revenons maintenant au cas d'une algèbre à division $D$ centrale sur $F$ quelconque d'indice $d$. On dispose de la correspondance globale de Jacquet-Langlands entre les représentations automorphes de $G(\AAA)$ et celles de $G^*(\AAA)$ avec  $G^*=GL(nd)$. Elle a été établie en toute généralité par Badulescu et Renard, cf. \cite{Badu,BaduR}. Cette correspondance peut servir de substitut aux paramètres globaux conjecturaux de Langlands.  Il est  alors tentant de prédire des liens entre la présence d'une représentation automorphe $\pi$ de $G(\AAA)$ dans le spectre automorphe de la variété $S$  et l'apparition de son transfert de Jacquet-Langlands dans celui de  $S^*$, la variété analogue attachée à $G^*$. Pour des résultats concernant le cas $n$ pair et la composante associée à l'élément $\theta_0$ ci-desssus, on renvoie le lecteur au travail récent de  Matringe-Offen-Yang  \cite[théorème 1.3]{matringe2025}. Une  façon d'aborder la question est d'utiliser une comparaison entre la  formule des traces relative pour $G$ et celle pour $G^*$ (comme, par exemple,  suggéré dans \cite{Zha2}). À l'instar de la correspondance de Jacquet-Langlands qui s'incarne dans des identités locales entre les caractères des représentations des groupes $G(\AAA)$ et $G^*(\AAA)$, identités   qui sont duales d'identités entre intégrales orbitales locales (dans une perspective plus large, à l'image de  la  théorie de l'endoscopie pour la formule des traces d'Arthur), il devrait exister  des identités entre intégrales orbitales relatives  duales d'identités  entre caractères relatifs.  Cette approche devrait donner également des factorisations explicites des périodes, dans l'esprit des formules d'Ichino-Ikeda,  en terme de fonctionnelles locales et de valeurs spéciales de fonctions $L$.

  La formule des traces relative pour $S^* $ et $G^*$  devrait également jouer un rôle pour comprendre le spectre automorphe d'autres espaces symétriques.  Ainsi, soit $E/F$ une extension quadratique  de corps de nombres dont on note $\eta$ le caractère quadratique de $\AAA^\times$  associé  par la théorie du corps de classe. On suppose de plus que l'algèbre matricielle  $M(n,D)$ contient $E$. Cela entraîne que   $nd$ est pair. Soit  $H$ le centralisateur dans $G$ de $E$. Soit $\pi$ une représentation automorphe cuspidale $\pi$ de $G(\AAA)$  et soit $\pi^*$ son transfert de Jacquet-Langlands à $G^*$. Alors  $\pi^* $ est une représentation automorphe de $G^*(\AAA)$, qu'on suppose, en outre,  cuspidale. Matringe-Offen-Yang, dans \cite[théorème 1.4]{matringe2025}, prouvent le résultat suivant,  conjecturé par   Guo-Jacquet dans \cite{Guo}, qui généralise un fameux résultat de Waldspurger \cite{Walds85}, cf. aussi \cite{Ja86}: si la représentation $\pi$ apparaît dans le spectre automorphe de $G/H$  alors  $\pi^*$ et $\pi^*\otimes\eta\circ\det$ apparaissent dans le spectre automorphe de $G^*/(G^*)^{\theta}$ avec
  \begin{align*}
  \theta=  \begin{pmatrix}
      I_{nd/2} & 0 \\ 0 & -I_{nd/2}
    \end{pmatrix}.
  \end{align*}
  Précisons que Matringe-Offen-Yang n'utilisent pas dans leur travail la formule des traces relative mais étudient certaines périodes tronquées de séries d'Eisenstein. Ces troncatures jouent aussi un rôle pour nous: dans l'article, nous relions les variantes modifiées des noyaux automorphes à des troncatures de ces noyaux.   Matringe-Offen-Yang démontrent  également une réciproque de  leur théorème,  plus difficile à formuler. Bien sûr,  on aimerait outrepasser l'hypothèse  $\pi^*$ cuspidale, voire décrire tout le spectre automorphe de   $G/H$ et factoriser les périodes en termes d'objets locaux. À cet égard, la formule des traces relatives devrait être un outil puissant (cf. \cite{Ja86, Jacquet-Chen,FMW,Zha2,Xue-Zhang}   pour des résultats suggestifs). 
\end{paragr}

\subsection{Présentation des résultats}

\begin{paragr}
  On continue avec les notations de la section précédente. Soit $G(\AAA)^1$ le noyau du module du  déterminant sur $G(\AAA)$. On a une décomposition  $G(\AAA)=G(\AAA)^1\times A_G^\infty$  où $A_G^\infty$ est un certain sous-groupe central de $G(\AAA)$, cf. § \ref{S:compact}. On fixe $P_0\subset G$ un sous-groupe parabolique défini sur $F$ minimal. On fixe aussi $M_0$ un facteur de Levi de $P_0$ défini sur $F$. On prend $\theta,\theta'\in \Xi$. On peut et on va supposer que $\theta$ et $\theta'$ sont des éléments de $M_0(F)$ qui commutent. On pose  $[G^\theta]= G(F)\back G^\theta(\AAA)$ et  $[G^\theta]^G=G^\theta(F)\back (G^\theta(\AAA)\cap G(\AAA)^1)$. Cette notation vaut aussi pour $\theta'$. Le problème de la formule des traces relatives est alors de donner une façon d'intégrer sur son domaine de définition l'application
  \begin{align*}
    (x,y)\in [G^{\theta'}]^G \times [G^\theta]\mapsto K_f(x,y)
  \end{align*}
  où $f\in \Sc(G(\AAA))$. Pour $n=1$, le quotient $[G^{\theta'}]^G$ est compact et l'application ci-dessus est à décroissance rapide. En particulier, elle est intégrable sur son domaine de définition.   En général, cette application est à décroissance rapide en une des variables mais à croissance lente en l'autre ce qui ne suffit pas pour garantir la convergence de l'intégrale sur $[G^{\theta'}]^G \times [G^\theta]$. Inspirés par les travaux d'Arthur sur la formule des traces, cf. \cite{ar1,Ar-cours}, nous introduisons un noyau modifié    $K^T_f(x,y)$ qui dépend, d'une part, du choix d'un sous-groupe compact maximal de $G(\AAA)$ en bonne position relative à $M_0$ et, d'autre part, d'un paramètre auxiliaire $T$ qui vit dans un certain cône d'un espace vectoriel réel; la définition est donnée dans la  sous-section \ref{ssec:maj noy tronq} où ce noyau modifié est plutôt noté $K^{T,\theta',\theta}_f(x,y)$ car la modification utilisée dépend du choix de  $\theta'$ et $\theta$, cf. remarque \ref{rq-intro:permutation}. Ce noyau modifié admet deux décompositions. La première est spectrale et repose sur l'interprétation de  $K_f(x,y)$ comme noyau de l'opérateur de convolution à droite $R(f)$  sur   l'espace $L^2([G])$ des fonctions de carré intégrable sur $[G]=G(F)\back G(\AAA)$.  La décomposition due à Langlands  
  \begin{align}\label{eq-intro:L2chi}
      L^2([G])=\hat\oplus_{ \chi\in \Xgo(G)}  L^2_\chi([G]),
  \end{align}
  selon l'ensemble $\Xgo(G)$ des données cuspidales  de $G$,  en sous-espaces fermés  invariants (ainsi que celle de certains espaces analogues  associés à des sous-groupes paraboliques de $G$) donne une décomposition
   \begin{align}\label{eq-intro:KTspec}
    K^T_f(x,y)=\sum_{\chi\in \Xgo(G)} K_{\chi,f}^T(x,y).
  \end{align}
  
  La seconde est géométrique et est introduite en \eqref{eq:sum-of}: c'est une somme indexée par les points rationnels du quotient catégorique $\cgo$ de $G$ par l'action à gauche et à droite respectivement de $G^{\theta'}$ et $G^\theta$:
  \begin{align}\label{eq-intro:KTgeo}
    K^T_f(x,y)=\sum_{\of \in \cgo(F)} K_{\of,f}^T(x,y).
  \end{align}
  
 Voici maintenant l'un des principaux résultats de convergence que nous obtenons.

 \begin{theoreme}\label{thm-intro:1} (pour des versions plus fortes et plus précises, cf. théorèmes \ref{thm:cv-spec} et \ref{thm:cv-geo})
    Pour toute fonction $f\in \Sc(G(\AAA))$ et tout $T$ dans un certain cône, on a
     \begin{align}
      \sum_{\chi\in\Xgo(G)}       \int_{[G^{\theta'}]^G\times [G^\theta]     }  |K_{\chi,f}^T(x,y)| \, dxdy <\infty\\
      \sum_{\of \in\cgo(F)}    \int_{[G^{\theta'}]^G\times [G^\theta]     }  |K^T_{\of,f}(x,y)| \, dxdy <\infty.
    \end{align}

      \end{theoreme}

      Soit $\eta$ un caractère de $\AAA^\times/F^\times$. Par composition avec le déterminant, on obtient un caractère de $G(\AAA)$ trivial sur le sous-groupe $G(F)$ qu'on note encore $\eta$.

 Aux §§ \ref{S:JchiT}  et \ref{S:dist geo}, on pose,  pour tous $\chi\in\Xgo(G)$ et $\of\in \cgo(F)$,
 \begin{align*}
   J_\chi^T(\eta,f)&= \int_{[G^{\theta'}]^G\times [G^\theta] }  K_{\chi,f}^T(x,y) \, \eta(x) dxdy\ ; \\
        J_\of^T(\eta,f)&= \int_{[G^{\theta'}]^G\times [G^\theta] }  K^T_{\of,f}(x,y) \, \eta(x) dxdy.
      \end{align*}
      La convergence absolue des intégrales ci-dessus est assurée par le théorème \ref{thm-intro:1}, du moins si $eta$ est unitaire, sinon par les  théorèmes \ref{thm:cv-spec} et \ref{thm:cv-geo}.
  Les formes linéaires $J_\of^T(\eta)$ et $J_\chi^T(\eta)$ ainsi obtenues sont  continues pour la topologie naturelle sur l'espace de Schwartz. On obtient alors l'énoncé suivant.
      
      \begin{theoreme} (cf. théorème \ref{thm:FTGJ avec T}) \label{cor-intro:1} Pour toute fonction $f\in \Sc(G(\AAA))$, on a 
        \begin{align*}
                   \sum_{\chi\in\Xgo(G)} J_\chi^T(\eta,f)=\sum_{\of \in\cgo(F)} J_\of^T(\eta,f).
        \end{align*}
      \end{theoreme}

    À ce stade, on suppose que $\eta$ est trivial sur $A_G^\infty$. Cela implique que le caractère $\eta$ est unitaire. On a remarqué déjà remarqué que, pour les questions de convergence, cette hypothèse était  superflue. Toutefois, sans elle, il faudrait reprendre la formulation de dépendance en $T$ ainsi que celle de la covariance des distributions, cf. § \ref{S-intro:cova} ci-dessous.  On montre, en effet, que $J_\chi^T(\eta,f)$, comme fonction du paramètre  $T$, coïncide avec une fonction polynôme-exponentielle, cf. proposition \ref{prop:JTchi poly}. On définit alors $J_\chi(\eta,f)$ comme le terme constant de cette expression. De façon analogue, on définit $J_\of(\eta,f)$, cf. § \ref{S:dist geo}. 
    
    On en déduit \og la formule des traces de Guo-Jacquet\fg{}.
    
      \begin{theoreme} (cf. théorème \ref{thm:FTGJ}) Les formes  linéaires $J_\chi(\eta)$ et $J_\of(\eta)$ sont  continues sur $\Sc(G(\AAA))$. De plus, pour toute fonction $f\in \Sc(G(\AAA))$,
      \begin{align*}
                   \sum_{\chi\in\Xgo(G)} J_\chi(\eta,f)=\sum_{\of \in\cgo(F)} J_\of(\eta,f).
       \end{align*}
      \end{theoreme} 
      
       C'est pour ainsi dire la formule \og  brute de décoffrage \fg{}: elle va nécessiter au fur et à mesure des applications d'être raffinée.  Dans cette optique, nous  expliciterons certaines des distributions  $J_\of$ et $J_\chi$ pour des $\chi$ et $\of$ assez généraux. Avant cela, nous allons passer en revue les propriétés de covariance de ces distributions.
      \end{paragr}

\begin{paragr}[Covariance.] --- \label{S-intro:cova} Les groupes $G^{\theta'}(\AAA)$ et $G^\theta(\AAA)$ agissent sur $\Sc(G(\AAA))$ respectivement par translations à gauche et à droite notées $f\mapsto ^g \!\!\!f$ et $f\mapsto f^g$, voir § \ref{S:covar-chi-T}. Les distributions $J_\chi(\eta)$ et  $J_\of(\eta)$ définies ci-dessus ne sont pas, en général, invariantes pour ces actions. On a le résultat suivant qui exprime le défaut d'invariance à l'aide de distributions analogues définies sur les sous-groupes de Levi de $G$.
  
\begin{proposition}(cf. propositions \ref{prop:covar} et \ref{prop:covar-geom}) \label{prop-intro:invariance}. Le symbole $\bullet$ désigne soit une donnée cuspidale $\chi\in\Xgo(G)$ soit un point rationnel  $\of$ du quotient catégorique $\cgo(F)$.
   \begin{enumerate}
  \item Pour tout $g\in G^\theta(\AAA)$ et   tout $f\in \Sc(G(\AAA))$,  on  a
    \begin{align*}
      J_\bullet(\eta,f^g)=J_\bullet(\eta,f).
    \end{align*}
  \item  Pour tout $g\in G^{\theta'}(\AAA)$ et   tout $f\in \Sc(G(\AAA))$, on  a
    \begin{align*}
      J_\bullet(\eta,  ^g\!\!f)-  J_\bullet(\eta, f)=\eta(g) \sum_{(Q, w_1',w_2')}  J^{Q,\theta_1,\theta_2}_{\bullet} ( \eta, f_{Q,\eta,g}^{w_1',w_2'})
    \end{align*}
    où  la somme porte sur un ensemble fini de triplets $(Q, w_1',w_2')$ formés d'un sous-groupe parabolique  propre $P_0\subset Q\subsetneq G$ muni du facteur de Levi $M_Q$ contenant $M_0$ et d'éléments $ w_1'$ et $w_2'$ du groupe de Weyl de $(G,M_0)$ auxquels sont associés des éléments $\theta_1$ et $\theta_2$ d'ordre au plus $2$ de $M_0(F)$ ainsi qu'une distribution $J^{Q,\theta_1,\theta_2}_{\bullet}$ sur $\Sc(M_Q(\AAA))$, définie   en \eqref{eq:JQtheta12}  comme ci-dessus comme le terme constant de polynômes-exponentielles en $T$ introduites en \eqref{eq:JQTthetachi} ou  \eqref{eq:JoQTtheta} et une fonction  $f_{Q,\eta,g}^{w_1',w_2'}\in \Sc(M_Q(\AAA))$  définie en \eqref{eq:fct-fQT}.
  \end{enumerate}    
\end{proposition}

\end{paragr}

\begin{paragr}[Distributions $J_\chi(\eta)$.] ---  Lorsque  $\chi$ est la classe d'un couple $(G,\pi)$, où $\pi$ est une représentation automorphe cuspidale irréductible de $G(\AAA)$ (dont le caractère central est trivial sur $A_G^\infty$),     la distribution  $f\mapsto J_\chi(\eta,f)$ est égale au caractère relatif
  \begin{align*}
    \sum_{\varphi\in \bc_\pi} \pc_{G^{\theta'},\eta}( \pi(f)\varphi) \overline{\pc_{G^\theta}(\varphi)}.
  \end{align*}
  Ici, $ \pc_{G^{\theta'},\eta}$ est la forme linéaire, définie sur le sous-espace $V_\pi\subset L^2(G(F)A_G^\infty \back G(\AAA))$ associée à $\pi$, et donnée  par 
  \begin{align*}
 \pc_{G^{\theta'},\eta}(\varphi)= \int_{[G^{\theta'}]_0 }\varphi(h)\eta(h)\, dh
  \end{align*}
  où $[G^{\theta'}]_0 =G^ {\theta'}(F)A_G^\infty \back G^{\theta'}(\AAA)$. La forme linéaire $\pc_{G^\theta}$ est définie de manière analogue relativement au groupe $G^\theta$ et au caractère trivial. La somme porte sur une certaine base hilbertienne $\bc_\pi$ du sous-espace $V_\pi$. Au risque de décevoir le lecteur, précisons que ces périodes sont souvent nulles. Si $n=1$, le produit des  périodes considérées est identiquement nulle sauf si $\eta$ est le caractère trivial et $\pi$ la représentation triviale. Si $n>1$, le produit des périodes est nul dès que les multiplicités des valeurs propres $\pm 1$ de $\theta$ (ou de $\theta'$) ne sont pas égales, ce qui arrive systématiquement si $n$ est impair. Cela résulte d'une généralisation de \cite[proposition 2.1]{FriJa}. Dans ce cas, le spectre cuspidal ne contribue pas du tout à la formule des traces. Un exemple de contribution  \og relativement cuspidale \fg{} est obtenu dans \cite{caractererelatifpondere} ; la contribution s'exprime alors à l'aide d'un caractère relatif pondéré. Pour une donnée cuspidale $\chi$ générale, on renvoie l'obtention d'une formule plus explicite de la distribution $J_\chi(\eta)$ à des travaux futurs. On établit cependant ici le résultat suivant qui devrait faciliter l'étude de $J_\chi(\eta)$ et qui, en tout cas, est utilisé dans \cite{caractererelatifpondere}.
  
  \begin{theoreme} \label{prop-intro:asym}(cf. théorème \ref{thm:comparaison-asym} pour un résultat plus fort et plus précis.)  Soit $f\in \Sc(G(\AAA))$  et $\chi\in \Xgo(G)$ une donnée cuspidale. L'expression $J_\chi^T(\eta,f)$, comme fonction du paramètre $T$, est asymptotique lorsque $T$ tend vers l'infini assez loin des murs dans une chambre positive,   à l'intégrale
    \begin{align*}
      \int_{[G^ {\theta'}]^G\times [G^ {\theta}]  } \La^T_{\theta'} K_{\chi}(x,y) \, \eta(x)dxdy.
    \end{align*}
  \end{theoreme}

  Précisons les notations utilisées. L'opérateur $\La_{\theta'}^T$ est un opérateur de troncature dit  \og relatif\fg{} (en référence  à l'opérateur de troncature  \og absolu \fg{} d'Arthur de \cite{ar2}). On le définit dans la sous-section \ref{ssec:op-tronc}. Notons qu'une variante de notre définition apparaît déjà dans le travail de Zydor \cite[section 3.7]{Zydor}. Notre opérateur a plusieurs propriétés essentielles : il transforme les fonctions sur $G(F)\back G(\AAA)$, qui sont lisses à croissance uniformément modérée, en des fonctions sur $[G^{\theta'}]^G$ qui sont à décroissance rapide (cf. proposition \ref{prop:dec-tronque} pour un résultat précis); il laisse inchangées les fonctions cuspidales sur $[G]$ (celles dont les termes constants sont nuls); enfin, pour toute fonction $\phi$ continue sur $G(F)\back G(\AAA)$, la fonction $\La^T_{\theta'}\phi$ converge simplement sur $G^{\theta'}(\AAA)\cap G(\AAA)^1$ vers  la restriction de $\phi$ à ce sous-groupe. Cette dernière propriété résulte immédiatement de la proposition \ref{prop:asym tronque}. En revanche, elle n'est pas satisfaite par l'opérateur de Zydor, cf. remarque \ref{rq:Zydor}. Elle est cependant cruciale dans \cite{caractererelatifpondere} qui explicite dans certains cas $J_\chi(\eta)$; le lecteur peut consulter \cite[théorèmes 2.4.3.1 et 2.5.1.1]{caractererelatifpondere} et leur preuve.
  
\end{paragr}

\begin{paragr}[Distributions $J_\of(\eta,f)$: cas régulier semi-simple.] --- Dans la suite, on identifie $G/G^\theta$ à la composante connexe $S_\theta$ qui contient $\theta$ de la variété $S$ définie en \eqref{eq-intro:S}. L'intégration  de $f$ le long de $G^\theta(\AAA)$, cf. \eqref{eq:intro-fxi}, donne une fonction de Schwartz $\varphi$ sur $\Sc(S_\theta(\AAA))$. On pose pour tout $\of\in\cgo(F)$
  \begin{align*}
    J_\of(\eta,\varphi)=J_\of(\eta,f).
  \end{align*}
Le membre de gauche définit alors une distribution sur $\Sc(S_\theta(\AAA))$. Le quotient catégorique $\cgo$ s'identifie au quotient catégorique $\cgo_\theta$ de  $ S_\theta$ par l'action par conjugaison de $G^{\theta'}$. De la sorte, on voit $\of$  comme un élément de $\cgo_\theta(F)$. 

Soit $\ga\in S_\theta(F)$. On dit que $\ga$ est $G^{\theta'}$-semi-simple, resp.  $G^{\theta'}$-régulier,  si l'orbite de $\ga$ sous l'action de $G^{\theta'}$ est fermée, resp. est de dimension maximale parmi les $G^{\theta'}$-orbites dans $S_\theta$. Il existe un ouvert dense   $\cgo_\theta^{\rs}\subset \cgo_\theta$ tel que, pour tout $\of\in \cgo_\theta^{\rs}$, la fibre $S_{\theta,\of}$ de l'application canonique $S_\theta\to \cgo_\theta$ au-dessus de $\of$ soit formée d'une seule  $G^{\theta'}$-orbite, forcément semi-simple régulière.

Soit $\of\in \cgo_\theta^{\rs}(F)$. Soit $M$ un sous-groupe de Levi de $G$ contenant $M_0$ et  tel que $M$ soit minimal pour la propriété suivante $(M\cap S_{\theta,\of})(F)\not=\emptyset$. On dit que $\of$ est $(G,\theta')$-anisotrope si $M=G$. Notons que, contrairement au cas de la formule des traces d'Arthur c'est-à-dire le cas de l'action de $G$ par conjugaison sur lui-même, en général il n'existe pas d'éléments  $(G,\theta')$-anisotropes. Pour qu'il en existe, il faut et il suffit que soit $n=1$  soit $n>1$ et les multiplicités des valeurs $\pm 1$ de $\theta$, resp.  $\theta'$, sont égales. La présence ou l'absence d'éléments   $(G,\theta')$-anisotropes semble se refléter dualement dans la présence ou l'absence de spectre cuspidal dans la formule des traces. Soit  $\ga\in (M\cap S_\theta)(F)$ d'image $\of$ dans $\cgo_\theta(F)$. Alors $\of$ est $(G,\theta')$-anisotrope si et seulement si le centralisateur $G_\ga^{\theta'}$ de $\ga$ dans  $G^{\theta'}$ est anisotrope modulo le centre de  $G$. On observera que   $G_\ga^{\theta'}$ est un groupe réductif connexe qui, en général, n'est ni un tore  ni ne contient des tores maximaux de  $G^{\theta'}$, ce qui diffère là encore du cas de la formule des traces d'Arthur. Soit $A$ le tore central déployé maximal de  $G_\ga^{\theta'}$. Notons que $A_M$ le tore central déployé maximal de $M$ est inclus dans $G_\ga^{\theta'}$ et donc on a $A\subset A_M$. Soit $L$ le sous-groupe de Levi de $G$ défini comme le centralisateur de $A$. Alors $G_\ga^{\theta'}\subset L$ et $M\subset L$. Alors $A_L$ est un sous-tore central déployé de $G_\ga^{\theta'}$, on a donc $A_L\subset A$ et, en fait $A_L=A$ par construction de $L$. Pour tout $g\in G(\AAA)$, on dispose alors du  poids d'Arthur  $v_L^G(g)$, cf. § \ref{S:defcar}: c'est celui qui apparaît dans  les intégrales orbitales pondérées de la formule des traces d'Arthur, cf. \cite[§8]{ar1}. Il dépend du choix du sous-groupe compact maximal de $G(\AAA)$ qui intervient dans la définition du noyau modifié $K^T$.

Supposons d'abord que le caractère $\eta$ n'est pas trivial sur le centralisateur $G_\ga^{\theta'}(\AAA)$. Dans ce cas, on obtient
\begin{align*}
  J_\of(\eta,f)= J_\of(\eta,\varphi)=0.
\end{align*}

Supposons désormais que le caractère $\eta$ est trivial sur $G_\ga^{\theta'}(\AAA)$. On peut alors exprimer  $J_\of(\eta,\varphi)$ comme une intégrale orbitale pondérée et tordue par le caractère $\eta$. Plus pécisément, on a

\begin{theoreme} (cf. théorème \ref{thm:IOP}) Sous les hypothèses ci-dessus, on a 
   \begin{align*}
   J_\of(\eta,f)=  J_\of(\eta,\varphi)= \vol(  G_\ga^{\theta'}(F) A_L^\infty  \back G_\ga^{\theta'}(\AAA)  ) \cdot \int_{G_\ga^{\theta'}(\AAA)\back G^{\theta'}(\AAA)}  \varphi( h^{-1}\ga h)  v_L^G(h)  \,\eta(h)dh
 \end{align*}
 où $A_L^\infty$ est un sous-groupe central de $L(\AAA)$, cf. § \ref{S:compact}. 
\end{theoreme}

Le volume de $G_\ga^{\theta'}(F) A_L^\infty  \back G_\ga^{\theta'}(\AAA)  $ est fini.  Le poids $v_L^G$ est une fonction qui est invariante à gauche par $L(\AAA)$ donc par $G_\ga^{\theta'}(\AAA)$  et qui dépend des choix de mesures de Haar sur $A_L^\infty $ et $A_G^\infty$. Ces choix sont compatibles avec ceux  qui interviennent dans la définition de la mesure quotient sur $G_\ga^{\theta'}(F) A_L^\infty  \back G_\ga^{\theta'}(\AAA)  $ et dans la décomposition de la mesure  $ G_\ga^{\theta'}(\AAA) = A_G^\infty \times (G_\ga^{\theta'}(\AAA) \cap G(\AAA)^1)$. Le choix de la mesure sur $G_\ga^{\theta'}(\AAA)$ est la même dans l'intégrale que dans la mesure sur le quotient  $G_\ga^{\theta'}(F) A_L^\infty  \back G_\ga^{\theta'}(\AAA) $. L'intégrale adélique est absolument convergente. Cependant, si $L\subsetneq G$, la distribution  $J_\of(\eta)$ n'est pas invariante sous l'action de $G^{\theta'}(\AAA)$. 
\end{paragr}

\begin{paragr}[Distributions $J_\of(\eta,f)$: descente dans le cas général.] --- Dans le cas général, où  $\of \in \cgo(F)$  est quelconque, nous donnons une étape décisive pour exprimer   $ J_\of(\eta,\varphi)$  en termes d'objets locaux. Plus précisément, nous réduisons ce problème au cas unipotent c'est-à-dire au cas de l'invariant $\of$ associé à $\Id_{D^n}$. Avant de donner quelques détails, nous remarquons qu'afin d'avoir une décomposition de Jordan pour l'espace symétrique, dans le reste de l'introduction et le corps de l'article, il sera plus commode d'étudier  $\wt{S_\theta}=S_\theta\theta'$ au lieu de $S_\theta$.  Soit $\wt{\cgo_\theta}$ le quotient catégorique de $\wt{S_\theta}$ par l'action de $G^{\theta'}$ par conjugaison. Il  s'identifie à $\cgo_\theta\simeq \cgo$.  Soit $(\wt{S_\theta})_\of$ la fibre du quotient catégorique  $\wt{S_\theta}\to \wt{\cgo_\theta}$ au-dessus de $\of$. Soit $M$ un sous-groupe de Levi de $G$ comme ci-dessus. Nous fixons un élément $\ga\in (M\cap(\wt{S_\theta})_\of)(F)$  dont la $G^{\theta'}$-orbite est fermée. Le centralisateur $G_\ga$ est alors un groupe réductif connexe qui est muni de l'involution induite par la conjugaison par $\theta'$ ; soit $G^{\theta'}_\ga=G^{\theta'}\cap G_\gamma$ le sous-groupe des points fixes. C'est ce qu'on appelle un descendant de $(G,G^{\theta'})$ et la liste des situations possibles est explicitement décrite dans la proposition \ref{prop:descendant}. D'après le lemme \ref{lem:var-unip}, on a 
\begin{align*}
    (\wt{S_\theta})_\of=\Int (G^{\theta'}) (\ga \cdot(\uc_{G_\ga}\cap \wt{S_{\theta'}}))
\end{align*}
où l'on note $\Int$ l'action par conjugaison et $\uc_{G_\ga}$ la variété formée des éléments unipotents de $G_\ga$. L'élément neutre appartient à  $\wt{S_{\theta'}}$ et l'espace tangent à $\wt{S_{\theta'}}$ en ce point s'identifie au sous-espace de l'algèbre de Lie $\ggo$ de $G$
\begin{align*}
    \sgo_{\theta'}=\{X\in \ggo| \Ad(\theta')(X)+X=0\},
\end{align*}
où $\Ad$ désigne l'action adjoint. L'application exponentielle induit un isomorphisme 
\begin{align*}
    \exp: \nc_{\ggo_\ga}\cap\sgo_{\theta'}\to\uc_{G_\ga}\cap \wt{S_{\theta'}}
\end{align*}
 où l'on note $\nc_{\ggo_\ga}$ le cône nilpotent de l'algèbre de Lie $\ggo_\gamma$ de $G_\ga$. Le groupe $G^{\theta'}_\ga$ agit sur $\sgo_{\theta',\ga}=\sgo_{\theta'}\cap\ggo_\ga$. Alors, on relie   $ J_\of(\eta,f)=  J_\of(\eta,\varphi) $  à la contribution nilpotente de la formule des traces infinitésimales pour l'action de $G^{\theta'}_\ga$ sur $\sgo_{\theta',\ga}$. Plus précisément, si $\ga=\Id$, on obtient l'action de $G^{\theta'}$ sur $\sgo_{\theta'}$ et la formule des traces correspondante est étudiée dans \cite{li1}. En général, l'étude de l'action de  $G^{\theta'}_\ga$ sur $\sgo_{\theta',\ga}$ se ramène à un produit d'actions de groupes plus simples : il y a au plus deux facteurs pour lesquels l'action est celle que l'on vient de décrire (mais en rang plus petit) et d'autres facteurs pour lesquels l'action est donnée par le modèle suivant. On se donne un entier $r\geq 1$, $E'/F$ une extension finie, $D'$ une algèbre simple centrale sur $E'$ ainsi que $E/E'$ une algèbre quadratique étale dont on note $\sigma$ l'unique $E'$-automorphisme non trivial. On note encore $\sigma$ l'automorphisme de  $E\otimes_{E'}D'$ donné par $\sigma\otimes \Id$. L'action à considérer est celle par conjugaison du groupe $G'=GL(r,D')$ sur l'espace
 \begin{align*}
     \{X\in M(r,E\otimes_{E'}D')| X+\sigma(X)=0 \}.
 \end{align*}
 Soit $\iota\in E$ un élément non nul qui vérifie $\iota+\sigma(\iota)=0$. L'application $X\mapsto \iota X$ identifie alors de manière équivariante l'action précédente à l'action adjointe de $G'$ sur son algèbre de Lie. La formule des traces pour l'action adjointe d'un groupe réductif a été étudiée dans  \cite{Cha}.
 
 Un sous-groupe parabolique $R'$ de $G_\ga$ est appelé \og relativement standard \fg{} s'il est $\theta'$-stable et contient un sous-groupe parabolique minimal de $G_\ga^{\theta'}$. Pour un tel $R'$, on le associé un caractère $\chi_{R'}$ du facteur de Levi $M_{R'}$ de $R'$ et
 une distribution nilpotente $J_{\nilp}^{R'}(\eta, \chi_{R'})$  définie dans \S \ref{S:def-dist-infi} sur l'espace tangent $\sgo_{\theta',R'}$ de $M_{R'}/M_{R'}^{\theta'}$ en l'élément neutre, qui est une version tronquée de la forme linéaire divergente 
 \begin{align*}
     \Phi\mapsto \int_{[M_{R'}^{\theta'}]^{R'}} \sum_{X\in(\nc_{\ggo_\ga}\cap\sgo_{\theta',R'})(F)} \Phi(\Ad(h^{-1})X) \, dh 
 \end{align*}
 pour toute fonction de Schwartz $\Phi$ sur $\sgo_{\theta',R'}(\AAA)$ où $[M_{R'}^{\theta'}]^{R'}$ est un analogue de $[G^{\theta'}]^G$. On associé également aux $\varphi$ et $g\in G^{\theta'}(\AAA)$ une fonction de Schwartz $\Phi_{g,R'}$  sur $\sgo_{\theta',R'}(\AAA)$ via \eqref{eq:defPsi}, \eqref{eq:def-fonctionPsi} et \eqref{eq:defPhigR'}. 

    \begin{theoreme} (cf. théorème \ref{thm:descente})
        Soit $\varphi\in \Sc(S_\theta(\AAA))$ une fonction à support compact. On a  
    \begin{align*}
	J_\of(\eta, \varphi)=\int_{G_\ga^{\theta'}(\AAA)\back G^{\theta'}(\AAA)} \sum_{R'} J_{\nilp}^{R'}(\eta, \chi_{R'}, \Phi_{g,R'}) \eta(g) \, dg
    \end{align*}  
où la somme porte sur un ensemble fini de sous-groupes paraboliques relativement standard $R'$ de $G_\ga$. 
    \end{theoreme}
\end{paragr}

\subsection{L'exemple du groupe \texorpdfstring{$G=GL(2,D)$}{G=GL(2,D)}} \label{ssec:intro GL2}

\begin{paragr}  On met de côté le cas évident $n=1$ pour lequel la variété $S$ est formé de deux points. Les intégrales des composantes spectrales du  noyau automorphe convergent directement et donnent des distributions invariantes qui sont en fait nulles sauf lorsque le caractère $\eta$ est trivial et que la donnée cuspidale correspond au caractère trivial. Le développement géométrique, quant à lui, se réduit à l'intégrale du noyau.
\end{paragr}

\begin{paragr} \label{S-intro:comp S}Nous allons considérer plus en détail le cas $n=2$.   On rappelle que le caractère $\eta$ est trivial sur $A_G^\infty$.  La variété $S$ a alors trois composantes connexes qui sont respectivement les $G$-orbites des matrices $I_2$, $-I_2$ et $ \begin{pmatrix}
      1 & 0 \\ 0 & -1
    \end{pmatrix}$.
  \end{paragr}
  
  \begin{paragr}   On va utiliser les notations suivantes. Soit $B\subset G$ le sous-groupe parabolique minimal formé des matrices triangulaires supérieures muni. Il est  muni de sa décomposition de Levi $B=M_B N_B$ où $M_B$ est le sous-groupe des matrices diagonales et $N_B$ est le radical unipotent de $B$. Soit $A_B\subset M_B$ le tore central déployé maximal et 
  $\al$ l'unique racine de $A_B$ dans $N_B$. Le paramètre de troncature $T$ est un point de l'espace vectoriel $\ago_B=\Hom(X^*(M_B),\RR)$ dual du groupe $X^*(M_B)$ des caractères rationnels définis sur $F$ de $M_B$. La racine $\al$ induit une forme linéaire sur $\ago_B$ notée $\bg \al,\cdot\bd$.  Soit $\hat\tau_B$ la fonction sur  $\ago_B$ caractéristique du cône des $H$ tels que $\bg \al,H\bd >0$.  On suppose $\bg \al,T\bd$ assez grand. 
  
  On fixe un sous-groupe compact maximal $K\subset G(\AAA)$ en bonne position par rapport à $M_B$: on a alors une application $K$-invariante à droite $H_B:G(\AAA)\mapsto \ago_B$ qui coïncide sur $B(\AAA)$ avec  l'homomorphisme donné par la composition de l'homomorphisme canonique de $M_B(\AAA)$ dans $\Hom(X^*(M_B),\AAA^\times)$ avec le logarithme du module usuel. On note $M_B(\AAA)^1$ le noyau de cet homomorphisme. 
    On normalise les mesures de Haar sur $K$, $N_B(\AAA)$ et $M_B(\AAA)^1$ de sorte que $K$, $N_B(F)\back N_B(\AAA)$ et $M_B(F)\back M_B(\AAA)^1$ (pour les mesures quotient) soient de volume $1$. Le morphisme $m\mapsto \bg \al ,H_B(m)\bd $  induit une suite exacte
    \begin{align*}
        1 \to M_B(\AAA)^1 \to M_B(\AAA)\cap G(\AAA)^1 \to \RR \to 1.
    \end{align*}
    On munit   $M_B(\AAA)\cap G(\AAA)^1$ de la mesure qui donne au quotient $\RR$  la mesure de Lebesgue. Alors $G(\AAA)^1$ est muni de la mesure compatible à la décomposition $G(\AAA)^1=(M_B(\AAA)\cap G(\AAA)^1) N_B(\AAA) K$.
    Finalement, on identifie naturellement $A_G^\infty$ à $\RR_+^\times$ qu'on équipe de la mesure de Lebesgue usuelle. La décomposition $G(\AAA)=G(\AAA)^1\times A_G^\infty$ détermine alors une mesure de Haar sur $G(\AAA)$ compatible à la mesure produit sur le membre de droite. De même, on obtient une mesure de Haar sur $M_B(\AAA)$.

    Soit $W=\{I_2,
      \begin{pmatrix}
        0 & 1 \\ 1 &0
      \end{pmatrix} \}$.   On utilisera les calculs ci-dessous
  \begin{align}\label{eq-intro:volume 1} &\forall x\in G(\AAA)\\
   \nonumber &\int_{ M_B(F)\back (M_B(\AAA)\cap G(\AAA)^1) }  (1- \sum_{w\in W} \hat\tau_B(H_B(w y x)-T) ) \eta(y)\, dy= \left\lbrace\begin{array}l
         0 \text{ si } \eta \text{ n'est pas trivial}  \\
         \bg \al, 2T-H_B(wx)-H_B(x)\bd \text{ sinon}.
    \end{array}\right.\\
    \label{eq-intro:volume 2}&\int_{ [G]^G }   (1-\sum_{\delta \in  B(F)\back G(F)}  \hat\tau_B(H_B(\delta  x)-T)   )  \, \eta(x) dx=\left\lbrace\begin{array}l
         0 \text{ si } \eta \text{ n'est pas trivial}  \\
         \vol([G]^G) - \exp(-\bg \al, T\bd) \text{ sinon}.
    \end{array}\right.
  \end{align}
    où $[G]^G=G(F)\back G(\AAA)^1$.

    L'algèbre     $\Sc(G(\AAA))$ agit par convolution à droite sur   l'espace  $ L^2( M_B(F)N_B(\AAA)\back G(\AAA))$. L'opérateur associé à $f$ a un noyau noté $K_{B,f}$. Soit  $K_{B,\chi,f}$ la composante de ce noyau  associée au facteur indexé par $\chi$ dans la décomposition
  \begin{align*}
    L^2( M_B(F)N_B(\AAA)\back G(\AAA))=\hat\oplus_{\chi'\in \Xgo(G)} L^2_{\chi'}( M_B(F)N_B(\AAA)\back G(\AAA))
  \end{align*}
  selon les données cuspidales de $G$. On définit aussi
  \begin{align}
    \nonumber(K_{\chi,f})_B(x,y)&=\int_{N_B(F)\back N_B(\AAA)} K_{\chi,f}(nx,y)\, dn\\
    \label{eq-intro:somme delta}&=\sum_{\delta  \in B(F)\back G(F)}  K_{B,\chi,f}(x,\delta y).
  \end{align}
  \end{paragr}

  \begin{paragr}
Soit $f\in \Sc(G(\AAA))$ et  $\chi\in \Xgo(G)$. Soit $K^{T,\theta',\theta}_{\chi,f}$ le noyau modifié défini à la sous-section \ref{ssec:maj noy tronq} et associé à  des éléments  $\theta,\theta'$ de $S$ parmi ceux donnés  au §  \ref{S-intro:comp S}. Soit $\chi_0\in \Xgo(G)$ la donnée cuspidale associée à $(M_B,1)$ où $1$ est la représentation triviale de $M_B(\AAA)$.
\end{paragr}

\begin{paragr}
  Dans les trois premiers exemples donnés aux  §§ \ref{S-intro:I-I}, \ref{S-intro:theta-I} et \ref{S-intro:I-theta}, le quotient géométrique est réduit à un point et le développement géométrique est réduit à un terme.
  \end{paragr}

\begin{paragr}[Cas $\theta=\theta'=I_2$.] --- \label{S-intro:I-I}Les mêmes calculs s'appliquent aux cas similaires $\theta=\pm I_2$ et $\theta'=\pm I_2$. Selon la définition de  la sous-section \ref{ssec:maj noy tronq}  et l'égalité \eqref{eq-intro:somme delta}, on a
  \begin{align*}
    K^{T,\theta',\theta}_{\chi,f}(x,y)&=K_{\chi,f}(x,y)-\sum_{\delta_1,\delta_2 \in B(F)\back G(F)} \hat\tau_B(H_B(\delta_1 x)-T) K_{B,\chi,f}(\delta_1 x,\delta_2y)\\
    &=K_{\chi,f}(x,y)-\sum_{\delta \in  B(F)\back G(F)}  \hat\tau_B(H_B(\delta  x)-T)  (K_{\chi,f})_B(\delta x,y)
  \end{align*}
   Par définition, on a 
  \begin{align}\label{eq-intro:integ GL2}
    J_\chi^T(\eta,f)&= \int_{  G(F)\back G(\AAA)^1\times G(F)\back G(\AAA) }  K_{\chi}^T(x,y) \, \eta(x) dxdy.
  \end{align}

  Dans l'intégrale du membre de droite ci-dessus, on commence par  intégrer sur la variable $y$. Tout d'abord, on a
  \begin{align*}
 \int_{   G(F)\back G(\AAA) }     K_{\chi,f}(x,y)\, dy=0 \text{ et }   \int_{   G(F)\back G(\AAA) }     (K_{\chi,f})_B(x,y)\, dy=0
  \end{align*}
  sauf si $\chi=\chi_0$. On pose
  \begin{align*}
  I(f)=\int_{   G(F)\back G(\AAA) }  f(y)\, dy.
    \end{align*}
  Alors, on a 
  \begin{align*}
&      \int_{   G(F)\back G(\AAA) }     K_{\chi_0,f}(x,y)\, dy= \int_{   G(F)\back G(\AAA) }     K_{f}(x,y)\, dy=  I(f)\\
&      \int_{   G(F)\back G(\AAA) }     (K_{\chi_0,f})_B(x,y)\, dy= \int_{   G(F)\back G(\AAA) }     (K_{f})_B(x,y)\, dy=  I(f). \end{align*}
  En tenant compte du calcul \eqref{eq-intro:volume 2}, on conclut qu'on a
    \begin{align*}
       J_{\chi_0}^T(\eta,f)&=  \left\lbrace\begin{array}l
         0 \text{ si } \eta \text{ n'est pas trivial ;}  \\
         I(f) (\vol([G]^G) - \exp(-\bg \al, T\bd)) \text{ sinon}.
    \end{array}\right. 
    \end{align*}
    On en déduit
     \begin{align*}
       J_{\chi}(\eta,f)&=  \left\lbrace\begin{array}l
         0 \text{ si } \eta \text{ n'est pas trivial ou } \chi\not=\chi_0 \ ;\\
        \vol([G]^G)  I(f)  \text{ sinon}.
    \end{array}\right. 
    \end{align*}
    La distribution obtenue est invariante à gauche et à droite par l'action par translations de $G(\AAA)$. 
       \end{paragr}

    \begin{paragr}[Cas $\theta=I_2$ et $\theta'=\begin{pmatrix}
      1 & 0 \\ 0 & -1
    \end{pmatrix}$.] --- \label{S-intro:theta-I}Cette fois-ci, on a, toujours  en suivant les définitions de  la sous-section \ref{ssec:maj noy tronq}  et en utilisant \eqref{eq-intro:somme delta},
    
    \begin{align*}
       K^{T,\theta',\theta}_{\chi,f}(x,y)=K_{\chi,f}(x,y)-\sum_{w\in W} \hat\tau_B(H_B(w x)-T)  (K_{\chi,f})_B(w x,y).
   \end{align*}
    
     Comme précédemment, mais en utilisant cette fois-ci \eqref{eq-intro:volume 1} pour $x=1$, on obtient
       \begin{align*}
       J_{\chi}^T(\eta,f)&=  \left\lbrace\begin{array}l
         0 \text{ si } \eta \text{ n'est pas trivial  ou } \chi\not=\chi_0 \ ; \\
         I(f) \bg \al, 2T\bd \text{ sinon}.
    \end{array}\right. 
    \end{align*}
    L'intégrale ci-dessus dépend donc linéairement de $T$: son terme constant est nul et on en déduit qu'on a    $J_{\chi}(\eta,f)=0$ pour tout $\chi\in \Xgo(G)$.
     
\end{paragr}

  \begin{paragr}[Cas $\theta=\begin{pmatrix}
      1 & 0 \\ 0 & -1
  \end{pmatrix}$ et $\theta'=I_2$.] --- \label{S-intro:I-theta} Dans ce cas, le noyau modifié s'écrit

  \begin{align*}
     K^{T,\theta',\theta}_{\chi,f}(x,y)=K_{\chi,f}(x,y)-\sum_{\delta_1 \in B(F)\back G(F)} \hat\tau_B(H_B(\delta_1 x)-T)  \sum_{w\in W} K_{B,\chi,f}(\delta_1 x,w y)
  \end{align*}
En utilisant l'égalité $   K^{T,\theta',\theta}_{\chi,f}(ax,y) =   K^{T,\theta',\theta}_{\chi,f}(x,a^{-1}y)$  pour tout $a\in A_G^\infty$, on voit qu'on a
  \begin{align*}
    J_{\chi}^T(\eta,f)&= \int_{  G(F)\back G(\AAA)^1 \times  M_B(F)\back M_B(\AAA) }  K^T_{\chi,f}(x,y) \, \eta(x) dxdy \\
    &= \int_{  G(F)\back G(\AAA) \times  M_B(F)\back (M_B(\AAA)\cap G(\AAA)^1) }  K^T_{\chi,f}(x,y) \, \eta(x) dxdy .
  \end{align*}
  On effectue alors l'intégration d'abord sur la variable $x$ ce qui donne 
  \begin{align*}
     \int_{  G(F)\back G(\AAA)} K_{\chi,f}(x,y) \eta(x)\, dx  -  \int_{  B(F)\back G(\AAA)} \hat\tau_B(H_B(x)-T)  \sum_{w\in W}K_{B,\chi,f}(x,wy) \eta(x)\, dx.
  \end{align*}
 Cette expression est nulle sauf si $\chi$ est égale à la donnée cuspidale $\chi_1$ associé à $(M_B,\eta)$. Pour $\chi=\chi_1$, on obtient
 \begin{align*}
   &\int_{  G(F)\back G(\AAA)} K_{\chi_1,f}(x,y) \eta(x)\, dx  -  \int_{  B(F)\back G(\AAA)} \hat\tau_B(H_B(x)-T)  \sum_{w\in W}K_{B,\chi_1,f}(x,w y) \eta(x)\, dx\\
  & =  \int_{  G(F)\back G(\AAA)} K_{f}(x,y) \eta(x)\, dx  -  \int_{  B(F)\back G(\AAA)} \hat\tau_B(H_B(x)-T)  \sum_{w\in W}K_{B,f}(x,wy) \eta(x)\, dx\\
  & = \int_{  G(\AAA)} {f}(x^{-1}y) \eta(x)\, dx  - \int_{N_B(F)\back N_B(\AAA)}  \int_{  G(\AAA)}\hat\tau_B(H_B(x)-T)\sum_{w\in W} f(x^{-1} u wy)   \eta(x)\, dx   \,du     \\
  & =  \eta(y) \int_{  G(\AAA)}  (1- \sum_{w\in W} \hat\tau_B(H_B(w y x)-T) )f(x^{-1})\eta(x)\, dx, 
  \end{align*}
  la dernière égalité étant obtenue par le changement de variables $x\mapsto uwyx$. Il résulte alors de \eqref{eq-intro:volume 1} qu'on a 
   \begin{align*}
       J_{\chi}^T(\eta,f)&=  \left\lbrace\begin{array}l
         0 \text{ si } \eta \text{ n'est pas trivial  ou } \chi\not=\chi_0 \ ; \\
         2\bg \al, T\bd I(f) + I'(f) \text{ sinon},
    \end{array}\right. 
    \end{align*}
  où l'on introduit
   \begin{align*}
  I'(f)= \int_{  G(\AAA)}  \bg \al, -H_B(wx^{-1})-H_B(x^{-1})\bd f(x)\, dx.  
  \end{align*}
 
  Cette fois-ci  $J_{\chi}^T(\eta,f)$ est une fonction affine de $T$ et on obtient
   \begin{align*}
       J_{\chi}(\eta,f)&=  \left\lbrace\begin{array}l
         0 \text{ si } \eta \text{ n'est pas trivial  ou } \chi\not=\chi_0 \ ; \\
          I'(f) \text{ sinon}.
    \end{array}\right. 
    \end{align*}
     On observera que la distribution $I'$ est invariante pour l'action à droite de $M_B(\AAA)$. En revanche, elle n'est pas invariante pour l'action à gauche de $G(\AAA)$. On a pour tout $y\in G(\AAA)$
     \begin{align*}
         I'(\,^y f)-I'(f)= \sum_{w\in W} I_{M_B}(f_{B,w,y})   
     \end{align*}
  où, pour toute fonction $\varphi\in \Sc(M_B(\AAA))$, on pose
  \begin{align*}
      I_{M_B}(\varphi)=\int_{M_B(\AAA)} \varphi(m)\, dm
  \end{align*}
  et, pour tout $m\in M_B(\AAA)$,
  \begin{align*}
      f_{B,w,y}(m) = \int_{N_B(\AAA)} \int_K -\bg \al, H_B(k^{-1} y)\bd      f( k n mw)\, dndk.       
  \end{align*}
  Cette formule est une explicitation de la proposition \ref{prop-intro:invariance}. 
  
  \begin{remarque}\label{rq-intro:permutation}
      Les noyaux modifiés $K^{T,\theta',\theta}_{\chi,f}$ et $K^{T,\theta,\theta'}_{\chi,f}$ sont distincts en général. Les distributions qu'on leur associe \emph{in fine} peuvent donc être distinctes comme on le voit en comparant les résultats des §§ \ref{S-intro:theta-I}  et \ref{S-intro:I-theta}.
  \end{remarque}
  
  \end{paragr}

   \begin{paragr}[Cas $\theta=\theta'=\begin{pmatrix}
      1 & 0 \\ 0 & -1
    \end{pmatrix}$.] --- Ce cas est en fait beaucoup plus intéressant mais, en un sens, exceptionnel : on montre, en effet, aux §§ \ref{S:GL2D spec} et \ref{S:GL2D geo}   qu'on obtient alors des distributions spectrales $J_{\chi}(\eta)$ et géométriques  $J_{\of}(\eta)$ qui sont invariantes à droite  et $\eta$-équivariantes à gauche pour l'action par translations de $M_B(\AAA)$.  Soulignons ce n'est pas du tout le cas en général. De fait, cette situation particulière doit rentrer dans le cadre des constructions de Sakellaridis dans \cite{Sak-Schwartz} (c'est-à-dire elle ne devrait pas avoir d'exposant critique au sens de cet article) ; nous invitons donc le lecteur à comparer nos constructions avec celles de  Sakellaridis. Pour une explicitation des distributions obtenues, nous renvoyons à l'article de Jacquet \cite{Ja86}, voir aussi  \cite{Jacquet-Chen} pour des compléments.
  \end{paragr}

\subsection{Organisation de l'article}

\begin{paragr} Dans la sous-section \ref{ssec:combi}, le lecteur trouvera les principales notations utilisées. Quelques résultats algébriques sur les involutions considérées dans cet article sont énoncés dans la sous-section \ref{ssec:involutions}. On donne ensuite dans la sous-section \ref{ssec:partition} une version \og relative \fg{} d'une partition utilisée par Arthur. On peut alors introduire dans la sous-section \ref{ssec:op-tronc} l'opérateur de troncature relatif et ses principales propriétés. La section \ref{sec:Dev spectral} est dédiée au développement spectral de la formule des traces relative. Les noyaux modifiés sont définis dans la sous-section \ref{ssec:maj noy tronq} et on en donne des majorations. Les principaux énoncés spectraux qui en résultent se trouvent dans la sous-section \ref{ssec:enonces spec}. On analyse le comportement en $T$ des distributions obtenues dans la sous-section \ref{ssec:comport T}. On peut alors définir en sous-section \ref{ssec:distr spec} les termes spectraux de la formule des traces comme leur terme constant en $T$. On donne des propriétés de covariance sous les actions naturelles des groupes qui interviennent. 

  Le reste de l'article, c'est-à-dire la section \ref{sec:dvpt geo}, est consacré au développement géométrique de la formule des traces relative. On commence par des préliminaires  algébriques (introduction de l'espace symétrique $S$, calcul d'un quotient catégorique, calcul d'orbites semi-simples etc.) en sous-section \ref{ssec:preparatif alg}. Une première version de la convergence du développement géométrique selon les points rationnels du quotient catégorique est énoncée dans la sous-section \ref{ssec:dvpt quotient}. On peut alors énoncer la formule des traces de Guo-Jacquet en sous-section \ref{ssec:FT-GJ}.  Dans la sous-section \ref{ssec:passage}, on prend le point de vue de l'action de $G^{\theta'}$ sur la variété $S$: on reformule le développement géométrique comme distribution sur l'espace de Schwartz de $S(\AAA)$. On donne divers préparatifs pour la descente semi-simple des distributions $J_\of^T$.  La sous-section \ref{ssec:desc-ss} est consacrée à cette descente semi-simple des distributions $J_\of^T$, le but étant de relier  une telle distribution à la contribution nilpotente de la formule des traces infinitésimales associée à un certain espace symétrique attaché au centralisateur d'un élément semi-simple d'invariant $\of$. Dans la sous-section \ref{ssec:IOP}, pour des invariants $\of$ réguliers semi-simples, on exprime la distribution $J_\of(\eta)$ à l'aide d'une intégrale orbitale pondérée. Enfin, la  sous-section \ref{ssec:var-infi} donne des  rappels et des compléments sur les formules des traces infinitésimales pour certains espaces symétriques.
  
\end{paragr}

\begin{paragr}[Remerciements.] --- Les auteurs remercient Jayce Getz pour une discussion sur les couples d'involutions qui commutent et pour avoir porté à leur connaissance les travaux \cite{HelWancommute, HelWansmooth} de Helminck-Schwarz. Ils remercient également Yiannis Sakellaridis pour une discussion  sur les résultats de cet article ainsi que pour avoir insisté sur le cas du groupe $GL(2)$. Pierre-Henri Chaudouard remercie l'Institut Universitaire de France pour lui avoir offert d'excellentes conditions de travail. Ce travail a en partie été effectué lorsque Huajie Li était en poste à Johns Hopkins University, à Max-Planck-Institut für Mathematik et à Aix–Marseille Université, institutions que le second auteur remercie pour les excellentes conditions de travail et le soutien financier apportés, y compris une aide du gouvernement français au titre du Programme Investissements d’Avenir, Initiative d’Excellence d’Aix–Marseille Université–A*MIDEX. Il remercie également l’Université Paris Cité pour l'invitation et le financement de son séjour à Paris, durant lequel cet article a été finalisé.
\end{paragr}

\section{Préliminaires} \label{sec:prelim}

\subsection{Notations générales} \label{ssec:combi}

\begin{paragr}  Dans cette sous-section, nous allons introduire les notations les plus couramment utilisées dans l'article: comme on va le voir, on utilise  autant que faire se peut les notations d'Arthur. 
\end{paragr}

\begin{paragr} Soit $F$ un corps et  $G$ un  groupe réductif connexe défini sur  $F$. Soit $A_G$ le tore déployé central maximal de $G$. Soit $(P_0,M_0)$ un couple formé d'un sous-groupe parabolique $P_0$ ainsi que d'un facteur de Levi $M_0$ tous deux définis sur $F$ et minimaux pour ces propriétés. Sauf mention contraire, les sous-groupes de $G$ seront toujours supposés définis sur $F$. Pour tout sous-groupe $P$ de $G$ soit  $N_P$ le radical unipotent de $P$. Un sous-groupe parabolique $P$ de $G$ est standard, resp. semi-standard, s'il contient $P_0$,  resp. $M_0$. Soit $\fc(P_0)\subset  \fc(M_0)$ les ensembles de sous-groupes paraboliques de $G$ qui sont  respectivement standard, semi-standard. Tout $P\in \fc(M_0)$ admet une unique facteur de Levi qui contient $M_0$ et qui est noté $M_P$.  On a donc  $P=M_PN_P$. Soit $A_P=A_{M_P}$. Tout facteur de Levi d'un sous-groupe parabolique de $G$ est appelé un sous-groupe de Levi de $G$. Soit $\lc(M_0)$ l'ensemble de sous-groupes de Levi de $G$ qui contient $M_0$. Pour tout $M\in\lc(M_0)$, on pose 
    \begin{align*}
      &\fc(M)=\{P\in\fc(M_0) : P\supset M\}, \\ 
      &\pc(M)=\{P\in\fc(M) : M_P=M\}, \\
      &\lc(M)=\{L\in\lc(M_0) : L\supset M\}. 
    \end{align*}
  \end{paragr}

  \begin{paragr}   Soit $P$   sous-groupe parabolique de $G$. Soit $X^*(P)$ le groupe des caractères algébriques de $P$ définis sur $F$. Soit $\ago_P^*=X^*(P)\otimes\RR$ et $ \ago_P=\Hom_\ZZ(X^*(P),\RR)$ : ce sont des $\RR$-espaces vectoriels en dualité. Pour $P\subset Q$ des  sous-groupes paraboliques standard, on dispose  de la restriction $\ago_Q^*\to \ago_P^*$  et  dualement de $\ago_P\to \ago_Q$. La première est injective et identifie  $\ago_Q^*$ à un sous-espace de $\ago_P^*$ et le noyau de la second est notée $\ago_P^Q$.  On a en fait aussi  $\ago_P^*=X^*(A_P)\otimes\RR$ et on dispose donc de morphismes $\ago_Q\to \ago_P$ qui est injectif et $\ago_P^*\to \ago_Q^*$ dont on note $\ago_P^{Q,*}$ le noyau. On dispose alors des sommes directes $\ago_P=\ago_P^Q\oplus \ago_Q$ et de  leurs duales. Les projections utilisées par la suite font implicitement référence à ces décompositions. Tout $X\in \ago_P$ s'écrira $X^Q+X_Q$ selon cette décomposition. On pose $\ago_{0}=\ago_{P_0}$,  $\ago_{0}^*=\ago_{P_0}^*$ etc.
  \end{paragr}

  \begin{paragr}[Algèbre de Lie.] ---
    En général, on note l'algèbre de Lie d'un groupe algébrique $G,P, M, N$ par la lettre gothique correspondante $\ggo,\pgo, \mgo,\ngo$. Cette convention  ne s'applique pas aux tores $A_G,A_P=A_{M_P}$ etc. 
  \end{paragr}
  
  \begin{paragr}  Soit $P$   sous-groupe parabolique standard. Soit $\Delta_{0}^P$, $\hat\Delta_0^P$, $\Delta_0^{P,\vee}$ et $\hat \Delta_0^{P,\vee}$ les ensembles respectifs des racines, des poids, des coracines, des copoids simples de $A_0$ dans $P_0\cap M_P$. On identifie  $\Delta_0^P$ et $\hat\Delta_0^P$ à des bases de  $ \ago_0^{P,*}$. Soit $\hat \Delta_0^{P,\vee}$ et $\Delta_0^{P,\vee}$ les bases duales respectives dans $\ago_0^P$. Pour $P\subset Q$ des  sous-groupes paraboliques standard, on définit aussi $\Delta_{P}^Q$ comme l'ensemble des racines simples de $A_P$ dans $M_Q\cap N_P$. C'est encore la base de $\ago_P^{Q,*}$ obtenue par projection de $\Delta_0^Q\setminus \Delta_0^P $ par $\ago_0^{Q,*}\to \ago_P^{Q,*}$. Soit $(\varpi_\al^\vee)_{\al\in \Delta_P^Q}$ la base de $\ago_P^Q$ duale de $\Delta_P^Q$ : l'ensemble des éléments de cette base est noté $\hat\Delta_P^{Q,\vee}$.
Pour tout $\al\in \Delta_P^Q$, on sait définir une coracine $\al^\vee$ (c'est clair si $P=P_0$ et en général $\al^\vee$ est la projection sur $\ago_P^Q$  de $\be^\vee$ pour $\be$ l'unique relèvement de $\beta$ à $\Delta_0^Q$). On note $\Delta_P^{Q,\vee}$ l'ensemble des $\al^\vee$ pour $\al\in \Delta_P^Q$. Soit $(\varpi_\al)_{\al\in \Delta_P^Q}$ la base de $\ago_P^{Q,*}$ duale de $(\al^\vee)_{\al\in \Delta_P^{Q}}$: ses éléments forment l'ensemble  $\hat\Delta_P^Q$.

  Pour tout sous-groupe parabolique standard $P$, soit 
$$\ago_P^{*,+}=\{\la\in \ago_P^*\mid \bg \la,\al^\vee\bd >0\, \forall\al\in \Delta_P\}.$$
Soit  $\overline{\ago_P^{*,+}}$ son adhérence dans $\ago^*_P$. De même, on définit $\ago_P^{+}$ en inversant le rôle des racines et coracines. On pose $\ago_P^{G,+}=\ago_P^{+}\cap \ago^G_P$.
\end{paragr}

\begin{paragr}\label{S:double-indice}
  Lorsqu'un sous-groupe parabolique est noté $P_i$, on pourra remplacer $P_i$ en exposant ou en indice dans les notations par $i$. Ainsi on note  $\ago_0=\ago_{P_0}$.  Un exposant omis dans une notation qui en exigerait  un est implicitement le groupe ambiant  $G$.
\end{paragr}

\begin{paragr}[Fonctions caractéristiques.] --- \label{S:car}
  Soit $P_1\subset P_2$ des sous-groupes paraboliques standard de $G$. On pose

\begin{align*}
\eps_1^2=\eps_{P_1}^{P_2}=(-1)^{\dim_\RR(\ago_{P_1}^{P_2})}.
\end{align*}

On définit  la fonction  $\tau_{P_1}^{P_2}$, resp. $\hat\tau_{P_1}^{P_2}$, resp. $\phi_{P_1}^{P_2}$,  caractéristique de l'ensemble  des  $H\in  \ago_0$ qui vérifient la condition 1 ci-dessous, resp. la condition 2, resp. la condition 3:

  \begin{enumerate}
  \item $\bg \al, H\bd >0$ pour tout $\al\in \Delta_1^2$;
  \item $\bg \varpi, H\bd >0$ pour tout $\varpi\in \hat\Delta_1^{2}$;
    \item $\bg \al, H\bd >0$ pour tout $\al\in \Delta_1^2$ et $\bg \al, H\bd \leq 0$ pour tout $\al\in \Delta_1\setminus \Delta_1^2$.
  \end{enumerate}
  On définit également $\sigma_{P_1}^{P_2}$ comme le produit de  $\phi_{P_1}^{P_2}$ et $\hat\tau_{P_2}^G$.  Observons que contrairement à  $\tau_{P_1}^{P_2}$ et  $\hat\tau_{P_1}^{P_2}$ la fonction $\sigma_{P_1}^{P_2}$  dépend non seulement de $P_1$ et $P_2$ mais aussi du groupe ambiant $G$. Pour $P_1=P_2$, on a d'ailleurs $\sigma_{P_1}^{P_1}$ est identiquement nulle sauf si $P_1=G$ auquel cas elle vaut identiquement $1$. Notons qu'on a pour $P_1\subset P\subset P_2$ (cf. \cite[lemme 6.1]{ar1})
  \begin{align}\label{eq:sigma 12}
  \tau_{P_1}^{P} \hat\tau_P= \sum_{P\subset P_2} \sigma_{P_1}^{P_2}.
  \end{align}
\end{paragr}

\begin{paragr}  \label{S:cdn}Supposons désormais que $F$ est, de plus,  un corps de nombres. 
Soit $V_F$ l'ensemble des places de $F$ et $\AAA$ l'anneau des adèles de $F$. Soit $F_\infty=F\otimes_\QQ \RR$. On a $\AAA=F_\infty\times \AAA_f$ où  $\AAA_f$ désigne l'anneau des \og adèles finis\fg.  Pour tout $v\in V_F$ soit $F_v$ le complété de $F$ en $v$ et $|\cdot|_v$ la valeur absolue normalisée correspondante sur $F_v$. Soit $\oc_v\subset F_v$ l'anneau des entiers. Soit $|\cdot|_\AAA$  le morphisme $\AAA^\times\to \CC^\times$ donné par le produit sur $v\in V_F$ des valeurs absolues normalisées $|\cdot|_v$.
\end{paragr}

\begin{paragr}[Sous-groupe compact maximal.]\label{S:compact} --- Soit $K=\prod_{v\in V}K_v$ où, pour tout $v\in V_F$, le sous-groupe $K_v\in G(F_v)$ est un sous-groupe compact maximal en bonne position par rapport au sous-groupe de Levi minimal  $M_0$  fixé au §\ref{S:P_0M_0} (plus précisément  $K_v$  doit satisfaire les conditions de \cite[p.9]{arthur2}) . 

Soit $P\in \fc(M_0)$  un sous-groupe parabolique  semi-standard de $G$. On a alors $G(\AAA)=P(\AAA)K$. L'accouplement $(\chi, x)\mapsto \log |\chi(x)|_\AAA$ sur $X^*(P) \times P(\AAA)$ définit un morphisme de groupe $P(\AAA)\to \ago_P$ qu'on étend en une application $H_P:G(\AAA)\to \ago_P$ par $H_P(pk)=H_P(p)$ pour tous $p\in P(\AAA)$ et $k\in K$.

Pour tout groupe réductif $H$ défini sur $F$, soit $A_H^\infty$ la composante neutre du groupe des $\RR$-points du sous-$\QQ$-tore déployé maximal de $\Res_{F/\QQ}(A_H)$. On pose $A_P^\infty=A_{M_P}^\infty$. Soit  $P(\AAA)^1$ le noyau de $H_P$. On a $P(\AAA)^1=M_P(\AAA)^1N_P(\AAA)$ où  $M_P(\AAA)^1=M_P(\AAA)\cap P(\AAA)^1$. Alors on a $M_P(\AAA)\simeq M_P(\AAA)^1\times A_P^\infty$ et la restriction de $H_P$ induit un isomorphisme  $A_P^\infty\simeq \ago_P$. 
Soit $A_P^{G,\infty}=A_P^\infty\cap G(\AAA)^1$.
On note 
$$[G]_P=N_P(\AAA)M_P(F)\back G(\AAA) \text{   et  } [G]_P^1=N_P(\AAA)M_P(F)\back G(\AAA)^1 .$$
Si $P=G$ on omet l'indice $G$. Plus généralement, on note $[N]=N(F)\back N(\AAA)$ pour tout sous-groupe $N$ défini sur $F$.

\end{paragr}

\begin{paragr}[Mesures de Haar.] ---\label{S:Haar}
On munit  $K$ de la mesure de Haar qui donne un volume total égal à $1$. Pour tout sous-groupe unipotent $N$ de $G$, on munit $N(\AAA)$ de la mesure de Haar qui donne le volume $1$ au quotient $[N]$ lorsque celui-ci est muni de la mesure quotient de la mesure sur $N(\AAA)$ par la mesure de comptage sur $N(F)$.

On fixe des mesures de Haar sur $G(\AAA)$ et sur $M_P(\AAA)$ pour tout sous-groupe parabolique semi-standard $P$  de sorte que l'application produit $M_P(\AAA)\times N_P(\AAA)\times K\to G(\AAA)$ préserve les mesures.

Soit $W$  le groupe de Weyl de $(G,M_0)$ c'est-à-dire le quotient du normalisateur de $M_0$ dans $G(F)$ par $M_0(F)$.  Le groupe $W$ agit sur $\ago_{P_0}$ et on fixe un produit scalaire $W$-invariant sur $\ago_{P_0}$. On note $\|\cdot\|$ la norme associée. Chaque sous-espace  est alors muni de cette norme et de la mesure euclidienne correspondante. Ainsi on obtient une mesure sur $\ago_P$ et $\ago_P^G$ Par transport par $H_P$, on en déduit une mesure de Haar sur $A_P^\infty$ et $A_P^{G,\infty}$. Finalement, on met sur $M_P(\AAA)^1$ la mesure telle que l'application  produit $A_P^\infty\times M_P(\AAA)^1\to M_P(\AAA)$ soit compatible aux mesures choisies.

Pour toute fonction absolument intégrable sur $[G]$ et  tout sous-groupe parabolique semi-standard $P$ de $G$, on a alors les formules d'intégration  
\begin{align}
    \label{eq:Iwasawa mes}
    \int_{P(F)\back G(\AAA)} f(x)\, dx&= \int_{[N]} \int_{[M_P]} \int_K \exp(-\bg 2\rho_P, H_P(m)\bd)  f(nmk)\,dndmdk \\
    \nonumber &=\int_{[N]} \int_{[M_P]^1} \int_{A_P^\infty}\int_K \exp(-\bg 2\rho_P, H_P(a)\bd)  f(namk)\,dndadmdk
\end{align}
où $2\rho_P\in \ago_P^*$ est la somme des racines de $A_P$ dans $N_P$. Bien sûr, on a aussi des variantes de ces formules pour l'intégration sur $P(F)\back G(\AAA)^1$ . On pourra noter $\rho_P^G$ l'élément $\rho_P$ si on éprouve le besoin de souligner la dépendance en le groupe $G$.
\end{paragr}

\begin{paragr}[Point $T_0$ et partition.] --- \label{S:T0} Soit $T_1,T\in \ago_0$ et $P$ un sous-groupe parabolique  standard  de $G$. On définit
  \begin{align*}
    A_{P_0}^{P,\infty}(T_1)=\{ a\in A_0^\infty\mid \langle \alpha, H_0(a)\rangle\geqslant \langle \alpha, T_1\rangle , \; \forall \alpha\in \Delta^P_0 \}
  \end{align*}
et
\begin{align*}
    A_{P_0}^{P,\infty}(T_1,T)=\{ a\in  A_0^{P,\infty}(T_1)\mid \langle \varpi, H_0(a)\rangle\leqslant \langle \varpi, T\rangle , \; \forall \varpi\in \hat\Delta^P_0 \}.
  \end{align*}
  On fixe  $T_-\in \ago_0^{G}$  et $\omega_0\subset P_0(\AAA)^1$ un ensemble  compact tel que $P_0(\AAA)^1=P_0(F)\omega_0$. On introduit l'ensemble de Siegel
\begin{equation*}
  \SG^P=\SG^P_{P_0}=\omega_0 A_{P_0}^{P,\infty}(T_-)K.
\end{equation*}
Comme il est loisible, on choisit en fait $T_-\in -\ago_0^{G,+}$ de sorte qu'on a $G(\AAA)=P(F)\SG^P$ (et ce quel que soit $P$).  Pour tout  $T\in \overline{\ago_0^+}$, soit $F^P(\cdot,T)$ la fonction caractéristique  de $P(F)\om_0 A_0^{P,\infty}(T_-,T)K\subset \SG^P$, cf.  \cite[section 6]{ar1}. On voit $F^P(\cdot,T)$ comme une fonction sur $[G]_P$.   D'après \cite[Lemma 2.1] {ar-unip}, il existe $T_0\in \overline{\ago_0^{G,+}}$ tel que, pour tout $T\in T_0+\overline{\ago_0^{G,+}}$, la fonction $F^P(\cdot,T)$ est la fonction caractéristique de l'ensemble 
\begin{align}\label{eq:def FP}
  \{g\in [G]_P \mid   \bg \varpi,H_0(\delta g)-T\bd \leq 0 \,\forall \varpi\in \hat\Delta^P_0, \, \forall\delta\in P(F) \},
\end{align}
et ce quel que soit $P$. Soit $Q$ un sous-groupe parabolique standard de $G$. Pour tout $T\in T_0+\overline{\ago_0^{G,+}}$, on dispose de la partition de $[G]_Q$ donnée par l'identité (cf. \cite[lemme 6.4]{ar1})
\begin{align}\label{eq:partition}
  \forall g\in G(\AAA) \   \sum_{P\in \fc^Q(P_0)} \sum_{\delta \in P(F)\back Q(F)}F^P(\delta g,T) \tau_P^Q(H_P(\delta g)-T)=1.
\end{align}

En suivant \cite[§ 2.3.4]{BPC}, pour tout $\la\in \ago_0^*$, on définit pour tout $g\in [G]_Q$ 
\begin{align}\label{eq:dQ}
  d^Q(\la,g)= \sum_{P\in \fc^Q(P_0)} \sum_{\delta \in P(F)\back Q(F)}F^P(\delta g,T_0) \tau_P^Q(H_P(\delta g)-T_0)\exp(\bg \la, H_{P}(\delta g)\bd).
\end{align}

\begin{lemme}\label{lem:unicite} Soit $Q\subset R$ des sous-groupes paraboliques standard. Il existe $c\geq 1$ tel que pour tout $g\in G(\AAA)$, il existe au plus un élément $\delta \in Q(F)\back R(F)$ such that $d^Q(\al,\delta g)>c$ pour tout $\al\in \Delta_0^R\setminus \Delta_0^Q$.
 \end{lemme}

 \begin{preuve} C'est une variante de \cite[proposition 2.3.4.3]{BPC}. On prend $c\geq 1$ tel que $\log(c) > \bg \al, T_{0}\bd$ pour tout  sous-groupe parabolique standard $P$ et tout $\al\in \Delta_P$.   
   Soit  $g\in G(\AAA)$  et $\delta_1,\delta_2\in R(F)$ tels que  qu'on ait $d^Q(\al, \delta_ig)>c$  pour tout $\al\in \Delta_0\setminus \Delta_0^Q$ et tout $i\in \{1,2\}$.  Il s'agit de prouver que $\delta_1\delta_2^{-1}\in Q(F)$. On peut multiplier à gauche $\delta_i$ par un élément de $Q(F)$ sans que cela ne change ni l'hypothèse ni la conclusion.   Soit  $i\in \{1,2\}$. En utilisant la partition \eqref{eq:partition} pour $[G]_Q$,    on peut donc supposer qu'il existe un sous-groupe parabolique standard  $P_i\subset Q$   tel que $F^{P_i}(\delta_i g,T_0) \tau_{P_i}^Q(H_{P_i}( \delta_i g)-T_0)=1$. Soit $\al\in \Delta_0^R\setminus \Delta_0^Q$. La condition $d^Q(\al, \delta_i g)>c$ implique alors qu'on a $\bg \al,H_P(g)>\log(c)$.  Par conséquent, on a $\bg \al, H_{P_i}(\delta_i g)-T_0\bd >0$ pour tout $\al\in \Delta_{P_i}^R\setminus \Delta_{P_i}^Q$. Il vient donc 
   \begin{align*}
     F^{P_i}(\delta_i g,T_0) \tau_{P_i}^R(H_{P_i}( \delta_i g)-T_0)=1.
   \end{align*}
On peut de nouveau invoquer la  partition \eqref{eq:partition} cette fois-ci pour $[G]_R$ : il s'ensuit qu'on a $P_1(F)\delta_1=P_2(F)\delta_2$ et $P_1=P_2$. Comme $P_i\subset Q$ le résultat est clair. 
\end{preuve}

\end{paragr}

\begin{paragr}[Hauteurs.] --- \label{S:hauteur}On fixe une hauteur notée $\|\cdot\|$ sur $G(\AAA)$ comme dans \cite[§ I.2.2]{MWlivre} auquel on renvoie pour ses principales propriétés.

 Soit $P\subset G$ un sous-groupe parabolique. Pour tout $g\in G(\AAA)$, on pose
\begin{align*}
  \|g\|_P=\inf_{\delta\in N_P(\AAA)M_P(F)}\|\delta g\|.
\end{align*}
On voit $\|\cdot\|_P$ comme une fonction sur $[G]_P$. Pour tout ensemble $X$ et toutes applications $f,g:X\to \RR_+$ on écrit $f\ll g$, resp $f\prec g$, s'il existe $c>0$, resp. et $d>0$, tel que pour tout $x\in X$ on a $f(x)\leq c g(x)$, resp.  $f(x)\leq c g(x)^d$. On dit que $f$ et $g$ sont équivalentes, resp. comparables,  et on note $f\sim g$, resp. $f\asymp g$,  si  $f\ll g$ et $g\ll f$, resp $f\prec g$ et $g\prec f$,. Comme fonction sur $G(\AAA)$, on a $\|\cdot\|_G\ll \|\cdot\|_P \ll\|\cdot\|$ comme il résulte de la compacité de  $[N_P]$. Par ailleurs, les restrictions à $\SG^P$ de  $\|\cdot\|_P$ et $\|\cdot\|$ sont équivalentes (cela résulte de  \cite[§ I.2.2 (ii) et (vii)]{MWlivre}).

Pour tout sous-groupe fermé $H\subset G$, la restriction de la hauteur sur $G(\AAA)$ à $H(\AAA)$ est comparable à une hauteur sur $H(\AAA)$.
\end{paragr}

\begin{paragr} 
  
\begin{lemme}
    \label{lem:Fsigma*d variante}
Soit $P_1\subset P_2\subset G$ des sous-groupes paraboliques standard.  Pour tous $N>0$ et $T\in T_0+\overline{\ago_0^+}$, il  existe $C,c>0$ tel que pour tout 
 \begin{align}\label{eq:la 12 variante}
\la=\sum_{\al\in \Delta_0^{2}\setminus \Delta_0^{1}}c_\al \al       \text{   avec }  c_\al>c
 \end{align}
 on a, pour tout $g\in [G]_{P_1}^1$,  
\begin{align}\label{eq:Fsigma*d variante}
  F^{P_1}(g,T)  \sigma_1^2(H_{P_0}(g)-T) d^{P_1}(-\la,g)\leq C  \|g\|_{P_1}^{-N}.
\end{align}
\end{lemme}

\begin{preuve} C'est une variante de la preuve de \cite[lemma 3.3.4.1]{BPC}. Pour la commodité du lecteur, on reprend l'argument. On peut et on va se restreindre au  cas où   $g$ appartient à  l'ensemble $\SG^{P_1}\cap G(\AAA)^1$. 
  Il existe $c_1,c_2>0$ tels que, pour tout $g\in \SG^{P_1}\cap G(\AAA)^1$, on a $\|g\|_{P_1}\leq c_1\exp(c_2 \|H_{P_0}(g)\|)$.  Soit  $T\in T_0+\overline{\ago_0^+}$. On suppose que $g$ vérifie en outre
  \begin{align}\label{eq:Fsinot=0}
    F^{P_1}(g,T)  \sigma_1^2(H_{P_0}(g)-T) \not=0.
  \end{align}
  Il s'agit pour de tels $g$ d'évaluer la norme de $\|H_{P_0}(g)\|$. Selon la décomposition $\ago_0^G=\ago_0^{P_1}\oplus \ago_{P_1}^{P_2}\oplus \ago_{P_2}^G$, on écrit $H_{P_0}(g)=  H_{P_0}(g))^{P_1}  +    H_{P_1}(g)^{P_2}  + H_{P_2}(g)  $. 
  \begin{align*}
    \|H_{P_0}(g)\|\leq \|  H_{P_0}(g)^{P_1}  \| +\|    H_{P_1}(g)^{P_2}  \| +\| H_{P_2}(g)  \|.
  \end{align*}
  Puisque    $F^{P_1}(g,T)  \not=0$, il existe $c_3>0$ tel que $\|  (H_{P_0}(g)^{P_1}  \|\leq c_3(1+\|T^{P_1}\|)$.  Puisque $\sigma_1^2(H_{P_0}(g)-T) \not=0$, il existe $c_4>0$ tel que $\|  H_{P_2}(g)  \|\leq c_4(1+\|     (H_{P_1}(g))^{P_2} -T_{P_1}^{P_2}\|+ \|T_{P_2}^G\| )$, cf. \cite[corollary 6.2]{ar1}. Il existe $c_5>0$ tel que  $\|     (H_{P_1}(g))^{P_2} -T_{P_1}^{P_2}\|\leq c_5  \sum_{\al\in \Delta_1^2 } \bg \al, H_0(g)-T\bd$ où, dans la somme, on observe qu'on a   $\bg \al, H_0(g)-T\bd>0$, cf. § \ref{S:car}. Toujours, sous la condition \eqref{eq:Fsinot=0}, on montre que pour tout $\al\in \Delta_0^{2}\setminus \Delta_0^{1}$, on a  $\bg \al, H_0(g)-T\bd \geq \bg \bar\al, H_0(g)-T\bd>0$ où $\bar\al$ est la projection de $\al$ sur $\ago_1^*$. Finalement, on obtient $c_6>0$ tel que
  \begin{align*}
     \|H_{P_0}(g)\|\leq c_6(1+\|T^G\|+ \sum_{ \al\in \Delta_0^{2}\setminus \Delta_0^{1}  } \bg \al, H_0(g)\bd). 
  \end{align*}
  On note que, dans la somme ci-dessus, on a $\bg \al, H_0(g)\bd>\bg \al, T\bd \geq 0$. L'énoncé est alors évident puisqu'on peut  remplacer $d^{P_1}(-\la,g)$ par $\exp(-\bg\la, H_{P_0}(g)\bd)$, cf. \cite[proposition 2.3.4.1]{BPC}. On observe également que les constantes $c_i$ qui apparaissent sont indépendantes du choix de $T$.
\end{preuve}
\end{paragr}

\begin{paragr}[Point suffisamment positif.] --- \label{S:suffisam}Pour $T\in \ago_0$, on définit $d(T)=\min_{\al\in \Delta_0} \bg \al,T\bd$. Dans toute la suite, on fixe $\eps>0$ assez petit et on dira qu'une assertion vaut  \og pour tout $T\in \ago_0$  suffisamment positif \fg{} si elle vaut pour  tout  $T\in T_0+\overline{\ago_0^+}$  tel que $d(T)\geq \eps\|T^G\|$.  Cette notion interviendra via le lemme suivant.

  \begin{lemme}
    \label{lem:Fsigma*d}
Soit $P_1\subset P_2\subset G$ des sous-groupes paraboliques standard. Pour tous $N,r>0$, il  existe $C,c>0$ tel que pour tout 
 \begin{align}\label{eq:la 12}
\la=\sum_{\al\in \Delta_0^{2}\setminus \Delta_0^{1}}c_\al \al       \text{   avec }  c_\al>c
 \end{align}
 on a 
\begin{align}\label{eq:Fsigma*d}
  F^{P_1}(g,T)  \sigma_1^2(H_{P_0}(g)-T) d^{P_1}(-\la,g)\leq C \exp(-r \|T^G\|) \|g\|_{P_1}^{-N}
  \end{align}
pour tout  $g\in [G]_{P_1}^1$  et tout $T\in \ago_0$   suffisamment positif.
\end{lemme}

\begin{preuve} C'est une conséquence immédiate de la preuve du lemme \ref{lem:Fsigma*d variante} vu que, si $T$ est  suffisamment positif et $g\in \SG^{P_1}\cap G(\AAA)^1$ vérifie \eqref{eq:Fsinot=0}, pour tout $\al\in \Delta_0^{2}\setminus \Delta_0^{1}$, on a  $\bg \al, H_0(g)\bd \geq \bg \al, T\bd\geq  \eps \|T^G\|$. 
\end{preuve}

\end{paragr}

\begin{paragr} Soit $\uc(\ggo_\CC)$ l'algèbre enveloppante de l'algèbre de Lie complexifiée de $G$.  Cette algèbre agit sur l'espace $ C^\infty(G(\AAA))$  des fonctions lisses sur $G(\AAA)$ par les représentations régulières à droite $R$ et à gauche $L$. Soit  $J\subset G(\AAA_f)$ un sous-groupe ouvert compact.  Soit $Q\subset G$ un sous-groupe parabolique. Soit $C^\infty([G]_Q)^J\subset C^\infty(G(\AAA)) $ le sous-espace des fonctions lisses sur $G(\AAA)$ invariante à gauche par $N_Q(\AAA)M_Q(F)$ et à droite par $J$. Pour tout entier $N\geq 0$, toute partie   finie $\Fgo\subset \uc(\ggo_\CC)$ et tout $\varphi\in C^\infty([G]_Q)^J$ on pose
\begin{align}
  \label{eq:norme}
&    \|\varphi\|_{N,\Fgo}=\sup_{x\in [G]_Q^1,X\in \Fgo}   \|x\|_Q^{-N} |R(X)\varphi(x)|.
\end{align}

Soit $\varphi\in C^\infty([G]_Q)^J$. Pour tout  sous-groupe parabolique standard $P\subset Q$, on définit un élément  $\varphi_P\in  C^\infty([G]_P)^J$ par 
\begin{align}\label{eq:terme cst}
\forall g\in [G]_P \ \  \varphi_P(g)=\int_{[N_P]} \varphi(ng)\, dn.
\end{align}
Pour des  sous-groupes paraboliques standard $P_1\subset P_2\subset Q$ et tout $g\in G(\AAA)$ on pose
\begin{align}\label{eq:phi12}
  \varphi_{1,2}(g)=\sum_{ P_1\subset P\subset P_2 }\eps_P^G \varphi_P(g).
\end{align}

\begin{lemme}\label{lem:phi12} Supposons $P_1\subsetneq P_2\subset Q$. Pour  tout
 \begin{align*}
\la=\sum_{\al\in \Delta_0^{2}\setminus \Delta_0^{1}}c_\al \al       \text{   avec }  c_\al\geq 0
 \end{align*}
 il existe un ensemble fini  $\Fgo\subset \uc(\ggo_\CC)$ tel que
 \begin{align}\label{eq:ineg 12}
   |  \varphi_{1,2}(g) |\leq \exp(-\bg\la,H_{P_0}(g)\bd) \|g\|_{Q}^N   \|\varphi\|_{N,\Fgo}
 \end{align}
  pour tous $g\in \SG^{P_2}\cap G(\AAA)^1$,  $N>0$,  et $\varphi\in C^\infty([G]_Q)^J$.
\end{lemme}

\begin{preuve} Le cas $P_2=Q$ implique le résultat plus général: en effet, ce cas  donne l'existence de $\Fgo\subset \uc(\ggo_\CC)$ fini tel que
   \begin{align*}
   |  \varphi_{1,2}(g) |\leq \exp(-\bg\la,H_{P_0}(g)\bd) \|g\|_{P_2}^N   \|\varphi_{P_2}\|_{N,\Fgo}
 \end{align*}
pour tout  $g\in \SG^{P_2}\cap G(\AAA)^1$,  $N>0$,  et $\varphi\in C^\infty([G]_Q)^J$. Cette majoration donne effectivement \eqref{eq:ineg 12}: d'une part les fonctions  $\|\cdot\|_{Q}$ et $\|\cdot\|_{P_2}$ sont équivalentes sur $\SG^{P_2}$ et, d'autre part, il existe $C>0$ tel que  pour tout $\varphi\in C^\infty([G]_Q)^J$, on a 
\begin{align*}
  \|\varphi_{P_2}\|_{N,\Fgo} \leq C  \|\varphi\|_{N,\Fgo}.
\end{align*}

On suppose désormais $P_2=Q$.  Il suffit ensuite de prouver le résultat lorsque $\la=c \al$ avec $c>0$ et $\al\in \Delta_0^{2}\setminus \Delta_0^{1}$. On est ramené au cas où $P_1\subsetneq Q$ est un sous-groupe parabolique maximal  \cite[preuve du corollaire I.2.11]{MWlivre}. Ce dernier  cas se traite comme dans \cite[preuve du lemme I.2.10]{MWlivre}).

\end{preuve}

\begin{lemme}\label{lem:majoration} Supposons $P_1\subsetneq P_2\subset Q$. 
\begin{enumerate}
    \item Pour  tous $r,N_1,N_2>0$, il existe un ensemble  fini $\Fgo\subset \uc(\ggo_\CC)$ tel que 
  \begin{align*}
     F^{P_1}(x,T ) \sigma_{P_1}^{P_2}(H_{P_1}(x)-T)  |\varphi_{P_1,P_2}(x)| \leq  \exp(-r\|T^G\|)\|x\|^{-N_1}_{P_1}   \|\varphi\|_{N_2,\Fgo}.   \end{align*}
pour tout $x\in P_1(F)N_2(\AAA)\back G(\AAA)^1$, tout $T\in \ago_0$ suffisamment positif et tout $\varphi\in  C^\infty([G]_Q)^J$.
\item Pour  tous $N_1,N_2>0$ et $T\in T_0+\overline{\ago_0^+}$, il existe un ensemble  fini $\Fgo\subset \uc(\ggo_\CC)$ tel que 
  \begin{align*}
     F^{P_1}(x,T ) \sigma_{P_1}^{P_2}(H_{P_1}(x)-T)  |\varphi_{P_1,P_2}(x)| \leq  \|x\|^{-N_1}_{P_1}   \|\varphi\|_{N_2,\Fgo}.   \end{align*}
pour tout $x\in P_1(F)N_2(\AAA)\back G(\AAA)^1$ et tout $\varphi\in  C^\infty([G]_Q)^J$.
\end{enumerate}
\end{lemme}

\begin{preuve} Soit $T\in T_0+\overline{\ago_0^+}$. Soit  $x\in P_1(F)N_2(\AAA)\back G(\AAA)^1$ tel que  $  F^{P_1}(x,T ) \sigma_{P_1}^{P_2}(H_{P_1}(x)-T)\not=0$. Soit $g\in G(\AAA)^1$ un relèvement de $x$. On peut bien sûr supposer qu'on a  $g\in \SG^{P_1}$. La condition   $F^{P_1}(g,T ) \sigma_{P_1}^{P_2}(H_{P_0}(g)-T) \not=0$ entraîne alors  que, pour tout $\al\in \Delta_0^2\setminus \Delta_0^1$, on a $\bg \al, H_0(g)-T\bd > 0$. Il s'ensuit qu'on a $g\in \SG^{P_2}$. Soit $\la$ comme dans le lemme \ref{lem:phi12}.  On a donc un ensemble fini $\Fgo\subset \uc(\ggo_\CC)$  tel que pour tout $N>0$ on ait
  \begin{align*}
     F^{P_1}(g,T ) \sigma_{P_1}^{P_2}(H_{P_1}(g)-T)  |\varphi_{P_1,P_2}(g)| \leq   F^{P_1}(g,T ) \sigma_{P_1}^{P_2}(H_{P_1}(g)-T) \exp(-\bg\la,H_{P_0}(g)\bd) \|g\|_{Q}^{N}   \|\varphi\|_{N,\Fgo}.
  \end{align*}
  Sur  $\SG^{P_1}$, \emph{a fortiori} sur $\SG^{P_2}$, les fonctions  $\exp(-\bg\la,H_{P_0}(\cdot)\bd)$ et $d^{P_1}(-\la,\cdot)$ d'une part et  $\|\cdot\|_{Q}$ et $\|\cdot\|_{P_1}$ d'autre part sont équivalentes. Il suffit alors d'utiliser  les lemmes \ref{lem:Fsigma*d} et \ref{lem:Fsigma*d variante} pour conclure. 
\end{preuve}

\end{paragr}

\begin{paragr}[Espace de Schwartz.] ---\label{S:Schwartz}
  Soit $G$ un groupe algébrique affine, connexe,  défini sur $F$ (dans ce § on ne le suppose  pas nécessairement réductif). Soit  $X$ une variété algébrique affine définie sur $F$ sur laquelle $G$ agit à droite\footnote{Bien sûr, tout le § vaut encore pour une action à gauche.} qui fait de $X$ un espace homogène. On suppose qu'il existe un point rationnel $e\in X(F)$. Ainsi $X\simeq H\back G$ où $H$ est stabilisateur de $e$: c'est un sous-groupe fermé de $G$ défini sur $F$. La seule situation qu'on rencontrera dans cet article est celle où le choix du point $e\in X(F)$ donne une identification de $X(\AAA)$ avec $H(\AAA)\back G(\AAA)$. Autrement dit, on suppose dans la suite que $G(\AAA)$ agit transitivement sur $X(\AAA)$. 

  L'ensemble  $X(\AAA)$ est muni naturellement d'une topologie associée à celle sur $\AAA$. On note $\|\cdot\|$ une hauteur sur $X(\AAA)$ (cf. la construction des \og normes abstraites \fg{} dans \cite[A.1]{RBP}).  Le groupe $G(\AAA)$ agit sur $X(\AAA)$ et par dualité sur l'ensemble des applications de $X(\AAA)$ dans $\CC$. Une application $f:X(\AAA)\to \CC$ est dite lisse s'il existe  $J\subset G(\AAA_f)$  un sous-groupe compact ouvert tel que $f$ est $J$-invariante et si pour tout $X\in X(\AAA)$ l'application $g\in G(\RR)\mapsto f(xg)$ est lisse au sens usuel.

   Soit $C\subset X(\AAA_f)$ un compact et $J\subset G(\AAA_f)$  un sous-groupe compact ouvert. Soit $\Sc(X(\AAA),C)^J$  l'espace des fonctions lisses,  invariantes par $J$, à support dans $X(F_\infty)\times C$ et  telles que les semi-normes définies par 
  \begin{align}\label{eq:Norm rXY}
\|f\|_{r,Y}=    \sup_{x\in X(\AAA)} \|x\|^r | R(Y)f(x)| 
  \end{align}
  pour $r\geq 1$ et  $Y\in \uc(\ggo_\CC)$ soient finies. Ces semi-normes définissent une topologie sur  $\Sc(X(\AAA),C)^J$ qui fait de celui-ci un espace de Fréchet. L'espace de Schwartz  $\Sc(X(\AAA))$ est la limite topologique localement convexe des espaces  $\Sc(X(\AAA),C)^J$. C'est un espace LF strict.

  Un cas particulier intéressant est lorsqu'on a $X=G$ (et le groupe $H$ est trivial). Le groupe $G$ agit en fait à gauche et à droite sur lui-même et $\Sc(G(\AAA))$ est une algèbre pour le produit de convolution. Pour tout  sous-groupe compact ouvert  $J\subset G(\AAA_f)$, on note   $\Sc(G(\AAA))^J$  la sous-algèbre des fonctions $J$-invariantes à gauche et à droite. La contruction s'applique en particulier à l'algèbre de Lie $\ggo$ de $G$. Dans ce cas, $\Sc(\ggo(\AAA))$ est engendré par les fonctions $f_\infty\otimes f^{\infty}$ où $f_\infty$ est une fonction de Schwartz sur le $\RR$-espace vectoriel $\ggo(F_\infty)$ et $f^\infty$ est une fonction localement constante à support compact sur $\ggo(\AAA_f)$. 
  Revenons au cas d'un espace $X=H\back G$.  Soit $dh$ une mesure invariante à droite sur le groupe $H(\AAA)$. L'application
  \begin{align}\label{eq:integ fibre}
    f\mapsto    \left( H(\AAA)x\mapsto \int_{H(\AAA)} f(hx)\, dh\right)
  \end{align}
induit un morphisme continu surjectif de $\Sc(G(\AAA))$ sur $\Sc(X(\AAA))$. La continuité résulte de \cite[proposition A.1.1 (vii)]{RBP}. C'est aussi une application ouverte par le théorème de l'application ouverte.
\end{paragr}

\subsection{Sur certaines involutions des formes intérieures des groupes généraux linéaires}\label{ssec:involutions}

\begin{paragr}\label{S:D}
Dans cette section,  $F$ est un corps de caractéristique $0$. Soit $M$ une algèbre simple centrale de dimension finie $>0$ sur $F$. Il existe $N\geq 1$ et $D$ une algèbre à division centrale et de dimension finie sur $F$ tels que, pour $V=D^N$ muni de sa structure de $D$-module à droite, $M$ s'identifie à l'algèbre $\End_D(V)$ des endomorphismes $D$-linéaires de $V$. 
Soit $G$ le groupe multiplicatif de $M$ vu comme groupe algébrique sur $F$. Il s'ensuit que   $G$ s'identifie au groupe $G=GL_D(V)$  des automorphismes $D$-linéaires  de $V$. 
\end{paragr}

\begin{paragr}
  Soit $\theta\in G(F)$ tel que  $\theta^2=1$.  Soit
  \begin{align*}
      V^{\pm \theta}=\{v\in V \mid \theta(v)=\pm v  \}.
  \end{align*}
  Lorsque le contexte est clair, on pose $ V^\pm=V^{\pm \theta}$.   On a alors une décomposition en sous-$D$-modules
$$V=V^+\oplus V^-.$$
Soit  $G^{\theta}$ le centralisateur de $\theta$. C'est un groupe réductif connexe défini sur $F$ qui s'identifie à $GL_D(V^+)\times GL_D(V^-)$. C'est aussi un facteur de Levi du sous-groupe parabolique maximal de $G$ qui stabilise $V^+$ (resp. qui stabilise $V^-$).

Plus généralement pour tout sous-groupe $P\subset G$  on note $P^\theta=P\cap G^\theta$ le centralisateur de $\theta$ dans $P$.

\begin{lemme}\label{lem:parab-H}
  \begin{enumerate}
  \item   Soit $P$ un sous-groupe parabolique de $G$. On a les équivalences entre les trois assertions :
  \begin{enumerate}
  \item $P^\theta$ est un sous-groupe parabolique de $G^\theta$ ;
  \item $\theta\in P$ 
  \item $\theta P\theta=P$
  \end{enumerate}
\item Soit  $P$ un sous-groupe parabolique de $G$ contenant $\theta$.
  \begin{enumerate}
  \item Pour tout facteur de Levi $M$ de $P$, le groupe  $M^\theta$ est un facteur de Levi de $P^\theta$ si et seulement si on a $\theta\in M$. 
  \item Il existe un facteur de Levi $M$ de $P$ contenant $\theta$.
  \end{enumerate}
 \item Tout sous-groupe parabolique de $G^\theta$ est de la forme $ P^\theta$ où  $P$ un sous-groupe parabolique de $G$ qui contient $\theta$. 
  \item Il existe $P_0\subset G$ un sous-groupe parabolique  minimal de $G$ et  $M_0$ un facteur de Levi de $P_0$ tels que $M_0$ contient $\theta$. Pour un tel couple $(P_0,M_0)$, on voit que $M_0$ est inclus dans $G^\theta$ et que $P_0^\theta$ est un sous-groupe parabolique minimal de $G^\theta$. 
  \end{enumerate}
\end{lemme}

\begin{preuve}
L'équivalence entre 1.b et 1.c est bien connue. L'assertion 1.a implique  l'assertion 1.b car, si  $P^\theta$ est un sous-groupe parabolique de $G^\theta$ alors il contient le centre de $G^\theta$ donc $\theta$. Montrons que l'assertion 1.b implique  l'assertion 1.a. Le sous-groupe parabolique $P$ est le stabilisateur d'un drapeau $V_0\subsetneq V_1\subsetneq V_2\subsetneq \ldots \subsetneq V$  de sous-$D$-modules de $V$. Comme $\theta\in P$, pour tout $i$, le sous-module $V_i$ est stable par $\theta$ et donc se décompose en $V_i=V_i^+\oplus V_i^-$ où $V_i^\pm=V_i\cap V^\pm$. Un élément appartient à $P^\theta$ si et seulement s'il stabilise les  drapeaux respectifs de $V^\pm$ définis par $V_0^\pm\subset V_1^\pm\subset V_2^\pm\subset\ldots \subset V^\pm$. Mais cette condition définit clairement un sous-groupe parabolique de $G^\theta$.

Prouvons l'assertion 2.a. La condition est nécessaire car si $M^\theta$ est un facteur de Levi de $P^\theta$ il contient le centre de $G^\theta$ donc $\theta$. La condition est suffisante : avec les notations ci-dessus, la donnée supplémentaire de $M$ est équivalente à la donnée pour tout $i$ d'une décomposition $V_{i}=V_{i-1}\oplus W_{i}$ et $M$ est le stabilisateur de chaque $W_i$. Puisque $\theta\in M$, on a $W_i=W_i^{+}\oplus W_i^-$ où $W_i^\pm=W_i\cap V^\pm$. Il s'ensuit que $M^\theta$ est le sous-groupe de $G^\theta$ qui stabilise chaque $W_i^{\pm}$ et ce dernier est un facteur de Levi de $M^\theta$.  Pour obtenir l'assertion 2.b, il suffit de choisir un supplémentaire arbitraire $W_{i+1}^\pm$ de $V_i^\pm$ dans $V_{i+1}^\pm$. Le facteur de Levi de $P$ défini comme le  stabilisateur des sous-modules $W_i=W_i^+\oplus W_i^-$ contient $\theta$.

Prouvons 3. Soit $Q$ un sous-groupe parabolique de $G^\theta$ défini comme le stabilisateur des drapeaux $(V_\bullet^\pm)$ des sous-modules $V^\pm$. Soit $P$ le stabilisateur dans $G$ du drapeau formé des sous-modules $V_0^+\subset V_1^+\subset \ldots\subset V^+\subset V^+\oplus V_1^-\subset \ldots \subset  V^+\oplus V^-$. Alors $P$ contient $\theta$ et $P^\theta=Q$. De plus, si $Q$ est minimal alors chaque quotient  $V_i^\pm/V_{i-1}^\pm$ est de rang $1$ et $P$ est aussi minimal. Réciproquement si $(P_0,M_0)$ est un couple minimal de $G$ avec $\theta\in M_0$ alors $M_0$ est le stabilisateur de sous-modules  de rang $1$ dont la  somme directe est $V$. Comme $\theta$ agit nécessairement par $\pm 1$ sur chacun de ces sous-modules  il est clair que $M_0\subset G^\theta$.  Il est évident aussi que $P_0^\theta$ est minimal dans $G_0^\theta$.
\end{preuve}
\end{paragr}

\begin{paragr}\label{S:P_0M_0}
  Soit  $P_0\subset G$ un sous-groupe parabolique défini sur $F$ et minimal pour ces propriétés et  $M_0$ un facteur de Levi de $P_0$ tels que $\theta\in M_0$ (de tels objets $M_0$ et $P_0$ existent, cf. lemme \ref{lem:parab-H} assertion 4).
  
Un sous-groupe parabolique de $G$ sera dit standard, resp. semi-standard, resp.   relativement standard, s'il contient $P_0$, resp. $M_0$, resp. $P_{0}^\theta$. Si le contexte n'est pas évident, on pourra  remplacer l'adverbe \og relativement \fg{} par \og $\theta$-relativement \fg{}. Standard implique relativement standard qui implique lui-même semi-standard. On rappelle que  $M_0$ est inclus dans $G^\theta$.  Un sous-groupe parabolique de $G^\theta$ sera dit standard, resp. semi-standard,  s'il contient $P_{0}^\theta$, resp. $M_0$.

 Soit 
$$\fc(P_0)\subset \fc(P_{0}^{\theta})\subset  \fc(M_0)\subset \fc(\theta)$$ les ensembles de sous-groupes paraboliques de $G$ qui sont  respectivement standard, relativement standard, semi-standard, resp. qui contiennent $\theta$. De même on définit $\fc^\theta(P_0^\theta)\subset   \fc^\theta(M_0)$  les ensembles de sous-groupes paraboliques de $G^\theta$ qui sont  respectivement standard, semi-standard. 

Le groupe $G^\theta(F)$ agit par conjugaison sur $\fc(\theta)$. Notons que l'application
$$P\mapsto P^\theta$$
induit des applications surjectives de $\fc(\theta)$  dans l'ensemble des sous-groupes paraboliques de $G^\theta$, de $\fc(P_{0}^\theta)$ dans $\fc^\theta(P_0^\theta)$, de $\fc(M_0)$ dans $\fc^\theta(M_0)$.
\end{paragr}

\begin{paragr} Rappelons que $W$ est le groupe de Weyl de $(G,M_0)$, cf. § \ref{S:Haar}.  De même, on définit $W^\theta$ le groupe de Weyl de  $(G^\theta,M_0)$.  Naturellement $W^\theta$ est un sous-groupe de $W$. Le groupe $W$ agit sur $\fc(M_0)$. Les orbites de $W$ rencontrent toutes $\fc(P_0)$. De même pour $W^\theta$ qui agit sur $\fc^\theta(M_0)$ relativement à $\fc(P_{0}^\theta)$. Pour tout $P\in \fc(M_0)$, soit $W^P$ le stabilisateur de  $P$ dans $W$ pour cette action.

Pour tout $w\in W$ et $P\in \fc(M_0)$, on pose
$$P_w=w^{-1}Pw  \text{    et    }  P_w^\theta= P_w\cap G^\theta.$$

Soit  $P\in \fc(M_0)$ et 
\begin{align*}
  _PW_\theta=\{w\in W\mid  M_P\cap P_0=M_P\cap wP_0 w^{-1}  \text{   et   } P_0^\theta \subset P_w\}.
\end{align*}

\begin{proposition}\label{prop:PWtheta}
  \begin{enumerate}
  \item   Pour tout  $P\in \fc(M_0)$ et toute double classe $\dot{w}\in W^P\back W/W^\theta$,  il existe un unique représentant de  $\dot{w}$ dans  $_PW_\theta$.
  \item L'application $(P,w)\mapsto P_w$ induit une bijection de la réunion disjointe 
\begin{align*}
  \bigcup_{P\in\fc(P_0)}  \,  _PW_\theta
\end{align*}
sur l'ensemble des classes de $G^\theta(F)$-conjugaison de $\fc(\theta)$.
\end{enumerate}
\end{proposition}

\begin{preuve}
 1.  Soit $w\in W$. On a $P_w^\theta \in \fc^\theta(M_0)$.  Il existe donc $w_\theta\in W^\theta$ tel que $w_\theta^{-1} P_w^\theta w_\theta\in \fc^\theta(P_{0}^\theta)$. Quitte à remplacer $w$ par $ww_\theta$, on peut supposer qu'on a $P_0^\theta\subset P_w$. Quitte à remplacer $w$ par $w'w$ avec $w'\in W^P$, ce qui ne change pas $P_w$, on peut supposer qu'on a aussi $M_P\cap P_0=M_P\cap wP_0 w^{-1}$. On a donc $w\in \,_PW_\theta$ ce qu'on suppose désormais. Soit $w_P\in W^P$ et $w_\theta\in W^\theta$ tels que $w_1=w_Pww_\theta\in  \,_PW_\theta$. On a alors $P_w^\theta$ et $P_{ww_\theta}^\theta$ sont deux sous-groupes paraboliques standard et conjugués de $G^\theta$ : ils sont donc égaux. Il s'ensuit que $w_\theta\in  W^{P_w^\theta}$. On a donc $w_1=w' w$ avec $w'=w_P ww_\theta w^{-1}\in W^P$. Il s'ensuit que $M_P\cap P_{0,w_1^{-1}}$ et $M_P\cap P_{0,w^{-1}}$ sont conjugués par $w'$. Comme ils sont tous deux égaux  à $M_P\cap P_0$, on a en fait $w'=1$ et donc $w_1=w$.

2.  Montrons d'abord que l'application en question est injective. Soit $(P,w)$ et $(P_1,w_1)$ tels que $P_w$ et $(P_1)_{w_1}$ soient $G^\theta(F)$-conjugués. Il existe $h\in G^\theta(F)$ tel que $(wh)^{-1} Pwh=w_1^{-1} P_1 w_1$. Comme $P$ et $P_1$ sont alors standard et conjugués, on a $P=P_1$.  Puis $P_w^\theta$ et $P_{w_1}^\theta$  appartiennent à $\fc^\theta(P_0^\theta)$ et sont $G^\theta(F)$-conjugués : ils sont donc aussi égaux et $h\in  P_w^\theta$. Ainsi  $P_w= P_ {w_1}$. Il est facile de conclure.

Montrons finalement que l'application en question est surjective.  Soit $P$ un  sous-groupe parabolique de $G$ qui contient $\theta$ c'est-à-dire $P^\theta$ est un sous-groupe parabolique de $G^\theta$. Quitte à  conjuguer $P$ par un élément de $G^\theta(F)$, on peut supposer que $P$ contient $P_0^\theta$. En particulier, $P$ est semi-standard. Il existe $w\in W$ tel que $wPw^{-1}$ est standard. Alors $P=w^{-1}(wPw^{-1})w$ ce qui conclut. 
\end{preuve}

On a aussi la variante suivante de cette proposition. 

\begin{corollaire}\label{cor:PWtheta}
  \begin{enumerate}
  \item   Pour tout  $P\in \fc(M_0)$ et toute double classe $\dot{w}\in P(F)\back P(F) W G^\theta(F)/G^\theta(F)$,  il existe un unique représentant de  $\dot{w}$ dans  $_PW_\theta$.
  \item L'application $(P,w)\mapsto P_w$ induit une bijection de la réunion disjointe 
\begin{align*}
  \bigcup_{P\in\fc(P_0)}  \,  _PW_\theta
\end{align*}
sur $\fc(P_0^\theta)$ c'est-à-dire l'ensemble de sous-groupes paraboliques relativement standard de $G$. 
\end{enumerate}
\end{corollaire}

\begin{preuve}
Pour l'assertion 1, avec l'existence dans la proposition \ref{prop:PWtheta}, il nous reste à montrer l'unicité. Soit $w, w_1\in\,_PW_\theta$ des représentants de $\dot{w}$. Il existe $p\in P(F)$ et $h\in G^\theta(F)$ tels que $w_1=p w h$. On a alors $P_{w_1}=h^{-1}P_w h$. Si $P$ est standard, l'injectivité dans la proposition \ref{prop:PWtheta} entraîne $w_1=w$ ce qui conclut. Mais il est évident que son argument est valable pour tout $P$ semi-standard. 

Pour l'assertion 2, en vertu de la bijectivité dans la proposition \ref{prop:PWtheta}, il suffit de montrer que l'application canonique de $\fc(P_0^\theta)$ sur l'ensemble des classes de $G^\theta(F)$-conjugaison de $\fc(\theta)$ est injective. Pour cela, on peut reprendre la preuve de l'injectivité dans la proposition \ref{prop:PWtheta}. 
\end{preuve}
\end{paragr}

\begin{paragr}   Soit $Q$ un sous-groupe parabolique semi-standard de $G$ et soit $\fc^Q$, resp. $\fc(Q)$, l'ensemble des sous-groupes paraboliques inclus dans $Q$, resp. contenant $Q$. On pose $\fc^Q(P_0)=\fc^Q\cap \fc(P_0)$, $\fc^Q(P_0^\theta)=\fc^Q\cap \fc(P_0^\theta)$ et $\fc^Q(\theta)=\fc^Q\cap \fc(\theta)$. On pose aussi $\pc^Q(P_0^\theta)=\pc(M_0)\cap\fc^Q(P_0^\theta)$.
Supposons que $Q$ est de plus standard.  On définit alors 
\begin{align*}
    \,_PW^Q_\theta= \,_PW_\theta\cap W^Q.
  \end{align*}
On a le corollaire immédiat.

\begin{corollaire}\label{cor:classe-Q_H}
  L'application $(P,w)\mapsto P_w$ induit une bijection de la réunion disjointe 
\begin{align*}
  \bigcup_{P\in\fc^Q(P_0)} \,   _PW^Q_\theta
\end{align*}
sur l'ensemble des classes de $Q^\theta(F)$-conjugaison de $\fc^Q(\theta)$ ou sur $\fc^Q(P_0^\theta)$.
\end{corollaire}
\end{paragr}

\begin{paragr} [Convention de notation.] ---   \label{S:wtheta}Pour tout $w\in W$, les résultats qui précèdent et ceux qui vont suivre valent lorsqu'on remplace $\theta$ par $w\theta w^{-1}$. Afin d'alléger  la notation des objets définis relativement à $w\theta w^{-1}$, on remplacera souvent l'indice ou l'exposant $w\theta w^{-1}$ par $w\theta$.  On espère que cela ne sera pas source de confusion. Notons qu'on a alors la formule $wG^\theta w^{-1}=G^{w\theta}$.
\end{paragr}

\begin{paragr} Finissons cette sous-section par quelques lemmes.
  
  \begin{lemme} \label{lem:Wdec}Soit $P\subset Q\subset R$ trois sous-groupes paraboliques standard de $G$.
  \begin{enumerate}
  \item Pour tout $w\in \, _QW^R_\theta$ et tout $w_1  \in \, _PW^Q_{w\theta}$, on a $w_1w\in   \, _PW^R_\theta$.
  \item Réciproquement pour tout $w_0\in  \, _PW^R_\theta$, il existe un couple unique $(w,w_1)$ comme ci-dessus tel que $w_0=w_1w$.
  \end{enumerate}
\end{lemme}

\begin{preuve}

  1. Soit $w_0=w_1w$. On a
  \begin{align*}
    M_P\cap w_0P_0 w_0^{-1}\subset M_Q\cap w_0P_0 w_0^{-1}&= w_1(M_Q\cap w P_0 w^{-1})w_1^{-1}\\
&=w_1(M_Q\cap P_0)w_1^{-1}\\
&=M_Q\cap w_1 P_0 w^{-1}_1.
  \end{align*}
Donc on a 
\begin{align*}
    M_P\cap w_0P_0 w_0^{-1}\subset M_P\cap w_1 P_0 w^{-1}_1=M_P\cap P_0\end{align*}
d'où l'égalité  $M_P\cap w_0P_0 w_0^{-1}=M_P\cap P_0$. Par ailleurs, on a $P_0^{w\theta}\subset P_{w_1}$. Il s'ensuit qu'on a  $    P_{0,w}^\theta=w^{-1} P_0^{w\theta} w   \subset P_{w_1w}=P_{w_0}$. On conclut avec le lemme \ref{lem:une-observation} ci-dessous appliqué à $w$ qu'on a $P_{0}^\theta=   P_{0,w}^\theta\subset P_{w_0}$.

2. Supposons $w_0\in  \, _PW^R_\theta$. On a donc  $P_0^\theta\subset P_{w_0}\subset Q_{w_0}$. Il existe alors un unique $w_1\in W^Q$ tel que  $w=w_1^{-1}w_0$ vérifie $M_Q\cap   w P_0 w^{-1}=M_Q\cap P_0$ et $P_0^\theta\subset Q_{w}$. Alors $w\in \,_QW^R_\theta$. On a donc l'unicité mais aussi  l'existence car  on obtient $w_1\in \, _PW^Q_{w\theta}$: en effet, on a d'une part, à l'aide du  lemme \ref{lem:une-observation},
\begin{align*}
  P_0^{w\theta} =w(P_{0,w}^\theta )w^{-1} =w(P_{0}^\theta )w^{-1}  \subset  P_{w_0 w^{-1}}=P_{w_1} 
\end{align*}
et d'autre part
\begin{align*}
  M_P\cap w_1P_0 w_1^{-1}\subset M_Q \cap w_1P_0 w_1^{-1}= w_1(M_Q\cap P_0)w_1^{-1}\\
= w_1(M_Q\cap   w P_0 w^{-1})w_1^{-1}=M_Q\cap (w_0 P_0 w_0^{-1})=M_Q\cap P_0 .
\end{align*}

\end{preuve}

\begin{lemme}
  \label{lem:une-observation}
  Soit $Q$ un sous-groupe parabolique standard de $G$. Pour tout  $w\in  \, _QW_\theta$, on a 
\begin{equation}
  \label{eq:conclusion-clef}
  P_0^{\theta}=P_{0, w}^{\theta}.
\end{equation}
\end{lemme}

\begin{preuve} Par hypothèse sur $w$, on a  $P_0^\theta\subset Q_w^\theta$. Puisque $Q$ est standard, on a aussi $P_{0, w}^{\theta}\subset Q_w^\theta$. Comme il s'agit de sous-groupes paraboliques de $G^\theta$, on en déduit qu'on a $P_0^\theta=(P_0\cap M_{Q_w})^\theta N_{Q_w^\theta}$ et $P_{0, w}^{\theta}= (P_{0, w}\cap M_{Q_w})^\theta N_{Q_w^\theta}$. Or, toujours  par hypothèse sur $w$, on a aussi $P_{0, w}\cap M_{Q_w}=w^{-1}(P_0\cap M_Q)w=w^{-1}(P_{0,w^{-1}}\cap M_Q)w=P_0\cap M_{Q_w}$. L'égalité \eqref{eq:conclusion-clef} s'ensuit.

\end{preuve}

\begin{lemme}\label{lem:Wconj}
    Soit $P\in\fc(M_0)$ et $w_1\in W$ tels que $M_P\cap P_0=M_P\cap w_1 P_0 w_1^{-1}$. L'application $w_2\mapsto w_1^{-1}w_2$ induit une bijection de $\, _PW_\theta$ sur $_{P_{w_1}}W_\theta$. 
\end{lemme}

\begin{preuve}
    On note que  
  \begin{align*}
    w_1^{-1}(M_P\cap P_0)w_1=w_1^{-1}(M_P\cap w_1P_0w_1^{-1})w_1=M_{P_{w_1}}\cap P_0.
  \end{align*}    
Il s'ensuit que pour $w_2\in W$
  \begin{align*}
    M_{P_{w_1}}\cap P_0=M_{P_{w_1}}\cap w_1^{-1}w_2P_0 w_2^{-1}w_1
  \end{align*} 
si et seulement si 
  \begin{align*}
    w_1^{-1}(M_P\cap P_0)w_1=w_1^{-1}(M_P\cap w_2P_0 w_2^{-1})w_1
  \end{align*}
c'est-à-dire 
  \begin{align*}
    M_P\cap P_0=M_P\cap w_2P_0 w_2^{-1}. 
  \end{align*}
Il est facile de conclure. 
\end{preuve}

\end{paragr}

\subsection{Une partition relative}\label{ssec:partition}

\begin{paragr} On se place dans la situation de la sous-section \ref{ssec:involutions} avec l'hypothèse supplémentaire que  $F$ est un corps de nombres. Les autres notations sont empruntées à la sous-section \ref{ssec:combi}. Plus précisément, on fixe $D$  et $V=D^N$ comme au § \ref{S:D}. Soit $(e_1,\ldots,e_N)$ la base canonique de $D^N$.  Soit  $G=GL_D(V)$ et $P_0\subset G$ le stabilisateur des sous-espaces respectivement engendrés par $(e_1,\ldots,e_i)$ pour $1\leq i \leq N$. Alors $P_0$ est un sous-groupe parabolique de $G$ défini sur $F$ et minimal pour ces propriétés.  Soit  $M_0\subset G$ le facteur de Levi défini sur $F$  de $P_0$  obtenu comme le stabilisateur des droites engendrées par  $e_i$ pour $1\leq i \leq N$.  On a $M_0(F)= (D^\times)^N$. 
\end{paragr}

\begin{paragr}[Choix de sous-groupes compacts maximaux.]\label{S:choixdeK} ---  Soit $\oc_D\subset D$ un ordre maximal de $D$ et $A\subset F$ l'anneau des entiers. Soit $v \in V_F$ une place non-archimédienne. Alors $\oc_D\otimes_{A} \oc_v$ est un ordre maximal dans $\oc_D\otimes_F F_v$. Il existe une algèbre à division $D_v$ sur $F_v$ et $n_v\geq 1$ un entier de sorte que, si $\oc_{D_v}$ désigne l'unique ordre maximal de $D_v$, l'algèbre $\oc_D\otimes_F F_v$ s'identifie à $M(n_v, D_v)$ et l'ordre maximal qu'on a considéré s'identifie à $M(n_v, \oc_{D_v})$. On peut alors identifier le groupe $G(F_v)$ à  $GL(n_vN,D_v)$. Soit $K_v\subset G(F_v)$ le sous-groupe compact maximal qui correspond à $GL(n_vN,\oc_{D_v})$.

Soit $v \in V_F$ une place archimédienne. Il existe une algèbre à division $D_v$ sur $F_v$ et $n_v\geq 1$  de sorte que l'algèbre $\oc_D\otimes_F F_v$ s'identifie à $M(n_v, D_v)$. L'algèbre $D_v$ est munie d'une involution notée $x\mapsto \bar x$ (qui est triviale si $D_v=\RR$ et qui est en fait une anti-involution si $D_v$ est une algèbre de quaternions). Le groupe $G(F_v)$ s'identifie alors à   $GL(n_vN,D_v)$. Soit $K_v\subset G(F_v)$ le sous-groupe compact maximal qui correspond au groupe des matrices $M\in GL(n_vN,D_v)$  telles que $^t \bar M M$ est l'identité.

On pose alors $K=\prod_{v\in V_F} K_v$. C'est un sous-groupe compact maximal de $G(\AAA)$ en bonne position par rapport à $M_0$, au sens de § \ref{S:compact}.
\end{paragr}

\begin{paragr}[Groupe de Weyl.] --- On identifiera dans toute la suite  les éléments du groupe de Weyl $W$ de $(G,M_0)$ à des matrices de permutations dans $G(F)$, à savoir les matrices qui permutent les éléments de la base canonique. On observe qu'alors $W$ est un sous-groupe de $G(F)\cap K$.
    \end{paragr}

\begin{paragr}[Élément $\theta$.] --- Soit $\theta\in M_0(F)$ un élément d'ordre au plus $2$. Soit $I_{\pm }=\{1\leq i\leq N| \theta(e_i)=\pm e_i  \}$. Soit $V_\pm$ le sous-$D$-module de base $(e_i)_{i\in I_\pm}$. Alors $V$ est la somme directe de $V_+$ et $V_-$ et le groupe $G^\theta$ s'identifie à $GL_D(V_+)\times GL_D(V_-)$. Vu le choix de $K$, on a $\theta\in K$. Pour tout sous-groupe parabolique $P$ de $G$ contenant $M_0$, on dispose de l'application $H_P$ définie au §\ref{S:compact} relativement au sous-groupe compact $K\subset G(\AAA)$. Observons que, pour tout $g\in G(\AAA)$, on a
\begin{align}
  \label{eq:invariance-HP}
H_P(\theta g)=H_P(g).
\end{align}

Soit $K^\theta=\prod_{v\in V_F}K_v^\theta$ où $K_v^\theta=K_v\cap G^\theta(F_v)$. On observe que $K_v^\theta$  est alors un sous-groupe compact maximal de $G^\theta(\AAA)$ en bonne position par rapport au sous-groupe de Levi minimal $M_0$ de $G^\theta$. C'est ce sous-groupe $K^\theta$ qu'on choisit comme sous-groupe compact maximal de référence de $G^\theta(\AAA)$.
\end{paragr}

\begin{paragr} On utilise les hypothèses et les notations du § \ref{S:T0}.  Soit $g\in G(\AAA)$ et  $Q\in \fc^G(P_0)$. Appelons   couple $(Q,T)$-canonique de $g$ l'unique  couple $(P, \delta)$ avec $P\in \fc^Q(P_0)$ et $\delta\in  P(F)\back Q(F)$  qui donne une contribution non-nulle dans \eqref{eq:partition}. 
En utilisant \eqref{eq:invariance-HP} et le fait que $\theta\in K$, on voit si $(P, \delta)$ est le couple $(Q,T)$-canonique de $g$ alors $(P, \delta\theta^{-1})$ est celui de $\theta g \theta^{-1}$. Par conséquent, pour $g\in G^\theta(\AAA)$, ce couple $(Q,T)$-canonique vérifie $\theta\in \delta^{-1}P\delta$. À l'aide du corollaire  \ref{cor:classe-Q_H}, on en déduit une partition de $[G^\theta]_{Q^\theta}$ donnée pour tout $h\in G^\theta(\AAA)$ par 
\begin{align}\label{eq:partition-H}
  \sum_{P\in \fc^Q(P_0)} \sum_{ w\in \, _PW^Q_\theta }    \sum_{\delta \in P_{w}^\theta(F)\back Q^\theta(F)}F^P(w\delta h,T) \tau_P^Q(H_P(w\delta h)-T)=1.
\end{align}
Notons que comme $w\in\,  _PW^Q_\theta $ on a $P_w\in \fc^Q(P_{0}^\theta)$ et donc on a $P_w^\theta\subset Q^\theta$.
\end{paragr}

\begin{paragr} \label{S:et les ss?} Les fonctions  $\tau_P^Q$ et $\hat\tau_P^Q$ introduites pour des sous-groupes paraboliques standard $P\subset Q$ valent plus généralement pour des sous-groupes paraboliques semi-standard : on vérifie que leur définition ne dépend pas du choix d'un élément $P_0\in \pc(M_0)$. En revanche, la définition de  la fonction $F^P(\cdot,T)$  dépend du choix de $P_0\in\pc(M_0)$.  On introduit alors  $F^P_B(\cdot,T)$ la fonction définie au § \ref{S:T0}  relativement à un
  sous-groupe $B\in\pc(M_0)$ et au sous-groupe parabolique semi-standard $P$ contenant $B$. Pour tout $P\in\fc(M_0)$ et $T\in\ago_0$, soit $T_P$ la projection de $wT$ sur $\ago_P$, où $w\in W$ est tel que $P_0\subset P_w$. Cette définition ne dépend pas du choix de $w$. Par exemple, pour $T\in\overline{\ago_0^+}$, on a $T_B\in\overline{\ago_B^+}$ pour tout $B\in\pc(M_0)$. Avec cette notation, on voit que $T_{P_w}=w^{-1}(T_P)$ est la projection de $w^{-1}T$ sur $\ago_{P_w}$ pour tout $w\in W$.

  \begin{lemme} \label{lem:nonstd} Soit $w\in W$, $g\in G(\AAA)$ et $T\in T_0+\overline{\ago_0^+}$
  \begin{enumerate}
      \item  Pour tout sous-groupe parabolique standard $P$, on a 
      \begin{align*}
          F^P(wg,T)=F^{P_w}_B(g, T_{P_w})
      \end{align*}
      avec $B=(P_0)_w$.
      \item pour tous sous-groupes paraboliques semi-standard $P\subset Q$, on a
      \begin{align}\label{eq:tauw}
  \tau_{P_w}^{Q_w}(H_{P_w}(g)-T_{P_w})=\tau_P^Q(H_P(wg)-T_P) \text{  et  } \hat\tau_{P_w}^{Q_w}(H_{P_w}(g)-T_{P_w})=\hat\tau_P^Q(H_P(wg)-T_P).
\end{align}
\end{enumerate}
  \end{lemme}
  
  \begin{preuve}
     On a  $W\subset G(F)\cap K$. Il s'ensuit que  si $B=(P_0)_w$ on a \begin{align*}
          H_{P_0}(wg)=wH_{B}(g).
      \end{align*}
  Le lemme s'ensuit.
  \end{preuve}

  Pour tout $P\in\fc(M_0)$ et $T\in T_0+\overline{\ago_0^+}$, on définit 
\begin{align*}
  F^P(\cdot,T)=F^P_B(\cdot,T_B) 
\end{align*}
où $B$ est un élément quelconque de $\pc(M_0)$ inclus dans $P$. Il résulte du lemme \ref{lem:nonstd} assertion 1 et du fait que la fonction $F^P(\cdot,T)$ est invariante à gauche par $P(F)$, que cette définition ne dépend pas du choix de $B$. Avec cette définition, l'assertion 1 du lemme \ref{lem:nonstd} devient 
\begin{align*}
  F^{P_w}(x,T)=F^P(wx,T)
\end{align*}
pour tout $w\in W$ et tout $x\in G(\AAA)$. 

  À l'aide du corollaire \ref{cor:classe-Q_H}, on peut réécrire l'identité \eqref{eq:partition-H}. Pour tout $Q\in\fc(P_0)$, $T\in T_0+\overline{\ago_0^+}$ et $h\in G^\theta(\AAA)$, on a 
\begin{align}\label{eq:varpartition-H}
  \sum_{P\in \fc^Q(P_0^\theta)} \sum_{\delta \in P^\theta(F)\back Q^\theta(F)}F^P(\delta h,T) \tau_P^Q(H_P(\delta h)-T_P)=1.
\end{align}
En fait,  cette identité vaut plus généralement pour tout $Q\in\fc(P_0^\theta)$. En effet, pour un tel $Q$, il existe  $B\in\pc(M_0)$ inclus dans $Q$. Il s'ensuit que $B^\theta$ est un sous-groupe parabolique minimal de $Q^\theta$ qui admet $M_0$ comme facteur de Levi. Il existe donc un élément $m$ de $M_{Q^\theta}(F)$ qui normalise $M_0$ et tel que $mB^\theta m^{-1}=P_0^\theta$.   Quitte à remplacer $B$ par $mB m^{-1}$ on peut supposer qu'on a de plus $B\in\pc(M_0)\cap\fc^Q(P_0^\theta)$. Il suffit alors d'appliquer \eqref{eq:partition-H} avec $B$ à la place de $P_0$. 

\begin{remarque}
  L'identité \ref{eq:varpartition-H} apparaît dans \cite[lemme 4.4]{li1}. La preuve qu'on donne ici est légèrement différente de celle de \cite{li1}.  
\end{remarque}
\end{paragr}

\begin{paragr} En suivant le § \ref{S:T0}, pour tout $P\in \fc(P_0^\theta)$, on dispose  d'ensembles de Siegel relatifs à $G^\theta(\AAA)$
  \begin{equation*}
  \SG^{P^\theta}=\SG^{P^\theta}_{P_0^\theta}=\omega_{P_0^\theta} A_{P_0^\theta}^{P^\theta,\infty}(T_-)K^\theta.
\end{equation*}
On suppose que $T_-\in -\ago_0^{G,+}$ vérifie $G(\AAA)=P(F)\SG^P$ et $G^\theta(\AAA)=P^\theta(F)\SG^{P^\theta}$. On définit alors pour tout $B\in\pc^P(P_0^\theta)$ et $T\in T_0+\overline{\ago_0^+}$ l'ensemble
\begin{align*}
  \SG^{M_{P}^{\theta}}(B,T_B)=\left(\omega_{P_0^\theta}\cap M_{P}^{\theta}(\AAA)\right) \left( A_{P_0^\theta}^{P^\theta,\infty}(T_-)    \cap  A_{B}^{P,\infty}(T_-,T_B)\right) \left( K^\theta\cap  M_{P}^{\theta}(\AAA)\right).
\end{align*}

On utilisera le lemme suivant : 
  
  \begin{lemme}(\cite[corollaire 4.13]{li1})\label{lem:relsiegel}
    Pour tout  $T\in T_0+\overline{\ago_0^+}$, on a
    \begin{align*}
  \{m\in M_P^\theta(\AAA) \mid F^P(m,T)=1\}= \bigcup_{B\in\pc^P(P_0^\theta)} M_P^\theta(F) \SG^{M_{P}^{\theta}}(B,T_B).
    \end{align*}
  \end{lemme}

\end{paragr}

\subsection{Un opérateur de troncature}\label{ssec:op-tronc}

\begin{paragr}
  On continue avec les notations de la section \ref{ssec:partition}. On définit aussi $[G^\theta]^G=G^\theta(F)\back (G^\theta(\AAA)\cap G(\AAA)^1)$. 
\end{paragr}

\begin{paragr}[Opérateur de troncature.] ---   Soit $Q\in \fc(P_0)$. Soit $T\in \ago_0$. On définit un opérateur  $\Lambda^{T,Q}_\theta$ de la façon suivante:   pour toute fonction continue $\varphi$ sur $[G]_Q$ on pose pour $x\in G(\AAA)$
  \begin{align}\label{eq:LaTtheta}
    (\Lambda^{T,Q}_\theta\varphi)(x)= \sum_{P\in \fc^Q(P_0)}  \eps_P^Q \sum_{w \in  \, _PW^Q_\theta}   \sum_{\delta \in P_{w}^\theta(F)\back Q^\theta(F)}\hat\tau_P^Q(H_P(w\delta x)-T)  \varphi_P(w\delta x)
  \end{align}
où $\varphi_P$ est le terme constant de $\varphi$ donné par \eqref{eq:terme cst}. On notera que dans \eqref{eq:LaTtheta} la somme sur $\delta$ est en fait finie (cf. \cite[lemme 5.1]{ar1}). On étend la définition ci-dessus à toute fonction continue $\varphi$ sur $[G]$ en posant : 
 \begin{align*}
    \Lambda^{T,Q}_\theta\varphi= \Lambda^{T,Q}_\theta\varphi_Q.
 \end{align*}
 Comme d'habitude, on omet l'exposant lorsqu'il s'agit du groupe $G$ lui-même. 
 
 \begin{remarque}\label{rq:Zydor}
   L'opérateur de troncature défini ci-dessus a de grandes similitudes avec celui défini par Zydor dans \cite[section 3.7]{Zydor} pour le sous-groupe $G^\theta$ (Zydor travaille dans un contexte bien plus général que nous). Ainsi, via la bijection $(P,w)\mapsto P_w$ du corollaire \ref{cor:PWtheta}, on peut identifier les ensembles de sommation avec ceux de Zydor. En utilisant le lemme  \ref{lem:nonstd} et nos choix de sous-groupes compacts maximaux, on voit qu'on a  
   \begin{align*}
   & \hat\tau_P^Q(H_P(wx)-T)=\hat\tau_P^Q(H_P(wx)-T_P) =\hat\tau_{P_w}^{Q_w}(H_{P_w}(x)-T_{P_w})=\hat\tau_{P_w}^{Q_w}(H_{P_w^\theta}(x)-T_{P_w}),\\
   & \varphi_P(wx)=\varphi_{P_w}(x),
   \end{align*} 
   pour tous $x\in G^\theta(\AAA)$,  $P\in \fc^Q(P_0)$ et $w \in  \, _PW^Q_\theta$. Cependant, la troncature que Zydor utilise repose sur l'expression $\hat\tau_{P_w}^{Q_w}(H_{P_w^\theta}(x)-T)$ \emph{qui ne coïncide pas}  avec celle définie ci-dessus.
   Notre opérateur n'est donc pas égal à celui de Zydor. Illustrons-le sur l'exemple du groupe $G=GL(2)$ sur le corps $F$ avec $\theta$  la matrice diagonale dont les entrées sont $1$ et $-1$. Le groupe $G^\theta$ est alors  le tore maximal diagonal.  Le paramètre $T$ s'identifie à un couple $(t_1,t_2)\in \RR^2$ tel que $t_1-t_2$ est assez positif. Soit $\phi$ la fonction sur $G(\AAA)$ constante égale à $1$. Soit $B$ le sous-groupe de Borel standard et $\bar B$ son opposé par rapport à $G^\theta$. On a alors $T_{\bar B}= (t_2,t_1)$. Soit $x\in G^\theta(\AAA)$ qu'on identifie de manière évidente à un couple $(x_1,x_2)\in \AAA^\times\times \AAA^\times$. On a alors $H_B(x)=H_{\bar B}(x)=(\log(|x_1|),\log(|x_2|))$.  En notant $\La^T_z$ l'opérateur de Zydor, on obtient
   \begin{align*}
       (\La^T_z \phi)(x)= 1 - \hat\tau_B(H_B(x)-T) - \hat\tau_{\bar B}(H_B(x)-T). 
   \end{align*}
   La fonction $(\La^T_z \phi)$ sur $G^\theta(\AAA)$  est la fonction caractéristique des $x\in G^\theta(\AAA)$ tels que $\log(|x_1||x_2|^{-1})=t_1-t_2$. Par ailleurs, on observe qu'on a 
   \begin{align*}
       (\La^T_\theta \phi)(x)= 1 - \hat\tau_B(H_B(x)-T) - \hat\tau_{\bar B}(H_B(x)-T_{\bar B})
   \end{align*}
   et que la fonction $(\La^T_\theta \phi)$ sur $G^\theta(\AAA)$  est ainsi la fonction caractéristique des $x\in G^\theta(\AAA)$ tels que $\log(|x_1||x_2|^{-1})\in [t_2-t_1; t_1-t_2]$. En conséquence, pour la convergence ponctuelle, on a 
   \begin{align*}
   \lim_{t_1-t_2\to +\infty}(\La^T_z \phi)= 0   \ \ \ \  \lim_{t_1-t_2\to +\infty}(\La^T_\theta \phi)= \phi.
   \end{align*}
   En particulier, l'opérateur de Zydor ne vérifie pas l'assertion de la proposition \ref{prop:asym tronque} ci-dessous. Or, pour ce travail et ses développements subséquents, cf. \cite{caractererelatifpondere}, il importera d'avoir la propriété de droite et la validité de  la proposition \ref{prop:asym tronque} ci-dessous.
 \end{remarque}
 
\end{paragr}

\begin{paragr}[Propriété de décroissance.] --- Soit  $J\subset G(\AAA_f)$ un sous-groupe ouvert compact. 

\begin{proposition}\label{prop:dec-tronque}   Soit $T\in T_0+\overline{\ago_0^+}$. Pour tous $N_1,N_2>0$, il existe  une famille finie $\Fgo$ d'éléments de $\uc(\ggo_\CC)$ telle que pour tout pour  tout $x\in [G^\theta]^G$   et tout $\varphi\in  C^\infty([G])^J$, on ait 
  \begin{align*}
     |(\Lambda^{T}_\theta\varphi)(x)|  \leq    \|x\|^{-N_1}_{G^\theta}   \|\varphi\|_{N_2,\Fgo}.
  \end{align*}

\end{proposition}

\begin{preuve} Compte tenu de l'assertion 2 de la proposition \ref{prop:asym tronque} ci-dessous, il suffit de  majorer la fonction $x\mapsto F^G(x,T)\varphi(x)$ sur  $[G^\theta]^G$,  donc de  majorer  la fonction $x\mapsto F^{G}( x,T )\|x\|_{G^\theta}^{N_1} \|x\|_G^{N_2}$ sur  $[G^\theta]^G$. C'est évident puisque sur $[G^\theta]^G$ la fonction  $F^{G}( \cdot,T )$ est à support compact.
  
\end{preuve}

\begin{proposition}\label{prop:asym tronque}

  \begin{enumerate}
  \item Pour tout $r>0$ et tous $N_1,N_2>0$, il existe une famille finie $\Fgo$ d'éléments de $\uc(\ggo_\CC)$ telle que, pour tout $T$ suffisamment positif, on ait 
  \begin{align*}
  |(\Lambda^{T}_\theta\varphi)(x)-F^G(x,T)\varphi(x)| \leq  \exp(-r\|T^G \|)\|x\|^{-N_1}_{G^\theta}   \|\varphi\|_{N_2,\Fgo}
  \end{align*}
pour  tout $x\in [G^\theta]^G$   et tout $\varphi\in  C^\infty([G])^J$.
\item Soit $T\in T_0+\overline{\ago_0^+}$. Pour tous $N_1,N_2>0$, il existe une famille finie $\Fgo$ d'éléments de $\uc(\ggo_\CC)$ telle qu'on ait 
  \begin{align*}
  |(\Lambda^{T}_\theta\varphi)(x)-F^G(x,T)\varphi(x)| \leq \|x\|^{-N_1}_{G^\theta}   \|\varphi\|_{N_2,\Fgo}
  \end{align*}
pour  tout $x\in [G^\theta]^G$   et tout $\varphi\in  C^\infty([G])^J$.
  \end{enumerate}

\end{proposition}

\begin{preuve} On part de l'expression \eqref{eq:LaTtheta} pour $Q=G$. Soit $h\in G^\theta(\AAA)$. On injecte dans le terme associé à $P,w$ et $\delta$ dans la définition \eqref{eq:LaTtheta}  de $(\Lambda^{T}_\theta\varphi)(h)$ l'identité \eqref{eq:partition-H} appliquée à $w\theta w^{-1}$, $Q=P$ et l'élément $w\delta h w^{-1}\in G^{w\theta}(\AAA)$. On obtient:
\begin{align*}
    (\Lambda^{T}_\theta\varphi)(h)= \sum_{P\in \fc^G(P_0)}  \eps_P^G \sum_{w \in  \, _PW^G_\theta}   \sum_{\delta \in P_{w}^\theta(F)\back G^\theta(F)} \hat\tau_P^G(H_P(w\delta h)-T)  \varphi_P(w\delta h) \times \\
\left(\sum_{P_1\in \fc^P(P_0)}   \sum_{w_1 \in  \, _{P_1}W^P_{w\theta}}   \sum_{\gamma \in P_{1,w_1}^{w\theta}(F)\back P^{w\theta}(F)}
 F^{P_1}(w_1\gamma w\delta h,T) \tau_{P_1}^P(H_{P_1}(w_1\gamma w\delta h)-T)\right) .
  \end{align*}

Avec les notations ci-dessus, on a  $w^{-1}\gamma w\in P_{1,w_1w}(F)\back P_w^\theta(F)$ et  $w_1w\in   \, _{P_1}W^G_\theta$.  À l'aide du lemme \ref{lem:Wdec}, on peut intervertir la somme sur $P$ et sur $P_1$ puis  utiliser \eqref{eq:sigma 12}. On obtient   que $(\Lambda^{T}_\theta\varphi)(h)$ est égal à
\begin{align*}
 \sum_{P_0\subset P_1\subset P}  \eps_P^G \sum_{w \in  \, _{P_1}W^G_\theta}   \sum_{\delta \in P_{1,w}^\theta(F)\back G^\theta(F)}  F^{P_1}(w\delta h,T ) \tau_{P_1}^P(H_P(w\delta h )-T)  \hat\tau_P^G(H_P(w\delta h)-T)  \varphi_P(w\delta h ) 
  \end{align*}
qui vaut
\begin{align*}
  \sum_{P_0\subset P_1\subset P_2\subset G}   \sum_{w \in  \, _{P_1}W^G_\theta}   \sum_{\delta \in P_{1,w}^\theta(F)\back G^\theta(F)}  F^{P_1}(w\delta h,T ) \sigma_{P_1}^{P_2}(H_{P_1}(w\delta h)-T)  \varphi_{1,2}(w\delta h ) 
\end{align*}
où   $\varphi_{1,2}$ est défini en \eqref{eq:phi12}
Si $P_1=P_2$ alors la fonction $\sigma_{P_1}^{P_2}$ est identiquement nulle sauf si $P_1=P_2=G$. Dans ce cas, la contribution de $P_1=P_2=G$ est simplement  $F^{G}( h,T )\varphi( h)$. Ainsi, il suffit d'établir la majoration cherchée pour les termes correspondant à $P_1\subsetneq P_2$ fixés et $w\in  \, _{P_1}W_\theta$ fixés. En utilisant le lemme \ref{lem:nonstd}, on voit que quitte à remplacer $T$ par $wT$, $P_0$ par $(P_0)_w$, $P_i$ par $(P_i)_w$, $\delta$ par $w\delta w^{-1}$ , $h$ par $w\delta h w^{-1}$, $\theta$ par $w\theta w^{-1}$ et $\varphi$ par $\varphi(\cdot w)$  il suffit de traiter la somme
\begin{align}\label{eq:la somme}
  \sum_{\delta \in P_{1}^\theta(F)\back G^{\theta}(F)}  F^{P_1}(\delta h,T ) \sigma_{P_1}^{P_2}(H_{P_1}(\delta h)-T)  \varphi_{P_1,P_2}(\delta h ).
\end{align}

On conclut alors facilement par  le lemme  \ref{lem:majoration}.
\end{preuve}

\end{paragr}

\begin{paragr}[Formule d'inversion.] --- Pour des références futures, nous énonçons aussi la proposition suivante.
 
 \begin{proposition}\label{prop:inv}
Soit $\varphi$ une fonction continue sur $[G]$. Pour tout sous-groupe parabolique standard $Q$ et tout $x\in G(\AAA)$,  on a 
  \begin{align*}
    \varphi_Q(x)=\sum_{ P\in \fc^Q(P_0) } \sum_{w   \in \, _PW^Q_\theta} \sum_{\delta\in P_w^\theta(F)\back  Q^\theta(F)} \tau_{P}^Q(H_{P}(w \delta x)-T) (\Lambda^{T,P}_{w\theta}\varphi) (w \delta x),
  \end{align*}
  la somme sur $\delta$ étant en fait à support fini.
\end{proposition}

\begin{preuve}
  On injecte la définition de l'opérateur de troncature  dans le membre de droite de l'expression à démontrer. On obtient

  \begin{align}\label{eq:pr-tronq}
    \sum_{P\in \fc^Q(P_0)} \sum_{w   \in \, _PW^Q_\theta} \sum_{\delta\in P_w^\theta(F)\back  Q^\theta(F)}  \tau_{P}^Q(H_{P}(w \delta x)-T) \\\nonumber
\sum_{P_1\in\fc^P(P_0)} \eps_{P_1}^P \sum_{w_1  \in \, _{P_1}W^P_{w\theta}} \sum_{\delta_1\in P_{1,w_1}^{w\theta}(F)\back P^{w\theta}(F)} \hat\tau_{P_1}^P(H_{P_1}(w_1\delta_1    w\delta x)-T)  \varphi_{P_1}(w_1\delta_1 w\delta x).
  \end{align}

Observons qu'on a $P^{w\theta}=P\cap Q^{w\theta}=P\cap w Q^\theta w^{-1}$. Il s'ensuit qu'on a 
\begin{align*}
  w^{-1}P^{w\theta}w= w^{-1}P w\cap Q^\theta.
\end{align*}
Comme dans la preuve de la proposition \ref{prop:dec-tronque}, on réécrit cette expression à l'aide du lemme \ref{lem:Wdec}. On trouve alors que  l'expression \eqref{eq:pr-tronq} est égale à 
\begin{align*}
  \sum_{P\in\fc^Q(P_0)} \sum_{w   \in \, _PW^Q_\theta} \sum_{\delta\in P_w^\theta(F)\back Q^\theta(F)} \\
\left[     \sum_{P\subset R \subset Q } \eps_P^R  \hat\tau_P^R(H_P(w\delta x)-T)  \tau_{R}^Q(H_{R}(w \delta x)-T)     \right] \varphi_P( w\delta x).
\end{align*}
L'expression entre crochets est nulle sauf si $P=Q$ auquel cas la somme se réduit à $\varphi_Q(  x)$ (c'est le \og lemme de Langlands\fg, cf. par exemple \cite[proposition 1.7.2]{labWal} pour une preuve).
\end{preuve}
\end{paragr}

\section{Développement spectral}\label{sec:Dev spectral}

\subsection{Noyaux automorphes}\label{ssec:noyau autom}

\begin{paragr}\label{S:notations-spec}
  Dans toute cette section, $F$ est un corps de nombres. Les autres notations sont celles de la section \ref{sec:prelim}.
\end{paragr}

\begin{paragr} Nous appellerons  \emph{données cuspidales de $G$} les classes d'équivalence de couples $(M_P,\pi)$ où $P$ est un sous-groupe parabolique standard de $G$ et $\pi$ une représentation irréductible de $M_P(\AAA)$ qui se réalise dans l'espace  $L^2_{\cusp}(A_P^{\infty} M_P(F)\back M_P(\AAA))$ des fonctions cuspidales de carré intégrable sur $A_P^{\infty} M_P(F)\back M_P(\AAA)$, la relation d'équivalence étant celle donnée dans  \cite[§ 3]{ar1}.
  Soit $\Xgo(G)$ l'ensemble des données cuspidales de $G$. 
\end{paragr}

\begin{paragr}\label{S:KPchi}
  Soit $P$ un sous-groupe parabolique semi-standard de $G$. Soit $L^2([G]_P)$ l'espace de Hilbert des fonctions de carré intégrable sur le quotient $[G]_P$ muni de la mesure quotient. L'algèbre de Schwartz $\Sc(G(\AAA))$, cf. §  \ref{S:Schwartz}, agit à droite sur cet espace par un opérateur intégrable dont le noyau est donné explicitement par:
\begin{align*}
  K_{P,f}(x,y)=\sum_{\gamma\in M_P(F)} \int_{ N_P(\AAA)} f(x^{-1}\gamma n y)\, dn
\end{align*}
pour $x,y\in G(\AAA)$ et $f\in \Sc(G(\AAA))$.
On a en fait une décomposition 
\begin{align*}
   K_{P,f}(x,y)=\sum_{\chi\in \Xgo(G)}  K_{P,\chi,f}(x,y)
\end{align*}
où $K_{P,\chi,f}$ s'interprète comme le noyau de l'opérateur induit par $f\in \Sc(G(\AAA))$ agissant sur le sous-espace fermé stable $L^2_\chi([G]_P)\subset L^2([G]_P)$ associé à $\chi$ et défini, par exemple, dans   \cite[§ 2.9.2]{BCZ}.
\end{paragr}

\begin{paragr}
  Sur l'ensemble $[G]_P$, on a la notion de poids au sens de \cite[§ 2.4.3]{BCZ}:  ce sont certaines fonctions positives sur $[G]_P$ dont nous ne rappellerons pas ici la définition. Il nous suffira de savoir que l'ensemble des poids contient les fonctions $\|\cdot\|_P$ et $d^P(\la,\cdot)$ pour tout $\la\in \ago_0$ et qu'il stable par produit, par les opérations $\min$ ou $\max$ et par tout automorphisme de $G$ défini sur $F$ qui préserve $P$.  Ces poids interviennent dans des énoncés de majoration comme l'énoncé suivant.

  \begin{lemme}\label{lem:rappel*BPCZ}(\cite[lemme 2.10.1.1]{BCZ} Il existe $N_0>0$ tel que, pour tout $N>0$ et tout poids $\om$ sur $[G]_P$, il existe une semi-norme continue  $\| \cdot\|_{\Sc}$ sur $\Sc(G(\AAA)$ telle que, pour  tout $y\in [G]_P$ et tout $f\in \Sc(G(\AAA))$, on ait

\begin{align}\label{eq:sum norm*chi2}
 \sum_{\chi\in \Xgo(G) } \sup_{x\in [G]_P}\left( \om(x)\|x\|_P^{-N_0-N} |K_{P,\chi,f}(x,y)|\right) \leq \| f\|_{\Sc}\om(y) \|y\|_P^{-N}.
\end{align}
      \end{lemme}
    \end{paragr}

\subsection{Majoration de noyaux modifiés} \label{ssec:maj noy tronq}

\begin{paragr}
  On reprend les notations et les hypothèses de la sous-section \ref{ssec:partition}. Soit $\theta,\theta'\in M_0(F)$ des éléments d'ordre $2$ au plus.
\end{paragr}

\begin{paragr}\label{S:maj noyau modif}
  Soit $f\in \Sc(G(\AAA))$. Soit $T\in T_0+\overline{\ago_0^+}$ un paramètre de troncature et $\chi\in \Xgo(G)$. Pour tous  $x\in G^{\theta'}(F)\back G(\AAA)$ et $y\in G^{\theta}(F)\back G(\AAA)$, on définit le noyau modifié
  \begin{align}\label{eq:KT}
    K^{T,\theta',\theta}_{\chi,f}(x,y)=\sum_{P\in \fc(P_0)} \eps_P^G \sum_{w_1 \in   \, _PW_{\theta'}} \sum_{\delta_1\in P_{w_1}^{\theta'}(F) \back G^{\theta'}(F)}  \hat\tau_P(H_P(w_1\delta_1x)-T) \times \\ \nonumber \left[\sum_{w_2\in   \, _PW_\theta} \sum_{\delta_2\in P_{w_2}^\theta(F) \back G^\theta(F)} K_{P,\chi,f}(w_1\delta_1x,w_2\delta_2y)\right].
  \end{align}

  \begin{remarque}\label{rq:delta1}
   Dans l'expression \eqref{eq:KT}, la somme sur $\delta_1$ est finie (cf. \cite[lemme 5.1]{ar1}). En revanche, dans l'expression entre crochets, celle sur $\delta_2$ ne l'est pas en général mais le lemme \ref{lem:rappel*BPCZ}  implique que la somme est du moins absolument convergente (prendre $\om=1$ et $N$ assez grand).
 \end{remarque}
 
  Dans la suite, $\theta'$ et $\theta$ sont fixés et on pourra les ôter en exposant si le contexte est clair.  On obtient également le noyau modifié $K^{T,\theta',\theta}_{f}$, noté simplement  $K^{T}_{f}$, en remplaçant  $K_{P,\chi,f}$ par  $K_{P,f}$. On a alors
  \begin{align}\label{eq:sum-chif}
    K^T_{f}=\sum_{\chi\in \Xgo(G)} K^T_{\chi,f}.
  \end{align}
  
On pourra aussi omettre l'indice $f$ si le contexte est clair.
 
\end{paragr}

\begin{paragr} Les principaux résultats de cette sous-section sont les  suivants.
  
  \begin{theoreme}
    \label{thm:maj noyau}
    Pour tous $N_1,N_2,r>0$, il existe une semi-norme  continue $\|\cdot\|_{\Sc} $ sur $\Sc(G(\AAA))$ telle que 
    \begin{align*}
\sum_{\chi\in \Xgo(G)} |K^T_{\chi,f}(x,y) - F^G(x,T)K_{\chi,f}(x,y) | \leq \exp(-r\|T^G\|) \|x\|_{G^{\theta'}}^{-N_1} \|y\|_{G^\theta}^{-N_2} \|f\|_{\Sc}
    \end{align*}
    pour  tous $f\in \Sc(G(\AAA))$, $x\in [G^{\theta'}]^G$, $y\in [G^\theta]$ et $T\in \ago_0$ suffisamment positif.
  \end{theoreme}

  \begin{theoreme}
    \label{cor:maj noyau}
Pour tous $N_1,N_2>0$ et $T\in T_0+\overline{\ago_0^+}$, il existe une semi-norme  continue $\|\cdot\|_{\Sc} $ sur $\Sc(G(\AAA))$ telle que 
    \begin{align*}
\sum_{\chi\in \Xgo(G)} |K^T_{\chi,f}(x,y) | \leq  \|x\|_{G^{\theta'}}^{-N_1} \|y\|_{G^\theta}^{-N_2} \|f\|_{\Sc}
    \end{align*}
    pour  tous $f\in \Sc(G(\AAA))$, $x\in [G^{\theta'}]^G$ et  $y\in [G^\theta]$.
  \end{theoreme}

\end{paragr}

\begin{paragr}[Démonstration des théorèmes   \ref{thm:maj noyau} et  \ref{cor:maj noyau} ] \label{S:dec KT-FT} --- Soit $\chi\in \Xgo(G)$ et $f\in \Sc(G(\AAA))$. En procédant comme dans la preuve de la proposition \ref{prop:dec-tronque}, on voit que pour tous $x\in G^{\theta'}(\AAA)$ et $ y\in G^\theta(\AAA)$, la différence
  \begin{align}\label{eq:ma difference}
    K^T_{\chi,f}(x,y) - F^G(x,T)K_{\chi,f}(x,y) 
  \end{align}
  est égale à la somme sur les sous-groupes paraboliques standard $P_1\subsetneq  P_2$ , la somme sur  $w_1 \in \,_{P_1}W_{\theta'}$, $\delta \in P_{1,w_1}^{\theta'}(F)\back G^{\theta'}(F)$ de 
  \begin{align*}
    F^{P_1}( w_1\delta_1 x,T)  \sigma_1^2(H_{P_1}(w_1\delta_1 x)-T)  K_{1,2,\chi}(w_1\delta x,y)
  \end{align*}
   où l'on pose pour tout $x\in G(\AAA)$
  \begin{align*}
    K_{1,2,\chi}(x,y)=\sum_{P_{1}\subset P\subset P_{2 }}\eps_P^G  \sum_{w_2\in   \, _PW_\theta} \sum_{\delta_2\in P_{w_2}^\theta(F) \back G^\theta(F)} K_{P,\chi}(x,w_2\delta_2 y).
  \end{align*}
  
Nous démontrerons au § \ref{S:etape 3} la proposition suivante en nous appuyant sur les résultats intermédiaires démontrés entretemps.

   \begin{proposition}
     \label{prop:etape 3} Soit $P_1\subsetneq  P_2$ des sous-groupes paraboliques.
     \begin{enumerate}
     \item     Pour tous $N_1,N_2,r>0$, il existe une semi-norme  continue $\|\cdot\|_{\Sc} $ sur $\Sc(G(\AAA))$ telle que
          \begin{align*}
           & \sum_{\chi\in \Xgo(G)} \sum_{w \in \,_{P_1}W_{\theta'}}\sum_{\delta \in P_{1,w}^{\theta'}(F)\back G^{\theta'}(F)}F^{P_1}( w \delta  x,T)  \sigma_1^2(H_{P_1}(w\delta x)-T)  |K_{1,2,\chi}(w \delta x,y)| \\
            &\leq  \exp(-r\|T^G\|) \|x\|_{G^{\theta'}}^{-N_1} \|y\|_{G^\theta}^{-N_2} \|f\|_{\Sc}
    \end{align*}
    pour  tout $f\in \Sc(G(\AAA))$, $x\in [G^{\theta'}]^G$, $y\in [G^\theta]$ et $T\in \ago_0$ suffisamment positif.
  \item Soit $T\in T_0+\overline{\ago_0^+}$. Pour tous $N_1,N_2>0$, il existe une semi-norme  continue $\|\cdot\|_{\Sc} $ sur $\Sc(G(\AAA))$ telle que
      \begin{align*}
           & \sum_{\chi\in \Xgo(G)} \sum_{w \in \,_{P_1}W_{\theta'}}\sum_{\delta \in P_{1,w}^{\theta'}(F)\back G^{\theta'}(F)}F^{P_1}( w \delta  x,T)  \sigma_1^2(H_{P_1}(w\delta x)-T)  |K_{1,2,\chi}(w \delta x,y)| \\
            &\leq  \|x\|_{G^{\theta'}}^{-N_1} \|y\|_{G^\theta}^{-N_2} \|f\|_{\Sc}
    \end{align*}
    pour  tout $f\in \Sc(G(\AAA))$, $x\in [G^{\theta'}]^G$ et  $y\in [G^\theta]$.
  \end{enumerate}
\end{proposition}

Le théorème \ref{thm:maj noyau} est alors une conséquence immédiate de la décomposition combinatoire de la différence \eqref{eq:ma difference} et la proposition \ref{prop:etape 3} assertion 1 ci-dessus.

Montrons maintenant le théorème \ref{cor:maj noyau}. D'après le  lemme \ref{lem:rappel*BPCZ} appliqué au  poids $\om=1$, il existe $N_0$ tel que, pour tout $N>0$, il existe une semi-norme continue  $\| \cdot\|_{\Sc}$ sur $\Sc(G(\AAA)$ telle que pour  tous $x,y\in G(\AAA)$ et tout $f\in \Sc(G(\AAA))$ on ait

\begin{align*}
  \sum_{\chi\in \Xgo(G) }  |K_{\chi,f}(x,y)|  \leq \| f\|_{\Sc}  \|x\|_G^{N_0+N}  \|y\|_G^{-N}.
\end{align*}
Soit  $T\in T_0+\overline{\ago_0^+}$ et $N'>0$. L'application $x\mapsto   F^G(x,T)\|x\|_G^{N_0+N+N'}$ est à support compact  sur $[G^{\theta'}]^G$: elle est donc bornée. Comme, pour tout $\theta \in M_0(F)$ d'ordre au plus $2$,  il existe $C>0$ et $r>0$ tel que pour  tout $y\in G^{\theta}(\AAA)$ on a (cf. \cite[proposition A.1.1 (ix)]{RBP})
\begin{align*}
   \|y\|_{G^{\theta}}\leq C \|y\|_G^{r},
\end{align*}
on obtient que, pour tous $N_1,N_2>0$, il existe une semi-norme  continue $\|\cdot\|_{\Sc} $ telle qu'on ait
\begin{align}\label{eq:maj-FT-K}
  \forall f\in \Sc(G(\AAA)), x\in [G^{\theta'}]^G \text{ et } y\in [G^\theta]\\
\nonumber  \sum_{\chi\in \Xgo(G)}  F^G(x,T)|K_{\chi,f}(x,y) |\leq  \|x\|_{G^{\theta'}}^{-N_1} \|y\|_{G^\theta}^{-N_2} \|f\|_{\Sc}.
\end{align}
On conclut alors à l'aide de la  décomposition combinatoire de la différence \eqref{eq:ma difference} et de la proposition \ref{prop:etape 3} assertion 2.

  \end{paragr}

\begin{paragr} Le reste de la sous-section est consacrée à des  préparatifs  à la démonstration de la proposition \ref{prop:etape 3} qui sera finalement donnée au  § \ref{S:etape 3}. Soit $P_1\subsetneq P_2$ des sous-groupes paraboliques standard.   Soit $\al\in \Delta_0^{P_2}\setminus \Delta_0^{P_1}$. Soit $P_1\subsetneq  P_1^\al \subset P_2$ défini par $\Delta_0^{P_1^\al }=\Delta_0^{P_1}\cup\{\al\}$. Pour tout  sous-groupe parabolique $ P_1^\al\subset P\subset P_2$ soit $P_\al$ défini par $\Delta_0^{P_\al}=\Delta_0^P\setminus\{\al\}$.

  Soit $\theta_1\in M_0(F)$ un élément d'ordre au plus deux. Soit $P$ un sous-groupe parabolique standard tel que $^\al P_1\subset P\subset P_2$. Soit $\chi\in \Xgo(G)$. Pour tout $x,y\in G(\AAA)$, on pose
  \begin{align*}
       K^{\al,\theta_1}_{P,\chi}(x,  y)=K_{P,\chi}(x,y)-\sum_{w_1\in \, _{P_\al}W^{P}_{\theta_1}}  \sum_{\delta_1 \in P_{\al,w_1}^{\theta_1}(F) \back   P^{\theta_1}(F)} K_{P_\al,\chi}(x,w_1\delta_1  y).
  \end{align*}
  Pour   $w\in   \, _PW_\theta$, on pose $       K^{\al,w}_{P,\chi}= K^{\al,\theta_1}_{P,\chi}$  avec $\theta_1=w\theta w^{-1}$.

  \begin{lemme}\label{lem:K12 chi sum}
Pour tous $x,y\in G(\AAA)$, on a
    \begin{align*}
      K_{1,2,\chi}(x,y)=\sum_{\,^\al P_1\subset P\subset P_2} \eps_P^G  \sum_{w \in   \, _PW_\theta} \sum_{\delta \in P_{w}^\theta(F) \back G^\theta(F)} K^{\al,w}_{P,\chi}(x,  w\delta y).
  \end{align*}
\end{lemme}

\begin{preuve} Soit  $ P_1^\al\subset P\subset P_2$.  En utilisant le lemme \ref{lem:Wdec}, on obtient pour tous $x,y\in G(\AAA)$
  
  \begin{align*}
    &\sum_{w\in   \, _{P_\al}W_\theta} \sum_{\delta\in P_{\al,w}^\theta(F) \back G^\theta(F)} K_{P_\al,\chi}(x,w\delta  y)\\
   & =\sum_{w_1\in   \, _{P}W_\theta} \sum_{w_2\in \, _{P_\al}W^{P}_{w_1\theta}} \sum_{\delta_2\in P_{\al,w_2w_1}^\theta(F) \back   P_{w_1}^\theta(F)}     \sum_{\delta_1\in P_{w_1}^\theta(F) \back   G^\theta(F)} K_{P_\al,\chi}(x,w_2w_1\delta_2\delta_1  y)\\
    &=\sum_{w_1\in   \, _{P}W_\theta}    \sum_{\delta_1\in P_{w_1}^\theta(F) \back   G^\theta(F)}     \sum_{w_2\in \, _{P_\al}W^{P}_{w_1\theta}}  \sum_{\delta_2\in P_{\al,w_2}^{w_1\theta}(F) \back   P^{w_1\theta}(F)} K_{P_\al,\chi}(x,w_2\delta_2 w_1\delta_1  y). 
  \end{align*}
Il s'ensuit qu'on a 
  \begin{align*}
&    \sum_{w \in   \, _PW_\theta} \sum_{\delta \in P_{w}^\theta(F) \back G^\theta(F)} K_{P,\chi}(x,w\delta y)- \sum_{w\in   \, _{P_\al}W_\theta} \sum_{\delta\in P_{\al,w}^\theta(F) \back G^\theta(F)} K_{P_\al,\chi}(x,w\delta  y)\\
    &=  \sum_{w \in   \, _PW_\theta} \sum_{\delta \in P_{w}^\theta(F) \back G^\theta(F)}  K^{\al,w}_{P,\chi}(x, w\delta y).
  \end{align*}
  Le lemme s'ensuit aisément.
\end{preuve}

\begin{lemme} \label{lem:etape 1} Soit  $ P_1^\al\subset P\subset P_2$ et $w\in  \, _PW_\theta$. Pour tout $N>0$, il existe $N'>0$ tel que pour tout $t>0$ il existe une semi-norme continue $\|\cdot\|_{\Sc}$ sur $\Sc(G(\AAA))$ de sorte qu'on ait pour tout $T\in T_0+\overline{\ago_0^+}$
  \begin{align}\label{eq:etape 1}
F^{P_1}(x,T)\sigma_1^2(H_{P_1}(x)-T)    \sum_{\chi\in \Xgo(G)}  |K^{\al,w}_{P,\chi,f}(x,  y)| \leq  \|f\|_{\Sc} d^{P_1}(-t\al,x)   \|x\|_{P_1}^{N'} \|y\|^{-N}_{P^{w\theta}}
  \end{align}
    pour tout $x\in  P_1(F)\back G(\AAA)^1$, tout $y\in G^{w\theta}(\AAA)$ et tout $f\in \Sc(G(\AAA))$.
  
\end{lemme}

\begin{preuve}  Soit  $J\subset G(\AAA_f)$ un sous-groupe ouvert compact. En utilisant le théorème de Banach-Steinhaus, on voit qu'il suffit de prouver l'énoncé pour le sous-espace $\Sc(G(\AAA))^J\subset \Sc(G(\AAA))$ des fonctions bi-invariantes par $J$. Soit  $ P_1^\al\subset P\subset P_2$ et $w\in  \, _PW_\theta$. On prendra un réel $N_0>0$ pour lequel  le lemme \ref{lem:rappel*BPCZ} vaut pour les sous-groupes paraboliques $P$ et $P_\al$. Soit $x,y\in G(\AAA)$.  On introduit 
 \begin{align*}
    K^{\al,\sharp}_{P,\chi}(x,  y)    &= K_{P,\chi}(x,y)- \int_{[N_{P_\al}]} K_{P,\chi}(nx,y)\, dn.
 \end{align*}
 Soit $\Om_{P,\al}^{w}$ l'ensemble des $\delta \in P_{\al}(F) \back   P(F)$ tels que $w\theta w^{-1}\notin   \delta^{-1}P_{\al}\delta$.
 On pose
 \begin{align*}
    K^{\al,w,\flat}_{P,\chi}(x,  y)=\sum_{\delta \in \Om_{P,\al}^{w}} K_{P_\al,\chi}(x,\delta y).
  \end{align*}
 On a
  \begin{align}\label{eq:K al w}
      K^{\al,w}_{P,\chi}(x,  y)= K^{\al,\sharp}_{P,\chi}(x,  y)+K^{\al,w,\flat}_{P,\chi}(x,  y)
  \end{align}
comme il  résulte  de l'égalité
  \begin{align*}
    \int_{[N_{P_\al}]} K_{P,\chi}(nx,y)\, dn=\sum_{\delta \in P_{\al}(F) \back   P(F)} K_{P_\al,\chi}(x,\delta y).
  \end{align*}

Il nous suffit, dès lors, de majorer chacun des deux termes du membre de droite de \eqref{eq:K al w}. Pour majorer $K^{\al,\sharp}_{P,\chi}(x,  y)$ on utilise le lemme \ref{lem:phi12} pour $P_\al\subsetneq P$ appliqué à la fonction $K_{P,\chi}(\cdot,y)\in C^\infty([G]_P)^J$. On obtient que, pour tout $y\in G(\AAA)$,  tout $x\in \SG^{P}\cap G(\AAA)^1$ et tout $f\in \Sc(G(\AAA))^J$
\begin{align*}
  |K^{\al,\sharp}_{P,\chi}(x,  y)|\leq \exp(-t \bg\al,H_{P_0}(x)\bd) \|x  \|_{P}^{N_0+N} \sup_{z\in [G]_P,X\in \Fgo} \left(\|z\|^{-N_0-N}_{P} |(R(X)K_{P,\chi,f})(z,y) | \right).
\end{align*}
Ici $R(X)K_{P,\chi,f}$ signifie qu'on applique l'opérateur différentiel $R(X)$ à la première variable de $K_{P,\chi,f}$. Il existe en fait une involution $X\mapsto X^*$ de $\uc(\ggo_\CC)$ tel que $R(X)K_{P,\chi,f}=K_{P,\chi,L(X^*)f}$.   En utilisant ensuite le lemme \ref{lem:rappel*BPCZ} pour le  poids $\om$ égal identiquement à $1$, on obtient une semi-norme  continue $\|\cdot\|_{\Sc} $ sur $\Sc(G(\AAA))$ telle que pour tout $f\in \Sc(G(\AAA))$ et tout $y\in G(\AAA)$ on ait
\begin{align}\label{eq:maj pour sharp}
    \sum_{\chi\in \Xgo(G)  } \sup_{z\in [G]_{P},X\in \Fgo} \left(\|z\|^{-N_0-N}_{P} |(R(X)K_{P,\chi,f})(z,y)|  \right) \leq \|y\|_{P}^{-N}\|f\|_{\Sc}.
\end{align}
On a donc  pour tout $y\in G(\AAA)$,  tout $x\in \SG^{P}\cap G(\AAA)^1$ et tout $f\in \Sc(G(\AAA))^J$
\begin{align}\label{eq:maj sharp}
   \sum_{\chi\in \Xgo(G)  }  |K^{\al,\sharp}_{P,\chi}(x,  y)|\leq \exp(-t \bg\al,H_{P_0}(x)\bd) \|x  \|_{P}^{N_0+N}  \|y\|_{P}^{-N}\|f\|_{\Sc}.
\end{align}

Majorons ensuite  $K^{\al,w,\flat}_{P,\chi}(x,  y)$ pour $y\in G^{w\theta}(\AAA)$. Introduisons le lemme suivant.

\begin{lemme}\label{lem:maj dP al}
  Il existe $c>0$ tel que     $d^{P_\al}(\al,\delta y)\leq c$ pour tout $\delta\in  \Om_{P,\al}^{w}$ et tout $y\in G^{w\theta}(\AAA)$.
\end{lemme}

\begin{preuve}
  Soit $c>0$ qui apparaît dans le lemme \ref{lem:unicite} (pour $Q=P_\al$ et $R=P$). Soit $\delta \in P_\al(F)\back P(F)$ et $y\in G^{w\theta}(\AAA)$ tel que $d^{P_\al}(\al,\delta y)>c$. On observe que comme on a $w\theta w^{-1}\in K$ on a  $d^{P_\al}(\al,\delta y)=d^{P_\al}(\al,\delta y w\theta w^{-1})=d^{P_\al}(\al,\delta  w\theta w^{-1}y)$. Le lemme  \ref{lem:unicite}  implique qu'on a $ \delta w\theta w^{-1} \in P_\al(F)\delta$ c'est-à-dire  $  w\theta w^{-1} \in \delta^{-1}P_\al(F)\delta$ et donc $\delta\notin \Om_{P,\al}^{w}$.
\end{preuve}

Le lemme \ref{lem:rappel*BPCZ} appliqué au  poids $\om=d^{P_\al}(t \al,\cdot)=d^{P_\al}(\al,\cdot)^t$ et  combiné au lemme \ref{lem:maj dP al} fournit pour tout $N'>0$  une  semi-norme  continue $\|\cdot\|_{\Sc} $ sur $\Sc(G(\AAA))$ telle que, pour tous $x\in G(\AAA)$, $y\in G^{w\theta}(\AAA)$ et $f\in \Sc(G(\AAA))$
\begin{align}\label{eq:maj flat}
  \sum_{\chi\in \Xgo(G)}   |K^{\al,w,\flat}_{P,\chi}(x,  y)|& \leq    \sum_{\delta \in \Om_{P,\al}^w} \sum_{\chi\in \Xgo(G)}    |K_{P_\al,\chi}(x,\delta y)|    \\
    \nonumber       &\leq  \|f\|_{\Sc}\|x\|_{P_\al}^{N_0+N'}  d^{P_\al}(-t \al,x)    \sum_{\delta \in \Om_{P,\al}^w} \|\delta y\|_{P_\al}^{-N'}\\
  \nonumber     &\leq  \|f\|_{\Sc}\|x\|_{P_\al}^{N_0+N'}  d^{P_\al}(-t \al,x)    \sum_{\delta \in P_\al(F)\back P(F)} \|\delta y\|_{P_\al}^{-N'}.
\end{align}

Expliquons comment on déduit \eqref{eq:etape 1} de \eqref{eq:maj sharp} et \eqref{eq:maj flat}. On observe tout d'abord que  $x\mapsto K^{\al,\sharp}_{P,\chi}(x,  y)$ et $x\mapsto K^{\al,w,\flat}_{P,\chi}(x,  y)$ sont  invariants à gauche par $P_\al(F)$, donc \emph{a fortiori} par $P_1(F)$.  On peut donc supposer qu'on a $x\in \SG^{P_1}\cap G(\AAA)^1$. Soit $T\in T_0+\overline{\ago_0^+}$. Supposons de plus $F^{P_1}(x,T)\sigma_1^2(H_{P_1}(x)-T)\not=0$. Cela implique que $x\in  \SG^{P_2}$. \emph{A fortiori}, $x\in \SG^{P}$ et on peut appliquer l'inégalité  \eqref{eq:maj sharp} à $x$. On peut y remplacer  $\exp(-t \bg\al,H_{P_0}(x)\bd) $ par  $d^{P_1}(-t \al,x)$ vu que ces fonctions  sont équivalentes sur $\SG^P$, cf. \cite[proposition 2.3.4.1]{BPC}. De même, on peut remplacer $\|x\|_{P}$ par $\|x\|_{P_1}$. Puis, il existe $C>0$ et $r>0$ tel que pour  tout $y\in G^{w\theta}(\AAA)$ on a (cf. \cite[proposition A.1.1 (ix)]{RBP})
\begin{align}\label{eq:equiv}
   \|y\|_P\leq   \|y\|_{P^{w\theta}}\leq C \|y\|_P^{r}.
\end{align}
Pour les mêmes raisons que ci-dessus, dans \eqref{eq:maj flat}, on peut remplacer $d^{P_\al}(-t \al,x)$ par $d^{P_1}(-t \al,x)$ et $\|x\|_{P_\al}$ par $\|x\|_{P_1}$. Enfin, pour tout $N>0$, il existe $N'$ tel que
\begin{align*}
  \sum_{\delta \in P_\al(F)\back P(F)} \|\delta y\|_{P_\al}^{-N'}\leq \|y\|_P^{-N}.
\end{align*}
Par \eqref{eq:equiv}, on a le même résultat avec $\|y\|_{P^{w\theta}}$ à la place de $\|y\|_P$.
\end{preuve}
\end{paragr}

 \begin{paragr}
  
\begin{lemme} \label{lem:etape 2}  Soit $\la =\sum_{\al\in\Delta_0^2\setminus \Delta_0^1}\al$. Pour tout $N>0$, il existe $N'>0$ tel que pour tout $t>0$ il existe une semi-norme continue $\|\cdot\|_{\Sc}$ sur $\Sc(G(\AAA))$ de sorte qu'on ait pour tout $T\in T_0+\overline{\ago_0^+}$
  \begin{align}\label{eq:etape 2}
F^{P_1}(x,T)\sigma_1^2(H_{P_1}(x)-T)    \sum_{\chi\in \Xgo(G)}  |K_{1,2,\chi,f}(x,  y)| \leq  \|f\|_{\Sc} d^{P_1}(-t\la ,x)   \|x\|_{P_1}^{N'} \|y\|^{-N}_{G^{\theta}}
  \end{align}
    pour tout $x\in  P_1(F)\back G(\AAA)^1$, tout $y\in G^{\theta}(\AAA)$ et tout $f\in \Sc(G(\AAA))$.
  \end{lemme}

\begin{preuve}
    Soit $\al\in\Delta_0^2\setminus \Delta_0^1$. Il résulte de la combinaison des lemmes \ref{lem:K12 chi sum} et \ref{lem:etape 1} que, pour tout $N>0$, il existe $N'>0$ tel que pour tout $t>0$ il existe une semi-norme continue $\|\cdot\|_{\Sc}$ sur $\Sc(G(\AAA))$ de sorte qu'on ait pour tout $T\in T_0+\overline{\ago_0^+}$
\begin{align*}
  &F^{P_1}(x,T)\sigma_1^2(H_{P_1}(x)-T)    \sum_{\chi\in \Xgo(G)}  |K_{1,2,\chi,f}(x,  y)| \\
  &\leq  \|f\|_{\Sc} d^{P_1}(-t \al,x)   \|x\|_{P_1}^{N'}   \sum_{P^\al\subset P\subset P_2}   \sum_{w \in   \, _PW_\theta} \sum_{\delta \in P_{w}^\theta(F) \back G^\theta(F)} \|w\delta yw^{-1}\|_{P^{w\theta}}^{-N} 
\end{align*}
pour tout $x\in  P_1(F)\back G(\AAA)^1$, tout $y\in G^{\theta}(\AAA)$ et tout $f\in \Sc(G(\AAA))$. Pour tout $N''>0$, il existe $N,c>0$ tels que
\begin{align*}
  \sum_{P^\al\subset P\subset P_2}   \sum_{w \in   \, _PW_\theta} \sum_{\delta \in P_{w}^\theta(F) \back G^\theta(F)} \|w\delta yw^{-1}\|_{P^{w\theta}}^{-N} \leq c\|y\|_{G^\theta}^{-N''}.
\end{align*}
Le lemme s'en déduit aisément vu que $d^{P_1}(-t \la,x)=\prod_{\al\in\Delta_0^2\setminus \Delta_0^1} d^{P_1}(-t \al,x)  $.
\end{preuve}
\end{paragr}

\begin{paragr}[Démonstration de la proposition \ref{prop:etape 3}.] --- \label{S:etape 3}   Soit $N_1,N_2>0$.    À l'aide de \eqref{eq:equiv}, on voit qu'il existe $N_1',c>0$ tels que pour tout $x\in G^{\theta'}(\AAA)$
    \begin{align}\label{eq:petite maj}
       \sum_{w \in \,_{P_1}W_{\theta'}}\sum_{\delta \in P_{1,w}^{\theta'}(F)\back G^{\theta'}(F)} \|w\delta x\|_{P_1}^{-N_1'} \leq c\|x\|_{G^{\theta'}}^{-N_1}.
    \end{align}
    Le lemme \ref{lem:etape 2} (dont on reprend les notations) assure l'existence de $N'>0$ tel que pour tout $t>0$ il existe une semi-norme continue $\|\cdot\|_{\Sc}$ sur $\Sc(G(\AAA))$  telle que  pour tout $T\in T_0+\overline{\ago_0^+}$ on ait:
      \begin{align*}
      & \sum_{\chi\in \Xgo(G)} \sum_{w \in \,_{P_1}W_{\theta'}}\sum_{\delta \in P_{1,w}^{\theta'}(F)\back G^{\theta'}(F)}F^{P_1}( w \delta  x,T)  \sigma_1^2(H_{P_1}(w\delta x)-T)  |K_{1,2,\chi}(w \delta x,y)| \\
      &\leq \|f\|_{\Sc}  \|y\|^{-N_2}_{G^{\theta}}    \sum_{w \in \,_{P_1}W_{\theta'}}\sum_{\delta \in P_{1,w}^{\theta'}(F)\back G^{\theta'}(F)} F^{P_1}( w \delta  x,T)  \sigma_1^2(H_{P_1}(w\delta x)-T)  d^{P_1}(-t\la ,w \delta x)    \|w \delta x\|_{G^{\theta'}}^{N'} .
      \end{align*}
          
On obtient alors les assertions 1 et 2 de  la proposition \ref{prop:etape 3} respectivement à l'aide des lemmes  \ref{lem:Fsigma*d}   et \ref{lem:Fsigma*d variante}.

\end{paragr}

\subsection{Énoncés de convergence spectrale}\label{ssec:enonces spec}

\begin{paragr} Les notations sont celles des sous-sections \ref{ssec:noyau autom} et \ref{ssec:maj noy tronq}. On a  $G^{\theta'}(\AAA)=A_G^{\infty}\times (G^{\theta'}(\AAA)\cap G(\AAA)^1)$. On munit alors $G^{\theta'}(\AAA)\cap G(\AAA)^1 $ de la mesure de Haar dont le produit avec celle sur $A_G^\infty$ donne la mesure de Haar fixée sur $G^{\theta'}(\AAA)$. En prenant le quotient de cette mesure par la mesure de comptage sur $G^{\theta'}(F)$, on obtient une mesure sur le quotient $[G^{\theta'}]^G$.
\end{paragr}

\begin{paragr}    Soit $f\in\Sc(G(\AAA))$. Rappelons qu'on a défini en \eqref{eq:KT} un noyau modifié $K^{T,\theta',\theta}_{\chi,f}$ associé à une donnée cuspidale $\chi$, un paramètre $T\in T_0+\overline{\ago_0^+}$ et des éléments $\theta'$ et $\theta$ de $M_0(F)$ d'ordre au plus $2$. On omet les exposants    $\theta'$ et $\theta$. On dispose aussi du  noyau tronqué $\Lambda^{T}_{\theta'} K_{\chi,f}$: ici la notation signifie qu'on a appliqué l'opérateur de troncature $\Lambda^{T}_{\theta'} $ à la variable de gauche du noyau $K_{\chi,f}$.  Ce noyau  tronqué est souvent un intermédiaire commode pour obtenir des décompositions spectrales plus fines que la décomposition selon les données cuspidales. Enfin, on peut aussi multiplier le noyau initial $K_{\chi,f}$ par la fonction $F^G(\cdot, T)$ appliquée à la première variable. Ces \og noyaux\fg{} sont en fait tous  à décroissance rapide sur  $[G^{\theta'}]^G\times [G^\theta]$. En particulier, on a l'énoncé suivant:

  \begin{theoreme}\label{thm:cv-spec}
    \begin{enumerate}
    \item Pour tout $T\in T_0+\overline{\ago_0^+}$ et $N_1,N_2>0$,   il existe une semi-norme continue $\|\cdot\| $ sur $\Sc(G(\AAA))$ telle que pour  tout $f\in \Sc(G(\AAA))$ on ait
       \begin{align*}
      \sum_{\chi\in\Xgo(G)}     \int_{[G^{\theta'}]^G\times [G^\theta]     }  F^G(x,T)|K_{\chi,f}(x,y)|   \|x\|_{G^{\theta'}}^{N_1}\|y\|_{G^{\theta}}^{N_2}   \, dxdy \leq \|f\|.
    \end{align*}
  \item  Pour tout $T\in T_0+\overline{\ago_0^+}$  et $N_1,N_2>0$, il existe une semi-norme continue $\|\cdot\| $ sur $\Sc(G(\AAA))$ telle que pour  tout $f\in \Sc(G(\AAA))$ on ait 
    \begin{align*}
      \sum_{\chi\in\Xgo(G)}     \int_{[G^{\theta'}]^G\times [G^\theta]      }  |K^T_{\chi,f}(x,y)|  \|x\|_{G^{\theta'}}^{N_1}\|y\|_{G^{\theta}}^{N_2} \, dxdy \leq \|f\|.
    \end{align*}
  \item Pour tout $T\in T_0+\overline{\ago_0^+}$  et $N_1,N_2>0$, il existe une semi-norme continue $\|\cdot\| $ sur $\Sc(G(\AAA))$ telle que pour  tout $f\in \Sc(G(\AAA))$ on ait 
    \begin{align*}
     \sum_{\chi\in \Xgo(G)}    \int_{[G^{\theta'}]^G\times [G^\theta]      } |\Lambda^{T}_{{\theta'}} K_{\chi,f}  (x, y)| \|x\|_{G^{\theta'}}^{N_1}\|y\|_{G^{\theta}}^{N_2}  \, dxdy\leq \|f\|.
    \end{align*}
  \end{enumerate}
\end{theoreme}

\begin{preuve}  Pour tous $N_1$ et $N_2$ assez grands, on  a 
\begin{align}\label{eq:cv-integ-norme}
      \int_{[G^{\theta'}]^G  }     \|x\|_{G^{\theta'}}^{-N_1}  \,dx  <\infty   \text{   et  }  \int_{[G^\theta]   } \|y\|_{G^{\theta}}^{-N_2}\,dy<\infty.
\end{align}
  
Il suffit donc de voir que les trois noyaux considérés sont à décroissance rapide sur  $[G^{\theta'}]^G\times [G^\theta]$ c'est-à-dire que  le noyau  $F^G(x,T)|K_{\chi,f}(x,y)| $ vérifie \eqref{eq:maj-FT-K} et le noyau $K^T_{\chi,f}(x,y)$ le  théorème  \ref{cor:maj noyau}. Pour le noyau $\Lambda^{T}_{{\theta'}} K_{\chi,f}  (x, y)$, cela résulte de la combinaison de la proposition \ref{prop:dec-tronque} et du lemme \ref{lem:rappel*BPCZ}. 
\end{preuve}

\end{paragr}

\begin{paragr}[Asymptotique en $T$.] ---  Le théorème \ref{thm:comparaison-asym} ci-dessous  montre qu'asymptotiquement toutes les expressions du théorème \ref{thm:cv-spec}  ci-dessus sont égales. 
  
  \begin{theoreme}\label{thm:comparaison-asym} Pour tous $N_1,N_2,r>0$, il existe une semi-norme  continue $\|\cdot\|_{\Sc} $ sur $\Sc(G(\AAA))$ telle que,  pour tout $T\in \ago_0$ suffisamment positif et tout  $f\in \Sc(G(\AAA))$, on ait
    \begin{align}
&\label{eq:asym 1}\sum_{\chi\in \Xgo(G)} \int_{[G^{\theta'}]^G\times [G^\theta]  }|K^T_{\chi,f}(x,y) - F^G(x,T)K_{\chi,f}(x,y) |   \|x\|_{G^{\theta'}}^{N_1}\|y\|_{G^{\theta}}^{N_2}   \, dxdy\leq e^{-r\|T^G\|}  \|f\|_{\Sc} \ ;\\
&\label{eq:asym 2}\sum_{\chi\in \Xgo(G)} \int_{[G^{\theta'}]^G\times [G^\theta]  } |\La^T_{\theta'} K_{\chi,f}(x,y) - F^G(x,T)K_{\chi,f}(x,y) |    \|x\|_{G^{\theta'}}^{N_1}\|y\|_{G^{\theta}}^{N_2} \, dxdy\leq e^{-r\|T^G\|}  \|f\|_{\Sc}\ ;\\
&\label{eq:asym 3}\sum_{\chi\in \Xgo(G)} \int_{[G^{\theta'}]^G\times [G^\theta]  }|K^T_{\chi,f}(x,y) -    \La^T_{\theta'} K_{\chi,f}(x,y) |  \|x\|_{G^{\theta'}}^{N_1}\|y\|_{G^{\theta}}^{N_2}\, dxdy\leq e^{-r\|T^G\|}  \|f\|_{\Sc}.
    \end{align}
      \end{theoreme}

      \begin{preuve}
        Tout d'abord, \eqref{eq:asym 1} est une conséquence immédiate du théorème \ref{thm:maj noyau} et de la convergence \eqref{eq:cv-integ-norme}.

        Traitons ensuite \eqref{eq:asym 2}. On peut fixer $J\subset G(\AAA_f)$ un sous-groupe ouvert compact. Par le théorème de Banach-Steinhaus, il suffit de prouver l'énoncé pour $f$ dans le sous-espace $\Sc(G(\AAA))^J\subset \Sc(G(\AAA))$ des fonctions bi-invariantes par $J$.  Il résulte de l'assertion 1 de la proposition \ref{prop:asym tronque} appliquée à la fonction $x\mapsto K_{\chi,f}(x,y)\in C^\infty([G])^J$ que, pour tout $r>0$ et tous $N_1,N_2>0$, il existe une famille finie $\Fgo$ d'éléments de $\uc(\ggo_\CC)$ telle que pour tout $T$ suffisamment positif, on ait 
  \begin{align*}
    &    |(\Lambda^{T}_{\theta'} K_{\chi,f})(x,y)-F^G(x,T)K_{\chi,f}(x,y) | \\
    &\leq  \exp(-r\|T^G \|)\|x\|^{-N_1}_{G^{\theta'}}   \sup_{z\in [G]^1,X\in \Fgo}\left(\|z\|_G^{-N_2} |R(X)K_{\chi,f})(z,y)|\right)
  \end{align*}
pour  tout $x\in [G^{\theta'}]^G$,  tout $f\in \Sc(G(\AAA))^J$ et tout $\chi\in \Xgo(G)$. En utilisant le lemme \ref{lem:rappel*BPCZ} pour le poids $\om=1$,  il existe $N_0>0$ tel que, pour tout $N_3>0$, il existe une semi-norme  continue $\|\cdot\|_{\Sc} $ sur $\Sc(G(\AAA))$ telle que
        \begin{align*}
    \sum_{\chi\in \Xgo(G)  } \sup_{z\in [G]^1,X\in \Fgo} \left(\|z\|^{-N_0-N_3}_{P} |(R(X)K_{P,\chi,f})(z,y)|  \right) \leq \|y\|_{G}^{-N_3}\|f\|_{\Sc}
        \end{align*}
        pour tout $f\in \Sc(G(\AAA))$ et tout $y\in G(\AAA)$. Il est aisé de conclure.

       Enfin \eqref{eq:asym 3} résulte de   la combinaison de \eqref{eq:asym 1} et \eqref{eq:asym 2}. 
      \end{preuve}
\end{paragr}

\subsection{Comportement en \texorpdfstring{$T$}{T}}\label{ssec:comport T}

\begin{paragr}\label{S:JchiT}
  On reprend les notations et les hypothèses des sous-sections \ref{ssec:maj noy tronq} et  \ref{ssec:enonces spec}.   Soit $\eta:F^\times\back \AAA^\times \to \CC^\times$ un caractère continu, trivial sur l'image de  $\RR_+^\times$ dans  $(F\otimes_\QQ \RR)^\times$ par le morphisme $1\otimes \Id_{\RR}$. Notons que $\eta$ est à valeurs dans le cercle unité. Par composition avec la norme réduite, on en déduit un caractère $G(\AAA)\to \CC^\times$ trivial sur les sous-groupes $A_G^\infty$ et $G(F)$, encore noté $\eta$.

  Soit $T\in T_0+\overline{\ago_0^+}$ et $\chi\in \Xgo(G)$. On introduit la distribution  $J^T_\chi(\eta)$ sur $\Sc(G(\AAA))$ donnée par l'intégrale suivante dont la convergence absolue et la continuité en $f\in \Sc(G(\AAA))$ est garantie par le théorème \ref{thm:cv-spec}:
  \begin{align}\label{eq:JTchieta}
    J^T_\chi(\eta,f)=  \int_{[G^{\theta'}]^G\times [G^\theta]      }  K^T_{\chi,f}(x,y) \, \eta(x)dxdy.
  \end{align}
\end{paragr}

\begin{paragr} Nous allons étudier la dépendance en $T$ de la distribution $J_\chi^T(\eta)$: nous allons voir que, comme fonction de $T$, elle coïncide, sur $T_0+\overline{\ago_0^+}$, avec une fonction  polynôme-exponentielle c'est-à-dire une application $\ago_0 \to \CC$ qui appartient au $\CC$-espace vectoriel engendré par les fonctions $T\in \ago_0 \mapsto p_\la(T) \exp(\bg \la,T\bd)$ où $\la\in \ago_{0,\CC}^*$ et $p_\la$ est une fonction  polynomiale. Les $\la$ pour lesquels $p_\la\not=0$ seront appelés exposants  de la fonction  polynôme-exponentielle. La partie polynomiale d'une fonction  polynôme-exponentielle  est, par définition, le polynôme $p_0$ associé à $\la=0$. Le terme constant est la valeur en $T=0$ de $p_0$.

Il nous faudra introduire aussi des analogues de  $J^T_\chi(\eta)$ pour les sous-groupes de Levi de $G$. Soit $Q$ un sous-groupe parabolique standard de $G$. Le groupe $P_0\cap M_Q$ est un sous-groupe parabolique de $M_Q$, défini sur $F$ et minimal, de facteur de Levi $M_0$. Soit $w_1'\in \, _QW_{\theta'}$ et $w_2'\in \, _QW_{\theta}$.  On pose:
  \begin{align}\label{eq:theta12}
    \theta_1=w_1'\theta' {w_1'}^{-1} \text{  et  }\theta_2=w_2' \theta {w_2'}^{-1}.
  \end{align}
Ce sont des éléments d'ordre au plus $2$ de $M_0(F)$.

  Soit  $f'\in \Sc(M_Q(\AAA))$ et  $\chi'\in \Xgo(M_Q)$. On dispose du noyau modifié $K^{T,\theta_1,\theta_2}_{f',\chi'}$ dont la définition est donnée par \eqref{eq:KT} où l'on substitue $M_Q$ à $G$. On définit alors
  \begin{align}\label{eq:JQTtheta}
  &  J^{Q,T,\theta_1,\theta_2}_{\chi'}(\eta,f')\\
   \nonumber& =  \int_{    [M_{Q}^{\theta_1}]^Q}  \int_{    [M_{Q}^{\theta_2}]}  \exp(-\bg 2\rho_{Q^{\theta_1}}^{G^{\theta_1}}, H_{Q^{\theta_1}   }( x)\bd+\bg 2\rho_Q^G- 2\rho_{Q^{\theta_2}}^{G^{\theta_2}}, H_{Q^{\theta_2}   }( y)\bd) K^{T,\theta_1,\theta_2}_{f',\chi'}(x,y)\, \eta(x)dxdy
  \end{align}
  où l'on pose
   \begin{align*}
  [M_{Q}^{\theta_1}]^Q=M_{Q}^{\theta_1}(F)\back M_{Q}^{\theta_1}(\AAA)\cap M_Q(\AAA)^1.
  \end{align*}
  Nous verrons plus bas l'utilité de l'introduction du facteur exponentiel. On observe que l'assertion 2 du  théorème \ref{thm:cv-spec}   (ou plutôt sa généralisation évidente à $M_Q$ qui est un produit de groupes auxquels le théorème s'applique) assure que l'intégrale ci-dessus est absolument convergente et que la distribution  $J^{Q,T,\theta_1,\theta_2}_{\chi'}(\eta)$ ainsi définie est continue. Il est commode de poser
  \begin{align}\label{eq:JQTthetachi}
      J^{Q,T,\theta_1,\theta_2}_{\chi}(\eta)=  \sum_{\chi' }J^{Q,T,\theta_1,\theta_2}_{\chi'}(\eta)
  \end{align}
 où la somme à droite porte sur l'ensemble fini des $\chi'\in \Xgo(M_Q)$ qui ont pour image $\chi$ par l'application évidente $\Xgo(M_Q)\to \Xgo(G)$.

Pour tout $f\in \Sc(G(\AAA))$, on introduit aussi la fonction $f_{Q,\eta}^{w_1',w_2'}\in \Sc( M_Q(\AAA))$ définie par
  \begin{align}\label{eq:fct-fQ}
 \forall m\in M_Q(\AAA)\ \    f_{Q,\eta}^{w_1',w_2'}(m)= \int_{K^{\theta_1}\times K^{\theta_2} }   \int_{N_Q(\AAA)} f(  (k_1w_1')^{-1}   m  n_Q ( k_2w_2' )) \,\eta(k_1)dn_Qdk_1dk_2.
  \end{align}

  Soit $T' \in \ago_0$. On pose:
  \begin{align}\label{eq:poly pw}
p_{w_1',w_2'}^Q(T,T')=     \int_{\ago_Q^{G}} \exp(\bg 2\rho_Q^G-2\rho_{Q^{\theta_1}}^{G^{\theta_1}}-2\rho_{Q^{\theta_2}}^{G^{\theta_2}}, H \bd) \Ga_Q( H-T,T')\, dH
  \end{align}
  où $\Ga_Q$ est la fonction introduite et notée $\Ga_Q'$ dans \cite[p. 13]{arthur2}. D'après \cite[lemma 2.2]{arthur2}, l'application $T'\mapsto p_{w_1',w_2'}^Q(T,T')$ est un polynôme-exponentielle en $T'$.
  
\end{paragr}

\begin{paragr} Nous sommes maintenant en mesure de formuler le principal énoncé de cette sous-section.

  \begin{proposition}\label{prop:comportement T} Pour tout  $T'\in \overline{\ago_0^+}$, $f\in \Sc(G(\AAA))$ et $\chi\in \Xgo(G)$, on a
    \begin{align*}
      J^{T+T'}_\chi(\eta,f)=  \sum_{Q\in \fc^G(P_0)}  \sum_{w_1'\in \, _QW_{\theta'}} \sum_{w_2'\in \, _QW_{\theta}}   p_{w_1',w_2'}^Q(T,T')   J^{Q,T,\theta_1,\theta_2}_{\chi} ( \eta, f_{Q,\eta}^{w_1',w_2'}).
    \end{align*}
  \end{proposition}

 On rappelle que $\theta_1$ et $\theta_2$ dépendent respectivement de $w_1'$ et $w_2'$ et sont définis en \eqref{eq:theta12}. La preuve de la proposition \ref{prop:comportement T} se trouve au § \ref{S:preuve comport T} ci-dessous.
\end{paragr}

\begin{paragr}[Preuve de la proposition \ref{prop:comportement T}.] ---  \label{S:preuve comport T}Soit  $T'\in \overline{\ago_0^+}$. D'après \cite[p. 14]{arthur2}, pour tout sous-groupe parabolique standard $P$ de $G$ et  tout $H\in \ago_P$, on a
  \begin{align*}
\hat{\tau}_P (H- (T+T')) = \sum_{P \subset Q \subset G} \eps_Q^G \hat{\tau}_P^Q (H - T) \Gamma_Q (H - T, T').
\end{align*}
En utilisant cette égalité et le lemme \ref{lem:Wdec}, on voit que pour tout $x\in G^{\theta'}(\AAA)$ et $y\in  G^{\theta}(\AAA)$, le noyau $K^{T+T'}_{\chi,f}(x,y)$ est égal à la somme sur les sous-groupes paraboliques  standard $Q$ de $G$, sur $w_1'\in \, _QW_{\theta'}$ et  $w_2'\in \, _QW_{\theta}$ de

\begin{align*}
  \sum_{P\subset Q}  \eps_P^Q \sum_{w_1\in _PW^Q_{\theta_1}w_1'}\sum_{w_2\in _PW^Q_{\theta_2}w_2'}\sum_{\delta_1\in P_{w_1}^{\theta'}(F) \back G^{\theta'}(F)} \sum_{\delta_2\in P_{w_2}^\theta(F) \back G^\theta(F)}  \\
    \Ga_Q( H_Q(w_1'\delta_1x)-T,T')   \hat\tau_P^Q(H_P(w_1\delta_1x)-T)K_{P,\chi,f}(w_1\delta_1x,w_2\delta_2y),
\end{align*}
 où, comme ci-dessus,  $\theta_1=w_1'\theta' {w_1'}^{-1}$ et $\theta_2=w_2' \theta {w_2'}^{-1}$.

Il s'ensuit  qu'on  a
\begin{align*}
    J^{T+T'}_\chi(\eta,f)=  \sum_{Q\in \fc^G(P_0)}  \sum_{w_1'\in \, _QW_{\theta'}} \sum_{w_2'\in \, _QW_{\theta}}   J^{Q,T,T'}_{\chi,w_1',w_2'}(\eta,f)
\end{align*}
où l'on définit,  pour $Q$, $w_1'\in \, _QW_{\theta'}$, $w_2'\in \, _QW_{\theta}$ comme ci-dessus, $J^{Q,T,T'}_{\chi,w_1',w_2'}(\eta,f)$ par l'intégrale:

 \begin{align*}
   \int_{    Q_{w_1'}^{\theta'}(F)\back G^{\theta'}(\AAA)\cap G(\AAA)^1}  \Ga_Q( H_Q(w_1'x)-T,T') \int_{    Q_{w_2'}^\theta(F)\back G^\theta(\AAA)}    \sum_{P\subset Q}  \eps_P^Q   \sum_{w_1\in _PW^Q_{\theta_1}}\sum_{w_2\in _PW^Q_{\theta_2}}  \\
   \sum_{\delta_1\in P_{w_1 w_1'}^{\theta'}(F) \back Q_{w_1'}^{\theta'}(F)} \sum_{\delta_2\in P_{w_2w_2'}^\theta(F) \back Q_{w_2'}^\theta(F)}  \hat\tau_P^Q(H_P(w_1w_1'\delta_1 x)-T)K_{P,\chi,f}(w_1w_1'\delta_1x,w_2w_2'\delta_2y) \, \eta(x)dxdy.
   \end{align*}

Désormais, nous fixons $Q$, $w_1'\in \, _QW_{\theta'}$ et $w_2'\in \, _QW_{\theta}$.  Par un changement de variables évident, l'expression ci-dessus devient:
   
\begin{align*}
   \int_{    Q_{}^{\theta_1}(F)\back G^{\theta_1}(\AAA)\cap G(\AAA)^1}  \Ga_Q( H_Q(x)-T,T') \int_{    Q_{ }^{\theta_2}(F)\back G^{\theta_2}(\AAA)}    \sum_{P\subset Q}  \eps_P^Q   \sum_{w_1\in _PW^Q_{\theta_1}}\sum_{w_2\in _PW^Q_{\theta_2}}  \\
  \sum_{\delta_1\in P_{w_1}^{\theta_1}(F) \back Q^{\theta_1}(F)} \sum_{\delta_2\in P_{w_2}^{\theta_2}(F) \back Q^{\theta_2}(F)}  \hat\tau_P^Q(H_P(w_1\delta_1 x)-T)K_{P,\chi,f}(w_1\delta_1x w_1',w_2\delta_2y w_2') \, \eta(x)dxdy.
\end{align*}

Observons que l'intégrande, comme fonction de $(x,y)$ est invariante à gauche par $N_{Q_{}^{\theta_1}}(\AAA)\times N_{Q_{}^{\theta_2}}(\AAA)$. Vu que le volume de $[N_{Q_{}^{\theta_1}}]\times [N_{Q_{}^{\theta_2}}]$ vaut  $1$, la formule de décomposition \eqref{eq:Iwasawa mes}  entraîne que l'intégrale ci-dessus est égale à

   \begin{align*}
     \int_{    [M_{Q}^{\theta_1}]^G}  \int_{K^{\theta_1}\times K^{\theta_2} }\exp(-\bg 2\rho_{Q^{\theta_1}}^{G^{\theta_1}}, H_{Q^{\theta_1}   }(x)\bd) \Ga_Q( H_Q(x)-T,T') \int_{    [M_{Q}^{\theta_2}]}  \exp(-\bg 2\rho_{Q^{\theta_2}}^{G^{\theta_2}}, H_{Q^{\theta_2}   }(y)\bd)  \\
     \sum_{P\subset Q}  \eps_P^Q   \sum_{w_1\in _PW^Q_{\theta_1}}\sum_{w_2\in _PW^Q_{\theta_2}}  
     \sum_{\delta_1\in (P_{w_1}\cap M_{Q}^{\theta_1})(F) \back M_{Q}^{\theta_1}(F) }\sum_{\delta_2\in (P_{w_2}\cap M_{Q}^{\theta_2})(F) \back M_{Q}^{\theta_2}(F) }  \\
     \hat\tau_P^Q(H_P(w_1\delta_1 x)-T)K_{P,\chi,f}(w_1\delta_1x k_1w_1',w_2\delta_2y k_2w_2') \, \eta(k_1)dk_1dk_2\eta(x)dxdy.
   \end{align*}
   où $[M_{Q}^{\theta_1}]^G=M_{Q}^{\theta_1}(F)\back (M_{Q}^{\theta_1}(\AAA)\cap G(\AAA)^1)$. Pour tout $m_i\in M^{\theta_i}(\AAA)$ et $k_i\in K^{\theta_i}$ pour $i=1,2$, on obtient, grâce à un changement de variables,
   
  \begin{align*}
    K_{P,f}(m_1k_1w_1',m_2k_2w_2')=\sum_{\ga\in M_P(F)} \int_{N_P^Q(\AAA)}   \int_{N_Q(\AAA)} f(  ( m_1k_1w_1')^{-1} \ga nn_Q (m_2 k_2w_2' ))\,dndn_Q\\
    =\exp(\bg 2\rho_Q^G,H_Q(m_2)\bd) \sum_{\ga\in M_P(F)} \int_{N_P^Q(\AAA)}   \int_{N_Q(\AAA)} f(  ( m_1k_1w_1')^{-1} \ga n  m_2 n_Q ( k_2w_2' )) \,dndn_Q.
  \end{align*}
  
  Pour le sous-groupe parabolique $P\cap M_Q$ de $M_Q$, toute fonction $f'\in \Sc( M_Q(\AAA))$ et toute donnée cuspidale $\chi'\in \Xgo(M_Q)$, on dispose du noyau $K_{P\cap M_Q,f',\chi'}$. On pose alors
  \begin{align*}
    K_{P\cap M_Q,f',\chi}=\sum_{\chi' }K_{P\cap M_Q,f',\chi'}
  \end{align*}
  où la somme à droite porte sur l'ensemble fini des $\chi'\in \Xgo(M_Q)$ qui ont pour image $\chi$ par l'application évidente $\Xgo(M_Q)\to \Xgo(G)$. De même, à partir du noyau modifié $K^{T,\theta_1,\theta_2}_{f',\chi'}$ relatif à  $\chi'\in \Xgo(M_Q)$ et  aux éléments $\theta_1, \theta_2$, on définit $K^{T,\theta_1,\theta_2}_{f',\chi}$ pour la donnée cuspidale $\chi\in \Xgo(G)$.
  
  Le calcul ci-dessus entraîne qu'on a pour tout $m_i\in M^{\theta_i}(\AAA)$, $i=1,2$,
  
  \begin{align*}
    \int_{K^{\theta_1}\times K^{\theta_2} }  K_{P,f,\chi}(m_1k_1w_1',m_2k_2w_2')\eta(k_1)dk_1dk_2 = \exp(\bg 2\rho_Q^G,H_Q(m_2)\bd)   K_{P\cap M_Q,f_{Q,\eta}^{w_1',w_2'},\chi}(m_1,m_2).
  \end{align*}

  On obtient alors:
  
  \begin{align*}
    J^{Q,T,T'}_{\chi,w_1',w_2'}(\eta,f)=\int_{    [M_{Q}^{\theta_1}]^G}  \exp(-\bg 2\rho_{Q^{\theta_1}}^{G^{\theta_1}}, H_{Q^{\theta_1}   }(x)\bd)  \Ga_Q( H_Q(x)-T,T') \\\int_{    [M_{Q}^{\theta_2}]}  \exp(\bg 2\rho_Q^G-2\rho_{Q^{\theta_2}}^{G^{\theta_2}}, H_{Q^{\theta_2}   }(y)\bd) K^{T,\theta_1,\theta_2}_{f_{Q,\eta}^{w_1',w_2'},\chi}(x,y)\, \eta(x)dxdy.
  \end{align*}

  On a  $[M_{Q}^{\theta_1}]^G=[M_{Q}^{\theta_1}]^Q \times A_Q^{G,\infty}$, cette décomposition étant compatible au choix des mesures. Notons que $A_Q^{G,\infty}\subset (M_{Q}^{\theta_1}\cap M_{Q}^{\theta_2})(\AAA)$ et que le caractère $\eta$ est trivial sur $A_Q^{G,\infty}$.  Pour tout $a\in  A_Q^{G,\infty}$ et $x\in [M_{Q}^{\theta_1}]^Q$, on a
  \begin{align*}
    \int_{    [M_{Q}^{\theta_2}]}  \exp(\bg 2\rho_Q^G-2\rho_{Q^{\theta_2}}^{G^{\theta_2}}, H_{Q^{\theta_2}   }(y)\bd) K^{T,\theta_1,\theta_2}_{f_{Q,\eta}^{w_1',w_2'},\chi}(a x,y)\, dy\\
    = \exp(\bg 2\rho_Q^G-2\rho_{Q^{\theta_2}}^{G^{\theta_2}}, H_{Q   }( a)\bd)  \int_{    [M_{Q}^{\theta_2}]}  \exp(\bg 2\rho_Q^G-2\rho_{Q^{\theta_2}}^{G^{\theta_2}}, H_{Q^{\theta_2}   }( y)\bd) K^{T,\theta_1,\theta_2}_{f_{Q,\eta}^{w_1',w_2'},\chi}(x,y)\, dy
  \end{align*}

  Finalement,  on obtient que $J^{Q,T,T'}_{\chi,w_1',w_2'}(\eta, f)$ est le produit de $p_{w_1',w_2'}^Q(T,T')$
  
  et de
  \begin{align*}
    \int_{    [M_{Q}^{\theta_1}]^Q}  \int_{    [M_{Q}^{\theta_2}]}  \exp(-\bg 2\rho_{Q^{\theta_1}}^{G^{\theta_1}}, H_{Q^{\theta_1}   }( x)\bd+\bg 2\rho_Q^G-2\rho_{Q^{\theta_2}   }^{G^{\theta_2}}, H_{Q^{\theta_2}   }( y)\bd) K^{T,\theta_1,\theta_2}_{f_{Q,\eta}^{w_1',w_2'},\chi}(x,y)\, \eta(x)    dxdy
  \end{align*}
  Cette dernière intégrale n'est autre que  $J^{Q,T,\theta_1,\theta_2}_{\chi} ( \eta, f_{Q,\eta}^{w_1',w_2'})$.
\end{paragr}

\subsection{Distributions spectrales}\label{ssec:distr spec}

\begin{paragr}
  Soit $\theta_1,\theta_2\in M_0(F)$ deux éléments d'ordre au plus $2$. Pour tout sous-groupe parabolique standard $Q$ de $G$, soit $\fc^{G,\flat}(Q,\theta_1,\theta_2)$  l'ensemble des sous-groupes paraboliques $R$ de $G$ qui contiennent  $Q$ et qui vérifient
    \begin{align}\label{eq:la condition}
      \rho_R^G=(\rho_{R^{\theta_1}}^{G^{\theta_1}}+\rho_{R^{\theta_2}}^{G^{\theta_2}})_R
    \end{align}
    où le dernier indice $R$ désigne la projection orthogonale sur $\ago_R$.
\end{paragr}

\begin{paragr}
    Pour tous sous-groupes paraboliques $P\subset Q$, on définit les fonctions  polynomiales de la variable $\la\in \ago_{0,\CC}^*$ :
$$\hat{\Theta}_P^Q(\la)= \vol(\ago_P^Q/\ZZ(\hat{\Delta}_P^{Q,\vee}))^{-1} \prod_{\varpi^\vee \in \hat{\Delta}_P^{Q,\vee}}\bg \la,\varpi^\vee\bd
$$
et
$$\Theta_P^Q(\la)= \vol(\ago_P^Q/\ZZ(\Delta_P^{Q,\vee}))^{-1} \prod_{\al \in \Delta_P^Q} \bg \la,\al^\vee\bd.
$$
Les facteurs volumes ci-dessus désignent les covolumes des réseaux engendrés respectivement par $\hat{\Delta}_P^{Q,\vee}$ et ${\Delta}_P^{Q,\vee}$. Comme d'habitude, on omet l'exposant $G$ lorsque $Q=G$ et que le contexte est clair.
\end{paragr}

\begin{paragr}\label{S:terme_cst}
  On reprend les notations de la sous-section \ref{ssec:comport T}. Soit $Q\subset G$ un sous-groupe parabolique standard. Soit $w_1'\in \, _QW_{\theta'}$ et $w_2'\in \, _QW_{\theta}$.  On définit alors $ \theta_1$ et $\theta_2$ comme en  \eqref{eq:theta12}.  Rappelons qu'on a défini une application  $p_{w_1',w_2'}^Q$ sur $\ago_0\times \ago_0$ en \eqref{eq:poly pw}.

  \begin{lemme}\label{lem:le terme cst}Soit $T\in \ago_0$. 
    L'application $T'\in\ago_0 \mapsto p_{w_1',w_2'}^Q(0,T'-T)$ est un polynôme-exponentielle. De plus, le terme constant de cette application est donnée par l'expression suivante:
    \begin{align*}
  c_{w_1',w_2'}^Q(T)=     \lim_{t\to 0}  \sum_{ R\in \fc^{G,\flat}(Q,\theta_1,\theta_2) } \eps_Q^R \frac{\exp(t \bg \la , -T_R\bd)}{ t^{\dim(\ago_R^G)}\hat\Theta_Q^R(t\la+2\rho_Q^G-2\rho_{Q^{\theta_1}}^{G^{\theta_1}}-2\rho_{Q^{\theta_2}}^{G^{\theta_2}}) \Theta_R^G(\la)}
    \end{align*}
    où $\la\in \ago_Q^{G,*}$ est un élément en position générale.
  \end{lemme}

  \begin{preuve}
    D'après \cite[lemma 2.2]{arthur2}, on a l'égalité suivante pour tout $\la\in \ago_Q^{G,*}$ en position générale
    \begin{align*}
      \int_{\ago_Q^{G}} \exp( \la+\bg 2\rho_Q^G-2\rho_{Q^{\theta_1}}^{G^{\theta_1}}-2\rho_{Q^{\theta_2}}^{G^{\theta_2}}, H \bd) \Ga_Q( H,T'-T)\, dH\\
      =  \sum_{Q\subset R}  \eps_Q^R \frac{\exp( \bg \la +  2\rho_R^G-2\rho_{R^{\theta_1}}^{G^{\theta_1}}-2\rho_{R^{\theta_2}}^{G^{\theta_2}} , T'_R-T_R\bd)}{ (\hat\Theta_Q^R\Theta_R^G)(\la+2\rho_Q^G-2\rho_{Q^{\theta_1}}^{G^{\theta_1}}-2\rho_{Q^{\theta_2}}^{G^{\theta_2}})}.
    \end{align*}
    L'intégrale définit une fonction analytique de $\la\in \ago_{Q,\CC}^{G,*}$ car $H\in \ago_Q^G \mapsto\Ga_Q( \cdot,T'-T)$ est à support compact (ce support dépendant de $T'-T$). Sa valeur en $\la=0$ est l'expression $p_{w_1',w_2'}^Q(0,T'-T)$. Il s'ensuit que le second membre de l'égalité ci-dessus se prolonge analytiquement à $\ago_{Q,\CC}^{G,*}$. Il résulte aussi du calcul ci-dessus que $T'\mapsto   p_{w_1',w_2'}^Q(0,T'-T)$ est un polynôme-exponentielle en $T'$. D'après \cite[Lemma 2.4.1.1]{chaudouardsymmetric}, la partie purement polynomiale de $T'\mapsto   p_{w_1',w_2'}^Q(0,T'-T)$ est alors donnée par la valeur en $\la=0$ de l'expression analytique au voisinage de $0$
    \begin{align}\label{eq:petite somme}
        \sum_{R}  \eps_Q^R \frac{\exp( \bg \la  , T'_R-T_R\bd)}{ \hat\Theta_Q^R(\la+2\rho_Q^G-2\rho_{Q^{\theta_1}}^{G^{\theta_1}}-2\rho_{Q^{\theta_2}}^{G^{\theta_2}})\Theta_R^G(\la)}
    \end{align}
    où l'on somme sur  les sous-groupes paraboliques $R$ contenant $Q$ qui vérifient \eqref{eq:la condition}. En particulier, pour obtenir le terme constant, il suffit de remplacer $T'$ par $0$ dans la somme \eqref{eq:petite somme}. Le lemme s'en déduit.
  \end{preuve}

  \begin{proposition}\label{prop:JTchi poly} Soit  $T\in T_0+\overline{\ago_0^+}$.     Soit $f\in \Sc(G(\AAA))$ et $\chi\in\Xgo(G)$. L'application $T'\mapsto J^{T'}_\chi(\eta,f)$ coïncide sur le cône $T+\overline{\ago_0^+}$ avec une fonction polynôme-exponentielle en $T'$ dont le terme constant est donné par
      \begin{align*}
&        J^{}_\chi(\eta,f)=  \\
  &      \sum_{Q\in \fc^G(P_0)}  \sum_{w_1'\in \, _QW_{\theta'}} \sum_{w_2'\in \, _QW_{\theta}}       c_{w_1',w_2'}^Q(T) \exp( \bg 2\rho_Q^G-2\rho_{Q^{\theta_1}}^{G^{\theta_1}}-2\rho_{Q^{\theta_2}}^{G^{\theta_2}}, T_Q\bd)   J^{Q,T,\theta_1,\theta_2}_{\chi} ( \eta, f_{Q,\eta}^{w_1',w_2'}).
    \end{align*}
  \end{proposition}

  \begin{preuve}
    Il résulte de la proposition \ref{prop:comportement T} qu'on a pour tout $T'\in T+\overline{\ago_0^+}$ 
     \begin{align*}
      J^{T'}_\chi(\eta,f)=  \sum_{Q\in \fc^G(P_0)}  \sum_{w_1'\in \, _QW_{\theta'}} \sum_{w_2'\in \, _QW_{\theta}}   p_{w_1',w_2'}^Q(T,T'-T)   J^{Q,T,\theta_1,\theta_2}_{\chi} ( \eta, f_{Q,\eta}^{w_1',w_2'}).
     \end{align*}
     Un changement de variables évident montre qu'on a
     \begin{align*}
        p_{w_1',w_2'}^Q(T,T'-T)  = \exp( \bg 2\rho_Q^G-2\rho_{Q^{\theta_1}}^{G^{\theta_1}}-2\rho_{Q^{\theta_2}}^{G^{\theta_2}}, T_Q\bd)       p_{w_1',w_2'}^Q(0,T'-T). 
     \end{align*}
     Il suffit alors de faire appel au lemme \ref{lem:le terme cst}.
  \end{preuve}

  L'application $f\mapsto   J^{}_\chi(\eta,f)$, notée   $J_\chi(\eta)$, est  alors une forme linéaire continue sur $\Sc(G(\AAA))$ comme il résulte de la  proposition \ref{prop:comportement T}  et du théorème \ref{thm:cv-spec} (ou de sa variante évidente appliquée aux sous-groupes de Levi standard de $G$). On utilisera dans la suite des variantes du résultat ci-dessus. Par exemple, la distribution $J^{Q,T,\theta_1,\theta_2}_{\chi}(\eta)$ définie en \eqref{eq:JQTtheta} et \eqref{eq:JQTthetachi} ne dépend que de $T^Q$: on  peut montrer que, sur un ensemble de $T$ assez positifs, elle coïncide avec un polynôme-exponentielle en $T^Q$. Dans la suite, on note simplement 
  \begin{align}\label{eq:JQtheta12}
  J^{Q,\theta_1,\theta_2}_{\chi}(\eta)
  \end{align}
  son terme constant.

\end{paragr}

\begin{paragr}[Propriétés de covariance de   $J^{T}_\chi(\eta,f)$.]\label{S:covar-chi-T} ---  Dans toute la suite, on prend  $T\in T_0+\overline{\ago_0^+}$. Soit $Q$ un sous-groupe parabolique  standard de $G$,  $w_1'\in \, _QW_{\theta'}$ et  $w_2'\in \, _QW_{\theta}$. On définit alors  $\theta_1$ et $\theta_2$ par la formule \eqref{eq:theta12}. Soit $g\in G^{\theta'}(\AAA)$. On définit alors $f_{Q,\eta,g}^{w_1',w_2'}\in \Sc( M_Q(\AAA))$ par la formule suivante:
  \begin{align}\label{eq:fct-fQT}
    &\forall m\in M_Q(\AAA)\ \    f_{Q,\eta,g}^{w_1',w_2'}(m)= \\
  &  \nonumber \int_{K^{\theta_1}\times K^{\theta_2} }   \int_{N_Q(\AAA)} f(  (k_1w_1')^{-1}   m  n_Q ( k_2w_2' )) p_{w_1',w_2'}^Q(0,-H_Q(k_1w_1'g)) \,\eta(k_1)dn_Qdk_1dk_2.
  \end{align}
où $p_{w_1',w_2'}^Q$ est défini en \eqref{eq:poly pw}.

Pour tout $f\in \Sc(G(\AAA))$ et $g\in G(\AAA)$, on note  $^g\!f$ et $f^g$ les fonctions dans $\Sc(G(\AAA))$ définies par    $^g\!f(x)=f(gx)$ et $f^g(x)=f(xg)$ pour tout $x\in G(\AAA)$.

\begin{proposition}\label{prop:covar-T}Soit $\chi\in\Xgo(G)$ et $f\in \Sc(G(\AAA))$.
  \begin{enumerate}
  \item Pour tout $g\in G^\theta(\AAA)$, on  a
    \begin{align*}
      J^{T}_\chi(\eta,f^g)=J^{T}_\chi(\eta,f).
    \end{align*}
  \item  Pour tout $g\in G^{\theta'}(\AAA)$, on  a
    \begin{align*}
      &      J^{T}_\chi(\eta,  ^g\!\!f)=\\
      &\eta(g)\sum_{Q\in \fc^G(P_0)}  \sum_{w_1'\in \, _QW_{\theta'}} \sum_{w_2'\in \, _QW_{\theta}}    \exp( \bg 2\rho_Q^G-2\rho_{Q^{\theta_1}}^{G^{\theta_1}}-2\rho_{Q^{\theta_2}}^{G^{\theta_2}}, T_Q\bd) J^{Q,T,\theta_1,\theta_2}_{\chi} ( \eta, f_{Q,\eta,g}^{w_1',w_2'}),
    \end{align*}
    où $J^{Q,T,\theta_1,\theta_2}_{\chi} $ est la distribution sur $\Sc( M_Q(\AAA))$ définie en \eqref{eq:JQTtheta} et  \eqref{eq:JQTthetachi}.
  \end{enumerate}
\end{proposition}

\begin{preuve} Il est clair sur la définition \eqref{eq:JTchieta} que $J^T_\chi(\eta)$ est invariante par translations à droite par $G^\theta(\AAA)$. Passons à  l'effet d'une translation à gauche. Pour cela, on suit essentiellement les mêmes calculs que ceux du § \ref{S:preuve comport T} dont on reprend  les notations.

Soit $P$ un  sous-groupe parabolique standard de $G$. On définit une application $k_P:P(\AAA)\back G(\AAA)\to (K\cap P(\AAA))\back K$ ainsi:   pour tout  $x\in G(\AAA)$, on a $P(\AAA)x=P(\AAA)k_{P}(x)$. On a alors
\begin{align}\label{eq:tau par g}
  \hat{\tau}_P (H_P(x g)- T) =  \hat{\tau}_P (H_P(x)+H_{P_0} (k_{P_0}(x)  g) - T) \\
  \nonumber = \sum_{P \subset Q \subset G} \eps_Q^G \hat{\tau}_P^Q (H_P(x) - T) \Gamma_Q (H_Q(x) - T, -H_Q (k_{Q}(x)  g) )
\end{align}
vu que $H_Q (k_{Q}(x)  g) =H_Q (k_{P_0}(x)  g)$.

  Soit $g\in G^{\theta'}(\AAA)$.    On constate par un changement de variables qu'on  a 
  \begin{align*}
    J^T_\chi(\eta, ^g\!\!f)=\int_{A_G^\infty \back [G^{\theta'}]}\int_{[G^\theta] }  K^{T}_{\chi,^g\!f}(x,y) \, \eta(x)dxdy\\
    = \eta(g) \int_{A_G^\infty \back [G^{\theta'}]}\int_{[G^\theta] } \sum_{P\in \fc(P_0)} \eps_P^G \sum_{w_1 \in   \, _PW_{\theta'}} \sum_{\delta_1\in P_{w_1}^{\theta'}(F) \back G^{\theta'}(F)}  \hat\tau_P(H_P(w_1\delta_1xg)-T) \times \\ \left[\sum_{w_2\in   \, _PW_\theta} \sum_{\delta_2\in P_{w_2}^\theta(F) \back G^\theta(F)} K_{P,\chi,f}(w_1\delta_1x,w_2\delta_2y)\right]\, \eta(x)dxdy
  \end{align*}
  En suivant le § \ref{S:preuve comport T} et en utilisant \eqref{eq:tau par g}, on obtient que $J^T_\chi(\eta, \,^g\!f)$ est le produit de $\eta(g)$ par la somme sur les sous-groupes paraboliques  standard $Q$ de $G$, sur $w_1'\in \, _QW_{\theta'}$ et  $w_2'\in \, _QW_{\theta}$ de
 \begin{align*}
   \int_{    Q_{w_1'}^{\theta'}(F)\back G^{\theta'}(\AAA)\cap G(\AAA)^1}  \int_{    Q_{w_2'}^\theta(F)\back G^\theta(\AAA)}    \sum_{P\subset Q}  \eps_P^Q   \sum_{w_1\in _PW^Q_{\theta_1}}\sum_{w_2\in _PW^Q_{\theta_2}}  
   \sum_{\delta_1\in P_{w_1 w_1'}^{\theta'}(F) \back Q_{w_1'}^{\theta'}(F)} \sum_{\delta_2\in P_{w_2w_2'}^\theta(F) \back Q_{w_2'}^\theta(F)}  \\ \Ga_Q( H_Q(w_1'x)-T,   -H_Q(k_{Q}(w_1'\delta_1 x) g)) \hat\tau_P^Q(H_P(w_1w_1'\delta_1 x)-T)K_{P,\chi,f}(w_1w_1'\delta_1x,w_2w_2'\delta_2y) \, \eta(x)dxdy.
 \end{align*}
 où les $\theta_i$ sont définis en \eqref{eq:theta12}.  Par un changement de variables, l'expression ci-dessus devient
  \begin{align*}
   \int_{    Q_{}^{\theta_1}(F)\back G^{\theta_1}(\AAA)\cap G(\AAA)^1}  \int_{    Q_{ }^{\theta_2}(F)\back G^{\theta_2}(\AAA)}    \sum_{P\subset Q}  \eps_P^Q   \sum_{w_1\in _PW^Q_{\theta_1}}\sum_{w_2\in _PW^Q_{\theta_2}}  
    \sum_{\delta_1\in P_{w_1}^{\theta_1}(F) \back Q^{\theta_1}(F)} \sum_{\delta_2\in P_{w_2}^{\theta_2}(F) \back Q^{\theta_2}(F)}\\
    \Ga_Q( H_Q(x)-T, -H_Q(k_{Q}(x) w_1'g))   \hat\tau_P^Q(H_P(w_1\delta_1 x)-T)K_{P,\chi,f}(w_1\delta_1x w_1',w_2\delta_2y w_2') \, \eta(x)dxdy.
  \end{align*}
  Par décomposition d'Iwasawa des mesures, l'expression ci-dessus est égale à
   \begin{align*}
     \int_{    [M_{Q}^{\theta_1}]^G}  \int_{K^{\theta_1}\times K^{\theta_2} }\exp(-\bg 2\rho_{Q^{\theta_1}}^{G^{\theta_1}}, H_{Q^{\theta_1}   }(x)\bd) \Ga_Q( H_Q(x)-T,-H_Q(k_1w_1'g)) \\
     \int_{    [M_{Q}^{\theta_2}]}  \exp(-\bg 2\rho_{Q^{\theta_2}}^{G^{\theta_2}}, H_{Q^{\theta_2}   }(y)\bd)       \sum_{P\subset Q}  \eps_P^Q   \sum_{w_1\in _PW^Q_{\theta_1}}\sum_{w_2\in _PW^Q_{\theta_2}}  \\
     \sum_{\delta_1\in (P_{w_1}\cap M_{Q}^{\theta_1})(F) \back M_{Q}^{\theta_1}(F) }\sum_{\delta_2\in (P_{w_2}\cap M_{Q}^{\theta_2})(F) \back M_{Q}^{\theta_2}(F) }  \\
     \hat\tau_P^Q(H_P(w_1\delta_1 x)-T)K_{P,\chi,f}(w_1\delta_1x k_1w_1',w_2\delta_2y k_2w_2') \, \eta(k_1)dk_1dk_2\eta(x)dxdy.
   \end{align*}

   À ce stade, on décompose  $[M_{Q}^{\theta_1}]^G= [M_{Q}^{\theta_1}]^Q\times A_Q^{G,\infty}$. On a pour tout $a\in  A_Q^{G,\infty}$ et $x\in [M_{Q}^{\theta_1}]^Q$
   \begin{align*}
     \int_{    [M_{Q}^{\theta_2}]}  \exp(-\bg 2\rho_{Q^{\theta_2}}^{G^{\theta_2}}, H_{Q^{\theta_2}   }(y)\bd)       \sum_{P\subset Q}  \eps_P^Q   \sum_{w_1\in _PW^Q_{\theta_1}}\sum_{w_2\in _PW^Q_{\theta_2}}      \\  \sum_{\delta_1,\delta_2}
     \hat\tau_P^Q(H_P(w_1\delta_1  a x)-T)K_{P,\chi,f}(w_1\delta_1a x k_1w_1',w_2\delta_2y k_2w_2') \, dy\\
     =    \exp(\bg 2\rho_Q^G-   2\rho_{Q^{\theta_2}}^{G^{\theta_2}}, H_{Q }(a)\bd)   \int_{    [M_{Q}^{\theta_2}]}  \exp(-\bg 2\rho_{Q^{\theta_2}}^{G^{\theta_2}}, H_{Q^{\theta_2}   }(y)\bd)       \sum_{P\subset Q}  \eps_P^Q   \sum_{w_1\in _PW^Q_{\theta_1}}\sum_{w_2\in _PW^Q_{\theta_2}}      \\  \sum_{\delta_1,\delta_2}
     \hat\tau_P^Q(H_P(w_1\delta_1  x)-T)K_{P,\chi,f}(w_1\delta_1x k_1w_1',w_2\delta_2y k_2w_2') \, dy.
   \end{align*}
   On peut d'abord intégrer sur  $A_Q^{G,\infty}\simeq \ago_Q^G$: cela donne l'intégrale
   \begin{align*}
     &\int_{\ago_Q^{G}} \exp(\bg 2\rho_Q^G-2\rho_{Q^{\theta_1}}^{G^{\theta_1}}-2\rho_{Q^{\theta_2}}^{G^{\theta_2}}, H \bd) \Ga_Q( H-T,-H_Q(k_1w_1'g))\, dH\\
    & = \exp(\bg 2\rho_Q^G-2\rho_{Q^{\theta_1}}^{G^{\theta_1}}-2\rho_{Q^{\theta_2}}^{G^{\theta_2}}, T_Q \bd)  p_{w_1',w_2'}^Q(0,   -H_Q(k_1w_1'g)).
   \end{align*}
   La formule annoncée s'en déduit alors assez facilement.
 \end{preuve}   
\end{paragr}

\begin{paragr}[Propriétés de covariance de   $J_\chi(\eta,f)$.]\label{S:covar-chi} ---

  \begin{proposition}\label{prop:covar}Soit $\chi\in\Xgo(G)$ et $f\in \Sc(G(\AAA))$.
  \begin{enumerate}
  \item Pour tout $g\in G^\theta(\AAA)$, on  a
    \begin{align*}
      J_\chi(\eta,f^g)=J_\chi(\eta,f).
    \end{align*}
  \item  Pour tout $g\in G^{\theta'}(\AAA)$, on  a
    \begin{align*}
      J_\chi(\eta,  ^g\!\!f)=\eta(g) \sum_{Q, w_1',w_2'}  J^{Q,\theta_1,\theta_2}_{\chi} ( \eta, f_{Q,\eta,g}^{w_1',w_2'})
    \end{align*}
    où 
    \begin{itemize}
        \item  la somme porte sur les triplets $(Q, w_1',w_2')$ formés d'un sous-groupe parabolique standard $Q$ et d'éléments $ w_1'\in \, _QW_{\theta'}$ et $w_2'\in \, _QW_{\theta}$ tels que $Q\in \fc^{G,\flat}(P_0,\theta_1,\theta_2)$ et où la fonction $f_{Q,\eta,g}^{w_1',w_2'}$ est définie en \eqref{eq:fct-fQT};
        \item la distribution $J^{Q,\theta_1,\theta_2}_{\chi}$ est celle définie en \eqref{eq:JQtheta12}. 
    \end{itemize}
   
  \end{enumerate}
\end{proposition}

\begin{preuve}
  On déduit les assertions  1 et 2 respectivement  des assertions 1 et 2 de la proposition \ref{prop:covar-T} en prenant les termes constants. C'est clair pour l'assertion 1. Pour l'assertion 2, il faut remarquer que seuls les termes associés à $Q\in \fc^{G,\flat}(P_0,\theta_1,\theta_2)$ peuvent donner un terme constant non nul. En effet, le coefficient de $T_Q$ dans l'exponentielle ne peut être compensé par l'expression $J^{Q,T,\theta_1,\theta_2}_{\chi} ( \eta, f_{Q,\eta,g}^{w_1',w_2'}) $ qui est un polynôme-exponentielle en $T^Q$ dont on note, rappelons-le, $J^{Q,\theta_1,\theta_2}_{\chi} ( \eta, f_{Q,\eta,g}^{w_1',w_2'}) $ le terme constant.
\end{preuve}
\end{paragr}

\begin{paragr}[Distributions $J^T(\eta)$ et $J(\eta)$.]\label{S:JTeta} ---
Pour tout    $f\in \Sc(G(\AAA))$ et tout $T\in T_0+\overline{\ago_0^+}$, on définit 
  \begin{align*}
   &J^T(\eta,f)=\sum_{\chi\in \Xgo(G)}J_\chi^T(\eta,f), \\
    &J(\eta,f)=\sum_{\chi\in \Xgo(G)}J_\chi(\eta,f).
  \end{align*}
  Ces sommes sont absolument convergentes, une fois encore par le  théorème \ref{thm:cv-spec}. Elles définissent des formes linéaires $J^T(\eta)$ et  $J(\eta)$ qui sont continues sur $\Sc(G(\AAA))$. Il résulte de la proposition \ref{prop:JTchi poly} que $J^T(\eta)$ coïncide avec une fonction polynôme-exponentielle dont le terme constant est $J(\eta)$.
\end{paragr}

\begin{paragr}[Description de $\fc^{G,\flat}(P_0,\theta_1,\theta_2)$.]\label{S:ens-F^flat} --- Dans ce paragraphe, dans deux cas particuliers, on donne un critère pour qu'un sous-groupe parabolique vérifie \eqref{eq:la condition}, qui est une généralisation de \cite[lemme 5.5]{li1}. Soit $\theta_1,\theta_2\in M_0(F)$ deux éléments d'ordre au plus $2$. 
Soit $R\in\fc^G(P_0)$ le stabilisateur d'un drapeau $V_0=0\subsetneq V_1\subsetneq V_2\subsetneq \ldots \subsetneq V_m=V$  de sous-$D$-modules de $V$. On a des décompositions en sous-$D$-modules 
    \begin{align*}
        V=V'^+\oplus V'^-=V^+\oplus V^-
    \end{align*}
    telles que $\theta_1$ et $ \theta_2$ agissent par $\pm1$ respectivement sur $V'^\pm$ et $ V^\pm$. Rappelons $N=\dim_D(V)$ et posons $p'=\dim_D (V'^+)$, $q'=\dim_D (V'^-)$, $p=\dim_D (V^+)$ et $q=\dim_D (V^-)$. 
Soit $1\leq i \leq m$. Soit $W_i\subset V_i$ de sorte que $V_i=V_{i-1}\oplus W_i$ et que $M_R$ soit le stabilisateur des sous-espaces $W_i$.  On pose 
    \begin{align*}        W'^\pm_i=W_i\cap V'^\pm \text{ et } W^\pm_i=W_i\cap V^\pm. \end{align*}
Soit $n_i=\dim_D(W_i)$, $p'_i=\dim_D (W'^+_i)$, $q'_i=\dim_D (W'^-_i)$, $p_i=\dim_D (W^+_i)$ et $q_i=\dim_D (W^-_i)$. Il est évident que 
\begin{align*}
    n_i=p'_i+q'_i=p_i+q_i \text{ pour tout } i, \sum_{1\leq i\leq m} n_i=N, \\
    \sum_{1\leq i\leq m} p'_i=p', \sum_{1\leq i\leq m} q'_i=q', \sum_{1\leq i\leq m} p_i=p \text{ et } \sum_{1\leq i\leq m} q_i=q.
\end{align*}
On écrit 
    $N_i=\sum\limits_{1\leq j\leq i} n_j$, $P'_i=\sum\limits_{1\leq j\leq i} p'_j$, $Q'_i=\sum\limits_{1\leq j\leq i} q'_j$, $P_i=\sum\limits_{1\leq j\leq i} p_j$ et $Q_i=\sum\limits_{1\leq j\leq i} q_j$. 

    Pour tout $i$ soit $e_i^*\in\ago_0^*$ le caractère pour l'action de $A_0$ sur $e_i$ où $(e_1,\cdots,e_N)$ est la base canonique de $V$. On note $(e_1^\vee,\cdots,e_N^\vee)$ la base duale de $(e_1^*,\cdots,e_N^*)$ c'est-à-dire $e_i^\vee\in\ago_0$ avec $e_i^*(e_j^\vee)=\delta_{ij}$ pour tous nombres $1\leq i,j\leq N$. Ainsi une base de $\ago_R$ est donnée par $(h_1^\vee,\cdots,h_N^\vee)$ où $h_i^\vee=\sum\limits_{N_{i-1} < j\leq N_i} e_j^\vee$ pour tout $i$. On note $(h_1^*,\cdots,h_N^*)$ sa base duale c'est-à-dire $h_i^*\in\ago_R^*$ avec $h_i^*(h_j^\vee)=\delta_{ij}$ pour tout $1\leq i,j\leq N$. Pour tout $1\leq k<m$ on définit 
    \begin{align*}
        \varpi_k^\vee=\frac{N-N_k}{N}\left(\sum_{1\leq i\leq k} h_i^\vee\right)-\frac{N_k}{N}\left(\sum_{k<j\leq m} h_j^\vee\right).
    \end{align*}
    Il est connu que $\hat\Delta_R^{G,\vee}=\{\varpi_k^\vee : 1\leq k< m\}$ est une base de $\ago_R^G$. Alors $R\in\fc^{G, \flat}(P_0, \theta_1,\theta_2)$ si et seulement si $(2\rho_R^G-2\rho_{R^{\theta_1}}^{G^{\theta_1}}-2\rho_{R^{\theta_2}}^{G^{\theta_2}})(\varpi_k^\vee)=0$ pour tout $1\leq k< m$. 
    On voit que 
    \begin{align*}
        (2\rho_R^G)_R=\dim_F(D) \sum_{1\leq i<j\leq m} n_i n_j (h_i^*-h_j^*), \\(2\rho_{R^{\theta_1}}^{G^{\theta_1}})_R=\dim_F(D) \sum_{1\leq i<j\leq m} (p'_i p'_j + q'_i q'_j)(h_i^*-h_j^*), \\(2\rho_{R^{\theta_2}}^{G^{\theta_2}})_R=\dim_F(D) \sum_{1\leq i<j\leq m} (p_i p_j + q_i q_j)(h_i^*-h_j^*). 
    \end{align*}
    Pour $1\leq i<j\leq m$ et $1\leq k< m$ on a 
      \begin{align*}
  (h_i^*-h_j^*)(\varpi_k^\vee) = \left\{ \begin{array}{ll}
0 & \text{si $i>k$ ou $j\leq k$ ; }\\
1 & \text{si $i\leq k<j$. }\\
\end{array} \right.
  \end{align*}
    On en déduit que pour tout $1\leq k< m$ 
    \begin{align}\label{eq:rhoRvarpi}
        (2\rho_R^G-2\rho_{R^{\theta_1}}^{G^{\theta_1}}-2\rho_{R^{\theta_2}}^{G^{\theta_2}})(\varpi_k^\vee)=\dim_F(D)\sum_{1\leq i\leq k} \sum_{k<j\leq m} (n_i n_j - p'_i p'_j - q'_i q'_j - p_i p_j - q_i q_j) \\ \nonumber
        =\dim_F(D) (N_k (N-N_k) - P'_k (p'-P'_k) - Q'_k (q'-Q'_k) - P_k (p-P_k) - Q_k (q-Q_k)).
    \end{align}
  
\begin{proposition}\label{prop:lacond1=2}
    Supposons  $\theta_1=\theta_2$. 
    Alors $R\in\fc^{G, \flat}(P_0, \theta_1,\theta_2)$ si et seulement si, pour tout $1\leq k< m$, on a $P_k-Q_k=0$ ou $=p-q$. 
\end{proposition}

\begin{preuve}
    Si $\theta_1=\theta_2$, on a $p'_i=p_i$ et $q'_i=q_i$ pour tout $1\leq i\leq m$. Donc \eqref{eq:rhoRvarpi} devient le produit de $\dim_F(D)$ avec 
    \begin{align*}
        N_k (N-N_k)  - 2 P_k (p-P_k) - 2 Q_k (q-Q_k). 
    \end{align*}
    Soit $a_k=p-P_k$ et $b_k=q-Q_k$. Comme $N_k=P_k+Q_k$ et $N=p+q$, la dernière expression est égale à 
    \begin{align*}
        (P_k+Q_k) (a_k+b_k)  - 2 P_k a_k - 2 Q_k b_k=(P_k-Q_k)(b_k-a_k). 
    \end{align*}
    Elle s'annule si et seulement si $P_k-Q_k=0$ ou $=p-q$. Cela conclut. 
\end{preuve}
  
\begin{proposition}\label{prop:fcGb}
 Supposons   
    \begin{align}\label{eq:cas interessant}
        \theta'=\theta \text{ et } \Trd(\theta)=0
    \end{align}
    ce qui implique que $N$ est pair, cf. \S \ref{S:D}. Soit $w'_1, w'_2\in W$. On pose : 
    \begin{align*}
        \theta_1=w_1'\theta {w_1'}^{-1} \text{  et  }\theta_2=w_2' \theta {w_2'}^{-1}. 
    \end{align*} 
    Alors $R\in\fc^{G, \flat}(P_0, \theta_1,\theta_2)$ si et seulement si, pour tout $1\leq i \leq m$, on a 
      \begin{align*}
          p'_i=q'_i=p_i=q_i=n_i/2. 
      \end{align*}
\end{proposition}

\begin{preuve}
 Il résulte de l'hypothèse \eqref{eq:cas interessant} que  
    \begin{align*}
        p'=q'=p=q=N/2.        
    \end{align*}
    Mais $N_k=P'_k+Q'_k=P_k+Q_k$. On a alors 
    \begin{align*}
        N_k N - P'_k p' - Q'_k q' - P_k p - Q_k q=N_k N - (P'_k+Q'_k+P_k+Q_k) N/2=0. 
    \end{align*}
    L'expression \eqref{eq:rhoRvarpi} devient le produit de $\dim_F(D)$ avec 
    \begin{align*}
        -N_k^2+P'^2_k+Q'^2_k+P^2_k+Q^2_k=-(P'_k+Q'_k)^2/2-(P_k+Q_k)^2/2+P'^2_k+Q'^2_k+P^2_k+Q^2_k \\
        =(P'_k-Q'_k)^2/2+(P_k-Q_k)^2/2\geq0. 
    \end{align*}
    L'égalité est vérifiée pour tout $1\leq k< m$ si et seulement si $p'_i=q'_i=p_i=q_i=n_i/2$ pour tout $1\leq i\leq m$. 
\end{preuve}

\begin{corollaire}
    Supposons \eqref{eq:cas interessant}. Soit $Q\subset G$ un sous-groupe parabolique standard. Soit $w_1', w_2'\in \, _QW_{\theta}$.  On définit alors $ \theta_1$ et $\theta_2$ comme en  \eqref{eq:theta12}. Si $Q\in\fc^{G, \flat}(P_0, \theta_1,\theta_2)$, alors $w'_1=w'_2$ et donc $\theta_1=\theta_2$. 
\end{corollaire}

\begin{preuve}
D'après la proposition \ref{prop:fcGb}, il existe $w\in W^Q$ tel que $w\theta_1w^{-1}=\theta_2$. Il s'ensuit que $w'^{-1}_2w w'_1\in W^\theta$. Par la proposition \ref{prop:PWtheta}, on en déduit que $w'_1=w'_2$. 
\end{preuve}
\end{paragr}

\begin{paragr}[Le cas $GL_2(D)$.] --- \label{S:GL2D spec}On se focalise dans ce § sur le cas $N=2$, de sorte que $G=GL_2(D)$, et \begin{align*}
    \theta=\theta'=\begin{pmatrix}
      1 & 0 \\ 0 & -1
    \end{pmatrix}.
\end{align*}
D'après la proposition \ref{prop:fcGb}, le seul élément possible de $\fc^{G,\flat}(P_0,\theta_1,\theta_2)$ est $G$ lui-même. En particulier, les relations de covariance de la proposition \ref{prop:covar} se simplifient drastiquement. On obtient alors des termes essentiellement invariants.

\begin{proposition}\label{prop:GL2 spec}
    Sous les conditions ci-dessous, pour tout $g\in G^\theta(\AAA)$, on  a
    \begin{align*}
      & J_\chi(\eta,f^g)=J_\chi(\eta,f)\\
      & J_\chi(\eta,  ^g\!\!f)= \eta(g) J_\chi(\eta,f).
    \end{align*}
  \end{proposition}
\end{paragr}

\section{Développement géométrique}\label{sec:dvpt geo}

\subsection{Espace symétrique et quotient catégorique}\label{ssec:preparatif alg}

\begin{paragr}
  On reprend les notations de la sous-section \ref{ssec:involutions}. En particulier, on suppose que $F$ est un corps commutatif de caractéristique $0$ et  que  $D$ est une algèbre à division centrale et de dimension finie sur $F$. On rappelle aussi qu'on a $V=D^N$, $G=GL_D(V)$ et $\theta\in G(F)$ qui vérifie $\theta^2=1$. On dispose d'un sous-groupe parabolique minimal $P_0\subset G$ et d'un facteur de Levi $M_0$ tel que $\theta\in M_0(F)$. Pour tout $g\in G$, on note $\Int (g): x\mapsto gxg^{-1}$ l'automorphisme intérieur de $G$ induit par la conjugaison de $g$. Ainsi $\Int(\theta)$ induit une involution de $G$.  On note $d$ le degré de $D$ c'est-à-dire $d^2=\dim_F(D)$. Soit $p=\dim_D(V^+)$ et $q=\dim_D(V^-)$.

  Soit  $\{e_1,\ldots,e_p, f_1,\ldots,f_q\}$ une $D$-base de $V$ telle que $\{e_1,\ldots,e_p\}$, resp. $\{f_1,\ldots,f_q\}$, soit une $D$-base de $V^+$, resp. $V^-$. Ce choix d'une base de $V$ identifie $G$ au groupe algébrique sur $F$ noté  $GL_{N,D}$ dont le groupe des  $F$-points est $GL_D(D^N)$.

  Pour tout entier $0\leq p'\leq N$, on pose 
  \begin{align*}
   q'=N-p' \text{ et }   \theta_{p'}=\theta_{p',q'}= \begin{pmatrix}
  I_{p'} & 0 \\  0 & -I_{q'}
\end{pmatrix}\in M_0(F).
  \end{align*}
  Alors, avec ces identifications, nous avons $\theta=\theta_p=\theta_{p,q}$ et $G^\theta=GL_{p,D}\times GL_{q,D}$.   On suppose également que $M_0$ est stabilisateur des $D$-droites engendrées respectivement par $e_i$ pour $1\leq i\leq p$ et $f_j$ pour $1\leq j\leq q$. Autrement dit, $M_0$ s'identifie  au sous-groupe des matrices diagonales. On peut et on va de plus supposer que $P_0^\theta$ s'identifie au sous-groupe de $G^\theta$ formé des matrices triangulaires supérieures. Notons que l'on peut toujours prendre pour $P_0$ le groupe  des matrices triangulaires supérieures. Ce n'est pas le seul choix possible et on ne l'imposera pas.
\end{paragr}

\begin{paragr}[Espace symétrique.] ---\label{S:esp sym} Soit \begin{align*}
      S=\{g\in G \mid g \Int(\theta)(g)=1\}.
  \end{align*}
Autrement dit, $g\in S$ si et seulement si $(g\theta)^2=1$. Le groupe $G$ agit sur $S$ par la conjugaison tordue $\Int_\theta(g)(s)= g  s \Int(\theta)(g)^{-1}$ pour tout $g\in G$ et tout $s\in S$. Notons que la restriction de cette action à $G^\theta$ n'est autre que la conjugaison usuelle, action que l'on note $\Int$ dans la suite. Comme dans \cite[\S 4.1]{JR2}, il y a alors $N+1$ composantes connexes de $S$, notées  $S_{p'}=S_{p',q'}$ et indexées par les entiers $0\leq p'\leq N$: la composante $S_{p'}$ est formée des $g\theta_{p'}g^{-1}\theta= g\cdot(\theta_{p'}\theta)\cdot\Int(\theta)(g)^{-1} $ pour $g\in G$ autrement dit c'est l'orbite de $\theta_{p'}\theta$. On a aussi 
  \begin{align}\label{eq:Sp' noyau}
      S_{p'}=\{g\in G \mid (g\theta)^2=1, \dim_F(\ker(g\theta-\Id_V))=dp'\}.
  \end{align}
  Le morphisme 
    \begin{align*}
  \rho_{p'}: G \to S 
  \end{align*}
 donné par  $g\mapsto g\theta_{p'}g^{-1}\theta$ identifie alors le quotient $G/G^{\theta_{p'}}$ avec $S_{p'}$. 
Ce morphisme est équivariant pour la multiplication à gauche de $G$ sur $G/G^{\theta_p'}$ et la conjugaison tordue $\Int_\theta$ de $G$ sur $S$. En particulier, il est équivariant pour les actions induites par le sous-groupe $G^\theta$.  Pour $p'=p$, on pose
  \begin{align}\label{eq:app-sym}
  \rho=\rho_p:G\to S_p.
  \end{align}
L'unique composante connexe  $S^\circ$ de $S$ qui contient l'élément neutre de $G$ est $S_p$. Si $p=0$ ou $q=0$, cette composante $S^\circ$ est réduite à un singleton. 
\end{paragr}

\begin{paragr}[Les équations de $S$.] ---   Soit $x\in \End_D(V)$. On note $\Prd_x$, resp. $\Nrd(x)$, resp. $\Trd(x)$, le polynôme caractéristique réduit, resp. la norme réduite, resp. la trace réduite, de $x$ dans $M$. Soit $x^\pm$ les projections de $x$ sur le facteur  $\End_D(V^\pm)$ selon la décomposition 
  \begin{align}\label{eq:EndDV}
  \End_D(V)=\End_D(V^+)\oplus\Hom_D(V^-,V^+)\oplus\Hom_D(V^+,V^-)\oplus\End_D(V^-). 
  \end{align}
  Si $x$ s'écrit $A+B+C+D$ selon la décomposition \eqref{eq:EndDV}, on peut matriciellement écrire 
  \begin{align*}
  x=\begin{pmatrix}
  A & B \\ C & D
\end{pmatrix}. 
  \end{align*}
  
  \begin{lemme}\label{lem:Smatrice}
  Soit $x=
  \begin{pmatrix}
  A & B \\ C & D
\end{pmatrix}
\in \End_D(V)$. 
\begin{enumerate}
	\item On a $x\in S$ si et seulement si les égalités suivantes sont satisfaites
  \begin{align*}
      A^2=I_p+BC, D^2=I_q+CB, AB=BD, CA=DC. 
  \end{align*}
  
  	\item Soit $0\leq p'\leq N$. Alors $x\in S_{p'}$ si et seulement si $x\in S$ et $\Trd(A)-\Trd(D)=d(p'-q')$. 
\end{enumerate}
  \end{lemme}
  
  \begin{preuve}Il s'agit d'un calcul direct. Précisons que les quatre égalités dans l'assertion 1 traduisent le fait que $(x\theta)^2=1$ alors que l'égalité sur les traces réduites dans l'assertion 2 donne la condition sur la dimension de $\ker( x\theta-\Id_V)$.  
  \end{preuve}

Soit $t$ une indéterminée et $F(t)$ le corps des fractions de l'anneau $F[t]$ des polynômes en $t$ à coefficients dans $F$.

  \begin{lemme}\label{lem:Prdx+-}
  Pour tout $0\leq p'\leq N$ et $x\in S_{p'}(F)$, on a l'égalité suivante dans $F(t)$ :
  \begin{align*}
      \Prd_{x^+}(t)=(t-1)^{d(p'-q')}\Prd_{x^-}(t). 
  \end{align*}
  \end{lemme}
  
  \begin{preuve}
  Soit $x=
  \begin{pmatrix}
  A & B \\ C & D
\end{pmatrix}
\in S_{p'}(F)$. À l'aide du lemme \ref{lem:Smatrice}, pour tout entier $m\geq2$, on calcule : 
  \begin{align*}
      A^{m+1}=A^{m-1}(I_p+BC)=A^{m-1}+A^{m-1}BC=A^{m-1}+BD^{m-1}C 
  \end{align*}
où l'on a utilisé $AB=BD$ dans la dernière égalité.  De même, on a 
  \begin{align*}
      D^{m+1}=D^{m-1}(I_q+CB)=D^{m-1}+D^{m-1}CB.  
  \end{align*}
On peut montrer par récurrence que 
  \begin{align*}
      \Trd(A^m)-\Trd(D^m)=d(p'-q')
  \end{align*}
pour tout entier $m\geq1$. Pour conclure on utilise les identités de Newton. 
  \end{preuve}

  \begin{lemme}\label{lem:PrdxS}
  Pour tout $x\in S(F)$, on a les égalités dans $F(t)$ : 
   \begin{align*}
     \Prd_x(t)&=\frac{  \Prd_{x^-}(t)}{  \Prd_{x^+}(t)} \Nrd(t^2 I_p -2t x^{+}+I_p) 
     \\     & =\frac{  \Prd_{x^+}(t)}{  \Prd_{x^-}(t)} \Nrd(t^2 I_q-2t x^- +I_q). 
   \end{align*}
  \end{lemme}
  
  \begin{preuve}
   Soit $x=
  \begin{pmatrix}
  A & B \\ C & D
\end{pmatrix}
\in S(F)$. On a $x^+=A$ et $x^-=D$. Montrons d'abord la première égalité. Par définition et des manipulations sur les colonnes des matrices on a

  \begin{align*}
  \Prd_x(t)&=\Nrd 
  \begin{pmatrix}
  t I_p-A & -B \\ -C & t I_q-D
\end{pmatrix}\\
 &= \Nrd 
  \begin{pmatrix}
  t I_p-A-B(t I_q -D)^{-1}C & -B \\ 0 & t I_q-D
\end{pmatrix}\\
    &=\Nrd (t I_p-A-B(t I_q -D)^{-1}C) \Prd_{D}(t). 
  \end{align*}
D'après l'égalité $AB=BD$ du lemme \ref{lem:Smatrice}, on a 
  \begin{align*}
  (t I_p-A)B=B(t I_q-D) 
  \end{align*}
et donc 
  \begin{align}\label{eq:pfPrdxS}
  B(t I_q -D)^{-1}=(t I_p-A)^{-1}B. 
  \end{align}
On obtient alors 
  \begin{align*}
  \Nrd (t I_p-A-B(t I_q -D)^{-1}C)=\Nrd (t I_p-A-(t I_p-A)^{-1}BC) \\=\Prd_A(t)^{-1} \Nrd((t I_p-A)^2-BC)=\Prd_A(t)^{-1} \Nrd(t^2 I_p -2t A+I_p)
  \end{align*}
  où l'on a utilisé l'égalité  $A^2=I_p+BC$ du lemme \ref{lem:Smatrice}.  On a prouvé la première égalité. 
  
  Pour la seconde égalité, notons que 
    \begin{align*}
  \Prd_x(t)&= \Nrd 
  \begin{pmatrix}
  t I_p-A & 0 \\ -C & t I_q-D-C(t I_p-A)^{-1}B
\end{pmatrix}\\
    &=\Prd_{A}(t) \Nrd (t I_q-D-C(t I_p-A)^{-1}B). 
  \end{align*}
  Mais \eqref{eq:pfPrdxS} implique que 
    \begin{align*}
  \Nrd (t I_q-D-C(t I_p-A)^{-1}B)=  \Nrd (t I_q-D-CB(t I_q -D)^{-1}) \\=\Prd_D(t)^{-1} \Nrd((t I_q-D)^2-CB)=\Prd_D(t)^{-1} \Nrd(t^2 I_q-2tD+I_q)
  \end{align*}
    où l'on a utilisé l'égalité  $D^2=I_q+CB$ du lemme \ref{lem:Smatrice}. La seconde égalité s'ensuit. 
  \end{preuve}
  
  \begin{corollaire}\label{cor:Prdx}
  Pour tout $0\leq p'\leq N$ et $x\in S_{p'}(F)$, on a les égalités dans $F(t)$ : 
  \begin{align*}
    \Prd_x(t)&=(t-1)^{d(q'-p')} \Nrd(t^2 I_p-2t x^+ +I_p)   \\       &=(t-1)^{d(q'-p')}  (2t)^{dp} \Prd_{x^+} (\frac{t^2+1}{2t})    \\     & =(t-1)^{d(p'-q')} \Nrd(t^2 I_q-2t x^- +I_q) 
    \\ &= (t-1)^{d(p'-q')}  (2t)^{dq} \Prd_{x^-} (\frac{t^2+1}{2t})   .
  \end{align*}
  \end{corollaire}
  
  \begin{preuve}
  Ces égalités sont des conséquences des lemmes \ref{lem:Prdx+-} et \ref{lem:PrdxS}. 
  \end{preuve}
\end{paragr}

\begin{paragr}[Éléments $H$-semi-simples, $H$-réguliers.] --- On adopte la terminologie générale suivante qui s'applique au cas  d'une variété affine connexe $X$  définie sur $F$ sur laquelle un groupe algébrique affine connexe $H$ agit.  Soit $x\in X$ un point géométrique. On dit que $x$ est $H$-semi-simple si l'orbite de $x$ sous $H$ est fermée dans $X$ pour la topologie de Zariski. On dit que $x$ est $H$-régulier si le centralisateur $H_x$ de $x$ dans $H$ est de dimension minimale (ou encore la $H$-orbite de $x$ est de dimension maximale). On pourra parfois omettre  le préfixe $H$ si le contexte est clair. Le cas qui nous intéresse dans ce § est le cas de l'action par conjugaison du groupe connexe $G^\theta$ sur les composantes connexes de $S$.

  \begin{lemme}\label{lem:closedorbits}
    Soit $x\in S$. Alors $x$ est un élément semi-simple de $G$ au sens usuel si et seulement si $x$ est $G^\theta$-semi-simple. 
  \end{lemme}

  \begin{preuve}
    C'est une conséquence de \cite[théorème C]{Rich}. On peut consulter aussi la preuve de \cite[proposition 9.3]{Rich} (qui concerne seulement la composante $S^\circ$ mais la preuve vaut encore pour $S$).
  \end{preuve}

Pour tous entiers $m,r,t\geq 0$ tels que $m\leq \min\{p-r,q-t\}$ et tout $A\in \gl_m(D)$, on pose
  \begin{align}\label{eq:xArt}
  x(A, r, t)=
\begin{pmatrix}
  A & 0 & 0 & A-I_m & 0 & 0 \\ 
  0 & I_{p-m-r} & 0 & 0 & 0 & 0 \\ 
  0 & 0 & -I_r & 0 & 0 & 0 \\ 
  A+I_m & 0 & 0 & A & 0 & 0 \\ 
  0 & 0 & 0 & 0 & I_{q-m-t} & 0 \\ 
  0 & 0 & 0 & 0 & 0 & -I_t \\ 
\end{pmatrix}.
  \end{align}
  Par le lemme \ref{lem:Smatrice}, on voit que $x(A, r, t)\in S(F)$. 

\begin{remarque}
Si $A$ est une matrice sans valeurs propres $\pm1$, le centralisateur de $x(A, r, t)$ dans $G^\theta$ a pour dimension
  \begin{align*}
    \dim_F((GL_{m,D})_A)+\dim_F(D)[(p-m-r)^2+r^2+(q-m-t)^2+t^2]
  \end{align*}
où $(GL_{m,D})_A$ est le centralisateur de $A$ dans $GL_{m,D}$. 
\end{remarque}
  
  \begin{lemme}\label{lem:signofSp}
  Soit $0\leq p'\leq N$. Alors $x(A, r, t)\in S_{p'}$ si et seulement si
\begin{align*}
    t-r=p'-p.
\end{align*}
Dans ce cas, on a aussi
 \begin{align*}
    0\leq m \leq \min\{p,q,p',q'\}.
  \end{align*}
  \end{lemme}
  
  \begin{preuve} D'après le lemme \ref{lem:Smatrice}, $x(A, r, t)\in S_{p'}$ si et seulement si
  \begin{align*}
    (p-m-r-r)-(q-m-t-t)=p'-q'
     \end{align*}
  c'est-à-dire 
  \begin{align*}
    (p-2r)-(N-p-2t)=p'-(N-p')
  \end{align*}
  qui est équivalent à 
  \begin{align*}
    t-r=p'-p.
\end{align*}
  On a toujours $0\leq r\leq p-m$ et $0\leq t\leq q-m$ donc
  \begin{align*}
    m-p\leq t-r=p'-p\leq q-m
  \end{align*}
  d'où $m\leq p'\leq N-m$.  Cela donne les dernières conditions.
\end{preuve}
  
  \begin{lemme}(\cite[lemme 4.3]{JR2})\label{lem:JRlem4.3}
  Pour tout entier $m\geq0$ et tout $A\in \gl_m(D)$ sans valeurs propres $\pm1$, la matrice
\begin{align*}
\begin{pmatrix}
  A  & A-I_m   \\ 
  A+I_m &   A  \\ 
\end{pmatrix}
  \end{align*}
  est inversible et sans valeurs propres $\pm 1$. Elle est semi-simple si et seulement si $A$ est semi-simple. 
  \end{lemme}
  
  \begin{proposition}\label{prop:elementssr}
  \begin{enumerate}
  \item Un élément $x\in S(F)$ est $G^\theta$-semi-simple si et seulement s'il existe des entiers $m,r,t\geq 0$ avec $m\leq \min\{p-r,q-t\}$ et $A\in \gl_m(D)$ un élément  semi-simple sans valeurs propres $\pm1$ tels que $x$  est conjugué sous $G^\theta(F)$ à  $x(A, r, t)$. 
  
  \item L'ensemble des classes de $G^\theta(F)$-conjugaison d'éléments semi-simples dans $S(F)$ est en bijection avec l'ensemble des quadruplets $(m, [A], r, t)$, où $m,r,t\geq 0$ sont des entiers tels que $m\leq \min\{p-r,q-t\}$ et $[A]$ est une classe de conjugaison semi-simple dans $\gl_m(D)$ sans valeurs propres $\pm1$. 
  \end{enumerate}
  \end{proposition}
  
  \begin{preuve}
La preuve est  celle de  \cite[proposition 4.1]{JR2} lorsque $D=F$ et celle de   \cite[proposition 5.1]{Zha2} si $p=q$. Il suffit en général de combiner leurs arguments. 
  \end{preuve}
  
  \begin{corollaire}\label{cor:elementssr}
  Soit $0\leq p'\leq N$. 
    \begin{enumerate}
  \item Un élément $x\in S_{p'}(F)$ est $G^\theta$-semi-simple si et seulement s'il existe des entiers $m,r,t\geq 0$  et un élément  semi-simple $A\in \gl_m(D)$  sans valeurs propres $\pm1$ tels que  $m\leq \min\{p-r,q-t\}$, $t-r=p'-p$ et $x$  est conjugué sous $G^\theta(F)$ à  $x(A, r, t)$. 
  
  \item L'ensemble des classes de $G^\theta(F)$-conjugaison d'éléments semi-simples dans $S_{p'}(F)$ est en bijection avec l'ensemble des quadruplets $(m, [A], r, t)$, où $m,r,t\geq 0$ sont des entiers tels que $m\leq \min\{p-r,q-t\}$, $t-r=p'-p$ et $[A]$ est une classe de conjugaison semi-simple dans $\gl_m(D)$ sans valeurs propres $\pm1$. 
  \end{enumerate}
  \end{corollaire}
  
  \begin{preuve}
  C'est une conséquence du lemme \ref{lem:signofSp} et de la proposition \ref{prop:elementssr}. 
  \end{preuve}

\begin{lemme}\label{lem:JRlem4.1}
Soit $x\in S(F)$ et $x=x_s x_u=x_u x_s$ sa décomposition de Jordan dans $G(F)$, où $x_s$ et $x_u$ sont respectivement les parties semi-simple et unipotente de $x$. 
\begin{enumerate}
	\item On a $x_s\in S(F)$ et $x_u\in S^\circ(F)$. 
	
	\item Soit $0\leq p'\leq N$. Alors $x\in S_{p'}(F)$ si et seulement si $x_s\in S_{p'}(F)$. 
\end{enumerate}
\end{lemme}

\begin{preuve} On suit la preuve de \cite[lemme 4.1]{JR2}. 

1. L'unicité de la décomposition de Jordan entraîne qu'on a  $x_s,x_u\in S(F)$. Soit $X\in \ggo(F)$ l'unique élément nilpotent  tel que $x_u=\exp(X)$. On a alors $\theta X \theta =-X$ et $x_u\in S^\circ(F)$.

2. Comme $S$ est la réunion disjointe des composantes connexes $S_{p'}$ où $0\leq p'\leq N$, il suffit de montrer que si $x\in S_{p'}(F)$ alors $x_s\in S_{p'}(F)$. Soit $x\in S_{p'}(F)$.  On a $xx_u=x_ux$. Soit $g\in G(F)$ tel que $x=\rho_{p'}(g)=g\theta_{p'}g^{-1}\theta$. On a alors
                  \begin{align*}
                    g\theta_{p'}g^{-1}\theta x_u=  x_u   g\theta_{p'}g^{-1}\theta. 
                  \end{align*}
Or $x_u\in S(F)$ implique qu'on a $(x_u\theta)^2=1$ et donc  
                  \begin{align*}
                       g\theta_{p'}g^{-1} x_u^{-1}= x_u   g\theta_{p'}g^{-1}. 
                  \end{align*}
                  Ainsi, avec $X$ défini en 1 ci-dessus, on a $\Ad(g\theta_{p'}g^{-1}) (X) =-X$. On pose $v=\exp(X/2)\in G(F)$. Alors 
                   \begin{align*}
                       g\theta_{p'}g^{-1} v= v^{-1}   g\theta_{p'}g^{-1}. 
                  \end{align*}
On en déduit que               
                  \begin{align*}
                    \rho_{p'}(v^{-1}g)=(v^{-1}g)\theta_{p'}( v^{-1}g)^{-1}\theta= v^{-2}g\theta_{p'}g^{-1}   \theta=  v^{-2} x  =x_s.
                  \end{align*}
                  Il s'ensuit que $x_s\in S_{p'}(F)$.
\end{preuve}

\begin{corollaire}\label{cor:JRlem4.1}
Soit $x\in S(F)$ et $x=x_s x_u=x_u x_s$ sa décomposition de Jordan dans $G(F)$, où $x_s$ et $x_u$ sont respectivement les parties semi-simple et unipotent. Alors $\Prd_{x^+}=\Prd_{(x_s)^+}$. 
\end{corollaire}

\begin{preuve}
Elle résulte de l'égalité $\Prd_{x}=\Prd_{x_s}$, du lemme \ref{lem:JRlem4.1} et du corollaire \ref{cor:Prdx}. 
\end{preuve}

    \begin{lemme}\label{lem:divisibilite}
                  Soit $0\leq p'\leq N$. Pour tout $x\in S_{p'}(F)$, le polynôme 
                  \begin{align*}
                    (t-1)^{d\max\{p'-q,0\}}(t+1)^{d\max\{p-p',0\}} 
                  \end{align*}
                  divise  $\Prd_{x^+}$ dans $F[t]$.
    \end{lemme}

                \begin{preuve}
D'après le corollaire \ref{cor:JRlem4.1}, on est ramené au cas où $x$ est semi-simple: on utilise alors le corollaire \ref{cor:elementssr}. Il nous faut montrer 
\begin{enumerate}
	\item $p-m-r\geq\max\{p'-q,0\}$ ; 
	
	\item $r\geq\max\{p-p',0\}$. 
\end{enumerate}

Comme $m\leq \min\{p-r,q-t\}$, on a $m\leq p-r$ et donc $p-m-r\geq0$. De plus, $m\leq q-t$ implique que 
\begin{align*}
p-m-r\geq p+t-q-r=p-q+p'-p=p'-q
\end{align*}
où l'on a utilisé $t-r=p'-p$. On a prouvé $p-m-r\geq\max\{p'-q,0\}$. 

On sait que $r\geq0$. Puisque $t-r=p'-p$, on a $r=t+p-p'\geq p-p'$. Donc $r\geq\max\{p-p',0\}$.

                    \end{preuve}
\end{paragr}
  
\begin{paragr}[Éléments semi-simples réguliers.] --- 
      Soit $0\leq p'\leq N$ et $x\in S_{p'}(F)$. On pose
      \begin{align}
          \nonumber \nu=\min\{p,q,p',q'\}\\
     \label{eq:defchix}
             \chi_x=       (t-1)^{-d\max\{p'-q,0\}}(t+1)^{-d\max\{p-p',0\}}  \Prd_{x^+}.
         \end{align}
        D'après le lemme \ref{lem:divisibilite}, on obtient ainsi un polynôme unitaire $\chi_x$ de $F[t]$. Son   degré est $d\nu$ comme il résulte du lemme suivant. 
         
         \begin{lemme}\label{lem:ssrequalites}
         	Soit $\nu=\min\{p,q,p',q'\}$. On a les égalités suivantes : 
	 \begin{align*}
             \nu+\max\{p'-q,0\}+\max\{p-p',0\}=p, \\
             \nu+\max\{q-p',0\}+\max\{p'-p,0\}=q,  \\
             \max\{p'-p,0\}-\max\{p-p',0\}=p'-p. 
         \end{align*}
         \end{lemme}
         
         \begin{preuve}
         Considérons d'abord le cas $\nu=p$. Dans ce cas, on a $\max\{p,q,p',q'\}=q$. Il s'ensuit que 
	 \begin{align*}
             \max\{p'-q,0\}=0, \max\{p-p',0\}=0, \max\{q-p',0\}=q-p', \max\{p'-p,0\}=p'-p. 
         \end{align*}
         Les égalités sont alors évidentes. Les autres cas se traitent de la même façon avec le tableau récapitulatif:

                    \begin{align*}                                              
                    &\nu &\max\{p,q,p',q'\} & & \max\{p'-q,0\} & &\max\{p-p',0\} & &\max\{q-p',0\} & &\max\{p'-p,0\}& \\
                                            & p  &q                        & & 0                     & & 0                  &  & q-p'  & & p'-p &\\
                      & q  &p                        & & p'-q                 & & p-p'               & &  0    & & 0   & \\
                      & p'  &q'                       & & 0                    & & p-p'               & &  q-p'   & & 0   & \\
                      & q'  &p'                       & & p'-q                & & 0                   & &  0  & & p'-p   & \\
                    \end{align*}
         \end{preuve}
         
     Pour tout $A\in \gl_\nu(D)$, on pose $x(A)=x(A, \max\{p-p',0\}, \max\{p'-p,0\})$ c'est-à-dire
  \begin{align}\label{eq:x(A)}
  x(A)=\begin{pmatrix}
  A & 0 & 0 & A-I_\nu & 0 & 0 \\ 
  0 & I_{\max\{p'-q,0\}} & 0 & 0 & 0 & 0 \\ 
  0 & 0 & -I_{\max\{p-p',0\}} & 0 & 0 & 0 \\ 
  A+I_\nu & 0 & 0 & A & 0 & 0 \\ 
  0 & 0 & 0 & 0 & I_{\max\{q-p',0\}} & 0 \\ 
  0 & 0 & 0 & 0 & 0 & -I_{\max\{p'-p,0\}} \\ 
\end{pmatrix}.
  \end{align}
  D'après les lemmes \ref{lem:signofSp} et \ref{lem:ssrequalites}, on a $x(A)\in S_{p'}$. Si $A$ est une matrice sans valeurs propres $\pm1$, il résulte des lemmes \ref{lem:closedorbits} et \ref{lem:JRlem4.3} que $x(A)$ est $G^\theta$-semi-simple si et seulement si $A$ est semi-simple.
 
 \begin{remarque}\label{rmq:nul}
 Notons que l'un des nombres $\max\{p'-q,0\}$ et $\max\{q-p',0\}$ est nul de même que l'un des nombres $\max\{p-p',0\}$ et $\max\{p'-p,0\}$. 
 \end{remarque}
         
   \begin{proposition}\label{prop:ssr-Prd}
  Soit $0\leq p'\leq N$ et $x\in S_{p'}(F)$. Les énoncés suivants sont équivalents. 
  \begin{enumerate}
  \item L'élément $x$ est conjugué sous $G^\theta(F)$ à un élément de la forme $x(A)$ où $A\in\gl_\nu(D)$ est semi-simple régulier et sans valeurs propres $\pm1$.   
   \item Le polynôme $\chi_x$ est séparable et $\chi_x(\pm1)\neq0$. 
  \item L'élément $x$ est à la fois $G^\theta$-semi-simple et  $G^\theta$-régulier.
  \end{enumerate}
  \end{proposition}
  
  \begin{preuve}
    
    1 implique 2. C'est évident car $\chi_{x(A)}=\Prd_A$.

    2 implique 3. Supposons que $\chi_x$ est séparable et $\chi_x(\pm1)\neq0$. Prouvons tout d'abord que $x$ est $G^\theta$-semi-simple. Soit $x=x_s x_u=x_u x_s$ la décomposition de Jordan de $x$ dans $G(F)$, où $x_s$ et $x_u$ sont respectivement les parties semi-simple et unipotente. Vu le lemme \ref{lem:closedorbits}, il s'agit de montrer qu'on a $x_u=1$. D'après le lemme \ref{lem:JRlem4.1} et le corollaire \ref{cor:JRlem4.1}, on a $x_s\in S_{p'}(F)$ et $\chi_x=\chi_{x_s}$. Alors $\chi_{x_s}$ est séparable et $\chi_{x_s}(\pm1)\neq0$. En utilisant le corollaire \ref{cor:elementssr}, quitte à conjuguer $x$ par un élément de $G^\theta(F)$, on peut et on va supposer que $x_s=x(A,r,t)$ avec des entiers $m,r,t\geq 0$ qui vérifient les conditions de ce  corollaire et un élément semi-simple $A\in\gl_m(D)$ sans valeurs propres $\pm1$. Par hypothèse, la multiplicité de la racine $1$, resp. $-1$, du polynôme $\Prd_{x_s^+}$ est $d\max\{p'-q,0\}$, resp. $d\max\{p-p',0\}$. Il s'ensuit que $r=\max\{p-p',0\}$ et $p-m-r=\max\{p'-q,0\}$. D'après le lemme \ref{lem:ssrequalites}, on a $m=\nu$. Comme $t-r=p'-p$, on obtient aussi  $t=\max\{p'-p,0\}$. Finalement, on a $x_s=x(A)$ avec les notations ci-dessus.  Par le corollaire \ref{cor:Prdx} (appliqué au cas $p=q=p'=q'=\nu$), on voit que la matrice
  $\begin{pmatrix}
  A & A-I_\nu \\ A+I_\nu & A
\end{pmatrix}$
est semi-simple régulière et sans valeurs propres $\pm1$. Comme $x_u$ est un élément unipotent commutant avec $x_s$, on trouve que $x_u$ est nécessairement de la forme 
  \begin{align*}
  \begin{pmatrix}
  I_\nu & 0 & 0 & 0 & 0 & 0 \\ 
  0 & A_1 & 0 & 0 & B_1 & 0 \\ 
  0 & 0 & A_2 & 0 & 0 & B_2 \\ 
  0 & 0 & 0 & I_\nu & 0 & 0 \\ 
  0 & C_1 & 0 & 0 & D_1 & 0 \\ 
  0 & 0 & C_2 & 0 & 0 & D_2 \\ 
\end{pmatrix}
  \end{align*}
  où $A_1$, $A_2$, $D_1$, $D_2$ sont carrées de tailles respectives $\max\{p'-q,0\}$, $\max\{p-p',0\}$, $\max\{q-p',0\}$, $\max\{p'-p,0\}$. Par la remarque \ref{rmq:nul}, les blocs $B_1$, $B_2$, $C_1$, $C_2$ ne figurent pas dans la matrice ci-dessus. On a donc $\theta x_u=x_u \theta$; or $(x_u\theta)^2=1$ puisque $x_u\in S(F)$, cf.  lemme \ref{lem:JRlem4.1}. On en déduit qu'on a $x_u^2=1$. Comme $x_u$ est unipotent, $x_u+1$ est inversible. On a alors $x_u=1$ et donc $x=x_s$ comme voulu.

On a donc obtenu $x=x_s=x(A)$ et cet élément est $G^\theta$-semi-simple. On calcule aisément le centralisateur dans $G^\theta$ de $x(A)$ :  on voit qu'il est isomorphe au produit d'un tore maximal de $GL_{\nu,D}$ (le centralisateur de $A$) et des groupes $GL_{\max\{p'-q,0\},D}$, $GL_{\max\{p-p',0\},D}$, $GL_{\max\{q-p',0\},D}$, $GL_{\max\{p'-p,0\},D}$. Sa dimension est donc donnée par la formule
\begin{align*}
  d\nu+d^2(\max\{p'-q,0\}^2+\max\{p-p',0\}^2+\max\{q-p',0\}^2+\max\{p'-p,0\}^2)\\
  =d\nu+d^2[(p'-q)^2+(p-p')^2]. 
\end{align*}
D'une part, les conditions $\chi_x$ séparable et $\chi_x(\pm1)\neq0$ définissent un ouvert non-vide de Zariski  de $S_{p'}$ dont on vient de voir qu'il est formé d'éléments $G^\theta$-semi-simples $x$ dont le centralisateur dans $G^\theta$, noté  $G^\theta_x$, a toujours la même dimension. D'autre part, la fonction $x\mapsto \dim(G^\theta_x)$ étant semi-continue supérieurement sur $S_{p'}$,  l'ensemble des éléments réguliers de $S_{p'}$ est également un ouvert de Zariski non-vide. Comme $S_{p'}$ est irréductible (c'est l'image de $G$ par le morphisme $\rho_{p'}$), ces deux ouverts ont une intersection non vide: on en déduit que tous les éléments du premier ouvert sont également $G^\theta$-réguliers.

3 implique 1. Soit   $x\in S(F)$ qui est  $G^\theta$-semi-simple et  $G^\theta$-régulier. D'après le corollaire \ref{cor:elementssr}, on peut supposer que $x=x(A,r,t)$ pour certains entiers $m,r,t\geq 0$ avec $m\leq \min\{p-r,q-t\}$, $t-r=p'-p$ et $A\in \gl_m(D)$ un élément  semi-simple sans valeurs propres $\pm1$. D'après le lemme \ref{lem:ssrequalites}, il suffit de montrer $m=\nu$, $r=\max\{p-p',0\}$ et que $A$ est régulier.  On commence par calculer le centralisateur $G_x^\theta$: il est isomorphe au produit du centralisateur de $A$ dans $GL_{m,D}$ et des groupes $GL_{p-m-r,D}$, $GL_{r,D}$,  $GL_{q-m-t,D}$ et $GL_{t,D}$.  Ce centralisateur est de dimension minorée par
\begin{align*}
  f(m,r,t)=dm+ d^2[(p-m-r)^2+r^2+(q-m-t)^2+t^2]. 
\end{align*}
Il est facile de voir que, pour $m\leq \min\{p-r,q-t\}$, on a 
\begin{enumerate}
\item $f(m,r,t)>f(m+1,r,t)$ si $(p-m-r)(q-m-t)\neq0$ ; 
\item $f(m,r,t)>f(m+1,r-1,t-1)$ si $rt\neq0$. 
\end{enumerate}
Donc, en un point où la fonction $f$ est minimale, on doit avoir $(p-m-r)(q-m-t)=0$ et $rt=0$. Dans ce cas, on a 
\begin{align*}
  f(m,r,t)=dm+ d^2[(p-m-r-q+m+t)^2+(r-t)^2] \\
  =dm+ d^2[(p-q+p'-p)^2+(p-p')^2] = dm +d^2[(p'-q)^2+(p-p')^2]
\end{align*}
où l'on a utilisé $t-r=p'-p$. Il nous faut discuter au cas par cas. 
\begin{enumerate}
\item Si $p-m-r=0$ et $r=0$, on a $m=p$. 
\item Si $p-m-r=0$ et $t=0$, on a $r=p-p'$ et alors $m=p'$. 
\item Si $q-m-t=0$ et $r=0$, on a $t=p'-p$ et alors $m=q'$. 
\item Si $q-m-t=0$ et $t=0$, on a $m=q$. 
\end{enumerate}
En tout cas, on a $m\geq\nu$. Il s'ensuit que 
\begin{align*}
  f(m,r,t)\geq d\nu +d^2[(p'-q)^2+(p'-p)^2] 
\end{align*}
qui est la dimension des centralisateurs des éléments  $G^\theta$-réguliers. Une condition nécessaire pour que $x(A,r,t)$ soit $G^\theta$-régulier est qu'on ait $m=\nu$,  $A$ est régulier dans $\mathfrak{gl}_\nu(D)$, $(p-m-r)(q-m-t)=0$ et $rt=0$. Comme $t-r=p'-p$, on voit que 
\begin{enumerate}
\item si $r=0$, $t=p'-p\geq0$ ; 
\item si $t=0$, $r=p-p'\geq0$. 
\end{enumerate}
En tout cas, on a $r=\max\{p-p',0\}$ et $t=\max\{p'-p,0\}$. 
  \end{preuve} 
\end{paragr}

\begin{paragr}[Invariants de l'action de $G^\theta$ sur $S_{p'}$.]\label{S:invariants} --- Soit $0\leq p'\leq N$. Rappelons que l'on note $d^2=\dim_F(D)$ et $\nu=\min\{p,q,p',q'\}$. Soit $\mathbf{A}^{d\nu}$ l'espace affine sur $F$ de dimension $d\nu$. Par le lemme \ref{lem:ssrequalites}, on peut définir un morphisme $c^\flat _{p'} : S_{p'}\to\mathbf{A}^{d\nu}$ en associant à $x$ la collection des $d\nu$ coefficients non dominants du polynôme $\chi_x$ défini par \eqref{eq:defchix}. Il est clair que $c^\flat_{p'}$ est $G^\theta$-invariant. Soit $\mathbf{A}^{d\nu}_{\rs}\subset\mathbf{A}^{d\nu}$ le sous-ensemble des éléments $(c_i)_{0\leq i\leq d\nu-1}$ tels que le polynôme 
  \begin{align*}
  P(t)=t^{d\nu}+\sum_{i=0}^{d\nu-1} c_it^i\in F[t]
  \end{align*}
est séparable et $P(\pm1)\neq0$. C'est un ouvert de Zariski de $\mathbf{A}^{d\nu}$. 

Dans la suite, si le contexte est clair, on omettra l'indice $p'$ dans $c^\flat_{p'}$.

  \begin{proposition}\label{prop:quot-cat}
  Soit $0\leq p'\leq N$. Le couple $(c^\flat_{p'}, \mathbf{A}^{d\nu})$ est un quotient catégorique de $S_{p'}$ par l'action de $G^\theta$. 
  \end{proposition}
  
  \begin{preuve}
    Il s'agit de montrer que le comorphisme $c^{\flat,*}:F[\mathbf{A}^{d\nu}] \to F[S_{p'}]^{G^\theta}$ induit par $c^\flat$ est un $F$-isomorphisme.  On va d'abord le montrer dans le cas de $G_0=GL_F(F^{dN})$ avec 
    \begin{align*}
      \theta_0=
      \begin{pmatrix}
        I_{dp} & 0 \\ 0 & -I_{dq}
      \end{pmatrix}
      \text{ et }
            \theta_{0,dp'}=
      \begin{pmatrix}
        I_{dp'} & 0 \\ 0 & -I_{dq'}
      \end{pmatrix}. 
    \end{align*}
Soit $S_{0,dp'}$ la variété correspondante définie comme au § \ref{S:esp sym} et $c_0^\flat:S_{0,dp'}\to \mathbf{A}^{d\nu}$ l'application définie comme ci-dessus. Pour $(c_i)_{0\leq i\leq d\nu -1}\in\mathbf{A}^{d\nu}$, on définit une matrice compagnon carrée de taille $d\nu$
  \begin{align*}
A((c_i)_{0\leq i\leq d\nu -1})=\left( \begin{array}{ccccc}
  0        & 0        & \cdots & 0         & -c_0         \\
  1        & 0        & \cdots & 0         & -c_1         \\
  0        & 1        & \ddots & \vdots  & -c_2         \\
  \vdots & \ddots & \ddots & 0         & \vdots       \\
  0        & \cdots & 0        & 1          & -c_{d\nu-1} \\
\end{array} \right). 
  \end{align*}
On définit un morphisme $j_0:\mathbf{A}^{d\nu}\to S_{0,dp'}$ par 
  \begin{align*}  
  (c_i)_{0\leq i\leq d\nu -1}\mapsto x(A((c_i)_{0\leq i\leq d\nu -1}), d\max\{p-p',0\}, d\max\{p'-p,0\}). 
  \end{align*}
  C'est une section de $c^\flat_0$ c'est-à-dire on a $c_0^\flat\circ j_0=\Id_{\mathbf{A}^{d\nu}}$. Dualement, $j_0^{*}\circ c_0^{\flat,*}$ est l'identité de  $F[\mathbf{A}^{d\nu}]$. Ainsi $j_0^{*}$ est surjectif; montrons que $j_0^*$ induit un monomorphisme de $F[S_{0,dp'}]^{G_0^{\theta_0}}$ dans $F[\mathbf{A}^{d\nu}]$. 
   Ce sera alors un isomorphisme de même que $c_0^{\flat,*}$. Soit $f\in F[S_{0,dp'}]$ une fonction régulière sur $S_{0,dp'}$ et $G_0^{\theta_0}$-invariante telle que $f\circ j_0$ soit identiquement nulle. Il résulte de la proposition \ref{prop:ssr-Prd} que l'orbite sous $G_0^{\theta_0}$ d'un élément  $G_0^{\theta_0}$-régulier  semi-simple rencontre $j_0(\mathbf{A}^{d\nu})$. Il s'ensuit que $f$ est identiquement nulle sur l'ensemble des éléments $G_0^{\theta_0}$-réguliers  semi-simples ; par conséquent, comme ce dernier  est  un ouvert de Zariski dense (ainsi qu'il résulte de la proposition \ref{prop:ssr-Prd}), $f$ est identiquement nulle. Cela conclut pour $S_{0,dp'}$ et $\theta_0$. 

Considérons maintenant le cas général. Soit  $\bar F $ est une clôture algébrique de $F$. En notant avec un indice $\bar F$ le changement de base de $F$  à $\bar F$, il existe $\varphi:G_{\bar F}\to G_{0,\bar F}$ un isomorphisme sur $\bar F$ ainsi qu'un $1$-cocycle $(u_\sigma)_{\sigma\in \Ga}$ du groupe de Galois $\Ga=\Gal(\bar F/F)$ à valeur dans le groupe adjoint $G_{0,\ad}(\bar F)$ de sorte qu'on ait $ \varphi\circ \sigma= \Int(u_\sigma) \circ \sigma \circ \varphi$ pour tout $\sigma\in \Ga$,  $\varphi(\theta)=\theta_0$ et $\varphi(\theta_{p'})=\theta_{0,dp'}$. Notons que $u_\sigma$ appartient à l'intersection des images de $G_{0}^{\theta_0}(\bar F)$ et $G_{0}^{\theta_{0,dp'}}(\bar F)$ dans le groupe adjoint $G_{0,\ad}(\bar F)$.  D'autre part, $\varphi$ induit un isomorphisme $S_{p', \bar F }\to S_{0,dp', \bar F }$. On en déduit que $\varphi$ induit un isomorphisme  $\bar F[S_{0,dp',\bar F }]^{G_{0,\bar F}^{\theta_0}}\to \bar F[S_{p',\bar F }]^{G_{\bar F}^\theta} $ qui se restreint en un $F$-isomorphisme  $F[S_{0,dp'}]^{G_{0}^{\theta_0}}\to F[S_{p'}]^{G_{}^\theta}$. On conclut en observant que $c^\flat=c_0^\flat\circ \varphi$.
\end{preuve}
\end{paragr}

\begin{paragr}\label{S:cat quot ss}Soit $0\leq p'\leq N$.

\begin{lemme}\label{lem:c-ss}
Soit $x\in S_{p'}$ et $x=x_sx_u$ la décomposition de Jordan de $x$ dans $G$. Alors $x_s\in S_{p'}$ et $c^\flat_{p'}(x)=c^\flat_{p'}(x_s)$. 
\end{lemme}

\begin{preuve}
La première assertion résulte du lemme \ref{lem:JRlem4.1} assertion 2, la seconde du corollaire \ref{cor:JRlem4.1} et de la définition \eqref{eq:defchix}.
\end{preuve}

   On pose $\cgo_{p'}=\Spec(F[S_{p'}]^{G^\theta})$ (qui est isomorphe à $\mathbf{A}^{d\nu}$ par la proposition \ref{prop:quot-cat}). Pour tout $\of\in \cgo_{p'}(F)$, soit $S_{p',\of}$ la fibre de $c^\flat_{p'}$ au-dessus de $\of$.
  
  \begin{lemme}\label{lem:parab-o}
  Soit $P\in\fc(M_0)$, $m\in M_P$ et $n\in N_P$. Pour tout $\of\in\cgo_{p'}$, $mn\in S_{p',\of}$ si et seulement si $m\in S_{p',\of}$ et $mn\in S$. 
  \end{lemme}
  
  \begin{preuve} Soit $m\in M_P$ et $n\in N_P$ tels $mn\in S$. Comme $\theta\in M_0$, la conjugaison par  $\theta$ normalise $M_P$ et $N_P$. Il s'ensuit que $m\in S$; en outre, vu \eqref{eq:Sp' noyau}, on a $mn\in S_{p'}$ si et seulement si $m\in S_{p'}$. Dans ce cas, comme $\Prd_{mn}=\Prd_m$, le corollaire \ref{cor:Prdx} implique que $c^\flat(mn)=c^\flat(m)$. L'assertion est donc évidente. 
  \end{preuve}
\end{paragr}

\begin{paragr}
  Pour tout sous-groupe semi-standard $P$ de $G$, soit $\ov{P}$ le sous-groupe parabolique semi-standard opposé à $P$ c'est-à-dire qu'on a $P\cap \ov{P}=M_P$. Pour tous sous-groupes paraboliques semi-standard $P\subset Q$ de $G$, on définit 
  \begin{align*}
      N_P^Q=N_P\cap M_Q, \ov{N}_P^Q=N_{\ov{P}}\cap M_Q \text{  et  } N_P^{Q,\theta}=N_P^Q\cap G^\theta. 
  \end{align*}

  \begin{lemme}\label{lem:LeviS}
    Soit $0\leq p'\leq N$. Pour tout $P\in\fc(M_0)$, on a 
  \begin{align*}
      (M_P\cap S_{p'})(F)=\bigsqcup\limits_{w\in \,_PW_{\theta_{p'}}} \rho_{p'}(M_P(F)w). 
  \end{align*}
  \end{lemme}
  
  \begin{preuve}
  Supposons que $M_P\simeq GL_{n_1,D}\times \cdots \times GL_{n_\ell, D}$ avec $\sum\limits_{i=1}^\ell n_i=p+q$. Tout élément dans $(M_P\cap S_{p'}\theta)(F)$ est semi-simple de valeurs propres $\pm 1$ donc est conjugué sous $M_P(F)$ à une matrice diagonale dont les coefficients diagonaux sont $\pm1$ respectivement de multiplicités $p'$ et $q'$. De plus, une classe de $M_P(F)$-conjugaison dans $(M_P\cap S_{p'}\theta)(F)$ est uniquement déterminée par les multiplicités de $\pm1$  dans chaque  bloc $GL_{n_i,D}$. Comme le groupe de Weyl $W$ agit sur $\theta_{p'}$ par permutation des coefficients diagonaux, d'après la proposition \ref{prop:PWtheta} et le corollaire \ref{cor:PWtheta} appliqués à $\theta_{p'}$, on voit que $\{w\theta_{p'} w^{-1}| w\in \,_PW_{\theta_{p'}}\}$ est exactement un système de représentants de classes de $M_P(F)$-conjugaison dans $(M_P\cap S_{p'}\theta)(F)$. Cela conclut. 
  \end{preuve}
  
  \begin{lemme}\label{lem:parabS}
  Soit $0\leq p'\leq N$. Soit $P\subset Q$ des sous-groupes paraboliques semi-standard de $G$. On a 
  \begin{align*}
      (P\cap M_Q\cap S_{p'})(F)=\Int_\theta(N_P^Q(F))((M_P\cap S_{p'})(F)). 
  \end{align*}
  \end{lemme}
  
  \begin{preuve}
  Soit $m\in M_P(F)$ et $n\in N_P^Q(F)$ tels que $mn\in (P\cap M_Q\cap S_{p'}\theta)(F)$. Par définition de $S_{p'}$, on a $m\in (M_P\cap S_{p'}\theta)(F)$. Comme $(mn)^2=m^2=1$, les éléments $mn$ et $m$ sont semi-simples dans $M_Q(F)$ au sens usuel. D'après \cite[lemme 2.1]{ar1} appliqué au sous-groupe parabolique $P\cap M_Q$ de $M_Q$ et la fonction caractéristique du singleton $\{n\}$, on obtient que $mn$ est $N_P^Q(F)$-conjugué à un élément $mn'$ avec $n'\in N_P^Q(F)$ tel que $mn'=n'm$. Puisque $mn'$ et $m$ sont semi-simples dans $M_Q(F)$, l'unicité de la décomposition de Jordan entraîne $n'=1$ c'est-à-dire que $mn$ est $N_P^Q(F)$-conjugué à $m\in(M_P\cap S_{p'}\theta)(F)$. Le résultat s'ensuit. 
  \end{preuve}
  
  \begin{corollaire}
  Soit $0\leq p'\leq N$. Soit $P\subset Q$ des sous-groupes paraboliques semi-standard de $G$. On a 
  \begin{align*}
      (P\cap M_Q\cap S_{p'})(F)=\bigsqcup\limits_{w\in \,_PW_{\theta_{p'}}} \rho_{p'}((P\cap M_Q)(F)w). 
  \end{align*}
  où $\bigsqcup$ désigne l'union mais précise que, dans celle-ci, les intersections deux à deux sont vides.
  \end{corollaire}
  
  \begin{preuve}
  Elle résulte des lemmes \ref{lem:LeviS} et \ref{lem:parabS}. 
  \end{preuve}
\end{paragr}

\subsection{Distributions géométriques}\label{ssec:dvpt quotient}

\begin{paragr}
  Les notations dans cette section sont celles de la section \ref{sec:Dev spectral}. Soit $\theta,\theta'\in M_0(F)$ des éléments d'ordre 2 au plus. 
\end{paragr}

\begin{paragr}\label{S:quot cat}   Le groupe $G^{\theta'}\times G^\theta$ agit à gauche sur $G$ par $(h_1,h_2)\cdot g=h_1gh_2^{-1}$. Soit $\cgo=\Spec(F[G]^{G^{\theta'}\times G^\theta})$ où  $F[G]^{G^{\theta'}\times G^\theta}$ est la $F$-algèbre des fonctions régulières sur $G$ invariante sous  $G^{\theta'}\times G^\theta$.  Soit $c:G \to \cgo$ le morphisme canonique. Le morphisme 
\begin{align}\label{eq:sym-rho}
    \rho_\theta:g\mapsto g\theta g^{-1}\theta',
\end{align}
dont on note $S_\theta$ l'image dans $G$, identifie  l'algèbre $F[G]^{G^{\theta'}\times G^\theta}$ à l'algèbre  $F[S_\theta]^{G^{\theta'}}$ des fonctions sur $S_\theta$ invariantes par conjugaison par $G^{\theta'}$.Notons que $S_\theta$ est une composante connexe de la variété $S=\{g\in G \mid (g\theta ')^2=1\}$, cf. sous-section \ref{ssec:preparatif alg}. D'après la proposition \ref{prop:quot-cat}, on peut et on va identifier  $\cgo$ à l'espace affine $\mathbf{A}^{d\nu}$ où 
\begin{align*}
    \nu={\min\{\dim_D(V^{\pm \theta}),\dim_D(V^{\pm \theta'})\}}. 
\end{align*}
Pour tout $g\in G$, on a donc $c(g)=c^\flat(\rho_\theta(g))$ où $c^\flat: S_\theta\to\mathbf{A}^{d\nu}$ est défini comme au \S \ref{S:invariants}. On dispose de l'ouvert $\cgo_{\rs}$ correspondant à l'ouvert $\mathbf{A}^{d\nu}_{\rs}\subset\mathbf{A}^{d\nu}$.

Pour tout $\of\in \cgo$, soit $G_\of$ la fibre de $c$ au-dessus de $\of$. 

\begin{proposition}\label{prop:inv-de-c}
  Soit $P=MN$ un sous-groupe parabolique semi-standard muni de sa décomposition de Levi et $w_1,w_2 \in W$.
 Pour tous $n\in N$ et  $m\in M$, on a 
    \begin{align*}
      c(w_1nmw_2)=c(w_1mw_2).
    \end{align*}
En particulier, pour tout $\of\in \cgo$, on a 
$$w_1Pw_2 \cap G_\of= w_1Nw_1^{-1} (w_1Mw_2\cap G_\of)= (w_1Mw_2\cap G_\of) w_2^{-1}Nw_2.$$
\end{proposition}

\begin{preuve} La seconde assertion est une conséquence immédiate de la première. Si l'on introduit $P'=w_1Pw_1^{-1}$, on voit que $P'$ est semi-standard de facteur de Levi $M'=w_1Mw_1^{-1}$. Quitte à remplacer $P$ par $P'$ et $w_2$ par $w_1w_2$, on est ramené au cas où $w_1=1$ ce qu'on suppose désormais. On est donc ramené à prouver qu'on a $c(nmw)=c(mw)$ pour $n\in N$,  $m\in M$ et $w\in W$. Soit  $\delta=  \rho_\theta( m w)$ et $\gamma=\rho_\theta(nm w)$. On a donc $c(nmw)=c^\flat(\gamma) $ et  $c(mw)=c^\flat(\delta) $.   On observe qu'on a $\delta\in M$ et $\gamma= n\delta \theta'(n)^{-1}\in \delta N$. Mais le lemme \ref{lem:parab-o} implique que pour tout $n\in N$ tel que $\delta n\in S_\theta$ on a $c^\flat(\delta n)= c^\flat(\delta)$. Le résultat s'ensuit.
\end{preuve}

\end{paragr}

\begin{paragr} \label{S:zetaof} Comme ci-dessus, on identifie $\cgo$ à l'espace affine de dimension $d\nu$. Soit $\vc$ un voisinage de $0$ dans $\cgo(\AAA)$ tel que $\vc\cap \cgo(F)=\{0\}$. Soit $\om$ un voisinage compact de $0$ dans $\cgo(\AAA)$ tel que $\om-\om \subset \vc$. Soit $\zeta$ une fonction  lisse sur $\cgo(\AAA)$,  à support inclus dans $\om$, à valeurs dans  l'intervalle réel $[0,1]$  et qui vaut $1$ dans un voisinage de $0$. Pour tout $\of \in \cgo(F)$ et $x\in G(\AAA)$,  on pose $\zeta_\of(x)=\zeta(c(x)-\of)$. Pour des éléments distincts $\of,\of'\in \cgo(F)$ les supports des fonctions  $\zeta_\of$ et $\zeta_{\of'}$ sont disjoints. En particulier, un élément $\delta \in G(F)$ qui appartient au support de $\zeta_\of$ vérifie nécessairement $c(\delta)=\of$. On pose alors  $f_\of=\zeta_{\of} f$ pour tout $f\in \Sc(G(\AAA))$.

  \begin{lemme}\label{lem:sum fof}
    \begin{enumerate}
    \item Pour tout $\of\in \cgo(F)$ et tout $f\in \Sc(G(\AAA))$, la fonction $f_\of$ appartient à  $\Sc(G(\AAA))$.
    \item Pour toute semi-norme continue $\|\cdot\|$ sur  $\Sc(G(\AAA))$, l'application $f\mapsto\sum_{\of\in \cgo(F)} \|f_\of\|$ est une semi-norme continue.
    \end{enumerate}
  \end{lemme}

  \begin{preuve}
    L'assertion 1 est évidente. Il suffit de prouver l'assertion 2 pour le sous-espace $\Sc(G(\AAA),C)^J$ muni de la topologie induite par les semi-normes $\|\cdot\|_{r,Y}$ définies en \eqref{eq:Norm rXY}. 
    Par le théorème de Banach-Steinhaus, il suffit de prouver la convergence de  $\sum_{\of\in \cgo(F)} \|f_\of\|$ pour tout $f\in \Sc(G(\AAA),C)^J$ et toute semi-norme continue. Cela résulte de la convergence, pour $t>0$ assez grand, de la somme  suivante:
    \begin{align*}
      \sum_{\of \in\cgo(F)}    \sup_{x\in  G(\AAA), c(x) \in \of +\om }   \|x\|^{-t}
    \end{align*}
    où $\om\subset \cgo(\AAA)$ est l'ensemble compact introduit plus haut. Via l'identification de $\cgo(\AAA)$ à $\AAA^{d\nu}$, un élément $y\in \cgo(\AAA)$ s'écrit $y=(y_1,\ldots,y_{d\nu})\in \AAA^{d\nu}$; on pose alors
    $$\|y\|=\prod_{v} \sup(1,|y_{1,v}|_v,\ldots, |y_{d\nu,v}|_v).$$ 
   Il existe $N,C>0$ tel que $\|c(x)\|\leq C \|x\|^N$ pour tout $x\in G(\AAA)$, la convergence ci-dessus résulte alors de la convergence suivante
    \begin{align*}
      \sum_{\of \in\cgo(F)}    \sup_{y\in \om }   \|y+\of\|^{-t/N}<\infty
    \end{align*}
    pour tout $t>0$ assez grand.
  \end{preuve}
\end{paragr}

\begin{paragr}
  Soit $f\in \Sc(G(\AAA))$.   Soit $P\in \fc(M_0)$ et  $\of\in \cgo(F)$. Pour tous $w_1,w_2\in W$ et tous $x,y\in G(\AAA)$, on pose

  \begin{align}\label{eq:Kpofw_1w_2 new}
  K_{P,\of,f}^{w_1,w_2}(x,y)=\sum_{\gamma\in (M_P\cap w_1G_\of w_2^{-1})(F)} \int_{N_P(\AAA)} f((w_1x)^{-1}\gamma  n w_2y)\, dn.
\end{align}
On obtient ainsi une fonction sur $P_{w_1}^{\theta'}(F)\back G(\AAA)\times P_{w_2}^\theta(F)\back G(\AAA)$.

Soit $K_{P,f}^{w_1,w_2}(x,y)=K_{P,f}(w_1x,w_2y)$ où $K_{P,f}$  est introduit au § \ref{S:KPchi}. Notons qu'on a 
\begin{align}\label{eq:sum of cgo KP}
 K_{P,f}^{w_1,w_2}(x,y)= \sum_{\of\in \cgo(F)}  K_{P,\of,f}^{w_1,w_2}(x,y).
\end{align}

\begin{lemme}\label{lem:KPof}Soit   $P\in \fc(M_0)$. Pour  tous $w_1,w_2\in W$,  $\of,\of'\in \cgo(F)$, $x\in G^{\theta'}(\AAA)$ et $y\in G^\theta(\AAA)$ on a
 \begin{align}\label{eq:dichot o}
       K_{P,\of',f_\of}^{w_1,w_2}(x,y)=\left\lbrace
       \begin{array}{l}
          K_{P,\of,f}^{w_1,w_2}(x,y)\text{ si } \of'=\of\, ;\\
0 \text{ sinon.} 
       \end{array}\right.
 \end{align}
 En particulier, on a
 \begin{align*}
    K_{P,f_\of}^{w_1,w_2}(x,y)=K_{P,\of,f}^{w_1,w_2}(x,y).
 \end{align*}
\end{lemme}

\begin{preuve} Soit  $w_1,w_2,\of,\of', x,y$ comme dans l'énoncé. Si le membre de gauche de \eqref{eq:dichot o} est non nul, alors il existe $\gamma\in (M_P\cap w_1G_{\of'} w_2^{-1})(F)$  et $n\in N_P(\AAA)$ tels que 
\begin{align*}
    \zeta( c(x^{-1}w_1^{-1}\gamma n w_2y)-\of)\not=0.
\end{align*}
Il résulte  de la proposition \ref{prop:inv-de-c}  qu'on a 
\begin{align*}
    c(x^{-1}w_1^{-1}\gamma n w_2y)=c(w_1^{-1}\gamma w_2)=\of'.
\end{align*}
Or  $\zeta( \of'-\of)=0$ sauf si $\of=\of'$ auquel cas on a   $\zeta( \of'-\of)=1$. Cela donne l'annulation dans \eqref{eq:dichot o}. Supposons $\of'=\of$. Dans ce cas, pour tous  $\gamma\in (M_P\cap w_1G_{\of} w_2^{-1})(F)$  et $n\in N_P(\AAA)$, on a 
\begin{align*}
    \zeta( c(x^{-1}w_1^{-1}\gamma n w_2y)-\of)=\zeta( \of-\of)=1.
\end{align*}
On a ainsi obtenu la première égalité. La dernière égalité résulte alors de  \eqref{eq:sum of cgo KP}. 
\end{preuve}

Soit $T\in T_0+\overline{\ago_0^+}$ un paramètre de troncature. Pour tous $x\in G^{\theta'}(F)\back G(\AAA)$ et $y\in G^\theta(F)\back G(\AAA)$,  on définit
\begin{align}\label{eq:KTo}
    K^T_{\of,f}(x,y)=K^{T,\theta',\theta}_{\of,f}(x,y)=\sum_{P\in \fc(P_0)} \eps_P^G \sum_{w_1 \in   \, _PW_{\theta'}} \sum_{\delta_1\in P_{w_1}^{\theta'}(F) \back G^{\theta'}(F)}  \hat\tau_P(H_P(w_1\delta_1x)-T) \times \\ \nonumber \left[\sum_{w_2\in   \, _PW_\theta} \sum_{\delta_2\in P_{w_2}^\theta(F) \back G^\theta(F)} K_{P,\of,f}^{w_1,w_2}(\delta_1x,\delta_2y)\right].
\end{align}

\begin{remarque}
  La somme ci-dessus sur $\delta_1$ est finie et celle sur $\delta_2$ est absolument convergente, cf. remarque \ref{rq:delta1}.
\end{remarque}

Soit $K^T_f=K^{T,\theta',\theta}_f$ le noyau modifié introduit au §\ref{S:maj noyau modif}. Vu l'égalité \eqref{eq:sum of cgo KP}, on a pour tous $x, y\in G(\AAA)$ 
\begin{align}
  \label{eq:sum-of}
K^T_f(x,y)=\sum_{\of\in\cgo(F)} K^T_{\of,f}(x,y)
\end{align}
dont la convergence absolue résulte à nouveau du lemme \ref{lem:rappel*BPCZ}. On déduit alors  du lemme \ref{lem:KPof} qu'on a pour tous $x\in G^{\theta'}(\AAA)$ et $y\in G^\theta(\AAA)$ 
\begin{align}\label{eq:Kof=Kf}
   K^T_{f_\of}(x,y)= K^T_{\of,f}(x,y).
\end{align}

\end{paragr}

\begin{paragr}
  \begin{theoreme}\label{thm:cv-geo}
   Pour tout $T\in T_0+\overline{\ago_0^+}$ et $N_1, N_2>0$, il existe une semi-norme continue $\|\cdot\| $ sur $\Sc(G(\AAA))$ telle que pour  tout $f\in \Sc(G(\AAA))$ on ait 
    \begin{align*}
      \sum_{\of \in\cgo(F)}  \int_{[G^{\theta'}]^G  \times [G^\theta]} |K^T_{\of,f}(x,y)| \|x\|_{G^{\theta'}}^{N_1}\|y\|_{G^{\theta}}^{N_2} \, dxdy \leq \|f\|.
    \end{align*}
  \end{theoreme}

  \begin{preuve}
   D'après le théorème \ref{thm:cv-spec}, il existe une semi-norme continue $\|\cdot\| $ sur $\Sc(G(\AAA))$ telle que pour  tout $f\in \Sc(G(\AAA))$ on ait 
$$
\int_{[G^{\theta'}]^G  \times [G^\theta]} |K^T_{f}(x,y)|\|x\|_{G^{\theta'}}^{N_1}\|y\|_{G^{\theta}}^{N_2} \, dxdy \leq \|f\|.
$$
En utilisant \eqref{eq:Kof=Kf}, on en déduit que 
\begin{align*}
   \sum_{\of \in\cgo(F)}  \int_{[G^{\theta'}]^G  \times [G^\theta]} |K^T_{\of,f}(x,y)| \|x\|_{G^{\theta'}}^{N_1}\|y\|_{G^{\theta}}^{N_2} \, dxdy \leq   \sum_{\of \in\cgo(F)} \|f_\of\|.
\end{align*}
D'après le lemme \ref{lem:sum fof},  $f\mapsto \sum_{\of \in\cgo(F)} \|f_\of\|$ est encore une semi-norme continue. Cela conclut.
\end{preuve}
\end{paragr}

\begin{paragr} Il sera parfois plus commode d'utiliser les noyaux associés à certains sous-groupes paraboliques semi-standard. Pour cela, on observe que, pour tout $P\in\fc(M_0)$, tous $x,y\in G(\AAA)$ et tous $w_1,w_2\in W$, on a
  \begin{align}\label{eq:KPow}
    K_{P,\of,f}^{w_1,w_2}(x,y) = K_{P_{w_1},\of,f}^{1,w_1^{-1}w_2}(x,y). 
  \end{align}
On utilisera alors l'expression suivante de $K^T_{\of,f}(x,y)$. 

\begin{proposition}
  Soit $\of\in\cgo(F)$. Pour tous $x,y\in G(\AAA)$ on a 
  \begin{align}\label{eq:varKTo}
  K^T_{\of,f}(x,y)=\sum_{P\in \fc(P_0^{\theta'})} \eps_P^G  \sum_{\delta_1\in P^{\theta'}(F) \back G^{\theta'}(F)}  \hat\tau_{P}(H_{P}(\delta_1x)-T_{P}) \times \\ \nonumber \left[\sum_{w\in   \, _{P}W_\theta} \sum_{\delta_2\in P_{w}^\theta(F) \back G^\theta(F)} K_{P,\of,f}^{1,w}(\delta_1x,\delta_2y)\right].
  \end{align}
\end{proposition}

\begin{preuve}
    D'après les égalités \eqref{eq:tauw} et \eqref{eq:KPow}, on a    
  \begin{align*}
    K^T_{\of,f}(x,y)=\sum_{P\in \fc(P_0)} \eps_P^G \sum_{w_1 \in   \, _PW_{\theta'}} \sum_{\delta_1\in P_{w_1}^{\theta'}(F) \back G^{\theta'}(F)}  \hat\tau_{P_{w_1}}(H_{P_{w_1}}(\delta_1x)-T_{P_{w_1}}) \times \\ \nonumber \left[\sum_{w_2\in   \, _PW_\theta} \sum_{\delta_2\in P_{w_2}^\theta(F) \back G^\theta(F)} K_{P_{w_1},\of,f}^{1,w_1^{-1}w_2}(\delta_1x,\delta_2y)\right].
\end{align*}
Par le lemme \ref{lem:Wconj}, on peut faire le changement de variables $w_2\mapsto w=w_1^{-1}w_2$; on obtient que $K^T_{\of,f}(x,y)$ est égal à 
\begin{align*}
    \sum_{P\in \fc(P_0)} \eps_P^G \sum_{w_1 \in   \, _PW_{\theta'}} \sum_{\delta_1\in P_{w_1}^{\theta'}(F) \back G^{\theta'}(F)}  \hat\tau_{P_{w_1}}(H_{P_{w_1}}(\delta_1x)-T_{P_{w_1}}) \times \\ \nonumber \left[\sum_{w\in   \, _{P_{w_1}}W_\theta} \sum_{\delta_2\in P_{w_1 w}^\theta(F) \back G^\theta(F)} K_{P_{w_1},\of,f}^{1,w}(\delta_1x,\delta_2y)\right].
\end{align*}
On conclut avec le corollaire \ref{cor:PWtheta}. 
\end{preuve}

\end{paragr}

\begin{paragr}[Distributions géométriques.]\label{S:dist geo} --- 
      Soit $\eta: G(\AAA) \to \CC^\times$ un caractère unitaire comme dans \S \ref{S:JchiT}. Soit $T\in T_0+\overline{\ago_0^+}$ et $\of\in\cgo(F)$. On introduit la distribution  $J^T_\of(\eta)$ sur $\Sc(G(\AAA))$ donnée par l'intégrale suivante dont la convergence absolue et la continuité en $f\in \Sc(G(\AAA))$ est garantie par le théorème \ref{thm:cv-geo}:
  \begin{align}\label{eq:JToeta}
    J^T_\of(\eta,f)=  \int_{[G^{\theta'}]^G\times [G^\theta]      }  K^T_{\of,f}(x,y) \, \eta(x)dxdy.
  \end{align}

  L'égalité  \eqref{eq:Kof=Kf} entraîne qu'on a
  \begin{align}\label{eq:JoTf=JTf}
       J^T_\of(\eta,f)=J^T(\eta,f_\of)
  \end{align}
  où $J^T(\eta)$ est la distribution introduite au § \ref{S:JTeta}. Il résulte alors des propriétés de $J^T(\eta)$, cf.    § \ref{S:JTeta}, que l'application $T\mapsto J^{T}_\of(\eta,f)$ coïncide dans un certain cône avec une fonction polynôme-exponentielle en $T$. On définit alors $J_\of(\eta,f)$ comme le terme constant de cette fonction de $T$. On a alors 
   \begin{align*}
       J_\of(\eta,f)=J(\eta,f_\of),
  \end{align*}
  où $J(\eta)$ est la distribution définie comme le terme constant de $J^T(\eta)$, cf.  § \ref{S:JTeta}.
  Comme $J(\eta)$ est continue, il résulte du lemme \ref{lem:sum fof}, que les distributions  $J_\of(\eta)$ et $\sum_{\of\in \cgo(F)} J_\of(\eta)$ sont continues sur $\Sc(G(\AAA))$. 
 
\end{paragr}

\begin{paragr}[Propriétés de covariance de   $J^{T}_\of(\eta,f)$ et $J_\of(\eta,f)$.] ---  Soit $Q\subset G$ un sous-groupe parabolique standard, $w_1'\in \, _QW_{\theta'}$ et  $w_2'\in \, _QW_{\theta}$. On pose $\theta_1=w_1'\theta' {w_1'}^{-1}$ et  $\theta_2=w_2' \theta {w_2'}^{-1}$. Comme au \S \ref{S:quot cat}, on introduit le morphisme canonique $c_{M_Q} : M_Q \to \cgo_{M_Q}=\Spec(F[M_Q]^{M_Q^{\theta_1}\times M_Q^{\theta_2}})$ où $F[M_Q]^{M_Q^{\theta_1}\times M_Q^{\theta_2}}$ est la $F$-algèbre des fonctions régulières sur $M_Q$ invariante respectivement à droite et à gauche par $M_Q^{\theta_1}$ et $M_Q^{\theta_2}$. On a le diagramme commutatif suivant : 
\begin{align}\label{eq:diag-fibre}
    \xymatrix{  M_Q \ar[r]^{i_{Q,w_1',w_2'}} \ar[d]^{c_{M_Q}}  & G  \ar[d]^{c}    \\
   \cgo_{M_Q} \ar[r]^{j_{Q,w_1',w_2'}} & \cgo    }
\end{align}
où $i_{Q,w_1',w_2'} : M_Q\to G$ est donné par $m\mapsto w_1'^{-1} m w_2'$ et $j_{Q,w_1',w_2'} : \cgo_{M_Q}\to \cgo$ est le morphisme induit par la propriété universelle du quotient catégorique.  

\begin{lemme}\label{lem:fibre-finie}
    Le morphisme $j_{Q,w_1',w_2'}$ est fini. En particulier, ses fibres sont finies.
\end{lemme}

\begin{preuve}
Le morphisme $G\to G$ donné par $g\mapsto w_1'^{-1} g w_2'$ identifie dualement  $F[G]^{G^{\theta'}\times G^{\theta}}$ à $F[G]^{G^{\theta_1}\times G^{\theta_2}}$. Il s'agit donc de voir que le morphisme de restriction $F[G]^{G^{\theta_1}\times G^{\theta_2}}\to F[M_Q]^{M_Q^{\theta_1}\times M_Q^{\theta_2}}$ est fini. Soit $S_1=\{g\theta_2g^{-1}\theta_1\mid g\in G \}$   et $S_1^Q=\{m\theta_2m^{-1}\theta_1\mid m\in G \}$ muni de l'action par conjugaison respectivement de $G^{\theta_1}$ et $M_Q^{\theta_1}$. Le morphisme précédent s'identifie à $F[S_1]^{G^{\theta_1}}\to F[S_1^Q]^{M_Q^{\theta_1}}$ où, comme d'habitude les exposants désignent les sous-algèbres de fonctions régulières et invariantes. Le sous-groupe de Levi est le fixateur d'une décomposition $V=V_1\oplus\ldots \oplus V_r$ de $V$ en sous-$D$-modules, cf. la sous-section \ref{ssec:preparatif alg}. Pour $1\leq j\leq r$ et $i=1,2$, on note $V_j^{\pm \theta_i}$ les espaces propres de la restriction de $\theta_i$ à $V_j$ pour la valeur propre $\pm 1$. On pose $p_j=\dim_D(V_j^{\theta_1})$ (on omet le $+$ en exposant) et   $q_j=\dim(V_j^{\theta_1})$. On note $p_j'$ et $q_j'$ les mêmes dimensions relatives à $\theta_2$. Soit  $\nu_j=\min\{p_j,q_j,p'_j,q'_j\}$.  Soit $p=\sum_{j=1}^r p_j$. De même, on définit $q=\sum_{j=1}^r q_j$, $p'$ et $q'$. Introduisons les entiers naturels  $\nu=\min\{p,q,p',q'\}$, $N_+=\sum_{j=1}^r \max(p_j'-q_j,0)-  \max(p'-q,0) $ et $N_-=\sum_{j=1}^r \max(p_j-p_j',0)-  \max(p-p',0)$. D'après la proposition \ref{prop:quot-cat}, le morphisme  $\Spec(F[S_1^Q]^{M_Q^{\theta_1}})\to \Spec(F[S_1]^{G^{\theta_1}})$ s'interprète comme le morphisme donné par
\begin{align*}
    (P_1,\ldots,P_r)\mapsto    (t-1)^{d N_+}  (t+1)^{d N_- }   \prod_{j=1}^r P_j
\end{align*}
entre le produit pour $j=1,\ldots,r$ des espaces affines  des polynômes unitaires de degré $d\nu_j$ vers l'espace affine  des polynômes unitaires de degré $d\nu$. Il est aisé de voir que ce morphisme est fini.

\end{preuve}

On reprend désormais les notations  du paragraphe \ref{S:covar-chi-T}. Soit  $f'\in \Sc(M_Q(\AAA))$ et  $\of'\in \cgo_{M_Q}(F)$. On dispose des noyaux modifiés $K^{T,\theta_1,\theta_2}_{f'}$ et $K^{T,\theta_1,\theta_2}_{\of',f'}$ dont la définition est donnée respectivement au \S \ref{S:maj noyau modif} et en \eqref{eq:KTo} où l'on substitue $M_Q$ à $G$. Soit
    \begin{align}\label{eq:sum-JQTtheta}
        J^{Q,T,\theta_1,\theta_2}(\eta,f')=\sum_{\chi'\in \Xgo(M_Q)} J^{Q,T,\theta_1,\theta_2}_{\chi'}(\eta,f') 
    \end{align}
    où les termes du membre de droite sont définis en \eqref{eq:JQTtheta}. Comme observé dans les lignes qui suivent \eqref{eq:JQTtheta}, l'assertion 2 du théorème \ref{thm:cv-spec} s'applique à notre situation et garantit la  convergence  et la continuité du membre de droite.     Le pendant géométrique de ces termes spectraux est donné par
      \begin{align}\label{eq:Jo'QTtheta}
  &  J^{Q,T,\theta_1,\theta_2}_{\of'}(\eta,f')\\
   \nonumber& =  \int_{    [M_{Q}^{\theta_1}]^Q}  \int_{    [M_{Q}^{\theta_2}]}  \exp(-\bg 2\rho_{Q^{\theta_1}}^{G^{\theta_1}}, H_{Q^{\theta_1}   }( x)\bd+\bg 2\rho_Q^G- 2\rho_{Q^{\theta_2}}^{G^{\theta_2}}, H_{Q^{\theta_2}   }( y)\bd) K^{T,\theta_1,\theta_2}_{\of',f'}(x,y)\, \eta(x)dxdy.
  \end{align}
  De même,  le théorème \ref{thm:cv-geo}  assure ici que l'intégrale \eqref{eq:Jo'QTtheta}, est absolument convergente et que la distribution  $J^{Q,T,\theta_1,\theta_2}_{\of'}(\eta)$ ainsi définie est continue. 
  
  Soit $\of\in\cgo(F)$. On a 
 \begin{align}\label{eq:fibre-dec}
M_Q(F)\cap      w_1' G_\of(F) w_2'^{-1}=\bigsqcup_{\of'} M_{Q,\of'}(F)
 \end{align}
 où $M_{Q,\of'}$ est la fibre de $c_{M_Q}$ au-dessus de $\of'$ et la réunion est prise sur l'ensemble (fini, cf. lemme \ref{lem:fibre-finie})  des $\of'\in \cgo_{M_Q}(F)$ tels que $j_{Q,w_1',w_2'}(\of')=\of$. Il est  alors commode de poser
  \begin{align}\label{eq:JoQTtheta}
      J^{Q,T,\theta_1,\theta_2}_{\of}(\eta)=  \sum_{\of' }J^{Q,T,\theta_1,\theta_2}_{\of'}(\eta)
  \end{align}
 où la somme à droite porte sur le même ensemble fini  des $\of'\in \cgo_{M_Q}(F)$ tels que $j_{Q,w_1',w_2'}(\of')=\of$. 
 Comme au \S  \ref{S:dist geo}, soit $J^{Q,\theta_1,\theta_2}_{\of}(\eta)$ le terme constant de $J^{Q,T,\theta_1,\theta_2}_{\of}(\eta)$. 
La proposition suivante étudie l'action de $G^\theta(\AAA)$ et $G^{\theta'}(\AAA)$ sur ces distributions, pour les notations $f^g, \ ^g \!f$, cf. § \ref{S:covar-chi-T}.

\begin{proposition}\label{prop:covar-geom-T}  Soit $\of\in\cgo(F)$ et $f\in \Sc(G(\AAA))$.
  \begin{enumerate}
  \item Pour tout $g\in G^\theta(\AAA)$, on  a
    \begin{align*}
      J^{T}_\of(\eta,f^g)=J^{T}_\of(\eta,f).
    \end{align*}
  \item  Pour tout $g\in G^{\theta'}(\AAA)$, on  a
    \begin{align*}
      &      J^{T}_\of(\eta,  ^g\!\!f)=\\
      &\eta(g)\sum_{Q\in \fc^G(P_0)}  \sum_{w_1'\in \, _QW_{\theta'}} \sum_{w_2'\in \, _QW_{\theta}}    \exp( \bg 2\rho_Q^G-2\rho_{Q^{\theta_1}}^{G^{\theta_1}}-2\rho_{Q^{\theta_2}}^{G^{\theta_2}}, T_Q\bd) J^{Q,T,\theta_1,\theta_2}_{\of} ( \eta, f_{Q,\eta,g}^{w_1',w_2'}). 
    \end{align*}
    où $f_{Q,\eta,g}^{w_1',w_2'}$ et $J^{Q,T,\theta_1,\theta_2}_{\of} ( \eta)$ sont définis respectivement en \eqref{eq:fct-fQT} et \eqref{eq:JoQTtheta}.
  \end{enumerate}
\end{proposition}

\begin{preuve}
     1. L'invariante par translations à droite par $G^\theta(\AAA)$ est évidente sur la définition \eqref{eq:JToeta}. 
     
     2.  Pour tout $g\in G^{\theta'}(\AAA)$, en utilisant l'égalité de l'assertion 2 de la proposition \ref{prop:covar-T} sommée sur les données cuspidales, on trouve 
    \begin{align*}
      &      J^{T}(\eta,  ^g\!\!f)=\\
      &\eta(g)\sum_{Q\in \fc^G(P_0)}  \sum_{w_1'\in \, _QW_{\theta'}} \sum_{w_2'\in \, _QW_{\theta}}    \exp( \bg 2\rho_Q^G-2\rho_{Q^{\theta_1}}^{G^{\theta_1}}-2\rho_{Q^{\theta_2}}^{G^{\theta_2}}, T_Q\bd) J^{Q,T,\theta_1,\theta_2} ( \eta, f_{Q,\eta,g}^{w_1',w_2'}). 
    \end{align*}
    Par construction même de $f_\of$, cf. \S  \ref{S:zetaof}, on a $(^g\!f)_\of=^g\!\!(f_\of)$. Avec \eqref{eq:JoTf=JTf}, on en déduit 
    \begin{align*}
      &      J^{T}_\of(\eta,  ^g\!\!f)=J^{T}(\eta,  (^g\!\!f)_\of)=J^{T}(\eta,  ^g\!\!(f_\of))=\\
      &\eta(g)\sum_{Q\in \fc^G(P_0)}  \sum_{w_1'\in \, _QW_{\theta'}} \sum_{w_2'\in \, _QW_{\theta}}    \exp( \bg 2\rho_Q^G-2\rho_{Q^{\theta_1}}^{G^{\theta_1}}-2\rho_{Q^{\theta_2}}^{G^{\theta_2}}, T_Q\bd) J^{Q,T,\theta_1,\theta_2} ( \eta, (f_\of)_{Q,\eta,g}^{w_1',w_2'}). 
    \end{align*}
    Il suffit donc de montrer 
    \begin{align*}
        J^{Q,T,\theta_1,\theta_2} ( \eta, (f_\of)_{Q,\eta,g}^{w_1',w_2'})=J^{Q,T,\theta_1,\theta_2}_{\of} ( \eta, f_{Q,\eta,g}^{w_1',w_2'}).
    \end{align*}
    D'après \eqref{eq:sum-JQTtheta}, \eqref{eq:Jo'QTtheta} et \eqref{eq:JoQTtheta}, on est ramené à prouver 
    \begin{align*}
        K^{T,\theta_1,\theta_2}_{(f_\of)_{Q,\eta,g}^{w_1',w_2'}}=\sum_{\of'\in j^{-1}_{Q,w_1',w_2'}(\of)} K^{T,\theta_1,\theta_2}_{\of',f_{Q,\eta,g}^{w_1',w_2'}}. 
    \end{align*}
    Soit $f'\in \Sc(M_Q(\AAA))$, $x\in M_Q^{\theta_1}(\AAA)$ et $y\in M_Q^{\theta_2}(\AAA)$. Les définitions s'explicitent ainsi:
  \begin{align*}
    K^{T,\theta_1,\theta_2}_{f'}(x,y)=\sum_{P\in \fc^{M_Q}(P_0\cap M_Q)} \eps_P^{M_Q} \sum_{w_3 \in   \, _PW^{M_Q}_{\theta_1}} \sum_{\delta_3\in P_{w_3}^{\theta_1}(F) \back M_Q^{\theta_1}(F)}  \hat\tau_P^{M_Q}(H_P(w_3\delta_3x)-T) \times \\ \nonumber \left[\sum_{w_4\in   \, _PW^{M_Q}_{\theta_2}} \sum_{\delta_4\in P_{w_4}^{\theta_2}(F) \back M_Q^{\theta_2}(F)} K_{P,f'}(w_3\delta_3x, w_4\delta_4y)\right]
  \end{align*}
  ainsi que 
  \begin{align*}
    K^{T,\theta_1,\theta_2}_{\of',f'}(x,y)=\sum_{P\in \fc^{M_Q}(P_0\cap M_Q)} \eps_P^{M_Q} \sum_{w_3 \in   \, _PW^{M_Q}_{\theta_1}} \sum_{\delta_3\in P_{w_3}^{\theta_1}(F) \back M_Q^{\theta_1}(F)}  \hat\tau_P^{M_Q}(H_P(w_3\delta_3x)-T) \times \\ \nonumber \left[\sum_{w_4\in   \, _PW^{M_Q}_{\theta_2}} \sum_{\delta_4\in P_{w_4}^{\theta_2}(F) \back M_Q^{\theta_2}(F)} K_{P,\of',f'}^{w_3,w_4}(\delta_3x, \delta_4y)\right]. 
  \end{align*} 
  L'expression  $K_{P,\of',f'}^{w_3,w_4}$ ci-dessus est donnée par la variante suivante de \eqref{eq:Kpofw_1w_2 new}
  \begin{align*}
  K_{P,\of',f'}^{w_3,w_4}(x,y)=\sum_{\gamma\in (M_P\cap w_3M_{Q,\of'} w_4^{-1})(F)} \int_{N_P(\AAA)} f'((w_3x)^{-1}\gamma  n w_4 y)\, dn.
\end{align*}

  Il suffit de montrer que, pour tous $x\in M_Q^{\theta_1}(\AAA)$ et $y\in M_Q^{\theta_2}(\AAA)$ et tous $(P,w_3,w_4)$ comme dans les sommes ci-dessus, on a 
  \begin{align}\label{eq:K=sumK}
      K_{P,(f_\of)_{Q,\eta,g}^{w_1',w_2'}}(w_3x, w_4y)=\sum_{\of'\in j^{-1}_{Q,w_1',w_2'}(\of)} K_{P,\of',f_{Q,\eta,g}^{w_1',w_2'}}^{w_3,w_4}(x, y). 
  \end{align}
    Mais la proposition \ref{prop:inv-de-c} et la définition \eqref{eq:fct-fQT} entraînent que pour tout $m\in M_Q(\AAA)$ 
    \begin{align*}
        (f_\of)^{w'_1,w'_2}_{Q,\eta,g}(m)=\zeta_\of(w'^{-1}_1 mw'_2) f_{Q,\eta,g}^{w'_1,w'_2}(m). 
    \end{align*} 
   Soit $x\in M_Q^{\theta_1}(\AAA)$ et $y\in M_Q^{\theta_2}(\AAA)$. Pour tous  $\gamma\in M_P(F)$ et $n\in N_P(\AAA)$, en utilisant de nouveau la proposition \ref{prop:inv-de-c}, on a  
  \begin{align*}
      \zeta_\of(w'^{-1}_1 (w_3x)^{-1}\gamma n (w_4y)w'_2)=\zeta_\of((w_3w'_1)^{-1} \gamma n (w_4w'_2))=\zeta_\of((w_3w'_1)^{-1} \gamma (w_4w'_2)). 
  \end{align*}
  Il s'ensuit que 
  \begin{align*}
      K_{P,(f_\of)^{w'_1,w'_2}_{Q,\eta,g}}(w_3x, w_4y)=\sum_{\gamma \in M_P(F)} \int_{N_P(\AAA)} (f_\of)^{w'_1,w'_2}_{Q,\eta,g} ((w_3x)^{-1}\gamma n (w_4y))\, dn\\
      =\sum_{\gamma\in M_P(F)} \zeta_\of((w_3w'_1)^{-1} \gamma (w_4w'_2)) \int_{N_P(\AAA)} f_{Q,\eta,g}^{w'_1,w'_2}((w_3x)^{-1}\gamma n (w_4y)) \, dn \\
      =\sum_{\gamma\in (M_P\cap (w_3w'_1)G_\of(w_4w'_2)^{-1})(F)}  \int_{N_P(\AAA)} f_{Q,\eta,g}^{w'_1,w'_2}((w_3x)^{-1}\gamma n (w_4y)) \, dn. 
  \end{align*}
  On a $(M_P\cap (w_3w'_1)G_\of(w_4w'_2)^{-1})\subset M_Q\cap (w_3w'_1)G_\of(w_4w'_2)^{-1}) = w_3( M_Q\cap (w'_1 G_\of(w'_2)^{-1}))w_4^{-1}$.
  On peut alors utiliser la décomposition \eqref{eq:fibre-dec} pour obtenir que l'expression précédente est égale à
  \begin{align*}
      \sum_{\of'\in j^{-1}_{Q,w_1',w_2'}(\of)} \sum_{\gamma\in (M_P\cap w_3M_{Q,\of'}w_4^{-1})(F)}  \int_{N_P(\AAA)} f_{Q,\eta,g}^{w'_1,w'_2}((w_3x)^{-1}\gamma n (w_4y)) \, dn\\
      =\sum_{\of'\in j^{-1}_{Q,w_1',w_2'}(\of)} K_{P,\of',f_{Q,\eta,g}^{w_1',w_2'}}^{w_3,w_4}(x, y). 
  \end{align*}
  On a donc prouvé \eqref{eq:K=sumK} et, par suite, la proposition. 
\end{preuve}

Comme conséquence de la proposition \ref{prop:covar-geom-T}, on obtient les relations de covariance pour les distributions  $J_\of(\eta)$ qui sont d'ailleurs  analogues à celles obtenues dans la proposition \ref{prop:covar} pour les distributions spectrales $J_\chi(\eta)$.

  \begin{proposition}\label{prop:covar-geom}Soit $\of\in\cgo(F)$ et $f\in \Sc(G(\AAA))$.
  \begin{enumerate}
  \item Pour tout $g\in G^\theta(\AAA)$, on  a
    \begin{align*}
      J_\of(\eta,f^g)=J_\of(\eta,f).
    \end{align*}
  \item  Pour tout $g\in G^{\theta'}(\AAA)$, on  a
    \begin{align*}
      J_\of(\eta,  ^g\!\!f)=\eta(g) \sum_{Q, w_1',w_2'}  J^{Q,\theta_1,\theta_2}_{\of} ( \eta, f_{Q,\eta,g}^{w_1',w_2'})
    \end{align*}
    où la somme porte sur les triplets $(Q, w_1',w_2')$ formés d'un sous-groupe parabolique standard $Q$ et d'éléments $ w_1'\in \, _QW_{\theta'}$ et $w_2'\in \, _QW_{\theta}$ tels que $Q\in \fc^{G,\flat}(P_0,\theta_1,\theta_2)$ et où  $f_{Q,\eta,g}^{w_1',w_2'}$ est défini en \eqref{eq:fct-fQT} et $J^{Q,\theta_1,\theta_2}_{\of} ( \eta)$ est le terme constant de \eqref{eq:JoQTtheta}. 
  \end{enumerate}
\end{proposition}

\begin{preuve}
On prend les termes constants dans la proposition \ref{prop:covar-geom-T}. La preuve est la même que celle de la proposition \ref{prop:covar}. 
\end{preuve}
\end{paragr}

\begin{paragr}[Le cas $GL_2(D)$.] --- \label{S:GL2D geo}On revient sur la situation du § \ref{S:GL2D spec} qui concerne $G=GL_2(D)$ et $\theta=\theta'$ d'ordre $2$ et distincts de $\pm I_2$. Dans ce cas, la proposition \ref{prop:covar-geom} se simplifie et on obtient le pendant géométrique de la proposition \ref{prop:GL2 spec}.

\begin{proposition}
    \label{prop:cas GL2 geo}
    Sous les conditions ci-dessous, pour tout $g\in G^\theta(\AAA)$, $\of\in \cgo(F)$ et $f\in \Sc(G(\AAA))$, on  a
    \begin{align*}
            & J_\of(\eta,  ^g\!\!f)= \eta(g) J_\of(\eta,f).
    \end{align*}
  \end{proposition}
\end{paragr}

\subsection{Formule des traces de Guo-Jacquet}\label{ssec:FT-GJ}

\begin{paragr} On rappelle que $\eta$ est le caractère de $F^\times\back \AAA^\times$ introduit au § \ref{S:JchiT}.   

\begin{theoreme} \label{thm:FTGJ avec T}
 Pour tout $T\in T_0+\overline{\ago_0^+}$ et tout $f\in\Sc(G(\AAA))$, on a 
    \begin{align*}
        \sum_{\chi\in\Xgo(G)} J_\chi^T(\eta, f)=\sum_{\of\in\cgo(F)} J_\of^T(\eta, f) 
    \end{align*}
    où $J_\chi^T(\eta, f)$ et $J_\of^T(\eta, f)$ sont les formes linéaires continues sur $\Sc(G(\AAA))$ définies respectivement en \eqref{eq:JTchieta} et \eqref{eq:JToeta}. 
\end{theoreme}

\begin{preuve}
C'est une conséquence de \eqref{eq:sum-chif}, \eqref{eq:sum-of}, des théorèmes \ref{thm:cv-spec} et \ref{thm:cv-geo} et des définitions \eqref{eq:JTchieta} et \eqref{eq:JToeta} des distributions. 
\end{preuve}
\end{paragr}

\begin{paragr} On obtient finalement la formule des traces de Guo-Jacquet.

\begin{theoreme} \label{thm:FTGJ} On a l'égalité de distribution sur $\Sc(G(\AAA))$
\begin{align*}
                   \sum_{\chi\in\Xgo(G)} J_\chi(\eta)=\sum_{\of \in\cgo(F)} J_\of(\eta).
       \end{align*}
       La convergence est absolue au sens où, pour chaque membre,  la somme des valeurs absolues des termes fournit une semi-norme continue  sur $\Sc(G(\AAA))$.
  \end{theoreme}
  
  \begin{preuve}
  Il suffit de prendre le terme constant en $T$ dans le théorème \ref{thm:FTGJ avec T} comme $J_\chi^T(\eta, f)$ et $J_\of^T(\eta, f)$ sont des fonctions polynôme-exponentielles en $T$ d'après la proposition \ref{prop:JTchi poly} et \S \ref{S:dist geo}. 
  \end{preuve}
\end{paragr}

\subsection{Passage à l'espace symétrique}\label{ssec:passage}

\begin{paragr}
    On reprend les notations de la sous-section \ref{ssec:preparatif alg} en supposant de plus que $F$ est un corps de nombres. Soit $0\leq p,p'\leq N$. On pose $\theta=\theta_p$.

    \begin{remarque}
    On utilisera également les résultats  de la sous-section \ref{ssec:dvpt quotient}. Cependant, on prendra garde à la légère incohérence de notations: l'élément que nous notons $\theta$ ici joue le rôle de l'élément  noté $\theta'$  là-bas.
        \end{remarque}
\end{paragr}

\begin{paragr}
  Soit $P\in \fc(M_0)$ et $\gamma\in (M_P\cap S_{p'})(F)$. D'après le lemme \ref{lem:parabS}, l'application 
  \begin{align*}
  n\mapsto \Int_\theta(n)(\gamma) 
  \end{align*}
induit un isomorphisme de $N_P/N_P^{\gamma\theta}$ sur $(\gamma N_P)\cap S_{p'}$. On obtient alors une mesure sur $((\gamma N_P)\cap S_{p'})(\AAA)$ par transport de la mesure quotient sur $N_P(\AAA)/N_P^{\gamma\theta}(\AAA)$. 
\end{paragr}

\begin{paragr}
  Soit $\varphi\in\Sc(S_{p'}(\AAA))$, cf. §\ref{S:Schwartz}. Soit $P\in \fc(M_0)$ et $\of\in \cgo_{p'}(F)$. Pour tout $x\in G(\AAA)$ on pose 
  \begin{align}\label{eq:defKPS}
  K_{P,\of,\varphi}(x)=\sum_{\gamma\in (M_P\cap S_{p',\of})(F)} \int_{((\gamma N_P)\cap S_{p'})(\AAA)} \varphi(\Int_\theta(x^{-1})(y))\, dy. 
  \end{align}
On obtient ainsi une fonction sur $N_P^\theta(\AAA)M_P^\theta(F)\back G(\AAA)$. On a aussi l'expression suivante de $K_{P,\of,\varphi}(x)$. 

\begin{lemme}\label{lem:varKPS}
  Pour tout $x\in G(\AAA)$ on a 
  \begin{align}\label{eq:varKPS}
  K_{P,\of,\varphi}(x)=\sum_{w\in \,_PW_{\theta_{p'}}} \sum_{\ga} \int_{N_P(\AAA)/N_P^{w\theta_{p'}}(\AAA)} \varphi(\rho_{p'}(x^{-1} \gamma n w))\, dn
  \end{align}
  où la somme porte sur les classes $\gamma\in M_P(F)/M_P^{w\theta_{p'}}(F)$ tels que $\gamma M_P^{w\theta_{p'}}(F) w\subset G_\of(F)$.
\end{lemme}

\begin{preuve}
  Par notre choix de mesure sur $((\gamma N_P)\cap S_{p'})(\AAA)$, on a 
  \begin{align}\label{eq:2varKPS}
  K_{P,\of,\varphi}(x)=\sum_{\gamma\in (M_P\cap S_{p',\of})(F)} \int_{N_P(\AAA)/N_P^{\gamma\theta}(\AAA)} \varphi(\Int_\theta(x^{-1}n)(\gamma))\, dn. 
  \end{align}
D'après le lemme \ref{lem:LeviS}, cette expression est égale à 
  \begin{align*}
  \sum_{w\in \,_PW_{\theta_{p'}}} \sum_{\gamma\in M_P(F)/M_P^{w\theta_{p'}}(F), \gamma w\in G_\of} \int_{N_P(\AAA)/N_P^{\Int(\gamma w)(\theta_{p'})}(\AAA)} \varphi(\rho_{p'}(x^{-1}n \gamma w))\, dn. 
  \end{align*}
En remarquant que l'action $\Int(\gamma^{-1})$ induit un isomorphisme 
  \begin{align*}
  N_P(\AAA)/N_P^{\Int(\gamma w)(\theta_{p'})}(\AAA) \to N_P(\AAA)/N_P^{w\theta_{p'}}(\AAA) 
  \end{align*}
qui préserve les mesures, on obtient l'expression désirée. 
\end{preuve}
\end{paragr}

\begin{paragr}
 Soit $\varphi\in\Sc(S_{p'}(\AAA))$. Pour $T\in T_0+\overline{\ago_0^+}$ un paramètre de troncature et $\of\in \cgo_{p'}(F)$, on définit 
  \begin{align}\label{eq:KToS}
  K_{\of,\varphi}^T(x)=\sum_{P\in\fc(P_0^\theta)} \eps_P^G \sum_{\delta\in P^\theta(F)  \back G^\theta(F)} \hat\tau_P(H_P(\delta x)-T_P) K_{P,\of,\varphi}(\delta x)
  \end{align}
  pour $x\in G(\AAA)$. À l'aide de \cite[lemme 5.1]{ar1}, on voit que la somme sur $\delta$ dans la dernière expression est finie. 

  \begin{theoreme}\label{thm:cvKphi}
  Pour tout $T\in T_0+\overline{\ago_0^+}$ et $N_1>0$, il existe une semi-norme continue $\|\cdot\|$ sur $\Sc(S_{p'}(\AAA))$ telle que pour tout $\varphi\in \Sc(S_{p'}(\AAA))$ on ait 
    \begin{align*}
      \sum_{\of\in\cgo_{p'}(F)}  \int_{[G^\theta]^G} |K_{\of,\varphi}^T(x)|\|x\|_{G^{\theta}}^{N_1} \, dx \leq \|\varphi\|. 
    \end{align*}
  \end{theoreme}
  
  \begin{preuve}   On a une application surjective continue $\Sc(G(\AAA))\to \Sc(S_{p'}(\AAA))$ donnée par l'intégration sur les fibres. En particulier, pour tout  $\varphi\in \Sc(S_{p'}(\AAA))$, il existe une fonction $f\in \Sc(G(\AAA))$ telle que 
  \begin{align*}
      \varphi(x)=\int_{G^{\theta_{p'}}(\AAA)} f(g_x y) \, dy 
  \end{align*}
pour tout $x\in S_{p'}(\AAA)$, où $g_x\in G(\AAA)$ est un élément tel que $\rho_{p'}(g_x)=x$.  Il suffit de montrer que l'on a
  \begin{align}\label{eq:preuve Kphi=Kf}
      \int_{[G^{\theta_{p'}}]} K^T_{\of,f}(x,y) \, dy = K^T_{\of,\varphi}(x) 
  \end{align}
  pour tous $T\in T_0+\overline{\ago_0^+}$, $\of\in \cgo_{p'}(F)$ et $x\in G(\AAA)$. En effet, le théorème \ref{thm:cv-geo} implique que l'application
  \begin{align*}
   f\in \Sc(G(\AAA))\mapsto  \sum_{\of\in\cgo_{p'}(F)}  \int_{[G^\theta]^G} |  \int_{[G^{\theta_{p'}}]} K^T_{\of,f}(x,y) \, dy | \|x\|_{G^{\theta}}^{N_1} \,dx
  \end{align*}
  est continue. Alors cette application se factorise,  via l'application ouverte $S(G(\AAA))\to \Sc(S_{p'}(\AAA))$, en une application continue qui est celle qui nous concerne:
  \begin{align*}
     \varphi\in  \Sc(S_{p'}(\AAA))\mapsto  \sum_{\of\in\cgo_{p'}(F)}  \int_{[G^\theta]^G} |K_{\of,\varphi}^T(x)|\|x\|_{G^{\theta}}^{N_1} \, dx. 
    \end{align*}

  On pourra alors conclure à l'aide du théorème \ref{thm:cv-geo}. Montrons \eqref{eq:preuve Kphi=Kf}. Pour cela, on considère un point $\of\in \cgo_{p'}(F)$. En comparant les expressions \eqref{eq:varKTo} et \eqref{eq:KToS}, on voit qu'il suffit de prouver 
  \begin{align}\label{eq:KPGvsKPS}
      \int_{[G^{\theta_{p'}}]} \sum_{w\in \,_PW_{\theta_{p'}}} \sum_{\delta_2\in P_w^{\theta_{p'}}(F)  \back G^{\theta_{p'}}(F)} K_{P,\of,f}^{1,w}(x,\delta_2 y) \, dy = K_{P,\of,\varphi}(x)
  \end{align}
pour tout $P\in \fc(M_0)$ et tout $x\in G(\AAA)$.  Or, en utilisant la décomposition 
\begin{align*}
    wP_w^{\theta_{p'}}(F)=M_P^{w\theta_{p'}}(F)N_P^{w\theta_{p'}}(F)w,
\end{align*} on voit que le membre de gauche de \eqref{eq:KPGvsKPS} est égal à 
  \begin{align*}
&    \sum_{w\in \,_PW_{\theta_{p'}}} \int_{P_w^{\theta_{p'}}(F)  \back G^{\theta_{p'}}(\AAA)} K_{P,\of,f}^{1,w}(x,y) \, dy \\ 
  &  = \sum_{w\in \,_PW_{\theta_{p'}}} \int_{P_w^{\theta_{p'}}(F)  \back G^{\theta_{p'}}(\AAA)} \sum_{\gamma\in (M_P\cap (G_\of w^{-1}))(F)} \int_{N_P(\AAA)} f(x^{-1}\gamma nwy) \, dndy\\
&=    \sum_{w\in \,_PW_{\theta_{p'}}} \sum_{\gamma \in M_P(F)/M_P^{w\theta_{p'}}(F), \gamma w\in G_\of(F)} \int_{N_P(\AAA)/N_P^{w\theta_{p'}}(F)} \int_{G^{\theta_{p'}}(\AAA)}f(x^{-1}\gamma nwy)\, dydn\\
    &=    \sum_{w\in \,_PW_{\theta_{p'}}} \sum_{\gamma \in M_P(F)/M_P^{w\theta_{p'}}(F), \gamma w\in G_\of(F)} \int_{N_P(\AAA)/N_P^{w\theta_{p'}}(F)} \varphi(\rho_{p'}(x^{-1}\gamma nw))\, dn.
  \end{align*}
Comme $\vol([N_P^{w\theta_{p'}}])=1$, l'expression ci-dessus est bien égale à   $K_{P,\of,\varphi}(x)$, cf.   \eqref{eq:varKPS}. 
  \end{preuve}
\end{paragr}

\begin{paragr}[Distributions sur l'espace symétrique.]\label{S:dist sur S} --- 
  Soit $\eta: G(\AAA) \to \CC^\times$ un caractère unitaire comme dans \S \ref{S:JchiT}. 
Avec le théorème \ref{thm:cvKphi}, on peut définir, pour tout $T\in T_0+\overline{\ago_0^+}$ et tout $\of\in\cgo_{p'}(F)$,
    \begin{align}\label{eq:def-JoTphi}
      J_\of^T(\eta,\varphi) = \int_{[G^\theta]^G} K_{\of,\varphi}^T(x) \eta(x) \, dx. 
    \end{align}
    Il résulte du théorème \ref{thm:cvKphi} que l'application
    \begin{align*}
      \varphi \mapsto \sum_{\of\in \cgo_{p'}(F)}   J_\of^T(\eta,\varphi)
    \end{align*}
  est définie par une somme absolument convergente et qu'elle est continue sur $\Sc(S_{p'}(\AAA))$.

  Par l'égalité \eqref{eq:preuve Kphi=Kf} et \S \ref{S:dist geo}, l'application $T\mapsto J^{T}_\of(\eta,\varphi)$ coïncide dans un certain cône avec une fonction polynôme-exponentielle en $T$. On définit alors $J_\of(\eta,\varphi)$ comme le terme constant de cette fonction de $T$. Via l'application ouverte $\Sc(G(\AAA))\to \Sc(S_{p'}(\AAA))$, on obtient ainsi une distribution  $J_\of(\eta)$ continue sur $\Sc(S_{p'}(\AAA))$. 
\end{paragr}

\begin{paragr}[Espace symétrique infinitésimal.] --- L'algèbre simple centrale $\End_D(V)$ munie du crochet de Lie usuel s'identifie à l'algèbre de Lie $\ggo$ de $G$. Le centralisateur de $\theta$ dans $\End_D(V)$ s'identifie à l'algèbre de Lie $\ggo^\theta$ de $G^\theta$. Pour tout sous-espace $W\subset \End_D(V)$, soit $W^\theta$ l'intersection de $W$ avec le centralisateur de $\theta$. Soit $\sgo$ l'espace tangent en $1$ de $S$. On le considérera toujours comme un sous-espace de $\End_D(V)$ via l'identification
  \begin{align*}
      \sgo=\{X\in \End_D(V) \mid X+\theta X\theta=0\}.
  \end{align*}
Alors $G^\theta$ agit à gauche sur $\sgo$ par restriction de l'action adjointe $\Ad$ de $G=GL_D(V)$ sur $\End_D(V)$. Cette action laisse stable le groupe des inversibles qui s'identifie à $G$. 
\end{paragr}

\begin{paragr}
Soit $P\subset Q$ des sous-groupes paraboliques semi-standard de $G$. On note $\mgo_P$, $\ngo_P^Q$, $\ov{\ngo}_P^Q$ les algèbres de Lie respectives de $M_P$, $N_P^Q$ et $\ov{N}_P^Q$. 

  Soit $\langle\cdot, \cdot\rangle$ la forme bilinéaire symétrique, non dégénérée et $G$-invariante sur $\End_D(V)$ définie par 
  \begin{align*}
      \langle X, Y\rangle=\Trd(XY)
  \end{align*}
pour tous $X, Y\in \End_D(V)$. Comme elle est invariante par $\Ad(\theta)$, on voit que $\ggo=\ggo^\theta\oplus\sgo$ est une somme directe orthogonale pour $\langle\cdot, \cdot\rangle$ et que la restriction de $\langle\cdot, \cdot\rangle$ à $\ggo^\theta$ ou $\sgo$ est non dégénérée. 

  Si $F$ est un corps de nombres, pour tout $F$-espace vectoriel $W$, on met sur $W(\AAA)$ la mesure de Haar qui donne le volume $1$ au quotient $W(F)\back W(\AAA)$ muni de la mesure quotient par la mesure de comptage. Dans ce cas, on fixe aussi un caractère additif, continu et non-trivial 
  \begin{align*}
      \psi : F\back\AAA \to \CC^\times. 
  \end{align*}
\end{paragr}

\begin{paragr}
  Pour $P\in\fc(M_0)$ on note $\rho_{\ngo_P\cap\sgo}$ l'élément de $\ago_{P^\theta}^\ast$ tel que la mesure $dn$ sur $(\ngo_P\cap\sgo)(\AAA)$ vérifie 
  \begin{align*}
  d(hnh^{-1})=\exp(\bg 2\rho_{\ngo_P\cap\sgo}, H_{P^\theta}(h)\bd) \, dn
  \end{align*}    
pour tout $h\in M_{P^\theta}(\AAA)$. Il est évident qu'on a 
\begin{align*}
    (\rho_{\ngo_P\cap\sgo})_{|\ago_P}=\rho_P^G-(\rho_{P^\theta}^{G^\theta})_{|\ago_P} 
\end{align*}
où l'indice $|\ago_P$ désigne la restriction à $\ago_P$. 
  
  Soit $P\in\fc(M_0)$. Soit $N$ un sous-groupe de $G^\theta$, resp. un sous-espace de $\ggo$, invariant sous $\Int(A_P)$, resp. sous $\Ad(A_P)$, et muni d'une mesure $dn$. S'il existe un élément $\rho$ de $\ago_P^\ast$ tel que 
  \begin{align*}
  d(hnh^{-1})=\exp(\bg 2\rho, H_P(h)\bd)dn 
  \end{align*} 
pour tout $h\in A_P^\infty$, on note $\rho_N=\rho$. 
\end{paragr}

\begin{paragr}[Espaces fibrés homogènes.] ---
  Soit $H$ un groupe algébrique, $H'$ un sous-groupe algébrique de $H$, et $X$ une variété quasi-projective munie d'une action à gauche de $H'$. On note $H\times^{H'}X$ le quotient du produit $H\times X$ par l'action 
  \begin{align*}
      h' \cdot (h,x)=(h{h'}^{-1},h'\cdot x). 
  \end{align*}
C'est une variété algébrique (cf. \cite[théorème 4.19]{PV}). 
\end{paragr}

\begin{paragr}[Un lemme d'Arthur pour les involutions intérieures.] --- 
  Pour tout $\gamma\in G$ et tout sous-groupe algébrique $H$ de $G$, on note $H_\gamma$ le centralisateur de $\gamma$ dans $H$. Soit $P=MN$ un sous-groupe parabolique semi-standard muni de sa décomposition de Levi. Soit $\gamma\in M\cap S_{p'}$ et $\gamma=\gamma_s\gamma_u$ sa décomposition de Jordan, où $\gamma_s$ est semi-simple et $\gamma_u$ est unipotent au sens usuel. D'après le lemme \ref{lem:JRlem4.1}, on a $\gamma_s\in M\cap S_{p'}$ et $\gamma_u\in M\cap S^\circ$. 
  
  \begin{lemme}\label{lem:thetacent}
  \begin{enumerate}
  \item On a $\Int(\theta)(G_\gamma)= G_\gamma$ et $\Int(\theta)(N_\gamma)=N_\gamma$. 
  \item Le morphisme de symétrisation $\rho$, cf. \eqref{eq:app-sym}, induit un isomorphisme de $F$-variétés algébriques 
    \begin{align*}
      N_\gamma/N_\gamma^\theta\to N_\gamma\cap S^\circ. 
  \end{align*}
  \end{enumerate}
  \end{lemme}  
  
  \begin{preuve}
    1. On a $\Int(\theta)(\gamma)=\gamma^{-1}$ car $\gamma\in S$. Si $x\in G_\gamma$, alors $\Int(\theta)(x)\in G_{\Int(\theta)(\gamma)}=G_{\gamma^{-1}}=G_\gamma$.  Comme $\theta\in M$, on a aussi $\Int(\theta)(N)\subset N$ d'où $\Int(\theta)(N_\gamma)=N_\gamma$.
    
    2. On utilise l'exponentielle comme dans la preuve du lemme \ref{lem:JRlem4.1}. Un  élément $n\in N$ s'écrit $n=\exp(X)$ avec $X\in \ngo$. On a $n\in N_\ga\cap S^\circ$ si et seulement si $X\in \ngo_\ga\cap \sgo$ (où $\ngo_\ga$ est l'algèbre de Lie de $N_\ga$) et alors  $n=\rho(\exp(X/2))$ avec $\exp(X/2)\in N_\ga$.  L'assertion 2 s'en déduit aisément.
  \end{preuve}

  \begin{lemme}\label{lem:arthur}
L'application 
$$(n,x) \mapsto \Int(n)(x)$$
induit un isomorphisme  de  $N^\theta\times^{ N^\theta_{\gamma_s}} ((\gamma N_{\gamma_s}) \cap S_{p'} )$ sur $(\gamma N )\cap S_{p'}$, où $N_{\gamma_s}^\theta$ agit  par conjugaison sur $(\gamma N_{\gamma_s}) \cap S_{p'}$. 
  \end{lemme}
  
    \begin{preuve}
Pour tout $n\in N$ et tout $x\in \gamma N_{\gamma_s}$,   la partie semi-simple de $\Int(n)(x)$ est $\Int(n)(\gamma_s)$. Pour $n, n'\in N$ et $x, x'\in \gamma N_{\gamma_s}$, si $\Int(n)(x)= \Int(n')(x')$ alors on a  $\Int(n)(\gamma_s)=\Int(n')(\gamma_s)$ et donc $n'\in n N_{\gamma_s}$. Il s'ensuit que l'application en question induit  un isomorphisme   $N \times^{ N_{\gamma_s}}(\gamma N_{\gamma_s})  $  sur $\gamma N $, cf.  également \cite[lemme 2.1]{ar1}. On en déduit clairement qu'on a une immersion de $N^\theta\times^{ N^\theta_{\gamma_s}}((\gamma N_{\gamma_s}) \cap S_{p'} )$  dans $(\gamma N )\cap S_{p'}$. 
Il suffit de voir la surjectivité. Soit $y\in (\gamma N )\cap S_{p'}$. Il existe $n\in N$ et $x\in  \gamma N_{\gamma_s}$ tel que $\Int(n)(x)=y$. On a donc $\theta y^{-1}\theta=y$. En prenant les parties semi-simples et en utilisant le fait que $\theta \gamma_s^{-1} \theta=\gamma_s$, on voit que $\Int(\theta n\theta)(\gamma_s)= \Int(n)( \gamma_s)$. Par conséquent $\rho(n^{-1})\in N_{\gamma_s}\cap S^\circ$. Par l'assertion 2 du lemme \ref{lem:thetacent}  appliquée à $\gamma_s\in M\cap S_{p'}$, il existe $u\in N_{\gamma_s}$ tel que $\rho(n^{-1})=\rho(u)$ et donc $nu\in  N^\theta$. De plus, on a $y=\Int(nu) (\Int(u^{-1})x)$ de sorte que on  a $\Int(u^{-1})x\in \gamma N_{\gamma_s}$ et $ \Int(u^{-1})x\in \Int(nu)^{-1}(S_{p'})=S_{p'}$.
  \end{preuve}
\end{paragr}

\begin{paragr}\label{S:sigma'}
  Soit $P=MN$ un sous-groupe parabolique semi-standard et  $\gamma\in (M\cap S_{p'})(F)$.  Soit $\sigma'$ l'anti-involution de $N$ donnée par $\sigma'(n)=\Int(\theta\gamma)(n^{-1})$. Alors $\sigma'$ laisse $N_{\gamma_s}$ invariant d'après l'assertion 1 du lemme \ref{lem:thetacent}. Il s'ensuit que l'application 
  \begin{align*}
      n\mapsto n\sigma'(n)
  \end{align*}
induit un isomorphisme de $N_{\gamma_s}/N_{\gamma_s}^{\theta\gamma}$ sur $N_{\gamma_s}^{\sigma'}$. À ce titre $N_{\gamma_s}^{\sigma'}(\AAA)$ hérite d'une mesure invariante normalisée, à savoir celle sur le quotient  $N_{\gamma_s}(\AAA)/N_{\gamma_s}^{\theta\gamma}(\AAA)$ . Comme $(\gamma N_{\gamma_s})\cap S_{p'}=\gamma N_{\gamma_s}^{\sigma'}$ on obtient une mesure sur $((\gamma N_{\gamma_s})\cap S_{p'})(\AAA)$. On peut raisonner autrement : on identifie via l'exponentielle 
$N_{\gamma_s}^{\sigma'}$ à  $\ngo_{\gamma_s}^{\sigma'}$ et on transporte la mesure additive, où l'on note encore $\sigma'$ l'anti-involution de $\ngo_{\gamma_s}$ donnée par $\sigma'(U)=\Ad(\theta\gamma)(-U)$. 
\end{paragr}

\begin{paragr}\label{S:tlK}
  Soit $\varphi\in \Sc(S_{p'}(\AAA))$. Soit $P\in\fc(M_0)$ avec $P=MN$ sa décomposition de Levi semi-standard. Soit $\of\in\cgo_{p'}(F)$. On considère pour $x\in G(\AAA)$ 
  \begin{align}\label{eq:KtildePo}
  \tlK_{P,\of,\varphi}(x)=\sum_{\gamma\in (M\cap S_{p',\of})(F)}\sum_{\nu\in N^\theta_{\gamma_s}(F)\back N^\theta(F)} \int_{((\gamma N_{\gamma_s})\cap S_{p'})(\AAA)} \varphi(\Int_\theta(\nu x)^{-1}(y))\,dy. 
  \end{align}
On obtient ainsi une fonction sur $P^\theta(F)\back G(\AAA)$. 

Pour $T\in T_0+\overline{\ago_0^+}$ un paramètre de troncature, on définit 
  \begin{align}\label{eq:tlKToS}
  \tlK_{\of,\varphi}^T(x)=\sum_{P\in\fc(P_0^\theta)} \eps_P^G \sum_{\delta\in P^\theta(F)  \back G^\theta(F)} \hat\tau_P(H_P(\delta x)-T_P) \tlK_{P,\of,\varphi}(\delta x)
  \end{align}
pour $x\in G(\AAA)$.  À l'aide de \cite[lemme 5.1]{ar1}, on voit que la somme sur $\delta$ dans la dernière expression est finie. 
  
    \begin{theoreme}\label{thm:cv-tlK}
  Soit $\of\in\cgo_{p'}(F)$ et $\varphi\in C_c^\infty(S_{p'}(\AAA))$. Pour tout $T$ suffisamment positif avec $\|T^G\|$ assez grand (par rapport au support de $\varphi$) on a 
    \begin{align*}
      \int_{[G^\theta]^G} |\tlK_{\of,\varphi}^T(x)| \, dx <\infty 
    \end{align*}
et 
    \begin{align}\label{eq:tlK=K}
      J_\of^T(\eta,\varphi)=\int_{[G^\theta]^G} \tlK_{\of,\varphi}^T(x) \eta(x) \, dx
    \end{align}
où le membre de gauche est défini par \eqref{eq:def-JoTphi}.  
  \end{theoreme}
  
  La preuve de ce théorème occupe les paragraphes suivants. 
\end{paragr}

\begin{paragr}[Majoration du noyau auxiliaire.]\label{S:maj-aux} ---
  Soit $x\in G^\theta(\AAA)$. On injecte l'identité \eqref{eq:varpartition-H} appliquée à $Q=P$ et l'élément $\delta x\in G^\theta(\AAA)$ dans la définition de $\tlK_{\of,\varphi}^T$. On obtient 
  \begin{align*}
  \tlK_{\of,\varphi}^T(x)=\sum_{P\in\fc(P_0^\theta)} \eps_P^G \sum_{\delta\in P^\theta(F)  \back G^\theta(F)} \hat\tau_P(H_P(\delta x)-T_P) \tlK_{P,\of,\varphi}(\delta x)\times \\ 
  \left(\sum_{P_1\in \fc^P(P_0^\theta)} \sum_{\delta_1 \in P_1^\theta(F)\back P^\theta(F)}F^{P_1}(\delta_1 \delta x,T) \tau_{P_1}^P(H_{P_1}(\delta_1 \delta x)-T_{P_1})\right). 
  \end{align*}
Comme les fonctions $H_P$ et $\tlK_{P,\of,\varphi}(x)$ sont invariantes à gauche par $P^\theta(F)$, on obtient que $\tlK_{\of,\varphi}^T(x)$ est égal à 
  \begin{align*}
  \sum_{P_0^\theta\subset P_1\subset P} \eps_P^G \sum_{\delta\in P_1^\theta(F)  \back G^\theta(F)} F^{P_1}(\delta x,T) \tau_{P_1}^P(H_{P_1}(\delta x)-T_{P_1}) \hat\tau_P(H_P(\delta x)-T_P) \tlK_{P,\of,\varphi}(\delta x) 
  \end{align*}
qui vaut 
  \begin{align*}
  \sum_{P_0^\theta\subset P_1\subset P_2} \sum_{\delta\in P_1^\theta(F)  \back G^\theta(F)} F^{P_1}(\delta x,T) \sigma_{P_1}^{P_2}(H_{P_1}(\delta x)-T_{P_1}) \tlK_{P_1,P_2,\of,\varphi}(\delta x) 
  \end{align*}
où l'on pose pour $g\in G(\AAA)$ et $M_0\subset P_1\subset P_2$ 
  \begin{align*}
  \tlK_{P_1,P_2,\of,\varphi}(g)=\sum_{P_1\subset P\subset P_2} \eps_P^G \tlK_{P,\of,\varphi}(g). 
  \end{align*}
On en déduit que 
    \begin{align*}
      \int_{[G^\theta]^G} |\tlK_{\of,\varphi}^T(x)| \, dx
    \end{align*}
se majore par la somme sur $P_0^\theta\subset P_1\subset P_2$ de 
    \begin{align*}
      \int_{P_1^\theta(F)\back (G^\theta(\AAA)\cap G(\AAA)^1)} F^{P_1}(x,T) \sigma_{P_1}^{P_2}(H_{P_1}(x)-T_{P_1}) | \tlK_{P_1,P_2,\of,\varphi}(x) | \, dx. 
    \end{align*}
Il suffit donc de majorer les termes correspondant à $P_1\subset P_2$ fixés. Si $P_1=P_2$ alors on a $\sigma_{P_1}^{P_2}=0$ sauf si $P_1=P_2=G$. Dans ce cas, on doit majorer 
      \begin{align*}
      \int_{[G^\theta]^G} F^G(x,T)  | \tlK_{G,\of,\varphi}(x) | \, dx. 
    \end{align*}
D'après le lemme \ref{lem:relsiegel}, on sait que la fonction $F^G(\cdot,T)$ sur $[G^\theta]^G$ est à support compact. Il s'agit donc de majorer l'intégrale d'une fonction continue sur un ensemble compact dont la convergence est claire. Dans toute la suite, on fixe $P_1\subsetneq P_2$. 

  \begin{lemme}\label{lem:tlK-res}
  Pour tout $P_1\subset P\subset P_2$, tout $T$ suffisamment positif avec $\|T^G\|$ assez grand (par rapport au support de $\varphi$) et tout $x\in P_1^\theta(F)\back G^\theta(\AAA)\cap G(\AAA)^1$ tel que 
    \begin{align*}
      F^{P_1}(x,T) \sigma_{P_1}^{P_2}(H_{P_1}(x)-T_{P_1})\neq 0, 
    \end{align*}
on a 
  \begin{align*}
    \tlK_{P,\of,\varphi}(x)=\sum_{\gamma\in (M_P\cap P_1\cap S_{p',\of})(F)} \sum_{\nu\in N^\theta_{P, \gamma_s}(F)\back N_P^\theta(F)} \int_{((\gamma N_{P,\gamma_s})\cap S_{p'})(\AAA)} \varphi(\Int(\nu x)^{-1}(y))\,dy. 
  \end{align*}
  \end{lemme}

  \begin{preuve}
  En utilisant la décomposition d'Iwasawa 
  \begin{align*}
    G^\theta(\AAA)\cap G(\AAA)^1=N_{P_1}^\theta(\AAA) \cdot (M_{P_1}^\theta(\AAA)\cap G(\AAA)^1) \cdot K^\theta
  \end{align*}
et le lemme \ref{lem:relsiegel} appliqué à $P_1$, on choisit un représentant de $x$ dans $G^\theta(\AAA)\cap G(\AAA)^1$ de la forme $n^* n_* m_0 a k$ où $a\in A_B^{P_1,\infty}( T_-, T_B)\cap G(\AAA)^1$ pour un certain $B\in\pc^{P_1}(P_0^\theta)$, $k\in K^\theta$ et $n^*, n_*, m_0$ restent respectivement dans certains compacts indépendants de $T$ de $N_{P_2}^\theta(\AAA), (N_{P_0^\theta}\cap M_{P_2}^\theta)(\AAA), M_0(\AAA)^1$. Comme $\sigma_{P_1}^{P_2}(H_{P_1}(a)-T_{P_1})\neq 0$, d'après \cite[p.944]{ar1}, $a^{-1}n_* m_0 a$ reste dans un compact indépendant de $T$. Soit $\gamma\in (M_P\cap S_{p',\of})(F)$, $\nu\in N_P^\theta(F)$ et $y\in ((\gamma N_{P,\gamma_s})\cap S_{p'})(\AAA)$ tels que $\varphi(\Int(\nu x)^{-1}(y))\neq0$. Notons $\Int(\nu n^* a)^{-1}(y)\in (a^{-1}\gamma a) N_P(\AAA)$. On trouve alors que $a^{-1}\gamma a$ appartient à un compact de $M_P(\AAA)^1\cap S_{p'}(\AAA)$ dépendant du support de $\varphi$. Mais cela implique que $\gamma\in (M_P\cap P_1)(F)$ pour $T$ suffisamment positif avec $\|T^G\|$ assez grand (cf. {\it{loc. cit.}}). 
  \end{preuve}
  
  D'après les lemmes \ref{lem:parab-o} et \ref{lem:arthur}, on a 
  \begin{align}
\label{eq:parab-o}    (M_P\cap P_1\cap S_{p',\of})(F)&=\bigsqcup\limits_{\gamma\in (M_1\cap S_{p',\of})(F)} ((\gamma N_1^P)\cap S_{p'})(F) \\ \nonumber
&=\bigsqcup\limits_{\gamma\in (M_1\cap S_{p',\of})(F)} \bigsqcup\limits_{\delta\in N_{1,\gamma_s}^{P,\theta}(F)\back N_1^{P,\theta}(F)} \Int(\delta^{-1})((\gamma N_{1,\gamma_s}^P)\cap S_{p'})(F). 
  \end{align}
Par le lemme \ref{lem:tlK-res}, on obtient que $\tlK_{P,\of,\varphi}(x)$ est égal à la somme sur $\gamma\in (M_1\cap S_{p',\of})(F)$ de 
  \begin{align*}
  \sum_{\delta\in N_{1,\gamma_s}^{P,\theta}(F)\back N_1^{P,\theta}(F)} \sum_{\xi\in((\gamma N_{1,\gamma_s}^P)\cap S_{p'})(F)} \sum_{\nu\in N^\theta_{P, \delta^{-1}\gamma_s\delta}(F)\back N_P^\theta(F)} \int_{((\xi N_{P,\gamma_s})\cap S_{p'})(\AAA)} \varphi(\Int(\delta\nu x)^{-1}(y))\,dy \\ 
=\sum_{\nu\in N_{1,\gamma_s}^\theta(F)\back N_1^\theta(F)} \sum_{\xi\in((\gamma N_{1,\gamma_s}^P)\cap S_{p'})(F)} \int_{((\xi N_{P,\gamma_s})\cap S_{p'})(\AAA)} \varphi(\Int(\nu x)^{-1}(y))\,dy. 
  \end{align*}
Les éléments de $(\gamma N_{1,\gamma_s}^P)\cap S_{p'}$ s'écrivent $\gamma\exp(U)$ avec $U\in\ngo_{1,\gamma_s}^{P,\sigma'}$ où $\sigma'$ est l'anti-involution définie dans le paragraphe \ref{S:sigma'}. Fixons un tel $U$ et posons $\xi=\gamma\exp(U)$. Tout élément de $(\xi N_{P,\gamma_s})\cap S_{p'}$ s'écrit $\xi n $ avec $n\in N_{P,\gamma_s}$ tel que $\exp(U)n\in N_{1,\gamma_s}^{\sigma'}$ c'est-à-dire $n'=\Int(\exp(U/2))(n)\in N_{P,\gamma_s}^{\sigma'}$. On en déduit que l'application 
  \begin{align*}
    (U,n')\mapsto \exp(U)n=\exp(U/2) n'\exp(U/2)
  \end{align*}
induit un isomorphisme de $\ngo_{1,\gamma_s}^{P,\sigma'}\times N_{P,\gamma_s}^{\sigma'}$ sur $N_{1,\gamma_s}^{\sigma'}$. 
Par la formule de Poisson, on a donc 
  \begin{align*}
  \sum_{\xi\in((\gamma N_{1,\gamma_s}^P)\cap S_{p'})(F)} \int_{((\xi N_{P,\gamma_s})\cap S_{p'})(\AAA)} \varphi(\Int(\nu x)^{-1}(y))\,dy \\ 
=\sum_{U\in\ngo_{1,\gamma_s}^{P,\sigma'}(F)} \int_{N_{P,\gamma_s}^{\sigma'}(\AAA)} \varphi(\Int(\nu x)^{-1}(\gamma\exp(U/2) n'\exp(U/2)))\,dn' \\ 
=\sum_{V\in \ov{\ngo}_{1,\gamma_s}^{P,\sigma'}(F)} \int_{\ngo_{1,\gamma_s}^{P,\sigma'}(\AAA)} \int_{N_{P,\gamma_s}^{\sigma'}(\AAA)} \varphi(\Int(\nu x)^{-1}(\gamma\exp(U/2) n'\exp(U/2))) \psi(\langle U,V\rangle)\,dn' dU \\ 
=\sum_{V\in \ov{\ngo}_{1,\gamma_s}^{P,\sigma'}(F)} \int_{\ngo_{1,\gamma_s}^{\sigma'}(\AAA)} \varphi(\Int(\nu x)^{-1}(\gamma\exp(U))) \psi(\langle U,V\rangle) \, dU. 
  \end{align*}
Soit $(\ov{\ngo}_{1,\gamma_s}^{2,\sigma'})'$ l'ensemble, éventuellement vide, des éléments de $\ov{\ngo}_{1,\gamma_s}^{2,\sigma'}$ qui ne appartiennent à aucun $\ov{\ngo}_{1,\gamma_s}^{Q,\sigma'}$ pour $P_1\subset Q\subsetneq P_2$. On note  $[G^\theta]_{P_1^\theta}^G=N_1^\theta(\AAA)M_1^\theta(F)\back G^\theta(\AAA)\cap G(\AAA)^1$. En utilisant \cite[proposition 1.1]{ar1}, on est ramené à majorer 
    \begin{align}\label{eq:maj-tlK12}
      \int_{[G^\theta]_{P_1^\theta}^G} \exp(\langle -2\rho_{P_1^\theta}^{G^\theta}, H_{P_1^\theta}(x)\rangle) F^{P_1}(x,T) \sigma_{P_1}^{P_2}(H_{P_1}(x)-T_{P_1}) \sum_{\gamma\in (M_1\cap S_{p',\of})(F)} \\ \nonumber 
      \int_{N_{1,\gamma_s}^\theta(F)\back N_1^\theta(\AAA)} \sum_{V\in (\ov{\ngo}_{1,\gamma_s}^{2,\sigma'})'(F)} \left|\int_{\ngo_{1,\gamma_s}^{\sigma'}(\AAA)} \varphi(\Int(n x)^{-1}(\gamma\exp(U))) \psi(\langle U,V\rangle) \, dU\right| dn dx. 
    \end{align}
Pour tous $\gamma\in (M_1\cap S_{p',\of})(F)$ et $V\in (\ov{\ngo}_{1,\gamma_s}^{2,\sigma'})'(F)$, la fonction de $g\in G^\theta(\AAA)$ 
    \begin{align*}
      \int_{\ngo_{1,\gamma_s}^{\sigma'}(\AAA)} \varphi(\Int(g^{-1})(\gamma\exp(U))) \psi(\langle U,V\rangle) \, dU
    \end{align*}
est invariante à gauche par $N_{2,\gamma_s}^\theta(\AAA)$. Puisque $\vol([N_{2,\gamma_s}^\theta])=1$, on peut alors écrire $n=n_2 n_1$ où $n_2\in [N_{1,\gamma_s}^{2,\theta}]$ et $n_1\in N_{1,\gamma_s}^\theta(\AAA)\back N_1^\theta(\AAA)$. On écrit également $x=mak$ selon la décomposition d'Iwasawa où $m\in M_1^\theta(F)\back (M_1^\theta(\AAA)\cap M_1(\AAA)^1)$, $a\in A_1^{G,\infty}$ et $k\in K^\theta$. Sous la condition $F^{P_1}(m,T)=1$, d'après le lemme \ref{lem:relsiegel} appliqué à $P_1$, $m$ reste dans un compact dépendant de $T$ de $M_1^\theta(\AAA)\cap M_1(\AAA)^1$. Comme $\varphi\in C_c^\infty(S_{p'}(F))$, la somme sur $\gamma$ est finie. Notons que $\Int(a^{-1})$ laisse $N_{1,\gamma_s}^\theta(\AAA)\back N_1^\theta(\AAA)$ invariant et que le changement de variables $n_1\mapsto an_1a^{-1}$ ajoute un facteur $\exp(\langle 2\rho_{N_1^\theta}-2\rho_{N_{1,\gamma_s}^\theta}, H_{P_1}(a)\rangle)$ à l'intégrande. Puis pour tout $\gamma$ fixé, par le lemme \ref{lem:arthur}, $n_1$ reste dans un compact dépendant du support de $\varphi$. Sous la condition $\sigma_{P_1}^{P_2}(H_{P_1}(a)-T_{P_1})\neq 0$, d'après \cite[p.944]{ar1}, $a^{-1}n_2 a$ reste dans un compact indépendant de $T$. Enfin, le changement de variables $U\mapsto aUa^{-1}$ ajoute un facteur $\exp(\langle 2\rho_{\ngo_{1,\gamma_s}^{\sigma'}}, H_{P_1}(a)\rangle)$ à l'intégrande. Il s'ensuit que \eqref{eq:maj-tlK12} se majore, à une constante multiplicative près, par la somme sur $\gamma\in (M_1\cap S_{p',\of})(F)$, la borne supérieure sur $y$ restant dans un compact de $G^\theta(\AAA)\cap G(\AAA)^1$ et l'intégrale sur $a\in A_1^{G,\infty}$ de 
    \begin{align*}
      \exp(\langle 2\rho_{\ngo_{1,\gamma_s}^{\sigma'}}-2\rho_{N_{1,\gamma_s}^\theta}, H_{P_1}(a)\rangle) \sigma_{P_1}^{P_2}(H_{P_1}(a)-T_{P_1}) \\ 
      \sum_{V\in (\ov{\ngo}_{1,\gamma_s}^{2,\sigma'})'(F)} \left|\int_{\ngo_{1,\gamma_s}^{\sigma'}(\AAA)} \varphi(y^{-1}\gamma\exp(U)y) \psi(\langle U,aVa^{-1}\rangle) \, dU\right|. 
    \end{align*}
En utilisant le mêmes arguments pour la majoration de \cite[(7.8) dans p.945]{ar1} et traitant le facteur $\exp(\langle 2\rho_{\ngo_{1,\gamma_s}^{\sigma'}}-2\rho_{N_{1,\gamma_s}^\theta}, H_{P_1}(a)\rangle)$ comme à la fin de la preuve de \cite[proposition 4.15]{li1}, on voit que cette dernière intégrale est convergente. 
\end{paragr}

\begin{paragr}[Preuve de l'égalité \eqref{eq:tlK=K}.] ---
On a déjà montré 
    \begin{align*}
      \int_{[G^\theta]^G} \tlK_{\of,\varphi}^T(x) \eta(x) \, dx \\ 
      =\sum_{P_0^\theta\subset P_1\subset P_2} \int_{P_1^\theta(F)\back (G^\theta(\AAA)\cap G(\AAA)^1)} F^{P_1}(x,T) \sigma_{P_1}^{P_2}(H_{P_1}(x)-T_{P_1}) \tlK_{P_1,P_2,\of,\varphi}(x) \eta(x) \, dx 
    \end{align*}
où  
    \begin{align*}
      \tlK_{P_1,P_2,\of,\varphi}(x)=\sum_{P_1\subset P\subset P_2} \eps_P^G \sum_{\gamma\in (M_P\cap S_{p',\of})(F)}\sum_{\nu\in N^\theta_{P,\gamma_s}(F)\back N_P^\theta(F)} \int_{((\gamma N_{P,\gamma_s})\cap S_{p'})(\AAA)} \varphi(\Int(\nu x)^{-1}(y))\,dy. 
    \end{align*}
On décompose l'intégrale sur $P_1^\theta(F)\back (G^\theta(\AAA)\cap G(\AAA)^1)$ en une intégrale double sur $x\in [G^\theta]_{P_1^\theta}^G$ et $n_1\in [N_1^\theta]$. Puis on fait passer l'intégrale sur $[N_1^\theta]$ à l'intérieur de la somme sur $P$ et sur $\gamma$. On considère donc 
    \begin{align*}
      \int_{[N_1^\theta]} \sum_{\nu\in N^\theta_{P,\gamma_s}(F)\back N_P^\theta(F)} \int_{((\gamma N_{P,\gamma_s})\cap S_{p'})(\AAA)} \varphi(\Int(\nu n_1 x)^{-1}(y))\,dydn_1. 
    \end{align*}
Comme $\vol({[N_P^\theta]})=\vol([N^\theta_{P,\gamma_s}])=1$, cette dernière expression est égale à 
    \begin{align*}
      &\int_{[N_1^\theta]} \int_{[N_P^\theta]} \sum_{\nu\in N^\theta_{P,\gamma_s}(F)\back N_P^\theta(F)} \int_{((\gamma N_{P,\gamma_s})\cap S_{p'})(\AAA)} \varphi(\Int(\nu n n_1 x)^{-1}(y))\,dydndn_1 \\ 
      =&\int_{[N_1^\theta]} \int_{N^\theta_{P,\gamma_s}(F)\back N_P^\theta(\AAA)} \int_{((\gamma N_{P,\gamma_s})\cap S_{p'})(\AAA)} \varphi(\Int(n n_1 x)^{-1}(y))\,dydndn_1 \\ 
      =&\int_{[N_1^\theta]} \int_{N^\theta_{P,\gamma_s}(\AAA)\back N_P^\theta(\AAA)} \int_{((\gamma N_{P,\gamma_s})\cap S_{p'})(\AAA)} \varphi(\Int(n n_1 x)^{-1}(y))\,dydndn_1 \\ 
      =&\int_{[N_1^\theta]} \int_{((\gamma N_P)\cap S_{p'})(\AAA)} \varphi(\Int(n_1 x)^{-1}(y))\,dydn_1 
    \end{align*}
où l'on a utilisé le lemme \ref{lem:arthur} dans la dernière égalité. On intervertit de nouveau la somme sur $P$ et sur $\gamma$ avec l'intégrale sur $[N_1^\theta]$. Par le lemme \ref{lem:maj-K12} ci-dessous, on peut recombiner cette dernière avec l'intégrale sur $[G^\theta]_{P_1^\theta}^G$. On trouve alors 
    \begin{align*}
    \sum_{P_0^\theta\subset P_1\subset P_2} \int_{P_1^\theta(F)\back (G^\theta(\AAA)\cap G(\AAA)^1)} F^{P_1}(x,T) \sigma_{P_1}^{P_2}(H_{P_1}(x)-T_{P_1}) K_{P_1,P_2,\of,\varphi}(x) \eta(x) \, dx 
    \end{align*}
où l'on pose pour $g\in G(\AAA)$ et $M_0\subset P_1\subset P_2$ 
  \begin{align*}
   K_{P_1,P_2,\of,\varphi}(g)=\sum_{P_1\subset P\subset P_2} \eps_P^G K_{P,\of,\varphi}(g). 
  \end{align*}
Avec les manipulations au début du paragraphe \ref{S:maj-aux}, on voit que c'est précisément $J_\of^T(\eta,\varphi)$. 

\begin{lemme}\label{lem:maj-K12}
Soit $P_0^\theta\subset P_1\subsetneq P_2\subset G$ des sous-groupes paraboliques. Pour tout $T$ suffisamment positif avec $\|T^G\|$ assez grand (par rapport au support de $\varphi$) on a
    \begin{align*}
    \int_{P_1^\theta(F)\back (G^\theta(\AAA)\cap G(\AAA)^1)} F^{P_1}(x,T) \sigma_{P_1}^{P_2}(H_{P_1}(x)-T_{P_1}) | K_{P_1,P_2,\of,\varphi}(x) | \, dx < \infty. 
    \end{align*}
\end{lemme}

\begin{preuve}
La preuve de cet énoncé est essentiellement identique à celle du paragraphe \ref{S:maj-aux}. Par un analogue du lemme \ref{lem:tlK-res}, on peut limiter la somme sur $\gamma$ dans 
    \begin{align*}
    K_{P_1,P_2,\of,\varphi}(x)=\sum_{P_1\subset P\subset P_2} \eps_P^G \sum_{\gamma\in (M_P\cap S_{p',\of})(F)} \int_{((\gamma N_P)\cap S_{p'})(\AAA)} \varphi(\Int(x^{-1})(y))\, dy
    \end{align*}
à $(M_P\cap P_1\cap S_{p',\of})(F)$. Il résulte de \eqref{eq:parab-o} qu'on a 
    \begin{align*}
    \sum_{\gamma\in (M_P\cap P_1\cap S_{p',\of})(F)} \int_{((\gamma N_P)\cap S_{p'})(\AAA)} \varphi(\Int(x^{-1})(y))\, dy \\ 
=\sum_{\gamma\in (M_1\cap S_{p',\of})(F)} \sum_{\xi\in((\gamma N_1^P)\cap S_{p'})(F)} \int_{((\xi N_P)\cap S_{p'})(\AAA)} \varphi(\Int(x^{-1})(y))\, dy. 
    \end{align*}
Les éléments de $(\gamma N_1^P)\cap S_{p'}$ s'écrivent $\gamma\exp(U)$ avec $U\in\ngo_1^{P,\sigma'}$ où $\sigma'$ est l'anti-involution définie dans le paragraphe \ref{S:sigma'}. Fixons un tel $U$ et posons $\xi=\gamma\exp(U)$. Tout élément de $(\xi N_P)\cap S_{p'}$ s'écrit $\xi n $ avec $n\in N_P$ tel que $\exp(U)n\in N_1^{\sigma'}$ c'est-à-dire $n'=\Int(\exp(U/2))(n)\in N_P^{\sigma'}$. On en déduit que l'application 
  \begin{align*}
    (U,n')\mapsto \exp(U)n=\exp(U/2) n'\exp(U/2)
  \end{align*}
induit un isomorphisme de $\ngo_1^{P,\sigma'}\times N_P^{\sigma'}$ sur $N_1^{\sigma'}$. Par la formule de Poisson, on a donc 
  \begin{align*}
  \sum_{\xi\in((\gamma N_1^P)\cap S_{p'})(F)} \int_{((\xi N_P)\cap S_{p'})(\AAA)} \varphi(\Int(x^{-1})(y))\,dy \\ 
=\sum_{U\in\ngo_1^{P,\sigma'}(F)} \int_{N_P^{\sigma'}(\AAA)} \varphi(\Int(x^{-1})(\gamma\exp(U/2) n'\exp(U/2)))\,dn' \\ 
=\sum_{V\in \ov{\ngo}_1^{P,\sigma'}(F)} \int_{\ngo_1^{P,\sigma'}(\AAA)} \int_{N_P^{\sigma'}(\AAA)} \varphi(\Int(x^{-1})(\gamma\exp(U/2) n'\exp(U/2))) \psi(\langle U,V\rangle)\,dn' dU \\ 
=\sum_{V\in \ov{\ngo}_1^{P,\sigma'}(F)} \int_{\ngo_1^{\sigma'}(\AAA)} \varphi(\Int(x^{-1})(\gamma\exp(U))) \psi(\langle U,V\rangle) \, dU. 
  \end{align*}
Soit $(\ov{\ngo}_1^{2,\sigma'})'$ l'ensemble, éventuellement vide, des éléments de $\ov{\ngo}_1^{2,\sigma'}$ qui ne appartiennent à aucun $\ov{\ngo}_1^{Q,\sigma'}$ pour $P_1\subset Q\subsetneq P_2$. En utilisant \cite[proposition 1.1]{ar1}, on est ramené à majorer 
    \begin{align}\label{eq:maj-K12}
      \int_{P_1^\theta(F)\back (G^\theta(\AAA)\cap G(\AAA)^1)} F^{P_1}(x,T) \sigma_{P_1}^{P_2}(H_{P_1}(x)-T_{P_1}) \sum_{\gamma\in (M_1\cap S_{p',\of})(F)} \\ \nonumber \sum_{V\in (\ov{\ngo}_1^{2,\sigma'})'(F)} \left| \int_{\ngo_1^{\sigma'}(\AAA)} \varphi(\Int(x^{-1})(\gamma\exp(U))) \psi(\langle U,V\rangle) \, dU \right| dx. 
    \end{align}
Soit $x=n^* n_* m a k$ où $n^*\in [N_2^\theta]$, $n_*\in [N_1^{2,\theta}]$, $m\in M_1^\theta(F)\back (M_1^\theta(\AAA)\cap M_1(\AAA)^1)$, $a\in A_1^{G,\infty}$ et $k\in K^\theta$. Ce changement de variables ajoute un facteur $\exp(\langle -2\rho_{P_1^\theta}^{G^\theta},H_{P_1^\theta}(ma) \rangle)$ à l'intégrande. Notons que $n^*$ s'absorbe dans l'intégrale sur $U$. Le changement de variables $U\mapsto aUa^{-1}$ ajoute un facteur $\exp(\langle 2\rho_{\ngo_1^{\sigma'}}, H_{P_1}(a)\rangle)$ à l'intégrande. D'après la discussion du paragraphe \ref{S:maj-aux}, on observe que \eqref{eq:maj-K12} est majorée, à une constante multiplicative près, par la somme finie sur $\gamma\in (M_1\cap S_{p',\of})(F)$, la borne supérieure sur $y$ restant dans un compact de $G^\theta(\AAA)\cap G(\AAA)^1$ et l'intégrale sur $a\in A_1^{G,\infty}$ de 
    \begin{align*}
      \exp(\langle 2\rho_{\ngo_1^{\sigma'}}-2\rho_{P_1^\theta}^{G^\theta}, H_{P_1}(a)\rangle) \sigma_{P_1}^{P_2}(H_{P_1}(a)-T_{P_1}) \\ 
      \sum_{V\in (\ov{\ngo}_1^{2,\sigma'})'(F)} \left|\int_{\ngo_1^{\sigma'}(\AAA)} \varphi(y^{-1}\gamma\exp(U)y) \psi(\langle U,aVa^{-1}\rangle) \, dU\right|. 
    \end{align*}
Comme à la fin du paragraphe \ref{S:maj-aux}, on peut montrer que cette intégrale est convergente. 
\end{preuve}
\end{paragr}

\subsection{Descente semi-simple}\label{ssec:desc-ss}

\begin{paragr}[Descendants.] ---
  Pour tout élément semi-simple $\sigma\in S(F)$, on a $\Int(\theta)(G_\sigma) \subset G_\sigma$, cf. lemme \ref{lem:thetacent} et on pose $G_\sigma^\theta= G_\sigma\cap G^\theta$. On obtient donc une paire symétrique $(G_\sigma,G_\sigma^\theta)$. On appelle  descendant de $(G,G^\theta)$ tout  triplet $(G_\sigma,G_\sigma^\theta,\Int(\theta)|_{G_\sigma})$ ainsi obtenu. On peut décrire explicitement les descendants: 
  
  \begin{proposition}\label{prop:descendant}
  Tout descendant de $(G, G^\theta, \Int(\theta))$ est un produit des triplets d'un des  types suivants : 
   \begin{enumerate}
  \item $(\Res_{L/F} GL_{r,D'}\times \Res_{L/F} GL_{r,D'}, \Delta \Res_{L/F} GL_{r,D'}, \delta)$ où $L/F$ est une extension finie de corps,   $D'$  une algèbre à division centrale sur $L$ et $\delta$ l'involution donnée par $(x,y)=(y,x)$ ; 
  \item $(\Res_{K/F} GL_{r,D'\otimes_L K}, \Res_{L/F} GL_{r,D'}, \gamma)$ où $L/F$ est une extension finie de corps,   $K/L$ une extension quadratique, $D'$  une algèbre à division centrale sur $L$ et $\gamma$ est l'involution galoisienne de $\Res_{L/F} GL_{r,D'}$ correspondant à l'élément non trivial de  $\Gal(K/L)$ ; 
  \item $(GL_{r+t,D}, GL_{r,D}\times GL_{t,D}, \Int(\theta_{r,t}))$ où $\theta_{r,t}=\begin{pmatrix}
  I_r &  0\\ 0 & -I_t
\end{pmatrix}$. 
  \end{enumerate}
\end{proposition}

\begin{remarque}
    Pour le type 2, $D'\otimes_L K$ n'est pas nécessairement une algèbre à division. 
\end{remarque}

\begin{preuve}
C'est une généralisation immédiate de \cite[proposition 4.3]{JR2} ou \cite[théorème 7.7.3]{AizGour} dans le cas $D=F$ et \cite[proposition 4.1]{Zha2} dans le cas $p=q$. On reprend l'argument ici pour la commodité des lecteurs et des auteurs. En utilisant la proposition \ref{prop:elementssr}, sans perte de généralité, on peut et on va supposer que $\sigma=x(A,r,t)$ défini en \eqref{eq:xArt} avec $A\in \gl_m(D)$ un élément  semi-simple sans valeurs propres $\pm1$. Grâce au lemme \ref{lem:JRlem4.3}, on trouve que le descendant $(G_\sigma,G_\sigma^\theta,\Int(\theta)|_{G_\sigma})$ est isomorphe au produit d'au plus deux triplets de type 3 avec le descendant de $(GL_{2m,D},GL_{m,D}\times GL_{m,D},\Int(\theta_{m,m}))$ associé au élément 
\begin{align}\label{eq:pf-des}
\begin{pmatrix}
  A  & A-I_m   \\ 
  A+I_m &   A  \\ 
\end{pmatrix}. 
  \end{align}
Alors il suffit d'étudier le dernier descendant autrement dit on peut et on va supposer $p=q=m$ et que $\sigma$ est de la forme \eqref{eq:pf-des} avec $A$ comme ci-dessus. D'après le corollaire \ref{cor:Prdx}, on est ramené au cas où $\Prd_A$ est une puissance d'un seul polynôme irréductible, ce que on va aussi supposer. Soit $\Prd_A=\chi^k$ avec $\chi$ un polynôme irréductible sur $F$. Comme $A$ est semi-simple, $\chi$ est le polynôme minimal de $A$. Rappelons $V=D^{2m}$ et $G=GL_D(V)$ avec les notations de \S \ref{S:D}. Alors $A$ agit sur $D^m$ et $\chi(A)=0$. Soit $E=F[t]/(\chi)\simeq F(A)$ qui est un corps où $F(A)$ désigne l'algèbre sur $F$ engendrée par $A$. L'action de $A$ munit $D^m$ d'une structure de $E\otimes_F D$-module. Soit $E\otimes_F D\simeq\gl_s(D') $ où $s\geq1$ est un entier et $D'$ est une algèbre à division centrale sur $E$. On a $D^m\simeq M_{t\times s}(D')$ pour un entier $t\geq1$ en tant que module semi-simple où $M_{t\times s}$ désigne une matrix de taille $t\times s$. Il s'ensuit que le centralisateur de $A$ dans $\gl_m(D)$, qui s'identifie avec $\ggo_\sigma^\theta$, est isomorphe à $\gl_{t}(D')$. L'égalité $\Prd_A=\chi^k$ implique que le degré de $\chi$ vaut $\deg\chi=\frac{dm}{k}$ où $d$ est le degré de $D$. En utilisant de nouveau le corollaire \ref{cor:Prdx}, on a 
\begin{align*}
    \Prd_\sigma(t)=(2t)^{dm}\Prd_A(\frac{t^2+1}{2t})=\left((2t)^{\deg\chi}\chi(\frac{t^2+1}{2t})\right)^k. 
\end{align*}
Soit $\tilde\chi(t)=(2t)^{\deg\chi}\chi(\frac{t^2+1}{2t})\in F[t]$ qui est séparable puisque $A$ n'a pas de valeurs propres $1$. Comme $\sigma$ est semi-simple, $\tilde \chi$ est le polynôme minimal de $\sigma$ dont le degré vaut $\deg\tilde\chi=2\deg\chi$. Soit $\tilde E=F[t]/(\tilde \chi)\simeq F(\sigma)$ qui est alors une algèbre étale sur $E$ de degré $2$ où $F(\sigma)$ désigne l'algèbre sur $F$ engendrée par $\sigma$. L'action de $\sigma$ munit $V=D^m\oplus D^m$ d'une structure de $\tilde E\otimes_F D$-module. Mais $V\simeq M_{2t\times s}(D')$ et 
\begin{align*}
    \tilde E\otimes_F D\simeq\tilde E\otimes_E (E\otimes_F D)\simeq\tilde E\otimes_E M_{t\times s}(D')\simeq M_{t\times s}(\tilde E\otimes_E D'). 
\end{align*}
On a $\ggo_\sigma\simeq\gl_t(\tilde E\otimes_E D')$. Notons que 
\begin{align*}
    \Int(\theta)(\sigma)=
\begin{pmatrix}
  A  & -A+I_m   \\ 
  -A-I_m &   A  \\ 
\end{pmatrix}=\sigma^{-1}.  
\end{align*}
Alors $\theta$ agit sur $\tilde E\simeq F(\sigma)$. De plus, on a 
\begin{align*}
    F(A)\simeq F\begin{pmatrix}
  A  & 0   \\ 
  0 &   A  \\ 
\end{pmatrix}=F(\frac{\sigma+\Int (\theta)(\sigma)}{2})\subset F(\sigma)^\theta 
\end{align*}
où l'on note $F(\sigma)^\theta$ la sous-algèbre de $F(\sigma)$ formé des éléments fixés par $\theta$. Donc $\theta$ est l'unique élément nontrivial du groupe d'automorphismes $\Aut_E(\tilde E)$. En résumé, on a 
\begin{align*}
    (G_\sigma, G_\sigma^\theta,\Int(\theta))\simeq(\Res_{\tilde E/F} GL_{t,\tilde E\otimes_E D'},\Res_{E/F} GL_{t,D'},\theta). 
\end{align*}
Si $\tilde E$ est isomorphe à $E\times E$, resp. est un corps, on trouve un triplet de type 1, resp. de type 2.  
\end{preuve} 
\end{paragr}

\begin{paragr}[Éléments anisotropes.]\label{S:anisotrope} ---
Soit $\sigma\in S(F)$ un élément semi-simple. On dit que $\sigma$ est $(G,\theta)$-anisotrope si la classe de $G^\theta(F)$-conjugaison de $\sigma$ ne rencontre pas $M(F)$ pour tout sous-groupe de Levi $M$ semi-standard et propre de $G$.  La condition est automatique si $G$ est anisotrope modulo son centre, c'est-à-dire si $p+q=1$, $G=GL_{1,D}$ et  $\sigma=\pm \Id_D$. Si $p+q>1$, il résulte de la proposition \ref{prop:elementssr} que $\sigma$ est $(G,\theta)$-anisotrope si et seulement si $p=q$ et $\sigma$ est conjugué sous $G^\theta(F)$ à $\begin{pmatrix}
  A & A-I_p \\ A+I_p & A
\end{pmatrix}$ avec $A\in \mathfrak{gl}_p(D)$ un élément semi-simple elliptique (au sens où son centralisateur dans $GL_{p,D}$ est anisotrope modulo le centre de $A_{GL_{p,D}}$) sans valeurs propres $\pm1$. Dans ce cas, on a, en particulier, $\sigma\in S^\circ(F)$. En tout cas, si $\sigma$ est $(G,\theta)$-anisotrope, alors $G_\sigma^\theta$ est anisotrope modulo $A_{G_\sigma^\theta}$.  Par ailleurs, on note que  $\sigma$ est $(G,\theta)$-anisotrope si et seulement si la classe de $G^\theta(F)$-conjugaison de $\sigma$ ne rencontre pas $P(F)$ pour tout $P\in \fc(P_0^\theta)$ sous-groupe parabolique propre de $G$. 

\begin{proposition}\label{prop:A_P_sigma}
    Soit $\sigma\in S(F)$ un élément semi-simple $(G,\theta)$-anisotrope. On a 
    \begin{align*}
        A_{G}=A_{G_\sigma}\cap A_{G^\theta}=A_{G_\sigma^\theta}=A_{G_\sigma}^\theta. 
    \end{align*}
\end{proposition}

\begin{preuve}
    Si $p+q=1$, comme $\sigma=\pm1$ et $\Int(\theta)$ est l'automorphisme identique de $G$, l'assertion est triviale. 
    
    Si $p=q$, d'après la discussion ci-dessus, il suffit de considérer  $\sigma=\begin{pmatrix}
  A & A-I_p \\ A+I_p & A
\end{pmatrix}$ avec $A\in \mathfrak{gl}_p(D)$ un élément semi-simple elliptique sans valeurs propres $\pm1$. L'égalité $A_G=A_{G_\sigma^\theta}$ résulte d'un calcul direct. Mais on a également  
    \begin{align*}
        A_G\subseteq A_{G_\sigma}\cap A_{G^\theta}  \text{ et } A_{G^\theta}\subset A_{G_\sigma^\theta}.
    \end{align*}
    On en déduit que 
        \begin{align*}
        A_{G}=A_{G_\sigma}\cap A_{G^\theta}=A_{G_\sigma^\theta}. 
    \end{align*}
    On sait que $(G_\sigma, G_\sigma^\theta)$ est une paire symétrique de type 1, 2 ou 3 dans la proposition \ref{prop:descendant}. De plus, comme $G_\sigma^\theta$ est anisotrope modulo $A_{G_\sigma^\theta}$, l'égalité $A_G=A_{G_\sigma^\theta}$ implique que $G_\sigma^\theta$ est de $F$-rang $1$. Pour les types 1 et 2, on a $A_{G_\sigma^\theta}=A_{G_\sigma}^\theta$ ce qui conclut. Enfin, le type 3 n'apparaît pas. Effectivement, comme $p=q$ et $\sigma\in S^\circ(F)$, \cite[proposition 4.1]{Zha2} implique $r=t$ et en particulier $r+t\neq1$.  
\end{preuve}
  
\begin{corollaire}\label{cor:explicite-G}
    Supposons $p=q$. Soit $\sigma=\begin{pmatrix}
  A & A-I_p \\ A+I_p & A
\end{pmatrix}$ avec $A\in \mathfrak{gl}_p(D)$ un élément semi-simple elliptique sans valeurs propres $\pm1$. 
\begin{enumerate}
    \item Si $\sigma$ n'est pas elliptique dans $G$, alors $(G_\sigma, G_\sigma^\theta)$ est une paire symétrique de type 1 avec $r=1$. 
    
    \item Si $\sigma$ est elliptique dans $G$, alors $(G_\sigma, G_\sigma^\theta)$ est une paire symétrique de type 2 avec $r=1$.
\end{enumerate}
\end{corollaire}

\begin{remarque}\label{rmq:type}
  Un élément $A\in \mathfrak{gl}_p(D)$ comme dans le corollaire \ref{cor:explicite-G} sera dit de type 1 (resp. de type 2) si $(G_\sigma, G_\sigma^\theta)$ est une paire symétrique de type 1 (resp. de type 2). D'après le corollaire \ref{cor:Prdx}, le type de $A$ ne dépend que $\Prd_A$. 
\end{remarque}
  
  Soit $P\in\fc(P_0^\theta)$. Un élément semi-simple $\sigma\in (M_P\cap S)(F)$ est dit $(P,\theta)$-anisotrope, resp.  $(M_P,\theta)$-anisotrope,   si la classe de $M_P^\theta(F)$-conjugaison de $\sigma$ ne rencontre pas $Q(F)$, resp. $M'(F)$, pour tout $Q\in\fc(P_0^\theta)$, $Q\subsetneq P$, resp. tout $M'\in\lc(M_0)$, $M'\subsetneq M_P$.   Grâce à la proposition \ref{prop:elementssr}, on peut montrer que pour $P\in\fc(P_0^\theta)$, un élément semi-simple $\sigma\in(M_P\cap S)(F)$ est $(P,\theta)$-anisotrope si et seulement s'il est $(M_P,\theta)$-anisotrope. Dans ce cas, on sait que $M_{P,\sigma}^\theta$ est anisotrope modulo $A_{M_{P,\sigma}^\theta}$ avec la discussion ci-dessus. 

 \begin{corollaire}\label{cor:A_P_sigma}
     Soit $P\in\fc(P_0^\theta)$ avec $P=MN$ sa décomposition de Levi semi-standard. Soit $\sigma\in (M\cap S)(F)$ un élément semi-simple $(P,\theta)$-anisotrope. On a 
         \begin{align*}
        A_{P}=A_{M_\sigma}\cap A_{M^\theta}=A_{M_\sigma^\theta}=A_{M_\sigma}^\theta. 
    \end{align*}
 \end{corollaire}
 
 Soit $\sigma=x(A, r, t)$ défini en \eqref{eq:xArt} où $m,r,t\geq 0$ sont des entiers tels que $m\leq \min\{p-r,q-t\}$ et $A\in \gl_m(D)$ est un élément  semi-simple sans valeurs propres $\pm1$. Supposons que $A$ est une matrice bloc-diagonale
 \begin{align}\label{eq:bloc-diag-A}
     \diag(\underbrace{A_1,\cdots,A_1}_{s_1},\cdots,\underbrace{A_i,\cdots,A_i}_{s_i},\underbrace{A'_1,\cdots,A'_1}_{k_1},\cdots,\underbrace{A'_j,\cdots,A'_j}_{k_j})
 \end{align}
 où $A_1,\cdots,A_i$ sont de type 1 alors que $A'_1,\cdots,A'_j$ sont de type 2, dont les polynômes caractéristiques réduits sont mutuellement distincts, cf. remarque \ref{rmq:type}. Effectivement, on déduit de \cite[théorème 4]{CFYu} que tous les polynômes $\Prd_{A_\alpha}, 1\leq \alpha\leq i$ et $\Prd_{A'_\beta}, 1\leq \beta\leq j$ sont des puissances de polynômes irréductibles mutuellement distincts sur $F$. Supposons $A_\alpha\in\gl_{m_\alpha}(D)$ et $A'_\beta\in\gl_{m'_\beta}(D)$ pour $1\leq \alpha\leq i$ et $1\leq \beta\leq j$. Soit $D_\alpha$, resp. $D'_\beta$, le centralisateur de $A_\alpha$ dans $\gl_{m_\alpha}(D)$, resp. de $A'_\beta$ dans $\gl_{m'_\beta}(D)$. D'après le corollaire \ref{cor:explicite-G}, chaque $D_\alpha$, resp. $D'_\beta$, est une algèbre à division centrale sur une extension finie $L_\alpha$, resp. $L'_\beta$, de $F$. Ainsi
 \begin{align}\label{eq:Gsigma}
     G_\sigma\simeq\prod_{1\leq \alpha\leq i} (\Res_{L_\alpha/F} GL_{s_\alpha,D_\alpha}\times \Res_{L_\alpha/F} GL_{s_\alpha,D_\alpha})\times\prod_{1\leq \beta\leq j} \Res_{K'_\beta/F} GL_{k_\beta,D'_\beta\otimes_{L'_\beta} K'_\beta} \\ \nonumber
     \times GL_{p+q-2m-r-t,D}\times GL_{r+t,D}
 \end{align}
 où $K'_\beta$ est une extension quadratique de $L'_\beta$ et 
 \begin{align}\label{eq:Gsigmatheta}
     G_\sigma^\theta\simeq\prod_{1\leq \alpha\leq i} \Delta\Res_{L_\alpha/F} GL_{s_\alpha,D_\alpha} \times\prod_{1\leq \beta\leq j} \Res_{L'_\beta/F} GL_{k_\beta,D'_\beta}\\ \nonumber
     \times GL_{p-m-r,D}\times GL_{q-m-t,D}\times GL_{r,D}\times GL_{t,D}.
 \end{align} 
 Cela nous donne une description explicite de la proposition \ref{prop:descendant}.
 
 Pour un tel $\sigma$, si $P_1\in\fc(P_0^\theta)$ et $\sigma\in(M_1\cap S)(F)$ est $(P_1,\theta)$-anisotrope, alors on trouve que 
 \begin{align}\label{eq:M1}
     M_1\simeq\prod_{1\leq\alpha\leq i}(\underbrace{GL_{2m_\alpha,D}\times\cdots\times GL_{2m_\alpha,D}}_{s_\alpha})\times\prod_{1\leq\beta\leq j}(\underbrace{GL_{2m'_\beta,D}\times\cdots\times GL_{2m'_\beta,D}}_{k_\beta}) \\ \nonumber
     \times \underbrace{\GmD\times\cdots\times\GmD}_{p+q-2m}.
 \end{align}
 Il s'ensuit que 
  \begin{align}\label{eq:M1sigma}
     M_{1,\sigma}\simeq\prod_{1\leq\alpha\leq i}(\underbrace{\Res_{L_\alpha/F}\mathbb{G}_{m,D_\alpha}\times\cdots\times \Res_{L_\alpha/F}\mathbb{G}_{m,D_\alpha}}_{2s_\alpha})\\ \nonumber
     \times\prod_{1\leq\beta\leq j}(\underbrace{\Res_{K'_\beta/F} \mathbb{G}_{m,D'_\beta\otimes_{L'_\beta} K'_\beta}\times\cdots\times \Res_{K'_\beta/F} \mathbb{G}_{m,D'_\beta\otimes_{L'_\beta} K'_\beta}}_{k_\beta})  
     \times \underbrace{\GmD\times\cdots\times\GmD}_{p+q-2m}
 \end{align}
 qui est le sous-groupe des matrices diagonales de $G_\sigma$. 
Pour tout sous-groupe $H$ de $G$, on note $\Cent_G(H)$ le centralisateur dans $G$ de $H$. On voit que 
 \begin{align}\label{eq:centAsigthe}
     \Cent_G(A_{G_\sigma}^\theta)\simeq\prod_{1\leq \alpha\leq i} GL_{2s_\alpha m_\alpha,D} \times \prod_{1\leq \beta\leq j} GL_{2k_\beta m'_\beta,D}\times GL_{p+q-2m-r-t,D}\times GL_{r+t,D} 
 \end{align}
qui appartient à $\lc(M_1)$. Il est évident que $G_\sigma\subset \Cent_G(A_{G_\sigma}^\theta)$. 
\end{paragr}

\begin{paragr}[Ensembles de Weyl.] ---
      Pour tous $P_1, P_2\in\fc(P_0^\theta)$, soit $W^\theta(\ago_1,\ago_2)$ l'ensemble, éventuellement vide, des isomorphismes distincts de $\ago_1$ sur $\ago_2$ obtenus par restriction d'un élément de $W^\theta$ à $\ago_1$. On identifie $w\in W^\theta(\ago_1,\ago_2)$ à  l'unique élément de $W^\theta$, encore noté $w$, tel que $w\alpha$ est une racine de $(P_0^\theta, A_0)$ pour toute racine $\alpha\in\Delta_0^{P_1^\theta}$ autrement dit $w^{-1}\beta$ est une racine de $(P_0^\theta, A_0)$ pour toute racine $\beta\in\Delta_0^{P_2^\theta}$ (cf. \cite[lemme 1.3.6]{labWal}). 
\end{paragr}

\begin{paragr}[Slices à la Luna.] \label{S:luna}--- Soit $\sigma\in G$ semi-simple.         Soit $\ggo_\sigma\subset\ggo$ l'algèbre de Lie de $G_\sigma$. Pour tout $x\in G_\sigma$, soit $\Ad$ l'action adjointe  de $G_\sigma$ sur $\ggo/\ggo_\sigma$. Soit
    \begin{align*}
        D^\sigma(x)=\det(\Ad(x)-\Id|_{\ggo/\ggo_\sigma})\in F[G_\sigma]^{G_\sigma}. 
    \end{align*} 
Soit $c^\flat_\sigma : G_\sigma\to\cgo_\sigma=\Spec(F[G_\sigma]^{G_\sigma})$ le quotient catégorique et $\cgo'_\sigma$ l'ouvert de $\cgo_\sigma$ défini par $D^\sigma(x)\neq0$. Soit $G'_\sigma\subset G_\sigma$ l'image réciproque de $\cgo'_\sigma$ par $c^\flat_\sigma$. Alors $G'_\sigma$ est un ouvert de $G_\sigma$ qui contient $\sigma$ et tel que pour tout $x\in G'_\sigma$ on a $G_x\subset G_\sigma$. 

Supposons de plus $\sigma\in S$. Soit $S'_\sigma=S_\sigma\cap G'_\sigma$: c'est un ouvert de $S_\sigma$ qui contient $\sigma$ et tel que pour tout $x\in S'_\sigma$ on a $G^\theta_x\subset G^\theta_\sigma$. Soit $c^\flat_{S_\sigma} : S_\sigma\to\cgo_{S_\sigma}=\Spec(F[S_\sigma]^{G_\sigma^\theta})$ le quotient catégorique. On a le diagramme commutatif suivant : 
\begin{align*}
    \xymatrix{  S_\sigma \ar[r]^{i} \ar[d]^{c^\flat_{S_\sigma}}  & G_\sigma  \ar[d]^{c^\flat_\sigma}    \\
   \cgo_{S_\sigma} \ar[r]^{j} & \cgo_\sigma    }
\end{align*}     
où $i : S_\sigma\to G_\sigma$ est l'inclusion naturelle et $j : \cgo_{S_\sigma}\to\cgo_\sigma$ est dual du morphisme de restriction $F[G_\sigma]^{G_\sigma}\to F[S_\sigma]^{G_\sigma^\theta}$. Soit $\cgo'_{S_\sigma}\subset\cgo_{S_\sigma}$ l'image réciproque de $\cgo'_\sigma$ par $j$. Alors $\cgo'_{S_\sigma}$ est un ouvert de $\cgo_{S_\sigma}$ dont l'image réciproque dans $S_\sigma$ est $S'_\sigma$.

    \begin{lemme}\label{lem:etale}(cf. \cite[lemme 3.2]{RadRal} et \cite[théorème A.2.3]{AizGour})
         Soit $\sigma\in S$ semi-simple. On considère l'action de $G^\theta_\sigma$ sur $G^\theta\times S_\sigma$ donnée par $y\cdot(h,x)=(hy^{-1}, \Int(y)(x))$. L'application 
        \begin{align*}
            \beta: G^\theta\times^{G^\theta_\sigma}S'_\sigma\rightarrow S
        \end{align*}
        donnée par $(h,x)\mapsto\Int(h)(x)$ est étale. 
    \end{lemme}

    \begin{preuve}
        Le morphisme est clairement $G^\theta$-équivariant. Il suffit donc de vérifier l'action au point $(1,x)\in G^\theta\times S'_\sigma$. La différentielle de $G^\theta\times S'_\sigma\rightarrow S$ en ce point est donnée par 
        \begin{align*}
            (Y,X)\mapsto\Ad(x^{-1})(Y)-Y+X.
        \end{align*}
        Ici $Y\in\ggo^\theta$ et $X\in T_x S_\sigma$ i.e. $\Ad(x^{-1}\theta) (X) + X = 0$ et $\Ad(\sigma) (X) = X$. On a
        \begin{align*}
            \Ad(x^{-1} \theta) (\Ad(x^{-1}) (Y) - Y + X) + \Ad(x^{-1}) (Y) - Y + X 
            \\
            = Y - \Ad(x^{-1}) (Y) - X + \Ad(x^{-1}) (Y) - Y + X = 0.
        \end{align*}
        Donc l’application est bien à image dans $T_x S$.

        Par définition de $S'_\sigma$, l’application $\iota : Y \mapsto \Ad(x^{-1}) (Y) - Y$ induit un automorphisme de $\ggo/\ggo_\sigma$. Notons que $\Ad(x^{-1} \theta)$ induit une involution de ce quotient. On a de plus $\iota \circ \Ad(\theta) = -\Ad(x^{-1} \theta) \circ \iota$. En particulier, on voit que $\iota$ induit un isomorphisme $\ggo^\theta/ \ggo^\theta_\sigma \rightarrow T_x S / T_x S_\sigma$. De là il est facile de conclure. 
    \end{preuve}
\end{paragr}

\begin{paragr}[Variété unipotente.] ---
  Soit $H$ un groupe réductif connexe défini sur $F$. On note $\uc_H$ la clôture de Zariski dans $H$ de l'ensemble des éléments unipotents dans $H(F)$.  C'est une sous-variété algébrique fermée de $H$ et définie sur $F$. Notons également $\nc_\hgo$ le cône nilpotent de $\hgo$ pour l'action adjointe de $H$, cf. \cite[exemple 2.3.13]{AizGour}. 
  
  \begin{lemme}\label{lem:var-unip}
  Soit $0\leq p'\leq N$ et $\sigma\in S_{p'}$ un élément semi-simple. On a 
   \begin{align*}
    \uc_{G_\sigma}\cap\sigma^{-1}S_{p'}=\uc_{G_\sigma}\cap\sigma^{-1}S=\uc_{G_\sigma}\cap S^\circ=\uc_{G_\sigma}\cap S. 
  \end{align*} 
\end{lemme}

\begin{preuve}
	Par le lemme \ref{lem:JRlem4.1}, on voit que  
  \begin{align*}
    \uc_{G_\sigma}\cap\sigma^{-1}S_{p'}=\uc_{G_\sigma}\cap\sigma^{-1}S\subset\uc_{G_\sigma}\cap S^\circ=\uc_{G_\sigma}\cap S. 
  \end{align*} 
  Soit $u\in \uc_{G_\sigma}\cap S^\circ$. Il existe un élément nilpotent $X\in\sgo$ tel que $u=\exp(X)$ et que $\Ad(\sigma)(X)=X$. Soit $g\in G$ un élément tel que $\sigma=\rho_{p'}(g)=g\theta_{p'}g^{-1}\theta$. On pose $v=\exp(X/2)\in G_\sigma$. Alors 
   \begin{align*}
    \rho_{p'}(vg)=vg\theta_{p'}(vg)^{-1}\theta=v(g\theta_{p'}g^{-1}\theta)\theta v^{-1}\theta=v\sigma \theta v^{-1}\theta=\sigma v\theta v^{-1} \theta=\sigma v^2=\sigma u. 
  \end{align*}   
Donc $\sigma u\in S_{p'}$ ce qui conclut. 
\end{preuve}

\begin{lemme}\label{lem:ssr-unip}
  Soit $0\leq p'\leq N$ et $\sigma\in S_{p'}$ un élément $G^\theta$-semi-simple et $G^\theta$-régulier. On a 
   \begin{align*}
    \uc_{G_\sigma}\cap\sigma^{-1}S_{p'}=\uc_{G_\sigma}\cap\sigma^{-1}S=\uc_{G_\sigma}\cap S^\circ=\uc_{G_\sigma}\cap S=1. 
  \end{align*} 
\end{lemme}

\begin{preuve}
Soit $u\in\uc_{G_\sigma}\cap\sigma^{-1}S_{p'}$. Notons que $x:=\sigma u=u\sigma$ est la décomposition de Jordan de $x\in S_{p'}$ dans $G$. Il résulte du lemme \ref{lem:c-ss} et de la proposition \ref{prop:ssr-Prd} que $x$ est $G^\theta$-semi-simple et $G^\theta$-régulier. La preuve de la proposition \ref{prop:ssr-Prd} entraîne également que sa partie unipotente $u=1$. On conclut avec le lemme \ref{lem:var-unip}. 
\end{preuve}
\end{paragr}

\begin{paragr}[Exponentielle.]\label{S:exp} --- Soit $0\leq p'\leq N$ et $\sigma\in S_{p'}(F)$ un élément semi-simple. Pour tout sous-espace $\vgo\subset\ggo$, on note $\vgo_\sigma$ le centralisateur dans $\vgo$ de $\sigma$. L’exponentielle $\exp$ réalise un isomorphisme du cône nilpotent $\nc=\nc_\ggo$ sur la variété unipotente $\uc=\uc_G$. Par restriction, on obtient en particulier des isomorphismes de $\nc_\sgo=\nc\cap\sgo$ sur $\uc_S=\uc\cap S=\uc\cap S^\circ$ et de $\nc_{\ggo_\sigma}\cap\sgo=\nc \cap \sgo_\sigma$ sur $\uc_{G_\sigma}\cap S=\uc_{G_\sigma}\cap S^\circ$, cf. lemme \ref{lem:var-unip}. D'après les lemmes \ref{lem:etale} et \ref{lem:c-ss}, on a un isomorphisme 
\begin{align*}
    G^\theta\times^{G^\theta_\sigma} (\nc \cap \sgo_\sigma) \to S_{p', c_{p'}^\flat(\sigma)}
\end{align*}
donné par $ (h, X) \mapsto h\sigma \exp(X)h^{-1} $ où $c^\flat _{p'} : S_{p'}\to\cgo_{p'}\simeq\mathbf{A}^{d\nu}$ est le morphisme défini dans \S \ref{S:invariants}. 

Soit $V$ un ensemble fini de places de $F$ contenant les places archimédiennes assez grand. On peut supposer que si $ X \in (\nc \cap \sgo_\sigma)(\AAA), h \in G^\theta(\AAA) $ sont tels que $ h\sigma \exp(X)h^{-1} \in S_{p'}(\oc_v) $ pour $ v \notin V $ alors quitte à changer $ X $ et $ h $ par le même élément de $ G^\theta_\sigma(\AAA) $ on peut supposer que $ h_v \in G^\theta(\oc_v) $ et $ X_v \in \sgo_\sigma(\oc_v) $ pour $ v \notin V $. Bien sûr $ V $ dépend de $ \sigma $.

On définit un ouvert $ \omega_V \subset \ggo(F_V) $ contenant $0$ par la condition suivante : $ (X_v)_{v \in V} \in \omega_V $ si et seulement si les coefficients (non dominants) du polynôme caractéristique de $ X_v $ sont de valeur absolue $ < \eps_v $ pour tout $ v \in V $. Pour $ \eps_v $ assez petit, l’exponentielle induit un isomorphisme analytique de $ \omega_V $ sur un ouvert $ \Omega_V $ de $ G(F_V) $. 
Il est facile de montrer que l'application $X\mapsto \sigma\exp(X)$ induit un isomorphisme analytique 
\begin{align*}
    \sgo_\sigma(F_V)\cap\omega_V\to  S_\sigma(F_V)\cap (\sigma\Omega_V). 
\end{align*}
Soit $\omega'_V$ l'image réciproque de $ S'_\sigma(F_V)\cap (\sigma\Omega_V)$ par cet isomorphisme. Alors $\omega'_V$ est un ouvert de $\sgo_\sigma(F_V)\cap\omega_V$. Il résulte du lemme \ref{lem:etale} que l’application 
\begin{align*}
    G^\theta(F_V) \times \omega_V' \to S(F_V)
\end{align*}
donnée par $ (h, X) \mapsto h\sigma \exp(X)h^{-1} $ est une submersion analytique sur un ouvert $ \Omega_V' \subset S(F_V) $.

On considère $ \varphi = \varphi_V \otimes \mathbf{1}_{S_{p'}(\oc^V)} $ avec $ \varphi_V \in C^\infty_c(S_{p'}(F_V)) $. On suppose que $ \supp(\varphi_V) \subset \Omega_V' $. Il existe une fonction $ f_V \in C^\infty_c(G^\theta(F_V) \times \omega_V') $ telle que on ait
\begin{align*}
    \int_{G^\theta_\sigma(F_V)} f_V(h^{-1}y^{-1}, \Ad(y)(X)) \, dy = \varphi_V(\Int(h^{-1})(\sigma \exp(X))) 
\end{align*} 
pour tout $ h \in G^\theta(F_V) $ et tout $ X \in \omega_V' $. On pose $ f^V = \mathbf{1}_{G^\theta(\oc^V)} \otimes \mathbf{1}_{\sgo_\sigma(\oc^V)} $ et $f=f_V\otimes f^V\in C_c^\infty((G^\theta\times\sgo_\sigma)(\AAA))$. On a alors
\begin{align*}
    \int_{G^\theta_\sigma(\AAA)} f(h^{-1}y^{-1}, \Ad(y)(X)) \, dy = \varphi(\Int(h^{-1})(\sigma \exp(X))) 
\end{align*} 
 pour tout $ h \in G^\theta(\AAA) $ et tout $ X \in (\nc \cap \sgo_\sigma)(\AAA) $. 
    
\end{paragr}

\begin{paragr}[Une formule préliminaire.] --- \label{S:prelim}
  Soit $0\leq p'\leq N$ et $\of\in\cgo_{p'}(F)$. Soit $P_1\in\fc(P_0^\theta)$ et  $\sigma\in (M_{P_1}\cap S_{p',\of})(F)$  un élément semi-simple  $(P_1, \theta)$-anisotrope. D'après le corollaire \ref{cor:A_P_sigma}, on a 
   \begin{align}\label{eq:A_1}
    A_{M_1}=A_{M_{1,\sigma}}\cap A_{M_1^\theta}=A_{M_{1,\sigma}^\theta}=A_{M_{1,\sigma}}^\theta. 
  \end{align} 
  Puisque $M_{1,\sigma}^\theta$ est anisotrope modulo $A_{M_{1,\sigma}^\theta}$ (voir \S \ref{S:anisotrope}), le couple $(P_{1,\sigma}^\theta,M_{1,\sigma}^\theta)$ est formé d'un sous-groupe parabolique minimal de $G_\sigma^\theta$ ainsi que d'un facteur de Levi. Il résulte de la proposition \ref{prop:descendant} que le centralisateur dans $G_\sigma$ d'un tore déployé maximal de $G_\sigma^\theta$ est un facteur de Levi d'un sous-groupe parabolique minimal de $G_\sigma$. Alors \eqref{eq:A_1} implique que le couple $(P_{1,\sigma},M_{1,\sigma})$ est formé d'un sous-groupe parabolique minimal de  $G_\sigma$ ainsi que d'un facteur de Levi. 
  
  Un sous-groupe parabolique de $G_\sigma$ sera dit relativement standard s'il contient  $P_{1,\sigma}^\theta$. Soit $\fc^{G_\sigma}(P_{1,\sigma}^\theta)$ l'ensemble de sous-groupes paraboliques de $G_\sigma$ qui sont relativement standard.

\begin{proposition}
    On a 
    \begin{align*}
    \fc^{G_\sigma}(P_{1,\sigma}^\theta)\subset\fc^{G_\sigma}(M_{1,\sigma}). 
    \end{align*}
\end{proposition}

\begin{preuve}
    Considérons  $R\in\fc^{G_\sigma}(P_{1,\sigma}^\theta)$. On a  $A_{M_{1,\sigma}^\theta}\subset R$. Par la discussion ci-dessus, le centralisateur dans $G_\sigma$ de $A_{M_{1,\sigma}^\theta}$ est un facteur de Levi minimal $M_{1,\sigma}$. Soit $T$ un tore déployé maximal de $G_\sigma$ tel que $A_{M_{1,\sigma}^\theta}\subset T\subset R$. Alors $R$ contient le centralisateur dans $G_\sigma$ de $T$. Ce centralisateur est un facteur de Levi minimal contenu dans $M_{1,\sigma}$, qui doit être $M_{1,\sigma}$. 
\end{preuve}

    Soit $P\in\fc(P_0^\theta)$ et $\gamma\in (M_P\cap S_{p',\of})(F)$. La partie semi-simple $\gamma_s$ de $\gamma$ est conjuguée sous $M_P^\theta(F)$ à un élément $(P_2,\theta)$-anisotrope de $(M_{P_2}\cap S)(F)$ où $P_2\in\fc^P(P_0^\theta)$. À l'aide de la proposition \ref{prop:elementssr} et du lemme \ref{lem:var-unip}, on peut alors écrire 
   \begin{align}\label{eq:dec-mu-w}
    \gamma=\mu^{-1} w \sigma u w^{-1}\mu 
  \end{align}   
  où  
   \begin{align*}
    \mu\in M_P^\theta(F), w\in W^\theta(\ago_1, \ago_2),  \text{ et } u\in (\uc_{G_\sigma}\cap M_{P_w}\cap S^\circ)(F). 
  \end{align*} 
Si $T$ est un tore d'un groupe réductif connexe $H$, on note $W(H,T)$ le groupe de Weyl associé. 

\begin{proposition}\label{prop:unique-gd}
    L'élément de Weyl $w$ est uniquement déterminé modulo multiplication à gauche par $W(M_P^\theta, A_2)$ et multiplication à droite par $W(G_\sigma^\theta, A_1)$.
\end{proposition}

\begin{preuve}
On trouve que  \eqref{eq:dec-mu-w} implique $\gamma_s=\mu^{-1} w \sigma w^{-1}\mu$. Soit $\gamma_s=\mu_*^{-1} w_* \sigma w_*^{-1}\mu_*$ avec $\mu_*\in M_P^\theta(F)$ et $w_*\in W^\theta(\ago_1, \ago_2)$. Comme $A_1=A_{M_{1,\sigma}^\theta}$ est un tore déployé maximal de $G_\sigma^\theta$, on voit que $A_2=A_{M_{2,w\sigma w^{-1}}^\theta}$ est un tore déployé maximal de $M_{P,w\sigma w^{-1}}^\theta$. Mais on a 
\begin{align}\label{eq:pf-mu-w}
    \mu_* \mu^{-1} (w\sigma w^{-1}) \mu \mu_*^{-1}=w_* \sigma w_*^{-1}. 
\end{align}
On en déduit que $\mu_* \mu^{-1} A_2 \mu \mu_*^{-1}$ et $A_2$ sont deux tores déployés maximaux de $M_{P,w_*\sigma w_*^{-1}}^\theta$. Il existe donc un $m\in M_{P,w_*\sigma w_*^{-1}}^\theta(F)$ tel que $m\mu_* \mu^{-1}$ normalise $A_2$. Soit $n=m\mu_* \mu^{-1}\in M_P^\theta(F)$. Il résulte de \eqref{eq:pf-mu-w} que $w_*^{-1} n w\in G^\theta_\sigma(F)$. De plus, $w_*^{-1} n w$ normalise $A_1$. On a montré que $w_*=nwv$ pour un certain $v\in G_\sigma^\theta(F)$ normalisant $A_1$ ce qui conclut. 
\end{preuve}

Si $w$ est fixé, $\mu$ est uniquement déterminé modulo multiplication à gauche par $(M_P^\theta\cap w G_\sigma w^{-1})(F)$. L'élément $u$ est uniquement déterminé par $\mu$ et $w$. Soit $W^\theta(\ago_1 ; P, G_\sigma)$ la réunion des éléments $w\in W^\theta(\ago_1,\ago_2)$ pour $P_2\in\fc^P(P_0^\theta)$ tel que $w^{-1}\alpha$ est positive pour toute racine $\alpha\in\Delta_{P_2^\theta}^{P^\theta}$ et tel que $w\beta$ est positive pour toute racine positive $\beta$ de $(G_\sigma^\theta,A_1)$ autrement dit 
\begin{align}\label{eq:W-2eme}
    w(P_{1,\sigma}^\theta)w^{-1}=P_{2,w\sigma w^{-1}}^\theta. 
\end{align}

Soit $\varphi\in C_c^\infty(S_{p'}(\AAA))$ et $x\in G^\theta(\AAA)$. Il résulte des propositions \ref{prop:elementssr} et \ref{prop:unique-gd} que l'expression $\tlK_{P,\of,\varphi}(x)$ défini en \eqref{eq:KtildePo} est égale à  
   \begin{align*}
    \sum_{w\in W^\theta(\ago_1 ; P, G_\sigma)} \sum_{\mu} \sum_{\nu} \sum_{u\in (\uc_{G_\sigma}\cap M_{P_w}\cap S^\circ)(F)} \int_{((\mu^{-1}w\sigma u w^{-1}\mu N_{P,\mu^{-1}w\sigma w^{-1}\mu})\cap S_{p'})(\AAA)} \varphi(\Int(\nu x)^{-1}(y)) \, dy
  \end{align*}
où les sommes sur $\mu$ et $\nu$ sont prises respectivement sur 
   \begin{align*}
    \mu\in (M_P^\theta\cap w G_\sigma w^{-1})(F)\back M_P^\theta(F) 
   \end{align*}
et 
   \begin{align*}
    \nu\in N_{P,\mu^{-1}w\sigma w^{-1}\mu}^\theta(F)\back N_P^\theta(F). 
    \end{align*}
En combinant les sommes sur $\mu$ et $\nu$, on trouve que 
   \begin{align*}
    \tlK_{P,\of,\varphi}(x)=\sum_{w\in W^\theta(\ago_1 ; P, G_\sigma)} \sum_{\pi} \sum_{u\in (\uc_{G_\sigma}\cap M_{P_w}\cap S^\circ)(F)} \int_{((w\sigma u w^{-1} N_{P,w\sigma w^{-1}})\cap S_{p'})(\AAA)} \varphi(\Int(\pi x)^{-1}(y)) \, dy
  \end{align*}
où la somme sur $\pi$ est prise sur 
   \begin{align*}
    \pi\in (P^\theta\cap w G_\sigma w^{-1})(F)\back P^\theta(F). 
   \end{align*}
Soit $T\in T_0+\overline{\ago_0^+}$ un paramètre de troncature. On injecte l'expression de $\tlK_{P,\of,\varphi}(x)$ dans la somme 
  \begin{align}\label{eq:sumtlK}
  \sum_{\delta\in P^\theta(F)  \back G^\theta(F)} \hat\tau_P(H_P(\delta x)-T_P) \tlK_{P,\of,\varphi}(\delta x)
  \end{align}
qui apparaît dans la définition \eqref{eq:tlKToS} de $\tlK_{\of,\varphi}^T(x)$. On change l'ordre des sommes sur $\delta$ et $w$ et puis on combine les sommes sur $\pi$ et $\delta$. L'expression \eqref{eq:sumtlK} devient 
   \begin{align*}
    \sum_{w\in W^\theta(\ago_1 ; P, G_\sigma)} \sum_{\xi} \hat\tau_P(H_P(\xi x)-T_P) \sum_{u\in (\uc_{G_\sigma}\cap M_{P_w}\cap S^\circ)(F)} \int_{((w\sigma u w^{-1} N_{P,w\sigma w^{-1}})\cap S_{p'})(\AAA)} \varphi(\Int(\xi x)^{-1}(y)) \, dy
  \end{align*}
où la somme sur $\xi$ est prise sur 
   \begin{align*}
    \xi\in (P^\theta\cap w G_\sigma w^{-1})(F)\back G^\theta(F). 
   \end{align*}
En remplaçant $\xi$ par $w\xi$ et utilisant \eqref{eq:tauw}, on voit que \eqref{eq:sumtlK} égale 
   \begin{align*}
    \sum_{w\in W^\theta(\ago_1 ; P, G_\sigma)} \sum_{\xi} \hat\tau_{P_w}(H_{P_w}(\xi x)-T_{P_w}) \sum_{u\in (\uc_{G_\sigma}\cap M_{P_w}\cap S^\circ)(F)} \int_{((\sigma u N_{P_w,\sigma})\cap S_{p'})(\AAA)} \varphi(\Int(\xi x)^{-1}(y)) \, dy
  \end{align*}
  où la somme sur $\xi$ est prise sur 
   \begin{align*}
    \xi\in (P_w^\theta\cap G_\sigma)(F)\back G^\theta(F). 
   \end{align*}
Soit 
   \begin{align*}
    R=P_{w,\sigma}.
   \end{align*}
D'après \eqref{eq:W-2eme}, c'est un sous-groupe parabolique relativement standard de $G_\sigma$.  
On obtient finalement que \eqref{eq:sumtlK} est égale à 
   \begin{align*}
    \sum_{w\in W^\theta(\ago_1 ; P, G_\sigma)} \sum_{\xi\in R^\theta(F)\back G^\theta(F)} \hat\tau_{P_w}(H_{P_w}(\xi x)-T_{P_w}) \sum_{u\in (\uc_{M_R}\cap S^\circ)(F)} \int_{((\sigma u N_R)\cap S_{p'})(\AAA)} \varphi(\Int(\xi x)^{-1}(y)) \, dy. 
  \end{align*}
  
On peut maintenant réécrire \eqref{eq:tlKToS}. Pour tout $R\in \fc^{G_\sigma}(P_{1,\sigma}^\theta)$ soit 
\begin{align}\label{eq:KRunip}
    K_{R,\varphi}^{\unip}(x) =      \sum_{u\in (\uc_{M_R}\cap S^\circ)(F)} \int_{((u N_R)\cap S^\circ)(\AAA)} \varphi(\Int(x^{-1})(\sigma y)) \, dy. 
\end{align}
D'après le lemme \ref{lem:var-unip}, on a $(\sigma u N_R)\cap S_{p'}=\sigma((u N_R)\cap S^\circ)$. Alors $K_{R,\varphi}^{\unip}(x)$ égale 
\begin{align*}
    \sum_{u\in (\uc_{M_R}\cap S^\circ)(F)} \int_{((\sigma u N_R)\cap S_{p'})(\AAA)} \varphi(\Int(x^{-1})(y)) \, dy. 
\end{align*}
Il s'ensuit que $\tlK_{\of,\varphi}^T(x)$ est égale à la somme sur $R\in \fc^{G_\sigma}(P_{1,\sigma}^\theta)$ et $\xi\in R^\theta(F)\back G^\theta(F)$ du produit de $K_{R,\varphi}^{\unip}(\xi x)$ avec 
   \begin{align*}
     \sum_{P, w} \eps_{P_w}^G \hat\tau_{P_w}(H_{P_w}(\xi x)-T_{P_w}) 
  \end{align*}
où la somme sur $P$ et $w$ est prise sur l'ensemble 
   \begin{align}\label{eq:ensPs}
     \{P\in\fc(P_0^\theta), w\in W^\theta(\ago_1; P, G_\sigma) | P_{w,\sigma}=R\}. 
  \end{align}
Pour tout $R\in \fc^{G_\sigma}(P_{1,\sigma}^\theta)$ on définit l'ensemble (éventuellement vide vu le lemme \ref{lem:Psigma})
    \begin{align*}
      \fc_R(M_1)=\{Q\in\fc(M_1) | Q_\sigma=R\}. 
    \end{align*}
    
\begin{proposition}
    Soit $R\in \fc^{G_\sigma}(P_{1,\sigma}^\theta)$. L'application 
    \begin{align}\label{eq:art189}
        (P,w)\mapsto P_w
    \end{align}
    induit une bijection de \eqref{eq:ensPs} sur $\fc_R(M_1)$.
\end{proposition}

\begin{preuve}
    Soit $Q\in\fc_R(M_1)$. Il existe un élément unique $P\in\fc(P_0^\theta)$ et un élément $w\in W^\theta$ tels que $Q=P_w$. On exige que $w^{-1}\alpha$ soit positive pour toute racine $\alpha\in\Delta_0^{P^\theta}$, de sorte que $w$ soit également uniquement déterminé. Comme $w\ago_1$ contient $w\ago_Q=\ago_P$, on trouve que $w\ago_1$ est de la forme $\ago_2$ pour un élément $P_2\in\fc^P(P_0^\theta)$ : sinon, $w^{-1}$ enverrait une combinaison linéaire positive non simple de racines de $\Delta_0^{P^\theta}$ vers une racine simple de $\Delta_0^{P_1^\theta}$, ce qui constitue une contradiction. De plus, puisque $P=wQw^{-1}\supset wRw^{-1}$, on a $P^\theta\supset (wRw^{-1})^\theta\supset wP_{1,\sigma}^\theta w^{-1}$. Avec l'hypothèse sur $w$, on voit que $w\beta$ est positive pour toute racine positive $\beta$ de $(G_\sigma^\theta, A_1)$. Il s'ensuit que la restriction de $w$ à $\ago_1$ définit un élément unique de $W^\theta(\ago_1;P,G_\sigma)$. De cette manière, on a trouvé une réciproque de l'application \eqref{eq:art189}. 
\end{preuve}

    En résumé, on a établi l'analogue suivant de \cite[lemme 3.1]{ar-family}. 

  \begin{lemme}\label{lem:art3.1}
  Soit $\varphi\in C_c^\infty(S_{p'}(\AAA))$ et $x\in G^\theta(\AAA)$. Pour tout $T\in T_0+\overline{\ago_0^+}$, $\tlK_{\of,\varphi}^T(x)$ égale la somme sur $R\in \fc^{G_\sigma}(P_{1,\sigma}^\theta)$ et $\xi\in R^\theta(F)\back G^\theta(F)$ du produit de $K_{R,\varphi}^{\unip}(\xi x)$ avec 
   \begin{align}\label{eq:sumfc_R}
     \sum_{P\in\fc_R(M_1)} \eps_P^G \hat\tau_P(H_P(\xi x)-T_P). 
  \end{align}
\end{lemme}
\end{paragr}

\begin{paragr} \label{S:noyaux}
On continue avec les notations de \S \ref{S:prelim}. En particulier, pour $\varphi\in C_c^\infty(S_{p'}(\AAA))$, $ x \in G^\theta(\AAA) $ et $R\in \fc^{G_\sigma}(P_{1,\sigma}^\theta)$, on a défini $K_{R,\varphi}^{\unip}(x)$ en \eqref{eq:KRunip}. 
Pour $ u\in (\uc_{M_R}\cap S^\circ)(F) $, on écrit $ u = \exp(U) $ avec $ U \in \nc_{\mgo_R\cap \sgo}(F) $. On identifie $((u N_R)\cap S^\circ)(\AAA)$ à $ U + (\ngo_R \cap \sgo)(\AAA) $ via l'exponentielle en admettant qu'on a préservation des mesures. On a montré dans \S \ref{S:exp} qu'il existe une fonction $f\in C_c^\infty((G^\theta\times\sgo_\sigma)(\AAA))$ dépendante de $\varphi$ telle que 
\begin{align}\label{eq:var-KRunip}
    K_{R,\varphi}^{\unip}(x) = \sum_{U \in \nc_{\mgo_R\cap \sgo}(F)} \int_{(\ngo_R \cap \sgo)(\AAA)} \varphi (\Int(x^{-1})(\sigma \exp(U+V))) \, dV 
    \\ \nonumber
    = \sum_{U \in \nc_{\mgo_R\cap \sgo}(F)} \int_{(\ngo_R \cap \sgo)(\AAA)} \int_{G^\theta_\sigma(\AAA)} f(x^{-1}z^{-1}, \Ad(z)(U + V)) \, dz dV.
\end{align}

Pour tout $ h \in G^\theta(\AAA) $, on peut introduire la fonction  
\begin{align}\label{eq:defPsi}
    \Psi_h(X) = \int_{G^\theta_\sigma(\AAA)} f(h^{-1}z^{-1}, \Ad(z)(X)) \, dz,
\end{align}
qui est lisse sur $ \sgo_\sigma(\AAA) $. Pour $ x \in G^\theta(\AAA) $ et $ h \in G^\theta_\sigma(\AAA) $ on a  
\begin{align*}
    \Psi_{hx}(X) = \Psi_x(\Ad(h^{-1})(X)).
\end{align*}  
On a donc  
\begin{align}\label{eq:noyaux}
    K_{R,\varphi}^{\unip}(x) = \sum_{U \in \nc_{\mgo_R\cap \sgo}(F)} \int_{(\ngo_R \cap \sgo)(\AAA)} \Psi_x(U + V) \, dV.
\end{align}
\end{paragr}

\begin{paragr}
    Soit $H$ un  groupe réductif connexe défini sur $F$. Soit $X_\ast(H)$ le groupe des cocaractères algébriques de $H$ définis sur $F$. Pour tout $\lambda\in X_\ast(H)$, on note $\PP_H(\lambda)$ le sous-groupe fermé  formé des $x\in H$ tels que $\lim\limits_{a\to 0} \lambda(a)x\lambda(a)^{-1}$ existe. C'est un sous-groupe parabolique de $H$ défini sur $F$ (cf. \cite[\S 13.4]{SpringerLAG}). 
\end{paragr}

\begin{paragr}[Lemmes combinatoires.]\label{S:lem-combi} ---
	La prochaine étape est de transformer la somme \eqref{eq:sumfc_R} sur des sous-groupes paraboliques de $G$ en une somme sur des sous-groupes paraboliques de $G_\sigma$. 

    Soit $\sigma\in S(F)$ un élément semi-simple. Supposons que $\sigma\in (M_1\cap S)(F)$ est $(M_1,\theta)$-anisotrope avec $M_1\in\lc(M_0)$ autrement dit la classe de $M_1^\theta(F)$-conjugaison de $\sigma$ ne rencontre pas $M(F)$ pour tout sous-groupe de Levi $M$ semi-standard et propre de $M_1$. Comme au début de \S \ref{S:prelim}, on obtient une paire symétrique $(G_\sigma,G_\sigma^\theta)$ avec les sous-groupes de Levi $M_{1,\sigma}$ et $M_{1,\sigma}^\theta$ définis sur $F$ et minimaux respectivement de $G_\sigma$ et $G_\sigma^\theta$. D'après la proposition \ref{prop:A_P_sigma}, on a encore \eqref{eq:A_1}. Un sous-groupe parabolique de $G_\sigma$ sera dit semi-standard s'il contient $M_{1,\sigma}$. Soit $\fc^{G_\sigma}(M_{1,\sigma})$ l'ensemble de sous-groupes paraboliques de $G_\sigma$ qui sont semi-standard. On note 
     \begin{align}\label{eq:defFGsigthe}
      \fc^{G_\sigma}(M_{1,\sigma},\theta)=\{R\in\fc^{G_\sigma}(M_{1,\sigma}) |\theta R\theta=R\}. 
    \end{align}    
    
   \begin{lemme}\label{lem:Psigma}
   L'application 
     \begin{align*}
      P\mapsto P_\sigma
    \end{align*}
induit une application surjective de $\fc(M_1)$ dans $\fc^{G_\sigma}(M_{1,\sigma},\theta)$. 
   \end{lemme}
    
    \begin{preuve}
	Comme $\sigma\in S$, on a $\theta\sigma\theta=\sigma^{-1}$. Pour tout $P\in\fc(M_1)$ on a alors 
    \begin{align*}
      \theta (P_\sigma)\theta=P_{\theta\sigma\theta}=P_{\sigma^{-1}}=P_\sigma 
    \end{align*}
    c'est-à-dire que $P_\sigma\in\fc^{G_\sigma}(M_{1,\sigma},\theta)$. Soit $R\in\fc^{G_\sigma}(M_{1,\sigma},\theta)$. D'après \cite[lemme 3.3]{HelWan}, il existe $\lambda\in X_\ast(A_{M_{1,\sigma}}^\theta)$ tel que $R=\PP_{G_\sigma}(\lambda)$. Soit $P=\PP_G(\lambda)$. Par \eqref{eq:A_1}, on a $\lambda\in X_\ast(A_{M_1})$ et alors $P\in\fc(M_1)$. Comme $R=P\cap G_\sigma$, la surjectivité est claire. 
    \end{preuve}
    
    Soit $R\in\fc^{G_\sigma}(M_{1,\sigma},\theta)$. Pour tout $\RR$-espace vectoriel $W$ sur lequel $\theta$ agit, on note $W^\theta$ le sous-espace de vecteurs fixés par $\theta$. Ainsi on a 
    \begin{align}\label{eq:dec-aRtheta}
      \ago_R^\theta=\ago_G\oplus(\ago_{G_\sigma}^G)^\theta\oplus(\ago_R^{G_\sigma})^\theta. 
    \end{align} 
    On définit les ensembles $\fc_R^0(M_1)\subset\fc_R(M_1)\subset\bar\fc_R(M_1)$ suivants : 
    \begin{align*}
      &\fc_R^0(M_1)=\{P\in\fc(M_1) | P_\sigma=R, \ago_P=\ago_{R}^\theta\}, \\
      &\fc_R(M_1)=\{P\in\fc(M_1) |P_\sigma=R\}, \\
      &\bar\fc_R(M_1)=\{P\in\fc(M_1) | P_\sigma\supset R\}. 
    \end{align*}
	
Puisque on n'utilisera que les sous-groupes paraboliques semi-standard de $G$ et $G_\sigma$ dans le reste de ce paragraphe, par la proposition \ref{prop:ssr-Prd}, sans perte de généralité, on peut supposer $\sigma=x(A,r,t)$ où $A$ est de la forme de \eqref{eq:bloc-diag-A} comme à la fin de \S \ref{S:anisotrope}. On adoptera les notations qui y sont introduites. D'après \eqref{eq:Gsigma} et \eqref{eq:M1sigma}, on a 
\begin{align}\label{eq:Risom}
    R\simeq\prod_{1\leq \alpha\leq i} (R_\alpha\times R_\alpha)\times\prod_{1\leq \beta\leq j} R'_\beta  
     \times R_+\times R_-
\end{align}
où $R_\alpha$, $R'_\beta$, $R_+$ et $R_-$ sont des sous-groupes paraboliques semi-standard respectivement de 
\begin{align*}
    \Res_{L_\alpha/F} GL_{s_\alpha,D_\alpha},  \Res_{K'_\beta/F} GL_{k_\beta,D'_\beta\otimes_{L'_\beta} K'_\beta},  GL_{p+q-2m-r-t,D} \text{ et } GL_{r+t,D} 
\end{align*}
pour tous $1\leq \alpha\leq i$ et $1\leq \beta\leq j$. 
Par \eqref{eq:M1} et \eqref{eq:centAsigthe}, il existe un unique élément $\tilde R$ de $\fc^{\Cent_G(A_{G_\sigma}^\theta)}(M_1)$ tel que $\tilde R_\sigma=R$. Effectivement, l'application 
\begin{align}\label{eq:bijtilde}
    R\mapsto \tilde R
\end{align}
induit une bijection de $\fc^{G_\sigma}(M_{1,\sigma},\theta)$ sur $\fc^{\Cent_G(A_{G_\sigma}^\theta)}(M_1)$. De plus, on a $A_R^\theta=A_{\tilde R}$. En particulier, on a $\wt{G_\sigma}=\Cent_G(A_{G_\sigma}^\theta)$ et $A_{G_\sigma}^\theta=A_{\wt{G_\sigma}}$. On peut alors réécrire la décomposition \eqref{eq:dec-aRtheta} en 
\begin{align*}
    \ago_{\tilde R}=\ago_G\oplus\ago_{\wt{G_\sigma}}^G\oplus\ago_{\tilde R}^{\wt{G_\sigma}}. 
\end{align*}
Pour $P\in\fc(M_1)$, on a $P^\ast=P\cap\Cent_G(A_{G_\sigma}^\theta)\in\fc^{\Cent_G(A_{G_\sigma}^\theta)}(M_1)$ et $P_\sigma=P^\ast_\sigma$. Autrement dit, $P_\sigma=R$, resp. $P_\sigma\supset R$, si et seulement si $P^\ast=\tilde R$, resp. $P^\ast\supset\tilde R$. Si $P_\sigma=R$, on a alors $A_{P^\ast}=A_R^\theta$. 

Pour tout $P\in \bar{\fc}_R(M_1)$, on pose 
	\begin{align*}
	\fc^{P_\sigma}(R,\theta)=\{Q\in\fc^{G_\sigma}(M_{1,\sigma},\theta) | R\subset Q\subset P_\sigma\}. 
	\end{align*} 
Soit $\fc^{P^\ast}(\tilde R)$ l'ensemble des sous-groupes paraboliques de $\Cent_G(A_{G_\sigma}^\theta)$ inclus dans $P^\ast$ et contenant $\tilde R$. L'application \eqref{eq:bijtilde} induit également une bijection de $\fc^{P_\sigma}(R,\theta)$ sur $\fc^{P^\ast}(\tilde R)$. 

\begin{lemme}\label{lem:existFR0}
    L'ensemble $\fc_R^0(M_1)$ est non-vide. 
\end{lemme}

\begin{preuve}
Soit $Q\in\fc(M_1)$ un sous-groupe parabolique tel que  $M_Q=\Cent_G(A_{G_\sigma}^\theta)$. Alors $P=\tilde R N_Q\in\fc(M_1)$ vérifie $P_\sigma=P^\ast_\sigma=\tilde R_\sigma=R$ et $A_P=A_{P^\ast}=A_R^\theta$.
\end{preuve}

\begin{lemme}\label{lem:convexe}
Soit $M$ le centralisateur de $A_{R}^\theta$ dans $G$. 
  \begin{enumerate}
  \item On a $M=M_{\tilde R}$ et $A_M=A_R^\theta$. 
  
  \item Le groupe $M$ est l'unique élément de $\lc(M_1)$ tel que 
    \begin{align*}
      \fc_R^0(M_1)\subset \pc(M). 
    \end{align*} 

  \item L'ensemble $\fc_R^0(M_1)$ est une famille convexe dans $\pc(M)$ au sens de l'appendice A de \cite{CZ}. 
\end{enumerate}
\end{lemme}

    \begin{preuve}
    Comme $A_R^\theta=A_{\tilde R}$, on a $M=\Cent_G(A_R^\theta)=\Cent_G(A_{\tilde R})=M_{\tilde R}$ et alors $A_M=A_{\tilde R}=A_R^\theta$ d'où l'assertion 1. 
    
    Puisque $A_R^\theta$ est connexe, pour tout $P\in\fc_R^0(M_1)$, on a $A_R^\theta=A_P$ et alors 
    $M=\Cent_G(A_R^\theta)=\Cent_G(A_P)=M_P$
    d'où l'assertion 2. 
    
	 Avec notre discussion ci-dessus, on a 
	\begin{align*}
      \fc_R^0(M_1)=\{P\in\fc(M_1) | P_\sigma=R, \ago_P=\ago_{R}^\theta\}=\{P\in\fc(M_1) | P^\ast=\tilde R, A_P=A_R^\theta\}. 
    \end{align*}
    D'après les premières deux assertions, on peut écrire 
	\begin{align}\label{eq:ensFR0}
      \fc_R^0(M_1)=\{P\in\pc(M) | P^\ast=\tilde R\}. 
    \end{align}
	Pour tout $P\in\pc(M)$, notons $\Phi(G, A_M)$, resp. $\Phi(P, A_M)$, l'ensemble des racines réduites de $A_M$ sur $G$, resp. $P$. Pour  toute $\alpha\in\Phi(G, A_M)$, soit $H(\alpha)^+$ l'ensemble de $P\in\pc(M)$ tels que $\alpha\in\Phi(P, A_M)$. Notons $\Phi(\tilde R, A_{\tilde R})\subset\Phi(G, A_M)$ le sous-ensemble des racines réduites de $A_{\tilde R}=A_M$ sur $\tilde R$. On a 
    \begin{align*}
      \fc_R^0(M_1)=\bigcap_{\alpha\in\Phi(\tilde R, A_{\tilde R})} H(\alpha)^+. 
    \end{align*} 
L'assertion 3 résulte alors de \cite[lemme A.0.3.1]{CZ}. 
    \end{preuve}

\begin{lemme}\label{lem:art5.1}
La fonction 
    \begin{align}\label{eq:sum-barF_R}
	\sum_{P\in\bar{\fc}_R(M_1)} \eps_P^G \hat\tau_P(X), \forall  X\in\ago_{R}^\theta
    \end{align}
est égale à la fonction caractéristique de l'ensemble 
    \begin{align}\label{eq:ens-art5.1}
	\{X\in\ago_G\oplus(\ago_{R}^{G_\sigma})^\theta | \varpi(X)\leq0, \forall  \varpi\in\hat\Delta_{R}^{G_\sigma}\}. 
    \end{align}
\end{lemme}

    \begin{preuve}
Soit $M=M_{\tilde R}$. Il existe $\lambda\in X_\ast(A_M)$ tel que $\tilde R=\PP_{\Cent_G(A_{G_\sigma}^\theta)}(\lambda)$. D'après notre discussion, on peut écrire 
    \begin{align*} 
	\bar\fc_R(M_1)=\{P\in\fc(M) | P^\ast\supset\tilde R\}. 
    \end{align*}
Pour $P\in\bar\fc_R(M_1)$, il existe $\mu\in X_\ast(A_P)$ tel que $P=\PP_G(\mu)$. On a $\PP_{\Cent_G(A_{G_\sigma}^\theta)}(\mu)=P^\ast\supset\tilde R$. Notons que la séquence $\PP_G(\lambda \mu^n), n\in\NN$ se stabilise lorsque l'entier $n$ est assez grand. Soit $Q=\lim\limits_{n\to\infty}\PP_G(\lambda \mu^n)$. Par \eqref{eq:ensFR0}, on trouve que $Q\in \fc_R^0(M_1)$ et $Q\subset P$. On a montré que 
    \begin{align}\label{eq:barcontient0} 
	\bar\fc_R(M_1)=\{P\in\fc(M) | \exists  Q\in \fc_R^0(M_1), Q\subset P\}. 
    \end{align} 
    Il résulte alors de \cite[lemme A.0.5.3]{CZ} et du lemme \ref{lem:convexe} que \eqref{eq:sum-barF_R} est égale à la fonction caractéristique de l'ensemble 
    \begin{align}\label{eq:ens-barF_R}
	\{X\in\ago_M | \varpi(X)\leq0, \forall P\in \fc_R^0(M_1), \forall  \varpi\in\hat\Delta_P\}. 
    \end{align}
   Il est évident que cet ensemble est invariant par translation de $\ago_G$. 
   
   Comme dans la preuve du lemme \ref{lem:existFR0}, soit $Q\in\fc(M_1)$ un sous-groupe parabolique tel que  $M_Q=\wt{G_\sigma}$ et soit $P^\dagger=\tilde R N_Q\in \fc_R^0(M_1)$. Notons $\ov{Q}$ le sous-groupe parabolique semi-standard opposé à $Q$. Alors $P^\ddagger=\tilde R N_{\ov{Q}}\in\fc_R^0(M_1)$. Soit $X\in\ago_M^G$ un élément de \eqref{eq:ens-barF_R}. On écrit $X=X_1+X_2$ selon la décomposition $\ago_M^G=\ago_{\wt{G_\sigma}}^G\oplus\ago_M^{\wt{G_\sigma}}$. Comme $\varpi_\alpha(X)\leq0$ pour toute $\alpha\in\Delta_{P^\dagger}$, en particulier, on a $\varpi_\alpha(X_1)=\varpi_\alpha(X)\leq0$ pour toute $\alpha\in\Delta_{P^\dagger}-\Delta_{P^\dagger}^Q$ c'est-à-dire $\varpi(X_1)\leq0$ pour tout $\varpi\in\hat\Delta_Q$. De même, en remplaçant $P^\dagger$ par $P^\ddagger$, on a $\varpi(X_1)\leq0$ pour tout $\varpi\in\hat\Delta_{\ov{Q}}$. On en déduit  $\varpi(X_1)=0$ pour tout $\varpi\in\hat\Delta_Q$ autrement dit $X_1=0$. On a montré que la projection dans $\ago_{\wt{G_\sigma}}^G=(\ago_{G_\sigma}^G)^\theta$ de tout élément de \eqref{eq:ens-barF_R} s'annule. 
   
   Soit $X\in\ago_M^{\wt{G_\sigma}}=(\ago_{R}^{G_\sigma})^\theta$ un élément de \eqref{eq:ens-barF_R}. La propriété $\varpi_\alpha(X)\leq0$ pour toute $\alpha\in\Delta_{P^\dagger}^Q\subset\Delta_{P^\dagger}$ implique $\varpi(X)\leq0$ pour tout $\varpi\in\hat\Delta_{P^\dagger}^Q=\hat\Delta_{\tilde R}^{\wt{G_\sigma}}$. On a prouvé que \eqref{eq:ens-barF_R} est un sous-ensemble de \eqref{eq:ens-art5.1}. 
   
   Pour tout cône $C$ de l'espace euclidien $\ago_M^G$, cf. \S \ref{S:Haar}, notons $C^\vee$ le cône dual et $-C=\{-X | X\in C\}$. Pour deux sous-ensembles non-vides $A, B\subset\ago_M^G$, soit $A+B$ leur somme de Minkowski. Soit $P\in \fc_R^0(M_1)$. D'après \eqref{eq:ensFR0}, on a $P\cap \wt{G_\sigma}=\tilde R$ et $\ago_P^{G,+}\subset\ago_{\wt{G_\sigma}}^G+\ago_{\tilde R}^{\wt{G_\sigma},+}$.
   Mais  
    \begin{align}\label{eq:pf5.1-cone1}
	-\ov{\ago_P^{G,+}}^\vee=\{X\in\ago_M^G | \varpi(X)\leq0, \forall  \varpi\in\hat\Delta_P\} 
    \end{align}
    et
    \begin{align}\label{eq:pf5.1-cone2}
	-\ov{\ago_{\wt{G_\sigma}}^G+\ago_{\tilde R}^{\wt{G_\sigma},+}}^\vee=\{X\in\ago_M^{\wt{G_\sigma}} | \varpi(X)\leq0, \forall  \varpi\in\hat\Delta_{\tilde R}^{\wt{G_\sigma}}\}. 
    \end{align}
    On en déduit que \eqref{eq:pf5.1-cone2} est un sous-ensemble de \eqref{eq:pf5.1-cone1}. Alors \eqref{eq:ens-art5.1} est un sous-ensemble de \eqref{eq:ens-barF_R}. 
    
    En résumé, on a établi l'égalité des ensembles \eqref{eq:ens-barF_R} et \eqref{eq:ens-art5.1} ce qui conclut. 
    \end{preuve} 

\begin{lemme}\label{lem:art5.2}
Pour $X\in\ago_R^\theta$ on a 
    \begin{align*}
	\sum_{P\in\fc_R(M_1)} \eps_P^G \hat\tau_P(X) = 
\left\{ \begin{array}{ll}
(-1)^{\dim(\ago_R^{G_\sigma})^\theta} \hat\tau_R^{G_\sigma}(X) & \text{si $X\in\ago_G\oplus(\ago_R^{G_\sigma})^\theta$ ; }\\
0 & \text{sinon. }\\
\end{array} \right. 
    \end{align*}
\end{lemme}

\begin{preuve}
Pour tout $P\in \bar{\fc}_R(M_1)$, l'application 
\begin{align*}
    Q\mapsto\Delta_R^Q, \text{ resp. } Q\mapsto\hat\Delta_Q^{P_\sigma}, 
\end{align*}
induit une bijection de $\fc^{P_\sigma}(R,\theta)$ sur 
l'ensemble des parties $\theta$-stables de $\Delta_R^{P_\sigma}$, resp. de $\hat\Delta_R^{P_\sigma}$, ce qui est un cas particulier de \cite[lemme 2.7.1]{labWal}. D'après la formule de binôme, cf. \cite[proposition 1.1]{ar1}, on a 
  \begin{align*}
   \sum_{Q\in\fc^{P_\sigma}(R,\theta)} (-1)^{\dim(\ago_R^Q)^\theta}=
  \left\{ \begin{array}{ll}
1 & \text{si $R=P_\sigma$ ; }\\
0 & \text{sinon. }\\
\end{array} \right. 
  \end{align*}
Ainsi on obtient 
  \begin{align}\nonumber
  \sum_{P\in\fc_R(M_1)} \eps_P^G \hat\tau_P(X) = \sum_{P\in\bar\fc_R(M_1)} \eps_P^G \hat\tau_P(X) \left(\sum_{Q\in \fc^{P_\sigma}(R,\theta)} (-1)^{\dim(\ago_R^Q)^\theta}\right) \\ \label{eq:pf-art5.2}
= \sum_{Q\in\fc^{G_\sigma}(R,\theta)} (-1)^{\dim(\ago_R^Q)^\theta} \sum_{P\in\bar\fc_Q(M_1)} \eps_P^G \hat\tau_P(X). 
  \end{align}
Si $X\notin \ago_G\oplus(\ago_R^{G_\sigma})^\theta$, alors $X\notin \ago_G\oplus(\ago_Q^{G_\sigma})^\theta$ pour tout $Q\in\fc^{G_\sigma}(R,\theta)$. Dans ce cas, il résulte du lemme \ref{lem:art5.1} que  la dernière expression s'annule. Supposons maintenant que $X\in\ago_G\oplus(\ago_R^{G_\sigma})^\theta$. Notons que l'ensemble $\{\varpi\in\hat{\Delta}_R^{G_\sigma} | \varpi(X)\leq 0\}$ est un sous-ensemble $\theta$-stable de $\hat{\Delta}_R^{G_\sigma}$. Alors il existe un unique $R'\in\fc^{G_\sigma}(R,\theta)$ tel que 
	\begin{align}\label{eq:pf5.2R'}
	\hat{\Delta}_{R'}^{G_\sigma}=\{\varpi\in\hat{\Delta}_R^{G_\sigma} | \varpi(X)\leq 0\}. 
	\end{align} 
Soit $Q\in\fc^{G_\sigma}(R,\theta)$. Notons que 
\begin{align}\label{eq:pf5.2-proj}
    \sum\limits_{P\in\bar\fc_Q(M_1)} \eps_P^G \hat\tau_P(X)=\sum\limits_{P\in\bar\fc_Q(M_1)} \eps_P^G \hat\tau_P(X_Q).
\end{align}
Si cette expression est non nulle, le lemme \ref{lem:art5.1} implique que $\varpi(X)=\varpi(X_Q)\leq0$ pour tout $\varpi\in\hat{\Delta}_Q^{G_\sigma}$ autrement dit $\hat{\Delta}_Q^{G_\sigma}\subset\hat{\Delta}_{R'}^{G_\sigma}$ ou encore $R'\subset Q$. Si $R'\subset Q$, on a $\varpi(X_Q)=\varpi(X)\leq0$ pour tout $\varpi\in\hat{\Delta}_Q^{G_\sigma}\subset\hat{\Delta}_{R'}^{G_\sigma}$ ce qui implique que \eqref{eq:pf5.2-proj} vaut $1$. Il s'ensuit que l'expression \eqref{eq:pf-art5.2} devient 
  \begin{align*}
   \sum_{Q\in\fc^{G_\sigma}(R',\theta)} (-1)^{\dim(\ago_R^Q)^\theta}=
  \left\{ \begin{array}{ll}
(-1)^{\dim(\ago_R^{G_\sigma})^\theta} & \text{si $R'=G_\sigma$ ; }\\
0 & \text{sinon. }\\
\end{array} \right. 
  \end{align*}
Mais d'après \eqref{eq:pf5.2R'}, $R'=G_\sigma$ si et seulement si $\varpi(X)>0$ pour tout $\varpi\in\hat{\Delta}_R^{G_\sigma}$ c'est-à-dire que $\hat\tau_R^{G_\sigma}(X)=1$. Le lemme s'ensuit. 
\end{preuve}

\begin{lemme}\label{lem:art5.3}
Pour $X\in\ago_R^\theta$ on a 
    \begin{align*}
	\sum_{Q\in\bar\fc_R(M_1)} \eps_Q^G \tau_R^{Q_\sigma}(X) \hat\tau_Q(X) = 
\left\{ \begin{array}{ll}
1 & \text{si $R=G_\sigma$ et $X\in\ago_G$ ; }\\
0 & \text{sinon. }\\
\end{array} \right. 
    \end{align*}
\end{lemme}

\begin{preuve}
	En utilisant le lemme \ref{lem:Psigma}, on voit que l'expression de gauche égale 
	\begin{align}\label{eq:pf5.3}
	\sum_{P\in\fc^{G_\sigma}(R,\theta)} \tau_R^P(X) \sum_{Q\in\fc_P(M_1)} \eps_Q^G \hat\tau_Q(X). 
	\end{align}
	En vertu du lemme \ref{lem:art5.2}, cette expression s'annule sauf si $X\in \ago_G\oplus(\ago_R^{G_\sigma})^\theta$. Dans ce cas, pour tout $P\in\fc^{G_\sigma}(R,\theta)$, on a 
	\begin{align*}
	 \sum_{Q\in\fc_P(M_1)} \eps_Q^G \hat\tau_Q(X)=\sum_{Q\in\fc_P(M_1)} \eps_Q^G \hat\tau_Q(X_P)=(-1)^{\dim(\ago_P^{G_\sigma})^\theta} \hat\tau_P^{G_\sigma}(X) \end{align*}
	 puisque $X_P\in\ago_G\oplus(\ago_P^{G_\sigma})^\theta$. 
	Alors l'expression \eqref{eq:pf5.3} devient 
	\begin{align*}
	\sum_{P\in\fc^{G_\sigma}(R,\theta)} (-1)^{\dim(\ago_P^{G_\sigma})^\theta} \tau_R^P(X)  \hat\tau_P^{G_\sigma}(X). 
	\end{align*}
D'après le lemme combinatoire de Langlands (voir \cite[proposition 1.7.2 et lemme 2.9.2]{labWal}), la dernière somme est nulle sauf si $R=G_\sigma$ auquel cas cette somme égale $1$. 
\end{preuve}

\begin{lemme}\label{lem:art4.2}
Pour $P\in\bar\fc_R(M_1)$ et $X\in\ago_R^\theta$ on a 
    \begin{align*}
	\sum_{Q\in\bar\fc_R(M_1), Q\subset P} \eps_Q^P \tau_R^{Q_\sigma}(X) \hat\tau_Q^P(X) = 
\left\{ \begin{array}{ll}
1 & \text{si $P\in\fc_R(M_1)$ et $X\in\ago_P$ ; }\\
0 & \text{sinon. }\\
\end{array} \right. 
    \end{align*}
\end{lemme}

\begin{preuve}
	Soit $P=MN\in\bar\fc_R(M_1)$ muni de sa décomposition de Levi. Alors $M_\sigma\cap R\in\fc^{M_\sigma}(M_{1,\sigma},\theta)$ qui est définie comme un analogue de \eqref{eq:defFGsigthe} avec $G_\sigma$ remplacé par $M_\sigma$. On note 
	\begin{align*}	           \bar\fc_{M_\sigma\cap R}^M(M_1)=\{Q'\in\fc^M(M_1) | Q'_\sigma\supset M_\sigma\cap R \}. 
	\end{align*}
	Il est connu que  l'application 
	\begin{align}\label{eq:pfart4.2}
	Q\mapsto M\cap Q
	\end{align}
    induit une bijection de $\fc^P(M_1)$ sur $\fc^M(M_1)$ dont la réciproque est donnée par $Q'\mapsto Q'N$. D'une part, si $Q\in\fc^P(M_1)$ vérifie $Q_\sigma\supset R$, on a $(M\cap Q)_\sigma=M_\sigma\cap Q_\sigma\supset M_\sigma\cap R$. D'autre part, si $Q'\in\fc^M(M_1)$ vérifie $Q'_\sigma\supset M_\sigma\cap R$, comme $P_\sigma\in\fc^{G_\sigma}(R,\theta)$, on a
    \begin{align}\label{eq:pf4.2R}
        R=(M_{P_\sigma}\cap R)N_{P_\sigma}=(M_\sigma\cap R) N_\sigma
    \end{align}
    et alors $(Q'N)_\sigma=Q'_\sigma N_\sigma\supset (M_\sigma \cap R)N_\sigma=R$. On a montré que la restriction de \eqref{eq:pfart4.2} induit une bijection de l'ensemble $\{Q\in\bar\fc_R(M_1) | Q\subset P\}$ sur $\bar\fc_{M_\sigma\cap R}^M(M_1)$. 
    
La somme en question est similaire à celle du lemme \ref{lem:art5.3} mais avec $(G, G_\sigma, R)$ remplacé par $(M, M_{\sigma}, M_{\sigma}\cap R)$ qui est effectivement un analogue du produit. Elle est donc nulle sauf si $M_{\sigma}\cap R=M_{\sigma}$ et $X\in\ago_M=\ago_P$ auquel cas elle vaut $1$. Pour conclure on remarque que $M_{\sigma}\cap R=M_{\sigma}$ si et seulement si $R=P_\sigma$ d'après \eqref{eq:pf4.2R}. 
\end{preuve}
\end{paragr}

\begin{paragr}
	En suivant \cite[\S2]{arthur2}, pour tout $P\in\fc(M_0)$ on définit 
	\begin{align*}
	\Gamma_P^G(X,Y)=\sum_{Q\in\fc(P)} \eps_Q^G\tau_P^Q(X)\hat\tau_Q(X-Y), \forall X, Y\in\ago_0. 
	\end{align*}
	Cette fonction ne dépend que des projections de $X,Y$ sur $\ago_P^G$. Elle vérifie 
	\begin{align}\label{eq:tau=Gam}
	\hat\tau_P(X-Y)=\sum_{Q\in\fc(P)} \eps_Q^G\hat\tau_P^Q(X)\Gamma_Q^G(X,Y), \forall X, Y\in\ago_0. 
	\end{align}
	Pour tout $Y$ fixé, la fonction $\Gamma_P^G(\cdot,Y)$ sur $\ago_P^G$ est à support compact d'après \cite[lemme 2.1]{arthur2}. Soit 
	\begin{align*}
	c'_P(\lambda, Y)=\int_{\ago_P^G} \Gamma_P^G(X,Y) \exp(\langle \lambda, X \rangle) dX, \forall \lambda\in \ago_{P,\CC}^\ast. 
	\end{align*}
C'est une fonction entière de $\lambda$. En particulier pour tout $x\in G(\AAA)$, on pose 
	\begin{align*}
	v'_P(\lambda, x)=\int_{\ago_P^G} \Gamma_P^G(X,-H_P(x)) \exp(\langle \lambda, X \rangle) dX, \forall \lambda\in \ago_{P,\CC}^\ast. 
	\end{align*}
On notera souvent $c'_P(Y)$ et $v'_P(x)$ les valeurs en $\lambda=0$ respectivement de $c'_P(\lambda, Y)$ et $v'_P(\lambda, x)$. 
\end{paragr}

\begin{paragr}
Soit $R\in\fc^{G_\sigma}(M_{1,\sigma},\theta)$. Supposons que 
    \begin{align*}
	\yc_R=\{Y_P | P\in\fc_R^0(M_1)\}
    \end{align*}
est un ensemble de points dans $\ago_R^\theta$ qui est $A_R^\theta$-orthogonal positif au sens de \cite[(A.0.4.1)]{CZ} \footnote{Nous abusons légèrement de la notion ici, car en toute rigueur on devrait associer un point dans $\ago_R^\theta$ à chaque sous-groupe parabolique dans $\pc(M_{\tilde R})$.} (cette définition est due à \cite[\S 3]{ar-character} et on utilise le lemme \ref{lem:convexe}). Pour tout $Q\in\bar\fc_R(M_1)$, d'après \eqref{eq:barcontient0}, il existe $P\in\fc_R^0(M_1)$ contenu dans $Q$ et on définit $Y_Q$ comme la projection de $Y_P$ sur $\ago_Q$. Grâce à notre hypothèse sur $\yc_R$, $Y_Q$ est indépendant du choix de $P\subset Q$. De cette manière, on obtient pour tout $R'\in\fc^{G_\sigma}(R,\theta)$ un ensemble $A_{R'}^\theta$-orthogonal positif
    \begin{align*}
	\yc_{R'}=\{Y_Q | Q\in\fc_{R'}^0(M_1)\}. 
    \end{align*}

Soit $X\in \ago_R^\theta$. On définit 
    \begin{align*}
	\Gamma_R^G(X,\yc_R)=\sum_{R'\in\fc^{G_\sigma}(R,\theta)} \tau_R^{R'}(X)\left(\sum_{Q\in\fc_{R'}(M_1)}\eps_Q^G\hat\tau_Q(X-Y_Q)\right). 
    \end{align*}
Soit $R'\in\fc^{G_\sigma}(R,\theta)$. On pose 
	\begin{align*}
	\fc^{R'}(R,\theta)=\{H\in\fc^{G_\sigma}(M_{1,\sigma},\theta) | R\subset H\subset R'\}.	\end{align*} 
Le lemme combinatoire de Langlands (voir \cite[proposition 1.7.2 et lemme 2.9.2]{labWal}) implique que 
	\begin{align*}
	\sum_{H\in\fc^{R'}(R,\theta)} (-1)^{\dim(\ago_R^H)^\theta} \tau_R^H(X)  \hat\tau_H^{R'}(X) = 
\left\{ \begin{array}{ll}
1 & \text{si $R=R'$ ; }\\
0 & \text{sinon. }\\
\end{array} \right. 
	\end{align*}
On en déduit que 
    \begin{align}\label{eq:tau=GamR}
	\sum_{P\in\fc_R(M_1)}\eps_P^G\hat\tau_P(X-Y_P)=\sum_{R'\in\fc^{G_\sigma}(R,\theta)} (-1)^{\dim (\ago_R^{R'})^\theta} \hat\tau_R^{R'}(X) \Gamma_{R'}^G(X,\yc_{R'}). 
    \end{align}
Soit 
	\begin{align*}
	c'_R(\lambda, \yc_R)=\sum_{P\in\fc_R^0(M_1)} c'_P(\lambda, Y_P), \forall \lambda\in \ago_{R,\CC}^{\theta, \ast}. 
	\end{align*}
On notera souvent $c'_R(\yc_R)$ la valeur en $\lambda=0$ de $c'_R(\lambda, \yc_R)$. En particulier, on obtient une fonction continue $v'_R(x)$ en $x\in G(\AAA)$ lorsque $Y_P=-H_P(x)$ pour tout $P\in\fc_R^0(M_1)$. 

\begin{lemme}\label{lem:art4.1}
Le support de la fonction 
    \begin{align*}
	X\mapsto\Gamma_R^G(X,\yc_R), \forall  X\in(\ago_R^G)^\theta
    \end{align*}
est compact et dépend continûment de $\yc_R$. De plus,  
    \begin{align*}
	\int_{(\ago_R^G)^\theta} \Gamma_R^G(X,\yc_R) \exp(\langle\lambda,X\rangle) \, dX=c'_R(\lambda, \yc_R), \forall  \lambda\in \ago^{\theta,\ast}_{R,\CC}.
    \end{align*}
\end{lemme}

\begin{preuve}
	L'égalité \eqref{eq:tau=Gam} entraîne que pour tout $X\in\ago_R^\theta$, 
    \begin{align*}
	\Gamma_R^G(X,\yc_R)=\sum_{R'\in\fc^{G_\sigma}(R,\theta)} \tau_R^{R'}(X)\left(\sum_{Q\in\fc_{R'}(M_1)}\sum_{P\in\fc(Q)} \eps_Q^P\hat\tau_Q^P(X)\Gamma_P^G(X,Y_P)\right).
    \end{align*}
    Par le lemme \ref{lem:Psigma}, c'est égal à 
    \begin{align*}
	\sum_{P\in\bar\fc_R(M_1)} \Gamma_P^G(X,Y_P) \left(\sum_{Q\in\bar\fc_R(M_1), Q\subset P} \eps_Q^P\tau_R^{Q_\sigma}(X)\hat\tau_Q^P(X)\right). 
    \end{align*}
    Soit $P\in\bar\fc_R(M_1)$; on a $\ago_P\subset \ago_R^\theta$.  Soit $\mathbf{1}_{\ago_P}$ la fonction caractéristique de $\ago_P$ définie sur $\ago_R^\theta$. D'après le lemme \ref{lem:art4.2}, on obtient 
    \begin{align*}
      \Gamma_R^G(X,\yc_R)&=\sum_{P\in\fc_R(M_1)} \Gamma_P^G(X,Y_P) \mathbf{1}_{\ago_P}(X), \forall X\in\ago_R^\theta
      \\   &=	\sum_{P\in\fc_R^0(M_1)} \Gamma_P^G(X,Y_P) 
    \end{align*}
pour tout $X\in \ago_R^\theta$ en dehors d'un ensemble de mesure nulle. Pour conclure on utilise \cite[lemme 2.1]{arthur2}. 
\end{preuve}
\end{paragr}

\begin{paragr}[Réduction au cas unipotent.] ---\label{S:red-unip}
	On reprend les notations du paragraphe \ref{S:prelim}. En particulier, on fixe $0\leq p'\leq N$, $\of\in\cgo_{p'}(F)$, $P_1\in\fc(P_0^\theta)$ et  $\sigma\in (M_{P_1}\cap S_{p',\of})(F)$ un élément semi-simple  $(P_1, \theta)$-anisotrope. Pour tout $Q\in \fc^{G_\sigma}(P_{1,\sigma}^\theta)$, on pose  
    \begin{align*}
	\fc^Q(P_{1,\sigma}^\theta,\theta)=\{R\in\fc^{G_\sigma}(P_{1,\sigma}^\theta) |  R\subset Q, \theta R\theta=R\}. 
    \end{align*}
    Pour tout $v\in V_F$ soit $K_{\sigma,v}\in G_\sigma(F_v)$ un sous-groupe compact maximal en bonne position par rapport à $M_{1,\sigma}$ (au sens de \cite[p.9]{arthur2}). On le choisira comme dans \S \ref{S:choixdeK} vu la proposition \ref{prop:descendant}. On en déduit que $\Int(\theta)(K_{\sigma,v})=K_{\sigma,v}$ et que $K_{\sigma,v}^\theta=K_{\sigma,v}\cap G_\sigma^\theta(F_v)$ est aussi un sous-groupe compact maximal de $G_\sigma^\theta(F_v)$ en bonne position par rapport à $M_{1,\sigma}^\theta$. Soit $K_\sigma=\prod_{v\in V}K_{\sigma,v}$ et $K_\sigma^\theta=\prod_{v\in V}K_{\sigma,v}^\theta$. Soit $R\in\fc^{G_\sigma}(M_{1,\sigma},\theta)$. On a alors une application $H_R : G_\sigma(\AAA)\to\ago_R$ de la manière habituelle selon la décomposition $G_\sigma(\AAA)=R(\AAA)K_\sigma$. Par restriction on obtient une application $G_\sigma^\theta(\AAA)\to\ago_R^\theta$. Pour tout $x\in G_\sigma^\theta(\AAA)$, soit $k_R(x)$ le composant de $x$ dans $K_\sigma^\theta$ par rapport à la décomposition $G_\sigma^\theta(\AAA)=R^\theta(\AAA)K_\sigma^\theta$. Il est uniquement déterminé modulo multiplication à gauche par $R^\theta(\AAA)\cap K_\sigma^\theta$. On note 
    \begin{align*}
	[M_R^\theta]^R=M_R^\theta(F) \back M_R^\theta(\AAA)\cap M_R(\AAA)^1. 
    \end{align*}       
Soit $T'\in\ago_{M_{1,\sigma}}^\theta\cap\ov{\ago_{P_{1,\sigma}}^+}$ un point suffisamment positif par rapport à $(G_\sigma,P_{1,\sigma})$. 
On note $T'_R$ la projection de $wT'$ sur $\ago_R^\theta$, où $w\in W(G_\sigma, A_{M_{1,\sigma}}^\theta)$ est tel que $P_{1,\sigma}\subset R_w$. 

\begin{remarque}\label{rmq:combi}
    Il nous faut vérifier l'existence de tels $T'$ et $w$. En particulier, si $(G_\sigma, G_\sigma^\theta)$ est de type 1 ou 2, resp. de type 2 ou 3, dans la proposition \ref{prop:descendant}, on peut supposer que $T'\in\ago_{M_{1,\sigma}^\theta}\cap\ov{\ago_{P_{1,\sigma}^\theta}^+}$, resp. $T'\in\ago_{M_{1,\sigma}}\cap\ov{\ago_{P_{1,\sigma}}^+}$, est un point suffisamment positif par rapport à $(G_\sigma^\theta,P_{1,\sigma}^\theta)$, resp. à $(G_\sigma,P_{1,\sigma})$, et on trouve que  $W(G_\sigma, A_{M_{1,\sigma}}^\theta)$ s'identifie à $W(G_\sigma^\theta, A_{M_{1,\sigma}^\theta})$, resp. $W(G_\sigma, A_{M_{1,\sigma}})$. Nous avons combiné ces éléments dans l'hypothèse ci-dessus sur $T'$ et $w$ d'où leur existence. 
\end{remarque}
    
	Soit $\varphi\in C_c^\infty(S_{p'}(\AAA))$ et $\eta: G(\AAA) \to \CC^\times$ un caractère unitaire comme dans \S \ref{S:JchiT}. D'après le théorème \ref{thm:cv-tlK} et les lemmes \ref{lem:art3.1} et \ref{lem:Psigma}, pour tout $T\in T_0+\overline{\ago_0^+}$ suffisamment positif avec $\|T^G\|$ assez grand (par rapport au support de $\varphi$), on a 
    \begin{align}\label{eq:var-prelim}
	J_\of^T(\eta, \varphi)=\int_{[G^\theta]^G} \tlK_{\of,\varphi}^T(x) \eta(x) \, dx
    \end{align}	
où $\tlK_{\of,\varphi}^T(x)$ égale la somme sur $R\in \fc^{G_\sigma}(P_{1,\sigma}^\theta,\theta)$ et $\xi\in R^\theta(F)\back G^\theta(F)$ du produit de $K_{R,\varphi}^{\unip}(\xi x)$ avec \eqref{eq:sumfc_R}. Comme dans \cite[\S6]{ar-family}, on fera d'abord formellement des changements de variables dont la justification sera reportée. On change la somme et intégrale sur 
    \begin{align*}
	(\xi,x)\in (R^\theta(F)\back G^\theta(F)) \times [G^\theta]^G
    \end{align*}	
en une somme et intégrale sur 
    \begin{align*}
	(\delta,x,g)\in (R^\theta(F)\back G_\sigma^\theta(F)) \times (G_\sigma^\theta(F)\back (G_\sigma^\theta(\AAA)\cap G(\AAA)^1)) \times (G_\sigma^\theta(\AAA)\back G^\theta(\AAA)). 
    \end{align*}	
    Soit $P\in\fc_R(M_1)$. Comme $H_P(\delta x g)$ est égal à la projection sur $\ago_P$ de $H_R(\delta x)+H_P(k_R(\delta x) g)$, on a
    \begin{align*}
      \hat\tau_P(H_P(\delta x g)-T_P)= \hat\tau_P(H_R(\delta x)-T'_R-Y_P^{T,T'}(\delta x, g)) 
    \end{align*}
où l'on pose
   \begin{align*}
     Y_P^{T,T'}(\delta x, g)=-H_P(k_R(\delta x) g)+T_P-T'_R. 
  \end{align*}
D'après \cite[lemme 3.6]{ar-character} et le fait que $T\in\overline{\ago_0^+}$, l'ensemble 
   \begin{align*}
     \yc_R^{T,T'}(\delta x, g)=\{Y_P^{T,T'}(\delta x, g) | P\in\fc_R^0(M_1)\}
  \end{align*}
  est $A_R^\theta$-orthogonal positif. Ainsi on déduit de \eqref{eq:tau=GamR} que la somme
  \begin{align*}
     \sum_{P\in\fc_R(M_1)} \eps_P^G \hat\tau_P(H_P(\delta x g)-T_P)
  \end{align*}
  est égale à la somme sur $R'\in\fc^{G_\sigma}(R,\theta)$ de 
    \begin{align}\label{eq:art(6.2)}
	(-1)^{\dim (\ago_R^{R'})^\theta} \hat\tau_R^{R'}(H_R(\delta x)-T'_R) \Gamma_{R'}^G(H_{R'}(\delta x)-T'_{R'},\yc_{R'}^{T,T'}(\delta x, g)). 
    \end{align}
L'expression $J_\of^T(\eta, \varphi)$ devient l'intégrale sur $g\in G_\sigma^\theta(\AAA)\back G^\theta(\AAA)$ et $x\in G_\sigma^\theta(F)\back (G_\sigma^\theta(\AAA)\cap G(\AAA)^1)$ et la somme sur $R'\in\fc^{G_\sigma}(P_{1,\sigma}^\theta, \theta)$, $R\in\fc^{R'}(P_{1,\sigma}^\theta, \theta)$ et $\delta\in R^\theta(F)\back G_\sigma^\theta(F)$ du produit de $K_{R,\varphi}^{\unip}(\delta xg)$, \eqref{eq:art(6.2)} et $\eta(xg)$. 

Ensuite, la somme sur $\delta\in R^\theta(F)\back G_\sigma^\theta(F)$ se décompose en une double somme sur 
   \begin{align*}
     (\mu, \xi)\in (R^\theta\cap M_{R'^\theta})(F)\back M_{R'^\theta}(F) \times R'^\theta(F)\back G_\sigma^\theta(F). 
  \end{align*}
On remplace la somme et intégrale sur 
   \begin{align*}
     (\xi, x)\in R'^\theta(F)\back G_\sigma^\theta(F)\times G_\sigma^\theta(F)\back (G_\sigma^\theta(\AAA)\cap G(\AAA)^1)
  \end{align*}
par une intégrale sur 
   \begin{align*}
     (v,a,m,k)\in [N_{R'}^\theta]\times (A_{R'}^\theta)^{G,\infty} \times [M_{R'}^\theta]^{R'}\times K_\sigma^\theta 
  \end{align*}
où un facteur $\exp(\langle -2\rho_{R'^\theta}^{G_\sigma^\theta}, H_{R'^\theta}(am) \rangle)$ est introduit. 
Notons que 
   \begin{align*}
     \hat\tau_R^{R'}(H_R(\mu vamk)-T'_R)=\hat\tau_R^{R'}(H_R(\mu m)-T'_R), \\ 
     \Gamma_{R'}^G(H_{R'}(\mu vamk)-T'_{R'},\yc_{R'}^{T,T'}(\mu vamk, g))=\Gamma_{R'}^G(H_{R'}(a)-T'_{R'},\yc_{R'}^{T,T'}(k, g)), \\
     \eta(\mu vamkg)=\eta(mkg).  
  \end{align*}
  Avec \eqref{eq:var-KRunip}, on voit que 
   \begin{align}\label{eq:art(6.3)}
     K_{R,\varphi}^{\unip}(\mu vamkg)=K_{R,\varphi}^{\unip}(\mu amkg)=\exp(\langle 2\rho_{\ngo_{R'}\cap\sgo}, H_{R'^\theta}(a) \rangle) K_{R,\varphi}^{\unip}(\mu mkg). 
  \end{align}
Rappelons que l'on a introduit une fonction $\Psi_h\in C_c^\infty(\sgo_\sigma(\AAA))$ en \eqref{eq:defPsi} dépendante de $\varphi$ pour tout $h\in G^\theta(\AAA)$. On définit également pour tout $U\in (\mgo_{R'}\cap\sgo)(\AAA)$
  \begin{align}\label{eq:def-fonctionPsi}
      \Psi_{h,R'}(U)=\int_{(\ngo_{R'}\cap\sgo)(\AAA)} \Psi_h(U+V) \, dV. 
  \end{align} 
Il est évident que $\Psi_{h,R'}\in C_c^\infty((\mgo_{R'}\cap\sgo)(\AAA))$ dépend continûment de $h$.
  On réécrit enfin l'intégrale sur $V\in (\ngo_R\cap\sgo)(\AAA)$ dans \eqref{eq:noyaux} en une intégrale double sur 
   \begin{align*}
     (V_1,V_2)\in (\ngo_R\cap\mgo_{R'}\cap\sgo)(\AAA)\oplus(\ngo_{R'}\cap\sgo)(\AAA). 
  \end{align*}
Après le changement de variables $V_2\mapsto\Ad(\mu m)(V_2)$, l'expression \eqref{eq:art(6.3)} devient le produit de $\exp(\langle 2\rho_{\ngo_{R'}\cap\sgo}, H_{R'^\theta}(am) \rangle) $ avec 
   \begin{align*}
     \sum_{U\in\nc_{\mgo_R\cap\sgo}(F)} \int_{(\ngo_R\cap\mgo_{R'}\cap\sgo)(\AAA)} \int_{(\ngo_{R'}\cap\sgo)(\AAA)} \Psi_{kg}(\Ad(\mu m)^{-1}(U+V_1)+V_2) \, dV_2 dV_1 \\
     =\sum_{U\in\nc_{\mgo_R\cap\sgo}(F)} \int_{(\ngo_R\cap\mgo_{R'}\cap\sgo)(\AAA)}  \Psi_{kg,R'}(\Ad(\mu m)^{-1}(U+V_1)) \, dV_1.
   \end{align*}
Comme $\vol([N_{R'}^\theta])=1$, l'expression $J_\of^T(\eta, \varphi)$ s'écrit alors l'intégrale et somme sur $g\in G_\sigma^\theta(\AAA)\back G^\theta(\AAA)$, $R'\in\fc^{G_\sigma}(P_{1,\sigma}^\theta, \theta)$, $k\in K_\sigma^\theta$, $a\in (A_{R'}^\theta)^{G,\infty}$ et $m\in[M_{R'}^\theta]^{R'}$ de
   \begin{align}\label{eq:art(6.4)}
     \eta(mkg) \exp(\langle 2\rho_{\ngo_{R'}\cap\sgo}-2\rho_{R'^\theta}^{G_\sigma^\theta}, H_{R'^\theta}(am) \rangle) \Gamma_{R'}^G(H_{R'}(a)-T'_{R'},\yc_{R'}^{T,T'}(k, g)) \\ \nonumber \sum_{R\in\fc^{R'}(P_{1,\sigma}^\theta, \theta)} (-1)^{\dim(\ago_R^{R'})^\theta}  \sum_{\mu\in(R^\theta\cap M_{R'^\theta})(F)\back M_{R'^\theta}(F)} \hat\tau_R^{R'}(H_R(\mu m)-T'_R) \\ \nonumber \sum_{U\in\nc_{\mgo_R\cap\sgo}(F)} \int_{(\ngo_R\cap\mgo_{R'}\cap\sgo)(\AAA)}  \Psi_{kg,R'}(\Ad(\mu m)^{-1}(U+V_1)) \, dV_1.
  \end{align}

On justifie maintenant les changements de variables ci-dessus. Par le théorème de Fubini, il suffit de montrer que \eqref{eq:art(6.4)} est absolument intégrable sur $g, R', k, a$ et $m$. Comme $\sigma$ est un élément semi-simple de $G(F)$ au sens usuel (vu le lemme \ref{lem:closedorbits}) et $S(\AAA)\subset G(\AAA)^1$, d'après \cite[lemme 6.1]{ar-family}, on peut et on va choisir un sous-ensemble compact $\Sigma$ de $G_\sigma^\theta(\AAA)\back G^\theta(\AAA)$ tel que l'intersection de $g^{-1}\sigma(\uc_{G_\sigma}\cap S)(\AAA)g$ avec le support de $\varphi$ est vide sauf si $g\in \Sigma$. Puisque $H_{R'}$ induit un homéomorphisme de $(A_{R'}^\theta)^{G,\infty}$ sur $(\ago_{R'}^G)^\theta$, le lemme \ref{lem:art4.1} implique que $\Gamma_{R'}^G(H_{R'}(a)-T'_{R'},\yc_{R'}^{T,T'}(k, g))$ s'annule pour tout $a$ en dehors d'un ensemble compact qui dépend continûment de $k$ et $g$. Par conséquent, l'expression \eqref{eq:art(6.4)} s'annule pour $(a,k,g)$ en dehors d'un ensemble compact et indépendant de $m$. Ainsi l'intégrale sur $m\in[M_{R'}^\theta]^{R'}$ de \eqref{eq:art(6.4)} est le produit de 
	\begin{align}\label{eq:JunipMT-supp}
	\eta(kg) \exp(\langle 2\rho_{\ngo_{R'}\cap\sgo}-2\rho_{R'^\theta}^{G_\sigma^\theta}, H_{R'^\theta}(a) \rangle)\Gamma_{R'}^G(H_{R'}(a)-T'_{R'},\yc_{R'}^{T,T'}(k, g)) 
	\end{align}
avec 
   \begin{align}\label{eq:JunipMT}
     J_{\nilp}^{R', T'}(\eta, 2\rho_{\ngo_{R'}\cap\sgo}-2\rho_{R'^\theta}^{G_\sigma^\theta}, \Psi_{kg,R'}) 
  \end{align} 
qui est défini en \eqref{eq:def-infi-Jo} ci-dessous. 
D'après le théorème \ref{thm:cv-infi}, l'intégrale \eqref{eq:JunipMT} est absolument convergente et majorée par $\|\Psi_{kg,R'}\|$ où $\|\cdot\|$ est une semi-norme continue sur $C_c^\infty((\mgo_{R'}\cap\sgo)(\AAA))$. Puisque le produit de \eqref{eq:JunipMT-supp} avec $\|\Psi_{kg,R'}\|$ est majoré par une fonction continue de $(a,k,g)$, \eqref{eq:art(6.4)} est absolument intégrable sur $g, R', k, a$ et $m$. On a montré que tous les changements de variables ci-dessus sont valables. De plus, on a prouvé que $J_\of^T(\eta, \varphi)$ égale 
   \begin{align}\label{eq:art(6.6)}
     \int_{G_\sigma^\theta(\AAA)\back G^\theta(\AAA)} \sum_{R'\in\fc^{G_\sigma}(P_{1,\sigma}^\theta, \theta)} \bigg(\int_{K_\sigma^\theta} \int_{(A_{R'}^\theta)^{G,\infty}} \eta(kg) \exp(\langle 2\rho_{\ngo_{R'}\cap\sgo}-2\rho_{R'^\theta}^{G_\sigma^\theta}, H_{R'^\theta}(a) \rangle) \\ \nonumber
      \Gamma_{R'}^G(H_{R'}(a)-T'_{R'},\yc_{R'}^{T,T'}(k, g)) J_{\nilp}^{R', T'}(\eta, 2\rho_{\ngo_{R'}\cap\sgo}-2\rho_{R'^\theta}^{G_\sigma^\theta}, \Psi_{kg,R'}) \, da dk\bigg) \, dg. 
  \end{align}
  
  Considérons l'intégrale 
    \begin{align}\label{eq:uTT'}
	u_{R'}^{T,T'}(k,g)=\int_{(A_{R'}^\theta)^{G,\infty}} \exp(\langle 2\rho_{\ngo_{R'}\cap\sgo}-2\rho_{R'^\theta}^{G_\sigma^\theta}, H_{R'^\theta}(a) \rangle) \Gamma_{R'}^G(H_{R'}(a)-T'_{R'},\yc_{R'}^{T,T'}(k, g)) \, da. 
    \end{align}   
En faisant un changement de variables, on voit que $u_{R'}^{T,T'}(k,g)$ égale
    \begin{align*}
	\exp(\langle 2\rho_{\ngo_{R'}\cap\sgo}-2\rho_{R'^\theta}^{G_\sigma^\theta}, T'^G_{R'} \rangle) \int_{(\ago_{R'}^G)^\theta} \exp(\langle 2\rho_{\ngo_{R'}\cap\sgo}-2\rho_{R'^\theta}^{G_\sigma^\theta}, X \rangle) \Gamma_{R'}^G(X,\yc_{R'}^{T,T'}(k, g)) \, dX. 
    \end{align*}    
D'après le lemme \ref{lem:art4.1}, c'est égale à 
    \begin{align*}
	\exp(\langle 2\rho_{\ngo_{R'}\cap\sgo}-2\rho_{R'^\theta}^{G_\sigma^\theta}, T'^G_{R'} \rangle) \sum_{P\in\fc_{R'}^0(M_1)} c'_P(2\rho_{\ngo_{R'}\cap\sgo}-2\rho_{R'^\theta}^{G_\sigma^\theta}, -H_P(kg)+T_P-T'_{R'}). 
    \end{align*}   
Ainsi le lemme \ref{lem:art2.2-li5.6} implique qu'il existe un polynôme $c_{R',P,Q}$ sur $\ago_Q^G$ pour tout $Q\in\fc(P)$ tel que 
    \begin{align*}
	u_{R'}^{T,T'}(k,g)=\exp(\langle 2\rho_{\ngo_{R'}\cap\sgo}-2\rho_{R'^\theta}^{G_\sigma^\theta}, T'^G_{R'} \rangle) \sum_{P\in\fc_{R'}^0(M_1)} \sum_{Q\in\fc(P)} \\ \nonumber 
	\exp(\langle 2\rho_{\ngo_{R'}\cap\sgo}-2\rho_{R'^\theta}^{G_\sigma^\theta}, (-H_P(kg)+T_P-T'_{R'})_Q^G \rangle) c_{R',P,Q}((-H_P(kg)+T_P-T'_{R'})_Q^G). 
    \end{align*}  
Il s'ensuit que $u_{R'}^{T,T'}(k,g)$ s'écrit 
    \begin{align*}
	\sum_{P\in\fc_{R'}^0(M_1)} \sum_{Q\in\fc(P)} \exp(\langle 2\rho_{\ngo_{R'}\cap\sgo}-2\rho_{R'^\theta}^{G_\sigma^\theta}, T_Q^G \rangle) \exp(\langle 2\rho_{\ngo_{R'}\cap\sgo}-2\rho_{R'^\theta}^{G_\sigma^\theta}, T'^Q_{R'} \rangle) \\ \nonumber 
	\exp(\langle 2\rho_{\ngo_{R'}\cap\sgo}-2\rho_{R'^\theta}^{G_\sigma^\theta}, -H_Q(kg)^G \rangle) c_{R',P,Q}((-H_P(kg)+T_P-T'_{R'})_Q^G). 
    \end{align*}   
On en déduit que $u_{R'}^{T,T'}(k,g)$ est un polynôme-exponentielle en $T$ et $T'$ dont le terme constant s'annule sauf si 
    \begin{align}\label{eq:cond-ctneq0}
 (2\rho_{\ngo_{R'}\cap\sgo}-2\rho_{R'^\theta}^{G_\sigma^\theta})|_{(\ago_{R'}^G)^\theta}=0
    \end{align}    
auquel cas $u_{R'}^{T,T'}(k,g)$ est un polynôme en $T$ et $T'$. 
Dans le dernier cas, le terme constant de $u_{R'}^{T,T'}(k,g)$ est sa valeur en $T=T'=0$ c'est-à-dire 
    \begin{align*}
	\int_{(\ago_{R'}^G)^\theta} \Gamma_{R'}^G(X,\yc_{R'}^{0,0}(k, g)) \, dX=v'_{R'}(kg) 
    \end{align*}  
où l'on fait à nouveau appel au lemme \ref{lem:art4.1}. 

  Avec la notation \eqref{eq:uTT'}, l'expression \eqref{eq:art(6.6)} s'écrit  
    \begin{align*}
     \int_{G_\sigma^\theta(\AAA)\back G^\theta(\AAA)} \sum_{R'\in\fc^{G_\sigma}(P_{1,\sigma}^\theta, \theta)} \bigg(\int_{K_\sigma^\theta} \eta(kg) u_{R'}^{T,T'}(k,g) J_{\nilp}^{R', T'}(\eta, 2\rho_{\ngo_{R'}\cap\sgo}-2\rho_{R'^\theta}^{G_\sigma^\theta}, \Psi_{kg,R'}) \,  dk\bigg) \, dg. 
    \end{align*} 
On peut changer l'ordre des intégrales sur $k\in K_\sigma^\theta$ et  $m\in[M_{R'}^\theta]^{R'}$ (dans la définition de $J_{\nilp}^{R', T'}$). La dernière expression devient 
    \begin{align*}
	\int_{G_\sigma^\theta(\AAA)\back G^\theta(\AAA)} \sum_{R'\in\fc^{G_\sigma}(P_{1,\sigma}^\theta, \theta)} J_{\nilp}^{R',T'}(\eta, 2\rho_{\ngo_{R'}\cap\sgo}-2\rho_{R'^\theta}^{G_\sigma^\theta}, \Phi_{g,R'}^{T,T'}) \eta(g) \, dg
    \end{align*}	
    où l'on définit pour tout $X\in (\mgo_{R'}\cap\sgo)(\AAA)$
    \begin{align*}
	\Phi_{g,R'}^{T,T'}(X)=\int_{K_\sigma^\theta} \Psi_{kg,R'}(X) u_{R'}^{T,T'}(k,g) \eta(k) \, dk. 
    \end{align*} 
Notons que $\Phi_{g,R'}^{T,T'}$ appartient à $C_c^\infty((\mgo_{R'}\cap\sgo)(\AAA))$ et dépend continûment en $g$. De plus, $\Phi_{g,R'}^{T,T'}(X)$ est un polynôme-exponentielle en $T$ et $T'$. Effectivement, en tant que fonctions de $T'$, $u_{R'}^{T,T'}(k,g)$ et $\Phi_{g,R'}^{T,T'}(X)$ ne dépendent que de $T'_{R'}$. Soit $\Phi_{g,R'}(X)$ le terme constant de $\Phi_{g,R'}^{T,T'}(X)$. Alors $\Phi_{g,R'}(X)=0$ sauf si $R'$ vérifie \eqref{eq:cond-ctneq0} auquel cas  
    \begin{align}\label{eq:defPhigR'}
	\Phi_{g,R'}(X)=\int_{K_\sigma^\theta} \Psi_{kg,R'}(X) v'_{R'}(kg) \eta(k) \, dk. 
    \end{align}  
D'après le théorème \ref{thm:expol-infi}, l'expression 
    \begin{align*}
	J_{\nilp}^{R',T'}(\eta, 2\rho_{\ngo_{R'}\cap\sgo}-2\rho_{R'^\theta}^{G_\sigma^\theta}, \Phi_{g,R'}^{T,T'})
    \end{align*}  
est un polynôme-exponentielle en $T$ et $T'$ dont le terme constant est 
    \begin{align*}
	J_{\nilp}^{R'}(\eta, 2\rho_{\ngo_{R'}\cap\sgo}-2\rho_{R'^\theta}^{G_\sigma^\theta}, \Phi_{g,R'}). 
    \end{align*}  

Rappelons qu'on note  $J_\of(\eta,\varphi)$ le terme constant du polynôme-exponentielle $T\mapsto J_\of^T(\eta,\varphi)$, cf. \S \ref{S:dist sur S}. Notre progrès est résumé par le théorème suivant qui est un analogue de \cite[lemme 6.2]{ar-family}. 

    \begin{theoreme}\label{thm:descente}
        Soit $\varphi\in C_c^\infty(S_{p'}(\AAA))$. On a  
    \begin{align}\label{eq:art(6.8)}
	J_\of(\eta, \varphi)=\int_{G_\sigma^\theta(\AAA)\back G^\theta(\AAA)} \sum_{R'} J_{\nilp}^{R'}(\eta, 2\rho_{\ngo_{R'}\cap\sgo}-2\rho_{R'^\theta}^{G_\sigma^\theta}, \Phi_{g,R'}) \eta(g) \, dg
    \end{align}  
où la somme est prise sur $R'\in\fc^{G_\sigma}(P_{1,\sigma}^\theta, \theta)$ vérifiant \eqref{eq:cond-ctneq0}. 
    \end{theoreme}
    
    \begin{remarque}
        La condition \eqref{eq:cond-ctneq0} est triviale pour les facteurs de type 1 ou 2 dans la proposition \ref{prop:descendant}. Donc seulement les facteurs de type 3 comptent. Pour cela, cette condition ressemble à \eqref{eq:la condition} avec $\theta_1=\theta_2$. On peut consulter la proposition \ref{prop:lacond1=2} pour une description plus explicite. 
    \end{remarque}
\end{paragr}

\subsection{Intégrales orbitales pondérées}\label{ssec:IOP}

\begin{paragr}\label{S:defcar}
  On continue avec les notations du paragraphe \ref{S:red-unip}. Mais dans cette sous-section, on suppose $\of\in\cgo_{p',\rs}(F)$ et  $S_\of(F)\not=\emptyset$ (voir § \ref{S:invariants} et § \ref{S:cat quot ss} pour la définition de $\cgo_{p', \rs}$). D'après la proposition \ref{prop:ssr-Prd}, tous les éléments de $S_{p',\of}$ sont alors $G^\theta$-semi-simples réguliers et  $S_{p',\of}(F)$ est une classe de $G^\theta(F)$-conjugaison. Rappelons que l'on fixe  $P_1\in\fc(P_0^\theta)$ et  $\sigma\in (M_{P_1}\cap S_{p',\of})(F)$ un élément $(P_1, \theta)$-anisotrope. 

  Soit $M\in\lc(M_0)$. Pour tout  $T\in\overline{\ago_0^+}$ et tout $g\in G(\AAA)$, soit $v_M(g,T)$ le volume de la projection sur $\ago_M^G$ de  l'enveloppe convexe des  points $T_P-H_P(g) $ lorsque $P$ décrit $\pc(M)$. C'est le poids usuel introduit par Arthur dans \cite[p.951]{ar1}.
  C'est une fonction positive sur $G(\AAA)$ qui est invariante à droite par $K$ et à gauche par $M(\AAA)$. Soit $v_M(g)=v_M(g,0)$. 
\end{paragr}

\begin{paragr}
  En utilisant la proposition \ref{prop:ssr-Prd}, sans perte de généralité, on peut supposer $\sigma=x(A)$ défini en \eqref{eq:x(A)} où $A\in\gl_\nu(D)$ est semi-simple régulier, sans valeurs propres $\pm1$ et elliptique dans une sous-algèbre de Levi standard de $\gl_\nu(D)$ disons de type $(\nu_1,\cdots,\nu_k)$ où $\sum_{1\leq i\leq k} \nu_i=\nu$.   Soit $\tilde G=GL_{2\nu,D}$, $H=GL_{|p'-q|,D}\times GL_{|p-p'|,D}$ et $\sigma_0= \begin{pmatrix}
          A & A-I_\nu \\
          A+I_\nu & A  \\
  \end{pmatrix}$. D'après le corollaire \ref{cor:Prdx}, le centralisateur $ \tilde G_{\sigma_0}$ est un tore maximal de $\tilde G$. On trouve que  
  \begin{align*}
      G_\sigma= \tilde G_{\sigma_0}\times H. 
  \end{align*}        
  Soit $\theta_0=\begin{pmatrix}
          I_\nu & 0  \\
          0 & -I_\nu  
  \end{pmatrix}$. Alors 
  \begin{align}\label{eq:tildeGsigthe0}
       \tilde G_{\sigma_0}^{\theta_0}=\left\{\begin{pmatrix}
          x & 0  \\
          0 & x  
  \end{pmatrix} : x\in GL_{\nu,D}, xA=Ax\right\} 
  \end{align}
  est isomorphe à un tore maximal de $GL_{\nu,D}$. Il résulte de la remarque \ref{rmq:nul} que l'action de $\theta$ sur $H$ est triviale. On en déduit que 
  \begin{align}\label{eq:Gsigthe}
      G_\sigma^\theta= \tilde G_{\sigma_0}^{\theta_0}\times H. 
  \end{align} 
  
  \begin{proposition}
  \begin{enumerate}
      \item Le groupe $H$ est réduit au singleton $\{1\}$ si et seulement si $p=q=p'=q'$. 
      
      \item Si $\eta(G_\sigma^\theta(\AAA))=1$, alors $\eta=1$ ou $H=\{1\}$. 
      
      \item Si $\eta^2=1$, la réciproque de l'assertion 2 est aussi vraie. 
  \end{enumerate}
  \end{proposition}
  
  \begin{preuve}
  L'assertion 1 est évidente puisque $p+q=p'+q'$. Il est connu que $\eta(H(\AAA))=1$ si et seulement si $\eta=1$ ou $H=\{1\}$, ce qui implique l'assertion 2. Si $\eta^2=1$, on a  $\eta(\tilde G_{\sigma_0}^{\theta_0}(\AAA))=1$ par \eqref{eq:tildeGsigthe0}, d'où l'assertion 3 vu \eqref{eq:Gsigthe}. 
  \end{preuve}
  
  Soit $\tilde M_1$ le centralisateur dans $\tilde G$ du sous-tore déployé maximal du tore $\tilde G_{\sigma_0}^{\theta_0}$. On a $\tilde M_{1,\sigma_0}=\tilde G_{\sigma_0}$ et $A_{\tilde M_1}=A_{\tilde G_{\sigma_0}^{\theta_0}}=A_{\tilde G_{\sigma_0}}^{\theta_0}$ (cf. corollaire \ref{cor:A_P_sigma}). Soit $T_0\subset B$ le sous-groupe des produits des matrices diagonales et le sous-groupe des produits des matrices triangulaires supérieures de $H$. On a $M_1=\tilde M_1\times T_0$ qui est un sous-groupe de Levi semi-standard de $G$ de type $(2\nu_1,\cdots,2\nu_k,\underbrace{1,\cdots,1}_{p+q-2\nu})$. 
  Alors 
  \begin{align*}
      P_{1,\sigma}=\tilde G_{\sigma_0}\times B \text{ et } P_{1,\sigma}^\theta=\tilde G_{\sigma_0}^{\theta_0}\times B. 
  \end{align*}
  Il existe donc une bijection 
  \begin{align}\label{eq:bij-par-ssr}
   \fc^H(B)\to \fc^{G_\sigma}(P_{1,\sigma}^\theta)=\fc^{G_\sigma}(P_{1,\sigma}^\theta,\theta)
  \end{align} 
  donnée par $Q\mapsto \tilde G_{\sigma_0}\times Q$. Pour $R\in\fc^{G_\sigma}(P_{1,\sigma}^\theta)$, on en déduit 
  \begin{align}\label{eq:theta-theta}
   A_R^\theta=A_{R^\theta} \text{ et } [M_R^\theta]^R=[M_{R^\theta}]^1. 
  \end{align}   
  
  Rappelons qu'en utilisant la bijection \eqref{eq:bijtilde}, nous avons noté $\wt{G_\sigma}\in\lc(M_1)$ le centralisateur dans $G$ de $A_{G_\sigma}^\theta=A_{G_\sigma^\theta}$ (vu \eqref{eq:theta-theta}). On écrit également 
  \begin{align*}
      v_\sigma(g,T)=v_{\wt{G_\sigma}}(g,T) \text{ et } v_\sigma(g)=v_{\wt{G_\sigma}}(g). 
  \end{align*}
  Effectivement, on a $\wt{G_\sigma}=\tilde M_1\times H$ qui est un sous-groupe de Levi semi-standard de $G$ de type $(2\nu_1,\cdots,2\nu_k,|p'-q|,|p-p'|)$. Dans \S \ref{S:lem-combi}, nous avons aussi vu que  
  \begin{align}\label{eq:sigma-L}
      \fc_{G_\sigma}(M_1)=\fc(\wt{G_\sigma}) \text{ et } A_{\wt{G_\sigma}}=A_{G_\sigma}^\theta=A_{G_\sigma^\theta}. 
  \end{align}

  \begin{lemme}
    \label{lem:cv}
    Pour tout  $T\in\overline{\ago_0^+}$ il existe une semi-norme continue $\|\cdot\|$ sur $\Sc(S_{p'}(\AAA))$ telle que pour tout $\varphi\in \Sc(S_{p'}(\AAA))$ on ait:
    \begin{align*}
       \int_{G_\sigma^\theta(\AAA)\back G^\theta(\AAA)} |\varphi(\Int(g^{-1})(\sigma)) |v_\sigma(g,T)  \, dg  \leq \|\varphi\|.
    \end{align*}
  \end{lemme}

  \begin{preuve}
    L'application $g\mapsto \Int(g^{-1})(\sigma)$ induit un isomorphisme de $G_\sigma^\theta\back G^\theta$ sur une sous-variété lisse et fermée $R$ de $S_{p'}$. On a un morphisme de restriction qui est continu et surjectif de $\Sc(S_{p'}(\AAA))$ sur $\Sc(R(\AAA))$. Par ailleurs, à l'aide du théorème 90 de Hilbert et du lemme de Shapiro, on déduit de \eqref{eq:Gsigthe} que le premier ensemble de cohomologie du groupe $G_\sigma^\theta(F_v)$ est trivial pour toute place $v$ de $F$. L'application considérée envoie alors bijectivement $G_\sigma^\theta(\AAA)\back G^\theta(\AAA)$ sur  $R(\AAA)$. On a ainsi une application continue, surjective et ouverte de  $\Sc(G^\theta(\AAA))$ sur  $\Sc(G_\sigma^\theta(\AAA)\back G^\theta(\AAA))\simeq \Sc(R(\AAA))$ donnée par l'intégration sur $G_\sigma^\theta(\AAA)$, cf. \S \ref{S:Schwartz}. Il suffit donc de prouver qu'il existe une  semi-norme continue sur $\Sc( G^\theta(\AAA))$ telle que  pour tout $f\in \Sc(G^\theta(\AAA))$,  on ait:
    \begin{align*}
        \int_{G_\sigma^\theta(\AAA)\back G^\theta(\AAA)} \left| \int_{G_\sigma^\theta(\AAA)   }f(hg) \, dh\right|v_\sigma(g,T)  \, dg  \leq \|f\|.
    \end{align*}
    Il suffit donc de prouver la continuité de
    \begin{align*}
f\in \Sc(G^\theta(\AAA)) \mapsto   \int_{G^\theta(\AAA)} \left| f(g) \right|v_\sigma(g,T)  \, dg.  
    \end{align*}
    La convergence et la continuité de l'intégrale sont évidentes ici. En effet, le poids $v_\sigma(g,T)$ est majoré par $c \max_{P\in \pc(\wt{G_\sigma})}\left( \|T_P - H_P(g)\|^{\dim(\ago_{\wt{G_\sigma}}^G)}\right)$ pour une certaine constante $c>0$. Par ailleurs, si $\|\cdot\|$ est une hauteur sur $G(\AAA)$, il existe $c_1>0$ tel que on a $\|H_P(g)\|\leq c_1(1+\log\|g\|)$ pour tout $g\in G(\AAA)$.  
  \end{preuve}
\end{paragr}

\begin{paragr}
  Rappelons que pour toute $\varphi\in\Sc(S_{p'}(\AAA))$ on note  $J_\of(\eta,\varphi)$ le terme constant du polynôme-exponentielle $T\mapsto J_\of^T(\eta,\varphi)$, cf. \S \ref{S:dist sur S}. La proposition suivante exprime $J_\of(\eta,\varphi)$ comme une intégrale pondérée tordue par le caractère $\eta$ construite à l'aide du poids $v_{\sigma}(g)$.

  \begin{theoreme}
    \label{thm:IOP}
    Soit $\varphi\in \Sc(S_{p'}(\AAA))$.
    On a $J_\of(\eta,\varphi)=0$ sauf si  $\eta(G_\sigma^\theta(\AAA))=1$, auquel cas on a 
    \begin{align*}
	    J_\of(\eta,\varphi)=\vol([G_\sigma^\theta]^1) \int_{G_\sigma^\theta(\AAA)\back G^\theta(\AAA)} \varphi(\Int(g^{-1})(\sigma)) v_\sigma(g) \eta(g) \, dg  
    \end{align*}    
    où $[G_\sigma^\theta]^1$ est défini dans \S \ref{S:compact}. 
  \end{theoreme}

  \begin{preuve} Il s'agit de montrer l'égalité de deux formes linéaires sur $\Sc(S_{p'}(\AAA))$ qui sont continues vu \S \ref{S:dist sur S} et le lemme \ref{lem:cv}. Il suffit de l'établir sur le sous-espace dense $\Cc(S_{p'}(\AAA))$.  Soit $\varphi\in \Cc(S_{p'}(\AAA))$ et $T\in T_0+\overline{\ago_0^+}$ suffisamment positif avec $\|T^G\|$ assez grand (par rapport au support de $\varphi$). L'argument suivant ressemble à celui du paragraphe \ref{S:red-unip}. On commence par l'expression \eqref{eq:var-prelim} de $J_\of^T(\eta,\varphi)$. Mais comme $\sigma$ est $G^\theta$-semi-simple et $G^\theta$-régulier, d'après le lemme  \ref{lem:ssr-unip}, l'expression \eqref{eq:KRunip} de  $K_{R,\varphi}^{\unip}(x)$ devient $\varphi(\Int(x^{-1})(\sigma))$ qui est indépendant de $R$. On en déduit que $J_\of^T(\eta,\varphi)$ égale 
   \begin{align}\label{eq:rss-prelim}
     \int_{[G^\theta]^G} \sum_{R\in \fc^{G_\sigma}(P_{1,\sigma}^\theta,\theta)} \sum_{\xi\in R^\theta(F)\back G^\theta(F)} \varphi(\Int(\xi x)^{-1}(\sigma)) \left(\sum_{P\in\fc_R(M_1)}\eps_P^G \hat\tau_P(H_P(\xi x)-T_P)\right) \eta(x) \, dx. 
  \end{align}
  Comme dans \S \ref{S:red-unip}, on fera également d'abord formellement des changements de variables dont la justification sera similaire et omise. En changeant la somme et intégrale sur $(\xi,x)$ et utilisant \eqref{eq:tau=GamR}, on obtient 
    \begin{align*}
      J_\of^T(\eta, \varphi)=\int_{ G_\sigma^\theta(\AAA)\back G^\theta(\AAA)}
      \int_{ G_\sigma^\theta(F)\back (G_\sigma^\theta(\AAA)\cap G(\AAA)^1)}  \sum_{R'\in\fc^{G_\sigma}(P_{1,\sigma}^\theta, \theta)} \sum_{R\in\fc^{R'}(P_{1,\sigma}^\theta, \theta)}
      \sum_{\delta\in R^\theta(F)\back G_\sigma^\theta(F)} \\
      \varphi(\Int(g^{-1})(\sigma)) (-1)^{\dim (\ago_R^{R'})^\theta} \hat\tau_R^{R'}(H_R(\delta x)-T'_R) \Gamma_{R'}^G(H_{R'}(\delta x)-T'_{R'},\yc_{R'}^{T,T'}(\delta x, g)) \eta(xg) \, dxdg.
    \end{align*}

Ensuite, on décompose la somme sur $\delta\in R^\theta(F)\back G_\sigma^\theta(F)$ et applique la décomposition d'Iwasawa. Puisque $\vol([N_{R'}^\theta])=1$, l'expression $J_\of^T(\eta, \varphi)$ s'écrit l'intégrale et somme sur $g\in G_\sigma^\theta(\AAA)\back G^\theta(\AAA)$, $R'\in\fc^{G_\sigma}(P_{1,\sigma}^\theta, \theta)$, $k\in K_\sigma^\theta$, $a\in (A_{R'}^\theta)^{G,\infty}$ et $m\in[M_{R'}^\theta]^{R'}=[M_{R'^\theta}]^1$ (cf. \eqref{eq:theta-theta}) de
   \begin{align*}
     \varphi(\Int(g^{-1})(\sigma)) \eta(mkg) \exp(\langle -2\rho_{R'^\theta}^{G_\sigma^\theta}, H_{R'^\theta}(a) \rangle) \Gamma_{R'}^G(H_{R'}(a)-T'_{R'},\yc_{R'}^{T,T'}(k, g)) \\ \nonumber \sum_{R\in\fc^{R'}(P_{1,\sigma}^\theta, \theta)} (-1)^{\dim(\ago_R^{R'})^\theta}  \sum_{\mu\in(R^\theta\cap M_{R'^\theta})(F)\back M_{R'^\theta}(F)} \hat\tau_R^{R'}(H_R(\mu m)-T'_R). 
  \end{align*}

  Considérons maintenant l'intégrale 
    \begin{align*}
	u_{R'}^{T,T'}(k,g)=\int_{(A_{R'}^\theta)^{G,\infty}} \exp(\langle -2\rho_{R'^\theta}^{G_\sigma^\theta}, H_{R'^\theta}(a) \rangle) \Gamma_{R'}^G(H_{R'}(a)-T'_{R'},\yc_{R'}^{T,T'}(k, g)) \, da. 
    \end{align*}    
Par un changement de variables et les lemmes \ref{lem:art4.1} et \ref{lem:art2.2-li5.6}, il existe un polynôme $c_{R',P,Q}$ sur $\ago_Q^G$ pour tout $Q\in\fc(P)$ tel que 
    \begin{align*}
	u_{R'}^{T,T'}(k,g)=\sum_{P\in\fc_{R'}^0(M_1)} \sum_{Q\in\fc(P)} \exp(\langle -2\rho_{R'^\theta}^{G_\sigma^\theta}, T_Q^G \rangle) \exp(\langle -2\rho_{R'^\theta}^{G_\sigma^\theta}, T'^Q_{R'} \rangle) \\ \nonumber 
	\exp(\langle -2\rho_{R'^\theta}^{G_\sigma^\theta}, -H_Q(kg)^G \rangle) c_{R',P,Q}((-H_P(kg)+T_P-T'_{R'})_Q^G). 
    \end{align*}    
Donc $u_{R'}^{T,T'}(k,g)$ est un polynôme-exponentielle en $T$ et $T'_{R'}$ dont le terme constant s'annule sauf si 
    \begin{align}\label{eq:cond-ctneq0-rss}
 (-2\rho_{R'^\theta}^{G_\sigma^\theta})|_{(\ago_{R'}^G)^\theta}=0
    \end{align}    
auquel cas $u_{R'}^{T,T'}(k,g)$ est un polynôme en $T$ et $T'_{R'}$. Mais la bijection \eqref{eq:bij-par-ssr} implique que la condition \eqref{eq:cond-ctneq0-rss} est équivalente à  $R'=G_\sigma$.  

Il faut prendre en compte aussi l'intégrale 

    \begin{align*}
        \int_{[M_{R'}^\theta]^{R'}} \eta(m) \sum_{\mu\in(R^\theta\cap M_{R'^\theta})(F)\back M_{R'^\theta}(F)} \hat\tau_R^{R'}(H_R(\mu m)-T'_R) \, dm. 
    \end{align*}    
Par la décomposition d'Iwasawa et $\vol([N_R^{R',\theta}])=1$, cette intégrale est égale à l'intégrale sur $a'\in (A_R^{R'})^{\theta,\infty}=A_{R^\theta}^{R'^\theta,\infty}$, $m'\in[M_R^\theta]^R=[M_{R^\theta}]^1$ (cf. \eqref{eq:theta-theta}) et $k'\in K_\sigma^\theta\cap M_{R'^\theta}(\AAA)$ de 
    \begin{align*}
        \eta(m'k') \exp(\langle -2\rho_{R^\theta\cap M_{R'^\theta}}^{M_{R'^\theta}}, H_{R^\theta}(a') \rangle)\hat\tau_R^{R'}(H_R(a')-T'_R). 
    \end{align*}
    Notons que 
        \begin{align*}
        \int_{a'\in (A_R^{R'})^{\theta,\infty}}  \exp(\langle -2\rho_{R^\theta\cap M_{R'^\theta}}^{M_{R'^\theta}}, H_{R^\theta}(a') \rangle)\hat\tau_R^{R'}(H_R(a')-T'_R) \, da' = C \exp(\langle -2\rho_{R^\theta\cap M_{R'^\theta}}^{M_{R'^\theta}}, {T'}_R^{R'} \rangle)
    \end{align*}
    avec un constant 
        \begin{align*}
        C=\int_{(\ago_R^{R'})^\theta} \exp(\langle -2\rho_{R^\theta\cap M_{R'^\theta}}^{M_{R'^\theta}}, H \rangle)\hat\tau_R^{R'}(H) \, dH < +\infty.
    \end{align*}  
    C'est une fonction exponentielle en ${T'}_R^{R'}$ dont le terme constant s'annule sauf si 
    \begin{align}\label{eq:Levi-ctneq0-rss}(-2\rho_{R^\theta\cap M_{R'^\theta}}^{M_{R'^\theta}})|_{(\ago_{R}^{R'})^\theta}=0. 
    \end{align}
D'après la bijection \eqref{eq:bij-par-ssr}, la condition \eqref{eq:Levi-ctneq0-rss} est équivalente à $R=R'$.

Comme la fonction $T'\mapsto u_{R'}^{T,T'}(k,g)$ ne dépend que $T'_{R'}$, par la discussion ci-dessus, on trouve que le terme constant  $J_\of(\eta,\varphi)$ provient de $R=G_\sigma$ dans \eqref{eq:rss-prelim}. Il est donc le terme constant de  
   \begin{align*}
     \int_{[G^\theta]^G} \sum_{\xi\in G_\sigma^\theta(F)\back G^\theta(F)} \varphi(\Int(\xi x)^{-1}(\sigma)) \left(\sum_{P\in\fc_{G_\sigma}(M_1)}\eps_P^G \hat\tau_P(H_P(\xi x)-T_P)\right) \eta(x) \, dx 
     \end{align*}
qui est égal à     
\begin{align*} \int_{G_\sigma^\theta(\AAA)\back G^\theta(\AAA)} \varphi(\Int(g^{-1})(\sigma)) \left(\int_{G_\sigma^\theta(F)\back (G_\sigma^\theta(\AAA)\cap G(\AAA)^1)}   \sum_{P\in\fc_{G_\sigma}(M_1)}\eps_P^G \hat\tau_P(H_P(xg)-T_P) \eta(xg) \, dx\right) \, dg. 
\end{align*}
Comme $G_\sigma^\theta(\AAA)\cap G(\AAA)^1=G_\sigma^\theta(\AAA)^1\times A_{G_\sigma^\theta}^{G,\infty}$, la dernière expression devient 
\begin{align*}
 \int_{G_\sigma^\theta(\AAA)\back G^\theta(\AAA)} \varphi(\Int(g^{-1})(\sigma)) \left(\int_{A_{G_\sigma^\theta}^{G,\infty}} \sum_{P\in\fc_{G_\sigma}(M_1)}\eps_P^G \hat\tau_P(H_P(ag)-T_P) \, da\right) \left(\int_{[G_\sigma^\theta]^1} \eta(yg) \, dy\right) \, dg.  
 \end{align*}
 Elle s'annule sauf si $\eta(G_\sigma^\theta(\AAA))=1$, auquel cas on a 
 \begin{align*}
     \int_{[G_\sigma^\theta]^1} \eta(yg) \, dy=\vol([G_\sigma^\theta]^1)\eta(g). 
 \end{align*}
 Par ailleurs, en utilisant  \eqref{eq:sigma-L}, on obtient 
\begin{align*}
    \int_{A_{G_\sigma^\theta}^{G,\infty}} \sum_{P\in\fc_{G_\sigma}(M_1)}\eps_P^G \hat\tau_P(H_P(ag)-T_P) \, da=\int_{\ago_{\wt{G_\sigma}}^G} \sum_{P\in\fc(\wt{G_\sigma})}\eps_P^G \hat\tau_P(X+H_P(g)-T_P) \, dX. 
\end{align*}
La famille  $(T_P-H_P(g) )_{ P\in\fc(\wt{G_\sigma})}$ est une famille $\wt{G_\sigma}$-orthogonale et régulière au sens de \cite[§ 1.5]{labWal}: cela résulte de \cite[lemme 3.6]{ar-character} et du fait que $T\in \ago_0^+$. On déduit  alors de \cite[proposition 1.8.7]{labWal} que 
     \begin{align*}
 X\in\ago_{\wt{G_\sigma}}^G\mapsto	\sum_{P\in\fc(\wt{G_\sigma})}\eps_P^G \hat\tau_P(X-(T_P-H_P(g))) 
     \end{align*}	
     est la fonction caractéristique de la projection sur $\ago_{\wt{G_\sigma}}^G $ de l'enveloppe convexe des points $(T_P-H_P(g) )_{ P\in\pc(\wt{G_\sigma})}$. On conclut en prenant le terme constant. 
  \end{preuve}

\end{paragr}

\subsection{Une variante infinitésimale}\label{ssec:var-infi}

\begin{paragr}
  Dans la suite, on change les notations et on note désormais $(G, G^\theta, \theta)$ une paire symétrique de type 1, 2 ou 3 dans la proposition \ref{prop:descendant}. Soit $\ggo$ l'algèbre de Lie de $G$ et 
    \begin{align*}
	\sgo=\{X\in\ggo \mid X+\theta(X)=0\}
    \end{align*}  
    qui s'identifie à l'espace tangent en $1$ de l'espace symétrique $G/G^\theta$. Dans cette sous-section, on rappelle ce qu'est le développement géométrique de la \og  formule des traces \fg{} pour l'action de $G^\theta$ sur $\sgo$. Pour les types 1 et 2, on a une identification  $G^\theta$-équivariante de $\ggo^\theta$ avec  $\sgo$ donnée respectivement  par $X\mapsto (X,-X)$ et $X\mapsto X \al$ où $\al$ est un  élément non nul de $K$ tel que  $\trace_{K/L}(\al)=0$. Dans ces cas, la formule des traces n'est qu'une reformulation \og relative \fg{} de la formule des traces pour les algèbres de Lie de \cite{Cha} alors que pour le type 3, elle est l'objet d'étude de  \cite{li1}.

On transposera dans notre cadre la plupart des notations utilisées précédemment sans plus de commentaire. Précisons-en cependant quelques-unes. Fixons un tore déployé maximal $\theta$-stable $A_0$ de $G$ tel que $A_0^\theta$ est un tore déployé maximal de $G^\theta$. Un sous-groupe parabolique de $G$, resp. de $G^\theta$, est dit semi-standard s'il contient $A_0$, resp. $A_0^\theta$. 
Fixons un sous-groupe parabolique minimal semi-standard $\theta$-stable $P_0$ de $G$. Alors $P_0^\theta$ est un sous-groupe parabolique minimal semi-standard de $G^\theta$. Soit $\fc(P_0^\theta,\theta)$ l'ensemble des sous-groupes paraboliques semi-standard $\theta$-stables de $G$ contenant $P_0^\theta$. Soit $\fc^\theta(P_0^\theta)$ l'ensemble des sous-groupes paraboliques de $G^\theta$ contenant $P_0^\theta$. Notons que l'application 
    \begin{align*}
	P\mapsto P^\theta
    \end{align*}  
de $\fc(P_0^\theta,\theta)$ dans $\fc^\theta(P_0^\theta)$ est surjective. Pour les types 1 et 2, cette application est une bijection. 
\end{paragr}

\begin{paragr}
On considère la relation d'équivalence suivante sur $\sgo(F)$ : deux éléments sont équivalents si et seulement si leurs parties semi-simples (pour la décomposition de Jordan dans $\ggo(F)$) sont conjuguées sous $G^\theta(F)$. On note $\oc$ l'ensemble des classes d'équivalence. La classe d'équivalence contenant $0\in\sgo(F)$ est appelée nilpotente. Si $\of\in\oc$ est nilpotente, on remplacera souvent $\of$ en indice par ``nilp''. 
\end{paragr}

\begin{paragr}[Majoration du noyau tronqué.] ---
  On note $\Sc(\sgo(\AAA))$ l'espace de Schwartz sur $\sgo(\AAA)$. Soit $\varphi\in\Sc(\sgo(\AAA))$, $\of\in\oc$ et $x\in G^\theta(\AAA)$. On définit pour tout $P=MN\in\fc(P_0^\theta,\theta)$ avec sa décomposition de Levi 
  \begin{align*}
  K_{P,\of,\varphi}(x)=\sum_{X\in \mgo(F)\cap\of} \int_{(\ngo\cap\sgo)(\AAA)} \varphi(\Ad(x^{-1})(X+U)) \, dU. 
  \end{align*}
En remplaçant le point $T_0\in\ov{\ago_0^{G,+}}$ dans \S \ref{S:T0} par $T_0+\theta(T_0)$ si nécessaire (seulement pour le type 1), on peut et on va supposer que $T_0\in\ago_0^\theta\cap\ov{\ago_0^{G,+}}$. On définit pour tout $T\in T_0+\ago_0^\theta\cap\ov{\ago_0^+}$  
  \begin{align*}
  K_{\of,\varphi}^T(x)=\sum_{P\in\fc(P_0^\theta,\theta)} (-1)^{\dim{(\ago_P^G)^\theta}} \sum_{\delta\in P^\theta(F)  \back G^\theta(F)} \hat\tau_P(H_P(\delta x)-T_P) K_{P,\of,\varphi}(\delta x). 
  \end{align*}
  À l'aide de \cite[lemme 5.1]{ar1}, on voit que la somme sur $\delta$ dans la dernière expression est finie. On note  $[G^\theta]^G=G^\theta(F)\back (G^\theta(\AAA)\cap G(\AAA)^1)$. 
  
\begin{theoreme}(\cite[théorème 3.1]{Cha} et \cite[théorème 4.14]{li1})\label{thm:cv-infi}
	Pour tout $T\in T_0+\ago_0^\theta\cap\overline{\ago_0^+}$ et tout $\chi\in \ago_{G^\theta}^\ast$, il existe une semi-norme continue $\|\cdot\|$ sur $\Sc(\sgo(\AAA))$ telle que pour tout $\varphi\in\Sc(\sgo(\AAA))$ on ait 
  \begin{align*}
  	\sum_{\of\in\oc} \int_{[G^\theta]^G} | K_{\of,\varphi}^T(x) | \exp(\langle \chi, H_{G^\theta}(x) \rangle) \, dx \leq \|\varphi\|. 
  \end{align*}	
\end{theoreme}

\begin{remarque}
    Pour les types 1 et 2,  on a $G^\theta(\AAA)\cap G(\AAA)^1 = G^\theta(\AAA)^1$ et alors $\exp(\langle \chi, H_{G^\theta}(x) \rangle)$ vaut $1$ pour tous $\chi\in \ago_{G^\theta}^\ast$ et $x\in [G^\theta]^G$.
\end{remarque}
\end{paragr}

\begin{paragr}[Distributions géométriques.] ---
  Soit $\varphi\in\Sc(\sgo(\AAA))$, $\of\in \oc$ et $\chi\in \ago_{G^\theta,\CC}^\ast$. 
  Soit $\eta:F^\times\back \AAA^\times \to \CC^\times$ un caractère unitaire comme dans \S \ref{S:JchiT}. Fixons un caractère $\iota\in X^\ast(G)$. Par composition avec $\iota$, on en déduit un caractère $G(\AAA)\to\CC^\times$ trivial sur les sous-groupes $A_G^\infty$ et $G(F)$, encore noté $\eta$. 
  On définit pour tout $T\in T_0+\ago_0^\theta\cap\overline{\ago_0^+}$ 
   \begin{align}\label{eq:infi-JoGT}
     J_\of^{G,T}(\eta, \chi, \varphi)=\int_{[G^\theta]^G} K_{\of,\varphi}^T(x) \eta(x) \exp(\langle \chi, H_{G^\theta}(x) \rangle)  \, dx
  \end{align}	
qui converge absolument d'après le théorème \ref{thm:cv-infi}. 

Plus généralement, soit $Q\in\fc(P_0^\theta,\theta)$ et $f\in\Sc((\mgo_Q\cap\sgo)(\AAA))$. Soit
\begin{align*}
    \fc^Q(P_0^\theta,\theta)=\{P\in \fc(P_0^\theta,\theta) | P\subset Q \}.
\end{align*}
On note $\oc^{\mgo_Q\cap\sgo}$ l'ensemble des classes de $M_Q^\theta(F)$-conjugaison semi-simples dans $(\mgo_Q\cap\sgo)(F)$. L'intersection $\of\cap (\mgo_Q\cap\sgo)(F)$ est une réunion finie, éventuellement vide, de classes $\of_1, \dots, \of_n \in \oc^{\mgo_Q\cap\sgo}$. Un sous-groupe parabolique de $M_Q$, resp. de $M_Q^\theta$, est dit semi-standard s'il contient $A_0$, resp. $A_0^\theta$. Le groupe $P'_0=P_0\cap M_Q$ est un sous-groupe parabolique minimal semi-standard $\theta$-stable de $M_Q$. Ainsi $B_Q=P'_0N_Q$ est un sous-groupe parabolique minimal semi-standard $\theta$-stable de $G$ inclus dans $Q$. Le groupe $P'^\theta_0=P_0^\theta\cap M_Q^\theta$ est un sous-groupe parabolique minimal semi-standard de $M_Q^\theta$. 
Soit $\fc^{M_Q}(P'^\theta_0,\theta)$ l'ensemble des sous-groupes paraboliques semi-standard $\theta$-stables de $M_Q$ contenant $P'^\theta_0$. L'application 
\begin{align}\label{eq:PcapM}
    P\mapsto P\cap M_Q
\end{align}
induit une bijection de $\fc^Q(P_0^\theta, \theta)$ sur $\fc^{M_Q}(P'^\theta_0,\theta)$ dont la réciproque est l'application $P'\mapsto P'N_Q$. 
Soit $\fc^{M_Q^\theta}(P'^\theta_0)$ l'ensemble des sous-groupes paraboliques de $M_Q^\theta$ contenant $P'^\theta_0$. On définit pour tout $P'=M'N'\in\fc^{M_Q}(P'^\theta_0,\theta)$ avec sa décomposition de Levi, tout $1\leq i\leq n$ et tout $x\in M_Q^\theta(\AAA)$
  \begin{align*}
  K_{P',\of_i,f}(x)=\sum_{X\in (\mgo'\cap\sgo)(F)\cap\of_i} \int_{(\ngo'\cap\sgo)(\AAA)} f(\Ad(x^{-1})(X+U)) \, dU. 
  \end{align*}
Pour tout $T'\in\ago_0^\theta\cap\ov{\ago_{P'_0}^+}$ et tout sous-groupe parabolique semi-standard $\theta$-stable $R$ de $M_Q$, on note $T'_R$ la projection de $w T'$ sur $\ago_R^\theta$, où $w\in W(M_Q, A_0^\theta)$ est tel que $P'_0\subset R_w$ (cf. remarque \ref{rmq:combi} pour l'existence). En particulier, pour $T\in T_0+\ago_0^\theta\cap\ov{\ago_0^+}$, on a $T_{B_Q}\in\ago_0^\theta\cap\overline{\ago_{B_Q}^+}\subset\ago_0^\theta\cap\ov{\ago_{P'_0}^+}$ et on note $T_R=(T_{B_Q})_R\in\ago_R^\theta$. 

\begin{remarque}\label{rmq:TR=TRN}
    Comme $P'_0\subset R_w$ si et seulement si $B_Q\subset (R N_Q)_w$, où $R N_Q$ est un sous-groupe parabolique semi-standard $\theta$-stable de $G$, on a $T_R=T_{R N_Q}$ avec notre notation. 
\end{remarque}

On définit pour tout $T\in T_0+\ago_0^\theta\cap\ov{\ago_0^+}$, tout $1\leq i\leq n$ et tout $x\in M_Q^\theta(\AAA)$
\begin{align*}
    K_{\of_i,f}^{Q,T}(x)=\sum_{P'\in\fc^{M_Q}(P'^\theta_0,\theta)} (-1)^{\dim{(\ago_{P'}^{M_Q})^\theta}}    \sum_{\delta\in {P'}^\theta(F)  \back M_Q^\theta(F)} \hat\tau_{P'}(H_{P'}(\delta x)-T_{P'}) K_{P',\of_i,f}(\delta x). 
\end{align*}
En utilisant la bijection \eqref{eq:PcapM} et la remarque \ref{rmq:TR=TRN}, on peut également écrire  
\begin{align*}
    K_{\of_i,f}^{Q,T}(x)=\sum_{P\in\fc^Q(P_0^\theta, \theta)} (-1)^{\dim{(\ago_P^Q)^\theta}} \sum_{\delta\in P^\theta(F)  \back Q^\theta(F)} \hat\tau_{P}(H_{P}(\delta x)-T_{P}) K_{P\cap M_Q,\of_i,f}(\delta x).  
\end{align*}
Soit $[M_Q^\theta]^Q=M_Q^\theta(F)\back (M_Q^\theta(\AAA)\cap M_Q(\AAA)^1)$. On définit pour tout $\xi\in \ago^\ast_{Q^\theta,\CC}$ 
\begin{align*}
    J_{\of_i}^{Q,T}(\eta, \xi, f)=\int_{[M_Q^\theta]^Q} K_{\of_i,f}^{Q,T}(x) \eta(x) \exp(\langle \xi, H_{Q^\theta}(x) \rangle)  \, dx.
\end{align*}
C'est essentiellement un produit de distributions de la forme \eqref{eq:infi-JoGT}. Enfin, on définit 
\begin{align}\label{eq:def-infi-Jo}
    J_{\of}^{Q,T}(\eta, \xi, f)=\sum_{1\leq i\leq n} J_{\of_i}^{Q,T}(\eta, \xi, f). 
\end{align}
\end{paragr}

\begin{paragr}[Comportement en $T$.] ---
Soit $\fc(Q,\theta)$ l'ensemble de sous-groupes paraboliques $\theta$-stables de $G$ contenant $Q$. On définit 
  \begin{align*}
   \Gamma_Q^G(H, T)=\sum_{R\in\fc(Q,\theta)} (-1)^{\dim (\ago_R^G)^\theta} \tau_Q^R(H) \hat\tau_R(H-T),  \forall H, T\in \ago_0^\theta. 
   \end{align*}	
Pour tout $T$ fixé, la fonction $\Gamma_Q^G(\cdot, T)$ sur $(\ago_Q^G)^\theta$ est à support compact d'après \cite[lemmes 1.8.3 et 2.9.2]{labWal}. Pour $\lambda\in\ago_{Q,\CC}^{\theta,\ast}$, on pose 
  \begin{align*}
 c'_Q(\lambda, T)=\int_{(\ago_Q^G)^\theta} \Gamma_Q^G(H,T)  \exp(\langle \lambda, H\rangle) \, dH, \forall T\in\ago_Q^\theta. 
  \end{align*}	
En particulier, on note  
  \begin{align*}
 p_{Q,\chi}(T)=c'_Q(\chi+2\rho_{\ngo_Q\cap\sgo}-2\rho_{Q^\theta}^{G^\theta}, T), \forall T\in\ago_Q^\theta. 
  \end{align*}	
  
\begin{lemme}\label{lem:art2.2-li5.6}
Il existe un polynôme $c_{\lambda,Q,R}$ sur $(\ago_R^G)^\theta$ pour tout $R\in\fc(Q,\theta)$ tel que 
  \begin{align*}
 c'_Q(\lambda, T)=\sum_{R\in\fc(Q,\theta)} \exp(\langle\lambda, T_R^G\rangle) c_{\lambda,Q,R}(T_R^G). 
  \end{align*}	
  En particulier, $c'_Q(\lambda, T)$ est un polynôme-exponentielle en $T$. 
\end{lemme}

\begin{preuve}
Il se déduit de la preuve de \cite[lemmes 1.9.1 et 2.9.2]{labWal} (voir aussi \cite[proposition 5.6]{li1}). 
\end{preuve}

  \begin{remarque}
      Pour les types 1 et 2, on a $\exp(\langle \chi+2\rho_{\ngo_Q\cap\sgo}-2\rho_{Q^\theta}^{G^\theta}, H\rangle)=1$ pour tout $H\in(\ago_Q^G)^\theta$. Il s'ensuit que  $p_{Q,\chi}(T)$ est effectivement un polynôme en $T$. 
  \end{remarque}
  
On définit $\varphi_Q^\eta\in\Sc((\mgo_Q\cap\sgo)(\AAA))$ par 
  \begin{align*}
 \varphi_Q^\eta(X)=\int_{K^\theta} \int_{(\ngo_Q\cap\sgo)(\AAA)} \varphi(\Ad(k^{-1})(X+V)) \eta(k) \, dVdk, \forall X\in(\mgo_Q\cap\sgo)(\AAA). 
  \end{align*}

\begin{theoreme}\label{thm:expol-infi}
Soit $T, T'\in T_0+\ago_0^\theta\cap\overline{\ago_0^+}$. 
  \begin{enumerate}
  \item On a 
  \begin{align*}
  	 J_\of^{G,T}(\eta,\chi,\varphi) = \sum_{Q\in\fc(P'_0,\theta)} p_{Q,\chi}(T^G_Q-T'^G_Q) \exp(\langle \chi+2\rho_{\ngo_Q\cap\sgo}-2\rho_{Q^\theta}^{G^\theta}, T'^G_Q\rangle) \\ 
  	 \times J_\of^{Q,T'}(\eta,\chi+2\rho_{\ngo_Q\cap\sgo}-2\rho_{Q^\theta}^{G^\theta},\varphi_Q^\eta). 
  \end{align*}	 
  
  \item L'expression $J_\of^{Q,T}(\eta, \xi, f)$ est la restriction d'un polynôme-exponentielle en $T$ indépendant de $T_Q$. 
\end{enumerate}
\end{theoreme}

\begin{preuve}
C'est une généralisation immédiate de \cite[théorème 4.2]{Cha} et \cite[théorème 5.8]{li1} qui concernent seulement un caractère $\eta$ vérifiant $\eta^2=1$ mais leurs preuves valent encore pour tout caractère unitaire. 
\end{preuve}
\end{paragr}

\begin{paragr}\label{S:def-dist-infi}
Avec le théorème \ref{thm:expol-infi}, on définit $J_\of^Q(\eta,\xi,f)$ comme le terme constant de $J_\of^{Q,T}(\eta, \xi, f)$. En particulier, on obtient la distribution nilpotente $J_{\nilp}^Q(\eta,\xi,\cdot)$ sur $\Sc((\mgo_Q\cap\sgo)(\AAA))$. Il est également clair que l'on peut généraliser les constructions et les résultats de cette section au cas d'un produit de paires symétriques des trois types dans la proposition \ref{prop:descendant}. 
\end{paragr}

\bibliography{bibliographie}
\bibliographystyle{alpha}

\begin{flushleft}
Pierre-Henri Chaudouard \\
Université Paris Cité\\
CNRS \\
IMJ-PRG \\
Bâtiment Sophie Germain\\
8 place Aurélie Nemours\\
F-75013 PARIS \\
France\\
\medskip
Institut Universitaire de France\\
\medskip

email:\\
Pierre-Henri.Chaudouard@imj-prg.fr \\
\end{flushleft}

\begin{flushleft}
Huajie Li\\
Yau Mathematical Sciences Center \\
Tsinghua University \\
Beijing 100084 \\
China \\
\medskip
email:\\
lihuajie@mail.tsinghua.edu.cn \\
\end{flushleft}

\end{document}